\documentclass[reqno,11pt]{amsart}
\usepackage{amsfonts,amsmath,amssymb}

\usepackage{hyperref}

\usepackage[margin=1.35in]{geometry}
\numberwithin{equation}{section}

\def\sgn{{\,\mathrm{sgn}}}

\def\pa{{\partial}}

\def\low{\mathrm{low}}

\def\g{\gamma}

\def\e{\varepsilon}

\def\eps{\epsilon}

\def\o{\omega}
\def\bar{\overline}
\def\what{\widehat}

\def\varep{\varepsilon}
\def\veps{\varepsilon}

\def\D{{\mathcal D}}

\def\R{{\mathbb R}}

\def\f12{{\frac 1 2}}

\def\curl{{\mbox curl\,}}

\def\f{\widetilde{f}}

\theoremstyle{plain}

\newtheorem{theorem}{Theorem}[section]
\newtheorem{lemma}[theorem]{Lemma}
\newtheorem{proposition}[theorem]{Proposition}

\newtheorem{definition}[theorem]{Definition}
\newtheorem{remark}[theorem]{Remark}

\begin{document}

\title[Water Waves with surface tension]{Global regularity for 2d water waves \\ with surface tension}


\thanks{The first author was partially supported by a Packard Fellowship and NSF Grant DMS 1265818.
The second author was partially supported by a Simons Postdoctoral Fellowship and NSF Grant DMS 1265875.}

\author{Alexandru D. Ionescu}\address{Princeton University}\email{aionescu@math.princeton.edu}
\author{Fabio Pusateri}\address{Princeton University}\email{fabiop@math.princeton.edu}

\begin{abstract}
We consider the full irrotational water waves system with surface tension and no gravity in dimension two (the capillary waves system),
and prove global regularity and modified scattering
for suitably small and localized perturbations of a flat interface.
An important point of our analysis is to develop a sufficiently robust method, based on energy estimates and dispersive analysis,
which allows us to deal simultaneously with strong singularities arising from time resonances 
in the applications of the normal form method and with nonlinear scattering. 
As a result, we are able to consider a suitable class of perturbations with finite energy, but no other momentum conditions.

Part of our analysis relies on a new treatment of the Dirichlet-Neumann operator in dimension two which is of independent interest.
As a consequence, the results in this paper are self-contained.
%
%
%

\end{abstract}

\maketitle

\setcounter{tocdepth}{1}
\tableofcontents

\section{Introduction}


\subsection{Free boundary Euler equations and water waves}\label{secWW}
The evolution of an inviscid perfect fluid that occupies a domain $\Omega_t \subset \R^n$, for $n \geq 2$, at time $t \in \R$,
is described by the free boundary incompressible Euler equations.
If $v$ and $p$ denote respectively the velocity and the pressure of the fluid (with constant density equal to $1$)
at time $t$ and position $x \in \Omega_t$, these equations are:
\begin{equation}
\label{E}
\left\{
\begin{array}{ll}
v_t + v \cdot \nabla v = - \nabla p - g e_n  &   x \in \Omega_t
\\
\nabla \cdot v = 0    &   x \in \Omega_t
\\
v (x,0) = v_0 (x)     &   x \in \Omega_0 \, ,
\end{array}
\right.
\end{equation}
where $g$ is the gravitational constant.
The free surface $S_t := \partial \Omega_t$ moves with the normal component of the velocity according to the kinematic boundary condition:
\begin{subequations}
\label{BC}
 \begin{equation}
\label{BC1}
\partial_t + v \cdot \nabla  \,\, \mbox{is tangent to} \,\, \bigcup_t S_t \subset \R^{n+1} \, .
\end{equation}
In the presence of surface tension the pressure on the interface is given by
\begin{equation}
\label{BC2}
p (x,t) = \sigma \kappa(x,t)  \,\, , \,\,\, x \in S_t  \, ,
\end{equation}
where $\kappa$ is the mean-curvature of $S_t$ and $\sigma > 0$.
\end{subequations}
At liquid-air interfaces, the surface tension force results from the greater attraction
of water molecules to each other than to the molecules in the air. In the case of irrotational flows, i.e.
\begin{equation}
\label{irro}
\rm{\curl} v = 0 \, ,
\end{equation}
one can reduce \eqref{E}-\eqref{BC} to a system on the boundary.
Such a reduction can be performed identically regardless of the number of spatial dimensions,
but here we only focus on the two dimensional case -- which is the one we are interested in --
and moreover assume that $\Omega_t \subset \R^2$ is the region below the graph of a function $h : \R_x \times \R_t \rightarrow \R$,
that is $\Omega_t = \{ (x,y) \in \R^2 \, : y \leq h(x,t) \}$ and $S_t = \{ (x,y) : y = h(x,t) \}$.

Let us denote by $\Phi$ the velocity potential: $\nabla \Phi(x,y,t) = v (x,y,t)$, for $(x,y) \in \Omega_t$.
If $\phi(x,t) := \Phi (x, h(x,t),t)$ is the restriction of $\Phi$ to the boundary $S_t$,
the equations of motion reduce to the following system for the unknowns $h, \phi : \R_x \times \R_t \rightarrow \R$:
\begin{equation}
\label{WWE}
\left\{
\begin{array}{l}
\partial_t h = G(h) \phi
\\
\partial_t \phi = - g h + \sigma  \dfrac{\partial_x^2 h}{ (1+h_x^2)^{3/2} }
  - \dfrac{1}{2} {|\phi_x|}^2 + \dfrac{{\left( G(h)\phi + h_x \phi_x \right)}^2}{2(1+{|h_x|}^2)}
\end{array}
\right.
\end{equation}
with
\begin{equation}
\label{defG0}
G(h) := \sqrt{1+{|h_x|}^2} \mathcal{N}(h)
\end{equation}
where $\mathcal{N}(h)$ is the Dirichlet-Neumann\footnote{By slightly abusing notation we will refer to
$G(h)$ as the Dirichlet-Neumann map, as this causes no confusion.} map associated to the domain $\Omega_t$.
We refer to  \cite[chap. 11]{SulemBook} or \cite{CraSul} for the derivation of the  water waves equations \eqref{WWE}.
This system describes the evolution of an incompressible perfect fluid of infinite depth and infinite extent, with a free moving
(one-dimensional) surface, and a pressure boundary condition given by the Young-Laplace equation.
One generally refers to \eqref{WWE} as the gravity water waves system when $g>0$ and $\sigma=0$,
and as the capillary water waves system when $g=0$ and $\sigma>0$.

The system \eqref{E}-\eqref{BC} has been under very active investigation in recent years.
Without trying to be exhaustive, we mention the early work on
the wellposedness of the Cauchy problem in the irrotational case and with gravity by
Nalimov \cite{Nalimov}, Yosihara \cite{Yosi}, and Craig \cite{CraigLim};
the first works on the wellposedness for general data in Sobolev spaces (for irrotational gravity waves) by Wu \cite{Wu1,Wu2};
and subsequent work on the gravity problem by Christodoulou-Lindblad \cite{CL}, Lannes \cite{Lannes}, Lindblad \cite{Lindblad},
Coutand-Shkoller \cite{CS2}, Shatah-Zeng \cite{ShZ1,ShZ3}, and Alazard-Burq-Zuily \cite{ABZ2,ABZ3}.
Surface tension effects have been considered in the work of Beyer-Gunther \cite{BG}, Ambrose-Masmoudi \cite{AM},
Coutand-Shkoller \cite{CS2}, Shatah-Zeng \cite{ShZ1,ShZ3}, Christianson-Hur-Staffilani \cite{CHS}, and Alazard-Burq-Zuily \cite{ABZ1}.
Recently, some blow-up scenarios have also been investigated \cite{CCFGL,CCFGG,CSSplash,IFL}.

The question of long time regularity of solutions with irrotational, small and localized initial data was also addressed in a few papers,
starting with \cite{WuAG}, where Wu showed almost global existence for the gravity problem ($g>0$, $\sigma=0$) in two dimensions ($1$d interfaces).
Subsequently, Germain-Masmoudi-Shatah \cite{GMS2} and Wu \cite{Wu3DWW}
proved global existence of gravity waves in three dimensions ($2$d interfaces).
Global regularity in $3$d was also proved in the case of surface tension and no gravity ($g=0$, $\sigma>0$)
by Germain-Masmoudi-Shatah \cite{GMSC}.

Global regularity for the gravity water waves system in dimension $2$ (the harder case) has been proved by the authors in
\cite{IoPu2}\footnote{See also our earlier paper \cite{IoPu1}
for the analysis of a simplified model (a fractional cubic Schr\"odinger equation),
and \cite{IoPunote} for an alternative description of the asymptotic behavior of the solutions constructed in \cite{IoPu2}.} and, independently by Alazard-Delort \cite{ADa,ADb}.
More recently, a new proof
of Wu's $2$d almost global existence result was given by Hunter-Ifrim-Tataru \cite{HIT},
and then complemented to a proof of global regularity in \cite{IT}.

\subsection{The main results}\label{MainResult}

Our results in this paper concern the capillary water waves system
\begin{equation}\label{CPW}
\left\{
\begin{array}{l}
\partial_t h = G(h) \phi,
\\
\partial_t \phi = \dfrac{\partial_x^2 h}{ (1+h_x^2)^{3/2} }
  - \dfrac{1}{2} {|\phi_x|}^2 + \dfrac{{\left( G(h)\phi + h_x \phi_x \right)}^2}{2(1+{|h_x|}^2)} .
\end{array}
\right.
\end{equation}
This is the system \eqref{WWE} when gravity effects are neglected ($g=0$) and the surface tension coefficient $\sigma$ is,
without loss of generality, taken to be $1$.
The system admits the conserved Hamiltonian
\begin{equation}\label{CPWHam}
\mathcal{H}(h,\phi) := \frac{1}{2} \int_\R G(h)\phi \cdot \phi \, dx + \int_\R \frac{{(\partial_x h)}^2}{1 + \sqrt{1+h_x^2} } \, dx
  \approx {\big\| |\partial_x|^{1/2} \phi \big\|}_{L^2}^2 + {\big\| |\partial_x|h \big\|}_{L^2}^2 .
\end{equation}

To describe our results we first introduce some basic notation.
Let
\begin{equation}\label{normC}
\mathcal{C}_0:=\{f:\mathbb{R}\to\mathbb{C} \,, f\text{ continuous and }\lim_{|x|\to\infty}|f(x)|=0\},
  \qquad {\|f\|}_{\mathcal{C}_0} := {\|f\|}_{L^\infty}.
\end{equation}
For any $N\geq 0$ let $H^N$ denote the standard Sobolev space of index $N$.
More generally, if $N\geq 0$, $b\in[-1,N]$, and $f\in\mathcal{C}_0$ then we define
\begin{equation}\label{normH}
\begin{split}
& {\|f\|}_{\dot{H}^{N,b}} := \Big\{\sum_{k\in\mathbb{Z}}
  {\|P_kf\|}^2_{L^2}(2^{2Nk}+2^{2kb}) \Big\}^{1/2} \approx {\big\| (|\partial_x|^N + |\partial_x|^b)f \big\|}_{L^2},
\\
& {\|f\|}_{\dot{W}^{N,b}} := \sum_{k\in\mathbb{Z}} {\|P_kf\|}_{L^\infty} (2^{Nk}+2^{bk}),
\end{split}
\end{equation}
where $P_k$ denote standard Littlewood-Paley projection operators (see subsection \ref{notation} for precise definitions).
Notice that the norms $\dot{H}^{N,b}$ define natural spaces of distributions for $b<1/2$ in dimension $1$, but not for $b\geq 1/2$.
This is the reason for the assumption $f\in\mathcal{C}_0$ in the definition.
Our main result is the following:

\begin{theorem}[Global Regularity] \label{MainTheo}
Let
\begin{equation}\label{constants0}
N_0=N_I:=9,\quad N_1=N_S:=3,\quad N_2=N_\infty:=5,\qquad 0<10^4p_1\leq p_0\leq 10^{-10}.
\end{equation}
Assume that $(h_0,\phi_0)\in(\mathcal{C}_0\cap \dot{H}^{N_0+1,p_1+1/2})\times \dot{H}^{N_0+1/2,p_1}$ satisfies
\begin{equation}\label{h0p0}
{\|h_0\|}_{\dot{H}^{N_0+1,p_1+1/2}} + {\|\phi_0\|}_{\dot{H}^{N_0+1/2,p_1}}
  + {\|(x\partial_x)h_0\|}_{\dot{H}^{N_1+1,p_1+1/2}} + {\|(x\partial_x)\phi_0\|}_{\dot{H}^{N_1+1/2,p_1}} = \e_0 \leq \bar{\e_0},
\end{equation}
where $\bar{\e_0}$ is a sufficiently small constant.
Then, there is a unique global solution
\begin{align*}
(h,\phi)\in C\big([0,\infty) : (\mathcal{C}_0\cap \dot{H}^{N_0+1,p_1+1/2})\times \dot{H}^{N_0+1/2,p_1}\big)
\end{align*}
of the system \eqref{CPW}, with $(h(0),\phi(0))=(h_0,\phi_0)$.
In addition, with $S:=(3/2)t\partial_t+x\partial_x$, we have
\begin{align}
\label{mainconcl1}
\langle t\rangle^{-p_0} {\|f(t)\|}_{\dot{H}^{N_0,p_1-1/2}} + \langle t\rangle^{-4p_0} {\|Sf(t)\|}_{\dot{H}^{N_1,p_1-1/2}}+ \langle t\rangle^{1/2} {\|f(t)\|}_{\dot{W}^{N_2,-1/10}}\lesssim \e_0,
\end{align}
for any $t\in[0,\infty)$, where $\langle t\rangle:=1+t$ and $f\in\{|\partial_x|h,|\partial_x|^{1/2}\phi\}$.
\end{theorem}

\begin{remark}
\normalfont
Very recently, Ifrim-Tataru \cite{IT2} independently obtained a similar global regularity result for the same system, 
in the case of data satisfying one momentum condition on the Hamiltonian variables. 
Their proof of global regularity still relies on the basic modified scattering mechanism, 
as introduced in the context of water waves in \cite{ADa,ADb,IoPu2}, 
but is somewhat simpler in large part because of their use of the better adapted holomorphic coordinates 
(which avoid the Dirichlet-Neumann map and some of the complications associated with paradifferential calculus) 
instead of Eulerian coordinates.
\end{remark}

\begin{remark}[Modified scattering]
\normalfont
The solution exhibits nonlinear scattering behavior as $t \to \infty$. 
A precise statement can be found in Theorem \ref{ScaThm} in section \ref{Sca1}, 
in which we provide two different descriptions of the asymptotic behavior of the solution. 
\end{remark}

\begin{remark}[Low frequencies and momentum conditions]\label{LowFreq}
\normalfont
The proof of the main theorem becomes easier if one makes the stronger low-frequency assumption
\begin{equation}\label{strongLow}
\|h_0\|_{\dot{H}^{N_0+1,1/2-}}+\|(x\partial_x)h_0\|_{\dot{H}^{N_1+1,1/2-}}
  + \|\phi_0\|_{\dot{H}^{N_0+1/2,0-}}+\|(x\partial_x)\phi_0\|_{\dot{H}^{N_1+1/2,0-}} \ll 1,
\end{equation}
which is the same condition as \eqref{h0p0}, but taking $p_1<0$.
Such low frequency norms are propagated by the flow, due to a suitable null structure at low frequencies of the nonlinearity.
However, the finiteness of the norm \eqref{strongLow} requires an unwanted momentum condition
on the natural Hamiltonian variables $(|\partial_x| h,|\partial_x|^{1/2}\phi)$, compare with \eqref{CPWHam}.

The choice of norm in \eqref{h0p0} accomplishes our main goals. 
On one hand it is strong enough to allow us to control the singular terms arising from the resonances of the normal form transformation
and deal with the ``division problem'', see \ref{introen} below.
On the other hand, it is weak enough to avoid momentum assumptions on the natural energy variables\footnote{Similar assumptions at low frequencies, designed
to avoid momentum conditions on the energy variables, were used by Germain-Masmoudi-Shatah \cite{GMSC} in their work on the capillary system in three dimensions.}
$|\partial_x| h$ and $|\partial_x|^{1/2}\phi$.
\end{remark}


As a direct byproduct of our energy estimates in sections \ref{secEE} and \ref{secEElow}, we also obtain the following:

\begin{theorem}[Long-time existence in Sobolev spaces]\label{e-2Theo}
Assume that
\begin{align*}
(h_0,\phi_0)\in(\mathcal{C}_0\cap \dot{H}^{N_0+1,p_1+1/2})\times \dot{H}^{N_0+1/2,p_1}
\end{align*}
satisfies
\begin{equation}\label{h0p01}
\|h_0\|_{\dot{H}^{N_0+1,p_1+1/2}}+\|\phi_0\|_{\dot{H}^{N_0+1/2,p_1}} = \e_0 \leq 1.
\end{equation}
Then there is a unique solution $$(h,\phi)
\in C\big([0,T_{\varepsilon_0}]:(\mathcal{C}_0\cap \dot{H}^{N_0+1,p_1+1/2})\times \dot{H}^{N_0+1/2,p_1}\big)$$
of the system \eqref{CPW}, with $(h(0),\phi(0))=(h_0,\phi_0)$ and $T_{\varepsilon_0}\gtrsim\varepsilon_0^{-2}$.
Moreover,
\begin{equation}\label{mainconcl11}
\|h(t)\|_{\dot{H}^{N_0+1,p_1+1/2}}+\|\phi(t)\|_{\dot{H}^{N_0+1/2,p_1}}\lesssim \varepsilon_0
\end{equation}
for any $t\in[0,T_{\varepsilon_0}]$.
\end{theorem}

A similar result is also proved, independently, by Ifrim-Tataru in \cite{IT2}, 
with weaker assumptions on the initial data. 
For gravity water waves such a result was obtained in\footnote{It is also a corollary
of the results in \cite{WuAG,IoPu2,ADb,HIT}.} \cite{WuNLS} in $2$d,
and in \cite{Totz} in $3$d, as a key step to proving the modulation approximation in infinite depth.
By analogy, Theorem \ref{e-2Theo} should be considered a step towards the rigorous justification of
approximate models and scaling limits for the water waves system with surface tension.
We refer the reader to \cite{CSS,SW1}, the book \cite{LannesBook}, and references therein, for works dealing with the
long-time existence of non-localized solutions of the water waves system and their modulation and scaling regimes.

\subsection{Main ideas of the proof}\label{secideas}

The system \eqref{CPW} is a time reversible quasilinear system.
In order to prove global regularity for solutions of the Cauchy problem for this type of equations, one needs to accomplish two main tasks:

\setlength{\leftmargini}{1.8em}
\begin{itemize}
  \item[1)] Propagate control of high frequencies (high order Sobolev norms);

  \item[2)] Prove pointwise decay of the solution over time.
\end{itemize}

In this paper we use a combination of improved energy estimates and asymptotic analysis to achieve these two goals.
This is a natural continuation of our work on the gravity water waves system \cite{IoPu2}.
However, here we adopt a more robust framework for some of the arguments and work entirely in Eulerian coordinates.\footnote{Besides
being the natural coordinates associated to the Hamiltonian formulation of the water waves problem,
Eulerian coordinates are in general more flexible because the analysis can be generalized
to higher dimensions.} More precisely, we perform a careful paralinearization of the Dirichlet-Neumann operator in two dimensions,
in the spirit of \cite{ABZ1,ADb}, and obtain new bounds consistent with the limited low-frequency structure we assume on the interface $h$. Then we set up our equations and construct high order energy functionals and low frequency energy functionals which can be controlled for long times.
Finally, we use the Fourier transform method to obtain sharp time decay rates for our solutions,
and prove modified scattering.
We elaborate on these main aspects of our proof below.

\subsection{Paralinearization and the Dirichlet-Neumann operator}\label{intropar}

One of the main difficulties in the analysis of the water waves system \eqref{WWE} comes from
the quasilinear nature of the equations, and their non-locality, due to the presence of the Dirichlet-Neumann operator $G(h)\phi$,
see \eqref{defG0}.
In a series of papers \cite{ABZ1,ABZ2,ABZ3} Alazard-Burq-Zuily proposed a systematic approach to these issues in the context
of the local-in-time Cauchy problem, based on para-differential calculus.
See also the earlier works of Alazard-M\'etivier \cite{AlMet1} and Lannes \cite{Lannes}.
This approach was then extended and adapted to study the problem of global regularity for gravity waves by Alazard-Delort in \cite{ADb}.

Our first step towards the proof of Theorem \ref{MainTheo} is a paralinearization of the system \eqref{CPW}
inspired by the works cited above. For our problem, however, we need bounds that depend only on $\||\partial_x|^{1/2+p_1}h\|_{L^2}$, not on $\|h\|_{L^2}$.
More precisely, let
\begin{align}
 \label{par0}
(h,\phi) \in C(\mathcal{C}_0 \cap \dot{H}^{N_0+1,p_1+1/2}(\R)\times H^{N_0+1/2,p_1}(\R))
\end{align}
be a real-valued solution of \eqref{CPW} on some time interval. Define
\begin{equation}
\label{par1}
\begin{split}
& B := \frac{G(h)\phi + h_x\phi_x}{(1+h_x^2)} , \qquad V := \phi_x-Bh_x,
\qquad \omega := \phi - T_B P_{\geq 1} h.
\end{split}
\end{equation}
Here, for any $a, b \in L^2(\R)$, we have denoted by $T_a b$ the paradifferential operator
\begin{equation}
\label{Tab}
\begin{split}
\mathcal{F}(T_a b)(\xi) & := \frac{1}{2\pi} \int_{\R} \what{a}(\xi-\eta) \what{b}(\eta) \chi_0(\xi-\eta,\eta)\,d\eta,
\\
\chi_0(x,y) & := \sum_{k \in \mathbb{Z}} \varphi_{\leq k-10}(x) \varphi_k(y),
\end{split}
\end{equation}
which restricts the product of $a$ and $b$ to the frequencies of $b$ much larger than those of $a$. The functions $(V,B)$ represent the restriction of the velocity field $v$ to the boundary of $\Omega_t$, and
the function $\omega$ is a variant of the so-called ``good-unknown'' of Alinhac.

In Appendix \ref{secDN}, we prove the following formula for the Dirichlet-Neumann operator:
\begin{equation}
\label{Gh}
G(h) \phi = |\partial_x|\omega - \partial_x T_V h + G_2 + G_{\geq 3},
\end{equation}
where $G_2$ is an explicit semilinear quadratic term, and $G_{\geq 3}$ denotes cubic and higher order terms.
This formula was already derived and used by Alazard-Delort \cite{ADb}; the new aspect here is that 
the remainder
$G_{\geq 3}$  satisfies better trilinear 
bounds of $L^2$, weighted $L^2$, and $L^\infty$-type, under the sole assumptions \eqref{par0}.
In particular we only need to assume that $|\partial_x|^{1/2+p_1}h$ and $|\partial_x|^{p_1} \phi$ are in  $L^2$,
for some $p_1 > 0$.

In section \ref{Equations}, using \eqref{Gh}, we diagonalize and symmetrize the system \eqref{CPW}
reducing it to a single scalar equation for one complex unknown $u$ of the form
\begin{align}
\label{par5}
u \approx |\partial_x| h - i |\partial_x|^{1/2} \omega + \text{higher order corrections} ,
\end{align}
see the formulas \eqref{symmsymbols}-\eqref{symm5}.
The capillary system \eqref{CPW} then takes the form
\begin{align}
\label{WWCpar}
\partial_t u - i \Lambda(\partial_x) u - i \Sigma(u) = - T_V \partial_x u + \mathcal{N}_u + \mathcal{R}_u,
\end{align}
where: $\Lambda(\xi) := |\xi|^{3/2}$ is the dispersion relation for linear waves of frequency $\xi$;
$\Sigma$ can be thought of as a self-adjoint differential operator of order $3/2$ which is cubic in $u$;
$\mathcal{N}_u$ are semilinear quadratic terms;
and $\mathcal{R}_u$ contains (semilinear) cubic and higher order terms in $u$.
See Proposition \ref{prosymm} for details.
Equation \eqref{WWCpar} is our starting point in establishing energy estimates for $u$,
hence for $(|\partial_x|h, |\partial_x|^{1/2}\phi)$.

\subsection{Energy estimates and the quartic energy inequality}\label{introen}
In order to construct high order energy functionals controlling Sobolev norms of our solution,
we first apply the natural differentiation operator associated to the equation \eqref{WWCpar}.
More precisely, we look at $W_k := \D^k u$, with $\D := |\partial_x|^{3/2} + \Sigma$, and show that
\begin{align}
\label{WWCparW}
\begin{split}
& \partial_t W_k - i \Lambda(\partial_x) W_k - i \Sigma(W_k) = - T_V \partial_x W_k + \mathcal{N}_{W_k} + \mathcal{R}_{W_k} .
\end{split}
\end{align}
Here $W_k \approx |\partial_x|^{3k/2} u$,
$\mathcal{N}_{W_k}$ denotes semilinear quadratic terms, and $\mathcal{R}_{W_k}$ are semilinear cubic and higher order terms in $W_k$.
If one looks at the basic functional $E(t) = {\| W_{N_0}(t) \|}_{L^2}^2$ associated to \eqref{WWCparW},
it is easy to verify that, as long as solutions are of size $\e$, such energy functional is controlled for $O(\e^{-1})$ times.

In order to go past this local existence time, one needs to rely on the dispersive properties of solutions.
One of the main difficulties in dealing with a one dimensional problem such as \eqref{WWCparW} is
the slow time decay, which is $t^{-1/2}$ for linear solutions.
A classical idea used to overcome the difficulties associated to weak dispersion is the use of normal forms \cite{shatahKGE},
which can sometimes be used to eliminate the slow decaying quadratic terms from the nonlinearity.

In the context of water waves in 2D (1D interfaces), the starting point is a {\it{quartic energy inequality}}, which is an estimate on the energy increment of the form
\begin{equation}\label{qin}
|\partial_t \mathcal{E}(t)| \lesssim \text{quartic semilinear terms},
\end{equation}
where $\mathcal{E}$ is a suitable energy functional. An inequality of this form was first proved by Wu \cite{WuAG}
for the gravity water wave model, and led to an almost-global existence result. 
All the later work on long term 2D water wave models, such as  \cite{IoPu2, IoPu3, ADa, ADb, HIT, IT, IT2} and this paper, 
relies on proving an inequality of this type as a first step.

For \eqref{WWCparW} a normal form transformation is formally available since the only time resonances, i.e. solutions of
\begin{align}
\label{phase1}
\Lambda(\xi) \pm \Lambda(\xi-\eta) \pm \Lambda(\eta) = 0 ,
\end{align}
occur when one of the three interacting frequencies $(\xi,\xi-\eta,\eta)$ is zero.
However, the superlinear dispersion relation $|\xi|^{3/2}$ makes these resonances very strong.
For example, in the case $|\xi| \approx |\eta| \approx 1 \gg |\xi-\eta|$, we see that
\begin{align}
\label{phase2}
| \Lambda(\xi) \pm \Lambda(\xi-\eta) - \Lambda(\eta) | \approx |\xi-\eta| .
\end{align}
Because of \eqref{phase1}, a standard normal transformation which eliminates the quadratic terms in \eqref{WWCparW}
will introduce a low frequency singularity. This is the so-called ``division problem''.
For comparison, in the gravity water waves case the dispersion relation is
$\Lambda(\xi)=|\xi|^{1/2}$ and one has the less singular behavior $|\Lambda(\xi) \pm \Lambda(\xi-\eta) - \Lambda(\eta) | \approx |\xi-\eta|^{1/2}$
in the case $|\xi| \approx |\eta| \approx 1 \gg |\xi-\eta|$.

The implementation of the method of normal forms is delicate in quasilinear problems,
due to the potential loss of derivatives.
Nevertheless this has been done in some cases, for example by using carefully constructed
nonlinear changes of variables as in Wu \cite{WuAG}, or via the ``iterated energy method'' of Germain-Masmoudi \cite{GM},
or the ``paradifferential normal form method'' of Alazard-Delort \cite{ADb},
or the ``modified energy method'' of Hunter-Ifrim-Tataru \cite{HIT}.

The common goal of these methods is to prove a quartic energy inequality like \eqref{qin}. The main ingredient for such an inequality to hold is the absence of time-resonant bilinear interactions (after symmetrization), and the methods described above are largely interchangeable as long as this ingredient is present.{\footnote{See also \cite{Delo,HITW} for earlier constructions proving quartic energy inequalities like \eqref{qin} in simpler models.}} In \cite{IoPu2} we proved a quartic energy inequality, suitable for global-in-time analysis, by adapting the change of variable of Wu \cite{WuAG}.
Here, as a starting point of our analysis, we adapt the more elegant and robust approach of Alazard--Delort \cite{ADb} and Hunter-Ifrim-Tataru \cite{HIT}.

In order to deal simultaneously with the issue of slow decay and low-frequency singularities introduced by normal forms,
we construct our modified energy functionals in the Fourier space,
 in a way that is similar to the I-method of Colliander--Keel--Staffilani--Takaoka--Tao \cite{CKSTT1,CKSTT2}. See also \cite{DIP,IoPu3}.
The use of Fourier analysis gives us a lot of flexibility in defining energy functionals and isolating the most singular contributions,
which can be expressed  as multilinear paraproducts with singular multipliers.

A similar construction, involving quadratic energy functionals and higher order corrections, was performed earlier in \cite{HIT} in the case of gravity water waves. 
In the gravity case, however, the problem is simpler because there are no singularities in the resulting quartic integrals, and the entire construction in \cite{HIT} was performed in the physical space. 
In our case, there are strong singularities (of the form $(\text{low frequency})^{-1/2}$)
in the quartic integrals resulting from the application of the normal form, see the longer discussion below.
To deal with these singularities it is important to construct energy functionals in the Fourier space and make careful assumptions on the low frequency structure of solutions.

\subsubsection{Low frequency singularities} We begin our analysis by looking the basic energy, expressed in time-frequency space,
\begin{align}
 \label{E2N}
E^{(2)}(t) = \frac{1}{2\pi} \int_{\R} {| \what{W_{N_0}}(\xi,t)|}^2 \, d\xi + \frac{1}{2\pi} \int_{\R} {|\what{S W_{N_1}}(\xi,t)|}^2 \, d\xi ,
\end{align}
for $N_0$ and $N_1$ as in \eqref{constants0}. Assuming that $E^{(2)}(0)$ is of size $\e_0^2$, our maim aim is to control $E^{(2)}(t)$ for all times, 
possibly allowing a small polynomial growth, under the a priori assumption that solutions are $\e_0 t^{-1/2}$ small in a suitable $L^\infty_x$ type space.
As we will see below, this is not possible without stronger information on the low-frequencies behavior of the solution.
Calculating the evolution of \eqref{E2N} by using the equations \eqref{WWCparW},
and performing appropriate symmetrizations to avoid losses of derivatives, we obtain
\begin{align}
\label{d_tE2}
\partial_t E^{(2)}(t) = \mbox{semilinear cubic terms}.
\end{align}
We then define a cubic energy functional $E^{(3)}$ which is a sum of cubic terms of the form
\begin{align}
 \label{E3N}
\begin{split}
\int_{\R\times\R} \what{F}(\xi,t) \what{G}(\eta,t) \what{H}(\xi-\eta,t) m(\xi,\eta)\, d\xi d\eta
\end{split}
\end{align}
where $F,G,H$ can be any of the functions $W_{N_0}, SW_{N_1}, u, Su$ or their complex conjugates,
and the symbols $m$ are obtained by diving the symbols of the cubic expressions in \eqref{d_tE2}
by the appropriate resonant phase function \eqref{phase1}.
These cubic energy functionals are a perturbation of \eqref{E2N} on each fixed time slice.
Moreover, by construction
\begin{align}
\label{d_tE2+E3}
\partial_t \big( E^{(2)} + E^{(3)} \big)(t) = \mbox{singular quartic semilinear terms} .
\end{align}
The ``division problem'' mentioned above manifests itself in the fact that
some of the symbols of the above quartic expressions have singularities of the form $(\text{low frequency})^{-1/2}$.
This is ultimately due to the lack of symmetries in the equation for $Su$ (or $SW_{N_1}$).

A typical example of a singular quartic term can be schematically written (in real space, for simplicity of exposition) as
\begin{align}
 \label{quarticsing}
\int_\R  Su  \cdot u \cdot |\partial_x|^{1/2} u \cdot |\partial_x|^{-1/2} S u \, dx. 
\end{align}
The difficulty in estimating such a term is the following:
in order to control the growth of the energy, one would need to bound \eqref{quarticsing} by $\varepsilon_0 t^{-1} E^{(2)}(t)$.
A sharp a priori decay assumption gives us that ${\| u \cdot |\partial_x|^{1/2} u \|}_{L^\infty} \lesssim \e_0^2 t^{-1}$, 
and the energy controls ${\| Su \|}_{L^2} \lesssim \sqrt{E^{(2)}}$,
but we have no control over the $L^2$ norm of $|\partial_x|^{-1/2} S u$. This difficulty is not present in the gravity water wave system, due to the less singular denominators. 

To deal with these singularities 
we need to need to make our suitable low frequency assumptions \eqref{h0p0} on the solutions.
Using these low frequency assumptions, and anticipating a sharp decay rate of $t^{-1/2}$ for our solution,
we then control the energy functional $E^{(2)}(t)$ for all times, allowing a slow growth of $t^{2p_0}$ 
(in fact, due to weaker symmetries, we need a faster growth rate on the weighted energies).
This is done in sections \ref{secEE} and \ref{secweighted}, for the Sobolev and the weighted Sobolev norm.

The additional assumptions made in order to close the energy estimates above
are then recovered by low frequency energy estimates in sections \ref{secEElow} and \ref{secweightedlow}.
By following a similar argument to the one above,
we control for long times a quadratic energy of the form
\begin{align}
\label{Elow}
E^{(2)}_{\low} := \int_{\R} \big( |\what{u}(\xi,t)|^2 + |\what{Su}(\xi,t)|^2 \big) \cdot |\xi|^{-1}
  \mathcal{P}_t (\xi) \,d\xi,
\end{align}
where $\mathcal{P}_t : \xi \in \R \rightarrow [0,1]$ is a smooth function that vanishes if $|\xi| \geq 20$,
equals $1$ if $(1+t)^{-2} \leq |\xi| \leq 1$, and is equal to $[(1+t)^2 |\xi|]^{2p_1}$, if $2|\xi| \leq (1+t)^{-2}$.
Notice that $E^{(2)}_{\low}$ has the following properties:

\setlength{\leftmargini}{1.8em}
\begin{itemize}
  \item[(1)]  it is consistent with our initial assumptions \eqref{h0p0}; 

  \item[(2)] for frequencies that are not too small with respect to time, 
  it controls $|\partial_x|^{-1/2} Su$, which is exactly what would be needed to bound the quartic singular term \eqref{quarticsing} above;

  \item[(3)] for very small frequencies it controls the $L^2$ norm of $|\partial_x|^{-1/2+p_1} Su$ with a bound that improves as $t \rightarrow \infty$.
\end{itemize}

To control the growth of $E^{(2)}_{\low}(t)$ we follow a similar strategy to the one above
and construct a suitable cubic correction $E^{(3)}_{\low}$ such that, schematically, we obtain
\begin{align*}
\partial_t \big( E^{(2)}_{\low} + E^{(3)}_{\low} \big)(t) = \mbox{special cubic semilinear terms}
  + \mbox{singular quartic semilinear terms} .
\end{align*}
The cubic terms in the above right-hand side are special because they are supported on small time-dependent sets.
The key to control these low frequency energy functionals is the null structure for low frequency outputs in the water waves system \eqref{CPW}.
In essence, here we exploit the fact that the nonlinear part of the system
has a better low frequency behavior than the linear evolution,
so that, without imposing moment conditions on the initial data, we can still recover strong enough low frequency information
for the nonlinear evolution.


\subsubsection{Compatible vector-field structures}\label{introvf}
As in other quasilinear problems, to prove global regularity we propagate
control not only of high Sobolev norms but also of other $L^2$ norms defined by vector-fields \cite{K1}. These norms are helpful both in proving decay and in controlling remainders from the stationary phase and integration by parts
arguments in section \ref{secdecay} and \ref{secprmainlem}. In our case, a natural vector-field to propagate is the scaling vector-field $S=(3/2)t\partial_t+x\partial_x$
which (essentially) commutes with the linear part of the equation.

In our case, propagating $L^2$ control of vector-fields carrying weights is challenging.
The reason for this is well-known: as in the semilinear case, the success of the I-method
ultimately depends on exploiting certain symmetries of the equation, which are related to its Hamiltonian structure.
Once weighted vector-fields, such as $S$, are applied to the equation, these symmetries are weakened.
Moreover, every weighted vector-field requires its own modified energy functional, and its own set of cubic corrections.
At the very least, this increases considerably the amount of work needed to prove weighted energy estimates.

In this paper we are able to use energy estimates to propagate control of a specific {\it compatible vector-field structure},
namely the vector-fields
\begin{equation}\label{vfstr}
\partial_x^{m_0},\qquad S\partial_x^{m_1},\qquad m_0\in\{0,\ldots,N_0\},\qquad m_1 \in\{0,\ldots,N_1\}.
\end{equation}
We also propagate control of a suitable low-frequency structure described by an energy functional as in \eqref{Elow}.

The compatible vector-field structure \eqref{vfstr},
using at most one weighted vector-field $S$ and many vector-fields $\partial_x$, was introduced, in the setting of water waves,
by the authors in \cite{IoPu2}.
It was then used, and played a critical role, in all the later papers on the subject such as \cite{HIT,IT,IoPu3,IT2}. 
The point of this choice is that:

\setlength{\leftmargini}{1.8em}
\begin{itemize}
  \item[(i)] The compatible vector-field structure \eqref{vfstr} can be propagated in time,
  with small $t^\e$ loss and reasonably manageable computations;

  \item[(ii)] It is still strong enough to provide almost optimal $t^{-1/2}$ dispersive decay,
  since we work in dimension one (see Lemmas \ref{dispersive} and \ref{interpolation}).
\end{itemize}

\subsection{Decay and modified scattering}\label{introdecay}
Having established the $L^2$ bounds described above,
we eventually prove a pointwise sharp decay rate of $t^{-1/2}$ for our solutions.
This is done in sections \ref{secdecay} and \ref{secprmainlem}
where we follow a similar strategy as in our previous works \cite{IoPu1,IoPu2,IoPu3},
and study the nonlinear oscillations in the spirit of the "method of space-time resonances", as in \cite{GMS2,GMSC,GNT1,IoPu1}.

Our starting point is again the main paralinearized equation \eqref{WWCpar}.
As a first step we perform a standard normal form transformation, i.e. a bilinear change of variables $v = u + B(u,u)$, to remove the slowly decaying quadratic terms from the equation, see \eqref{defv}-\eqref{nf}.
In this part of the argument we do not need to pay attention to losses of derivatives, but only keep track of the singularities introduced by the normal forms.
It then suffices to prove decay for the new unknown $v$ using the equation containing only cubic (and higher order) terms.
To do this we write Duhamel's formula in terms of the linear profile $f = e^{-it\Lambda} v$ and study the oscillations in time and space of the resulting integral.
In particular we need to deal with trilinear expressions of the form
\begin{align}
\label{Duhamelintro}
\begin{split}
\int_{t_1}^{t_2} \int_{\R\times\R} &
  e^{it(-|\xi|^{3/2} \pm|\xi-\eta|^{3/2} \pm|\eta-\sigma|^{3/2} \pm|\sigma|^{3/2})}
\\
& \times c^{\pm\pm\pm}(\xi,\eta,\sigma) \what{f^{\pm}}(\xi-\eta,s) \what{f^{\pm}}(\eta-\sigma,s) \what{f^{\pm}}(\sigma,s)\,d\eta d\sigma \,ds,
\end{split}
\end{align}
where $f^+ = f$, $f^- = \bar{f}$.
We remark once again that the symbols in the above expressions have singularities of the form $(\mbox{low frequency})^{-1/2}$, see \eqref{cbounds}.

We proceed by splitting the integrals \eqref{Duhamelintro} into several types of interactions, depending on the size of the 
frequencies of the inputs relative to each other and to time, and estimate all the different contributions in several 
Lemmas in sections \ref{prooftech1} and \ref{prooftech2}.
The main contribution to the integrals \eqref{Duhamelintro} comes from space-time resonances, which occur when the size of the frequencies of the three functions is comparable to the 
size of the output frequency $|\xi|$, and the three inputs are $f,f$ and $\bar{f}$.

In this case, a stationary phase analysis argument reveals that a correction to the asymptotic behavior
is needed, similarly to the case of gravity waves \cite{IoPu2}.
We refer the reader to \cite{HN,DelortKG1d,KP,IoPu2,BosonStar} and references therein,
for more works related to modified scattering in dispersive equations.
We take this into account by defining a proper norm, the $Z$-norm in \eqref{Znorm}, which essentially measures $\what{f}$ in $L^\infty$.
Using also the carefully chosen $L^2$-norms \eqref{h0p0}, we are able to control uniformly, over time and frequencies, the $Z$-norm.
This gives us the desired pointwise decay at the sharp $t^{-1/2}$ rate, as well as modified scattering.

Some of the proofs in section \ref{secprmainlem} are similar to parts of the arguments in our previous works \cite{IoPu1,IoPu2,IoPu3}.
However, the decay analysis here is more complicated, once again, because of the low frequency singularities introduced by the strong quadratic time resonances.
This is especially evident in this part of the argument where non $L^2$-based norms need to be estimated, and
meaningful symmetrizations cannot be performed.
 In particular, we need to deal with singularities in several places:
to prove bounds on the normal form that recasts the quadratic nonlinearity into a cubic one (subsection \ref{secnf});
to control all cubic terms once the main asymptotic contribution is factored out (subsection \ref{prooftech1});
to estimate the quartic terms arising from the renormalization of the cubic equation needed to correct the asymptotic behavior (subsection \ref{prooftech3}).
Moreover, the superlinear dispersion relation creates additional cubic resonant interactions (subsection \ref{prooftech2})
which are not present in the case of gravity water waves.

\subsection{Organization} The rest of the paper is organized as follows: in section \ref{prelim} we summarize the main definitions and notation in the paper and state the main bootstrap proposition, which is Proposition \ref{MainProp}.

In sections \ref{Equations}--\ref{secweightedlow} we prove the main improved energy estimate.
The key components of the proof are Proposition \ref{prosymm} (derivation of the main quasilinear scalar equation),
Proposition \ref{proEE1} (improved energy estimate on the high Sobolev norm), Proposition \ref{proEElow}
(improved energy estimate on the low frequencies), Proposition \ref{proEEZ}
(improved weighted energy estimate on the high frequencies), and Proposition \ref{proEEZlow} (improved weighted energy estimate on the low frequencies).
The proofs in these sections use also the material presented in the appendices, 
in particular the paralinearization of the Dirichlet--Neumann operator in Proposition \ref{ra102}.

In sections \ref{secdecay}--\ref{secprmainlem} we prove the main improved decay estimate. 
In these sections we work with the Eulerian variables. The key components of the proof are the normal form transformation in \eqref{defv}, 
the construction of the profile in \eqref{prof} and its renormalization in \eqref{nf50}, and the improved control of $Z$ norm in Lemma \ref{mainlem}.

In section \ref{Sca1} we discuss the asymptotic behavior of nonlinear solutions, and provide two descriptions of the modified scattering,
one in the Fourier space and one in the physical space.

\section{Preliminaries}\label{prelim}

\subsection{Notation and basic lemmas}\label{notation}
In this subsection we summarize some of our main notation and recall several basic formulas and estimates.
We fix an even smooth function $\varphi: \R\to[0,1]$ supported in $[-8/5,8/5]$ and equal to $1$ in $[-5/4,5/4]$,
and define, for any $k\in\mathbb{Z}$,
\begin{equation*}
\varphi_k(x) := \varphi(x/2^k) - \varphi(x/2^{k-1}) , \qquad \varphi_{\leq k}(x):=\varphi(x/2^k),
  \qquad\varphi_{\geq k}(x) := 1-\varphi(x/2^{k-1}).
\end{equation*}
For $k\in\mathbb{Z}$ we denote by $P_k$, $P_{\leq k}$, and $P_{\geq k}$ the operators defined by the Fourier multipliers $\varphi_k$,
$\varphi_{\leq k}$, and $\varphi_{\geq k}$ respectively.
Moreover, let
\begin{align}
 \label{P'_k}
P'_k := P_{k-1} + P_k + P_{k+1} \qquad \mbox{and} \qquad \varphi'_k := \varphi_{k-1} + \varphi_k + \varphi_{k+1}.
\end{align}

Given $s\geq 0$ let $H^s$ denote the usual space of Sobolev functions on $\mathbb{R}$. Recall the space $\mathcal{C}_0$ defined in \eqref{normC}.
We use 3 other main norms: assume $N\geq 0$, $b\in[-1,N]$, and $f\in \mathcal{C}_0$ then
\begin{equation}\label{norms0}
\begin{split}
&\|f\|_{\dot{H}^{N,b}}:=\big\{\sum_{k\in\mathbb{Z}}\|P_kf\|^2_{L^2}(2^{2Nk}+2^{2kb})\big\}^{1/2},\\
&\|f\|_{\dot{W}^{N,b}}:=\sum_{k\in\mathbb{Z}}\|P_kf\|_{L^\infty}(2^{Nk}+2^{bk}),\\
&\|f\|_{\widetilde{W}^{N}}:=\|f\|_{L^\infty}+\sum_{k\geq 0}2^{Nk}\|P_kf\|_{L^\infty}.
\end{split}
\end{equation}
Notice that $\|f\|_{\dot{H}^{N,0}}\approx \|f\|_{H^N}$ and $\|f\|_{\widetilde{W}^{N}}\lesssim \|f\|_{\dot{W}^{N,0}}$. The spaces $\widetilde{W}^N$ are often used in connection with Lemma \ref{algebra} to prove iterative bounds on products of functions.

\subsubsection{Multipliers and associated operators} We will often work with multipliers $m:\mathbb{R}^2\to\mathbb{C}$ or $m:\mathbb{R}^3\to\mathbb{C}$, and operators defined by such multipliers. We define the class of symbols
\begin{equation}
\label{Sinfty}
S^\infty := \{m: \R^d \to \mathbb{C} : \,m \text{ continuous and } {\| m \|}_{S^\infty} := {\|\mathcal{F}^{-1}(m)\|}_{L^1} < \infty \}.
\end{equation}
Given a suitable symbol $m$ we define the associated bilinear operator $M$ by
\begin{equation}
\label{Moutm}
\mathcal{F}\big[M(f,g)\big](\xi)=\frac{1}{2\pi}\int_{\mathbb{R}}m(\xi,\eta)\widehat{f}(\xi-\eta)\widehat{g}(\eta)\,d\eta ,
\end{equation}
We often use the identity
\begin{align}
\label{lemcomm}
 SM(f,g) = M(Sf,g)+M(f,Sg)+\widetilde{M}(f,g)
\end{align}
where $S=(3/2)t\partial_t+x\partial_x$, $f,g$ are suitable functions defined on $I\times\mathbb{R}$,
and the symbol of the bilinear operator $\widetilde{M}$ is given by
\begin{align}
\label{lemcomm2}
\widetilde{m}(\xi,\eta) = -(\xi\partial_\xi +\eta \partial_\eta ) m(\xi,\eta).
\end{align}
This follows by direct calculations and integration by parts.

Lemma \ref{touse} below summarizes some properties of symbols and associated operators (see \cite[Lemma 5.2]{IoPu2} for the proof).

\begin{lemma}\label{touse}

\setlength{\leftmargini}{1.8em}
\begin{itemize}
 
  \item[(i)] We have $S^\infty\hookrightarrow L^\infty(\mathbb{R}\times\mathbb{R})$. If $m,m'\in S^\infty$ then $m\cdot m'\in S^\infty$ and
\begin{equation}\label{al8}
\|m\cdot m'\|_{S^\infty}\lesssim \|m\|_{S^\infty}\|m'\|_{S^\infty}.
\end{equation}
Moreover, if $m\in S^\infty$, $A:\mathbb{R}^2\to\mathbb{R}^2$ is a linear transformation, $v\in \mathbb{R}^2$, and $m_{A,v}(\xi,\eta):=m(A(\xi,\eta)+v)$ then
\begin{equation}\label{al9}
\|m_{A,v}\|_{S^\infty}=\|m\|_{S^\infty}.
\end{equation}

  \item[(ii)] Assume $p,q,r\in[1,\infty]$ satisfy $1/p+1/q=1/r$, and $m\in S^\infty$. Then, for any $f,g\in L^2(\mathbb{R})$,
\begin{equation}\label{mk6}
\|M(f,g)\|_{L^r} \lesssim \|m\|_{S^\infty} \|f\|_{L^p} \|g\|_{L^q}.
\end{equation}
In particular, if $1/p+1/q+1/r=1$,
\begin{equation}\label{mk6.01}
\Big|\int_{\mathbb{R}^2}m(\xi,\eta)\widehat{f}(\xi)\widehat{g}(\eta)\widehat{h}(-\xi-\eta)\,d\xi d\eta\Big|\lesssim \|m\|_{S^\infty}\|f\|_{L^p}\|g\|_{L^q}\|h\|_{L^r}.
\end{equation}

  \item[(iii)] If $p_1,p_2,p_3,p_4\in[1,\infty]$ are exponents that satisfy $1/p_1+1/p_2+1/p_3+1/p_4=1$ then
\begin{equation}\label{mk6.5}
\begin{split}
\Big|\int_{\mathbb{R}^3}\widehat{f_1}(\xi)\widehat{f_2}(\eta)\widehat{f_3}(\rho-\xi)&\widehat{f_4}(-\rho-\eta)
m(\xi,\eta,\rho)\,d\xi d\rho d\eta\Big|
\\
&\lesssim \|f_1\|_{L^{p_1}}\|f_2\|_{L^{p_2}}\|f_3\|_{L^{p_3}}\|f_4\|_{L^{p_4}}
\|\mathcal{F}^{-1}m\|_{L^1}.
\end{split}
\end{equation}

\end{itemize}

\end{lemma}

Given any multiplier $m: \mathbb{R}^d \to \mathbb{C}$, $d\in\{2,3\}$, and any $k,k_1,k_2,k_3,k_4\in \mathbb{Z}$, we define
\begin{equation}
\begin{split}
\label{lemEE1pr1}
&m^{k,k_1,k_2}(\xi,\eta) := m(\xi,\eta)\cdot\varphi_k(\xi) \varphi_{k_1}(\xi-\eta) \varphi_{k_2}(\eta),\\
&m^{k_1,k_2,k_3,k_4}(\xi,\eta,\rho) := m(\xi,\eta,\rho) \cdot\varphi_{k_1}(\xi) \varphi_{k_2}(\eta)\varphi_{k_3}(\rho-\xi)\varphi_{k_4}(-\rho-\eta).
\end{split}
\end{equation}
Let
\begin{align}
\label{lemEE1pr2}
\begin{split}
&\mathcal{X} := \{(k,k_1,k_2) \in \mathbb{Z}^3 : \max(k,k_1,k_2) - \mathrm{med}(k,k_1,k_2) \leq 4\},\\
&\mathcal{Y} := \{(k_1,k_2,k_3,k_4) \in \mathbb{Z}^4 : 2^{k_1}+2^{k_2}+2^{k_3}+2^{k_4}\geq (1+2^{-10})2^{\max(k_1,k_2,k_3,k_4)}\},\\
\end{split}
\end{align}
and notice that $m^{k,k_1,k_2}\equiv 0$ unless
$(k,k_1,k_2)\in\mathcal{X}$, and $m^{k_1,k_2,k_3,k_4}\equiv 0$ unless
$(k,k_1,k_2)\in\mathcal{Y}$.
Moreover, we will often use the notation
\begin{align}
\label{Onot}
m(\xi,\eta) & = O\big( f(|\xi|,|\xi-\eta|,|\eta|) \big) \, \Longleftrightarrow \,
{\| m^{k,k_1,k_2}(\xi,\eta) \|}_{S^\infty} \lesssim f(2^k,2^{k_1},2^{k_2}) \mathbf{1}_{\mathcal{X}}(k,k_1,k_2)
\end{align}
So, for example,
\begin{align*}
m(\xi,\eta) = O\big( |\xi-\eta|^{3/2} \big)\quad\text{ means }\quad\| m^{k,k_1,k_2}(\xi,\eta) \|_{S^\infty} \lesssim 2^{3k_1/2} \mathbf{1}_{\mathcal{X}}(k,k_1,k_2).
\end{align*}
We use a similar notation for symbols of three variables,
\begin{align}
\label{Onot2}
\begin{split}
& m(\xi,\eta,\rho) = O\big( f(|\xi|,|\eta|,|\rho-\xi|,|\rho+\eta|) \big)
\\
& \, \Longleftrightarrow \,
\big\| \mathcal{F}^{-1}\big(m^{k_1,k_2,k_3,k_4}\big)\big\|_{L^1} \lesssim f(2^{k_1},2^{k_2},2^{k_3},2^{k_4}) \mathbf{1}_{\mathcal{Y}}(k_1,k_2,k_3,k_4).
\end{split}
\end{align}

\subsubsection{Paraproducts}\label{paraprod}
For any $a,b\in L^2(\mathbb{R})$ we define the paraproduct $T_ab$ by the formula
\begin{equation}\label{on3}
\begin{split}
& \mathcal{F}(T_ab)(\xi) := \frac{1}{2\pi}\int_{\mathbb{R}}\widehat{a}(\xi-\eta)\widehat{b}(\eta)\chi(\xi-\eta,\eta)\,d\eta,
\\
& \chi(x,y) := \sum_{k\in\mathbb{Z}}\varphi_k(y)\varphi_{\leq k-10}(x).
\end{split}
\end{equation}
We also use the general formulas
\begin{equation}\label{on7}
a b = T_a b + T_b a + R(a,b),\qquad F(a) = T_{F'(a)} a + R_F(a),
\end{equation}
where $R(a,b)$ and $R_F(a)$ are (substantially) more smooth remainders.
The precise bounds on the remainders $R(a,b)$ and $R_F(a)$ depend on the context.

\subsubsection{A dispersive estimate and an interpolation lemma}
The following is our main linear dispersive estimate, which we use to control pointwise decay of solutions.

\begin{lemma}\label{dispersive}
For any $t\in\mathbb{R}\setminus\{0\}$, $k\in\mathbb{Z}$, and $f\in L^2(\mathbb{R})$ we have
\begin{equation}
\label{disperse}
 \|e^{i t \Lambda}P_kf\|_{L^\infty}\lesssim |t|^{-1/2}2^{k/4}\|\widehat{f}\|_{L^\infty}
  + |t|^{-3/5}2^{-2k/5}\big[2^k\|\partial \widehat{f}\|_{L^2}+\|\widehat{f}\|_{L^2}\big]
\end{equation}
and
\begin{equation}\label{disperseEa}
 \|e^{i t \Lambda}P_kf\|_{L^\infty}\lesssim |t|^{-1/2}2^{k/4}\|f\|_{L^1}.
\end{equation}
\end{lemma}

A more precise version is proved in Lemma \ref{ScaDecay} in section \ref{Sca1}. See also \cite[Lemma A.1]{IoPu3}.

We also use the following simple interpolation lemma (see also Lemma A.2 in \cite{IoPu3}).

\begin{lemma}\label{interpolation}
For any $k \in \mathbb{Z}$, and $f\in L^2(\mathbb{R})$ we have
\begin{equation}
\label{interp1}
{\big\| \what{P_kf} \big\|}_{L^\infty}^2 \lesssim {\big\| P_kf \big\|}_{L^1}^2 \lesssim
  2^{-k} {\|\what{f}\|}_{L^2} \big[ 2^k {\|\partial \what{f}\|}_{L^2} + {\|\what{f}\|}_{L^2} \big].
\end{equation}
\end{lemma}

\begin{proof} By scale invariance we may assume that $k=0$. It suffices to prove that
\begin{equation}\label{interp2}
 \big\|P_0f\big\|^2_{L^1}\lesssim \|\widehat{f}\|_{L^2}\big[\|\partial \widehat{f}\|_{L^2}+\|\widehat{f}\|_{L^2}\big].
\end{equation}
For $R\geq 1$ we estimate
\begin{equation*}
\begin{split}
\big\|P_0f\big\|_{L^1}&\lesssim \int_{|x|\leq R}|P_0f(x)|\,dx+\int_{|x|\geq R}|xP_0f(x)|\cdot \frac{1}{|x|}\,dx\\
&\lesssim R^{1/2}\|P_0f\|_{L^2}+R^{-1/2}\|xP_0f(x)\|_{L^2_x}\\
&\lesssim R^{1/2}\|\widehat{f}\|_{L^2}+R^{-1/2}\big[\|\partial \widehat{f}\|_{L^2}+\|\widehat{f}\|_{L^2}\big].
\end{split}
\end{equation*}
The desired estimate \eqref{interp2} follows by choosing $R$ suitably.
\end{proof}

\subsection{The main proposition}
Given $p_1\in[0,10^{-3}]$ we fix $\mathcal{P}=\mathcal{P}_{p_1}: [0,\infty) \rightarrow [0,1]$ an increasing function,
smooth on $(0,\infty)$, such that
\begin{equation}
\label{p0}
\mathcal{P}(x) = x^{2p_1}\,\,\text{ if }\,\,x\leq 1/2,\qquad \mathcal{P}(x) =1\,\,\text{ if }\,\, x\geq 1,\qquad x\mathcal{P}^\prime (x) \leq 10 p_1\mathcal{P}(x).
\end{equation}

Our main theorem follows using the local existence theory and a continuity argument from the following main proposition:

\begin{proposition}[Main bootstrap]\label{MainProp}
Assume that
\begin{equation}\label{constants}
\begin{split}
& N_0=N_I:=9,\quad N_1=N_S:=3,\quad N_2=N_\infty:=5,\qquad 0<10^4p_1\leq p_0\leq 10^{-10},
\\
& 0 < \varepsilon_0 \leq \varepsilon_1 \leq \varepsilon_0^{2/3} \ll 1.
\end{split}
\end{equation}
Assume $T\geq 1$ and $(h,\phi)\in C\big([0,\infty):(\mathcal{C}_0\cap \dot{H}^{N_0+1,p_1+1/2})\times \dot{H}^{N_0+1/2,p_1}\big)$
is a real-valued solution of the system \eqref{CPW},
\begin{equation}\label{on1}
\partial_th = G(h)\phi,\qquad \partial_t\phi = \dfrac{\partial_x^2 h}{(1+h_x^2)^{3/2}}
  - \dfrac{1}{2}\phi_x^2 + \dfrac{(G(h)\phi+h_x\phi_x)^2}{2(1+h_x^2)}.
\end{equation}
Let $U:=|\partial_x|h-i|\partial_x|^{1/2}\phi$ and assume that, for any $t\in[0,T]$,
\begin{equation}\label{bing1}
\langle t\rangle^{-p_0}\mathcal{K}_{I}(t)+\langle t\rangle^{-4p_0}\mathcal{K}_{S}(t)+\langle t\rangle^{1/2}\sum_{k\in\mathbb{Z}}(2^{N_2k}+2^{-k/10})\big\|P_k U(t)\big\|_{L^\infty}\leq\e_1,
\end{equation}
where $\langle t\rangle=1+t$ and, for $\mathcal{O}\in\{I,S\}$,
\begin{equation}\label{bing2}
\big[\mathcal{K}_{\mathcal{O}}(t)\big]^2:=\int_{\mathbb{R}}|\widehat{\mathcal{O}U}(\xi,t)|^2\cdot (|\xi|^{-1}+|\xi|^{2N_{\mathcal{O}}})\mathcal{P}((1+t)^2|\xi|)\,d\xi.
\end{equation}
Assume also that the initial data $U(0)$ satisfy the stronger bounds
\begin{equation}\label{bing3}
\sum_{\mathcal{O}\in\{I,S\}}\big\|(|\partial_x|^{-1/2+p_1}+|\partial_x|^{N_{\mathcal{O}}})(\mathcal{O}U)(0)\big\|_{L^2}\leq\e_0.
\end{equation}

Then we have the improved bound, for any $t\in[0,T]$,
\begin{equation}\label{bing4}
\langle t\rangle^{-p_0}\mathcal{K}_{I}(t)+\langle t\rangle^{-4p_0}\mathcal{K}_{S}(t)+\langle t\rangle^{1/2}\sum_{k\in\mathbb{Z}}(2^{N_2k}+2^{-k/10})\big\|P_kU(t)\big\|_{L^\infty}\lesssim \e_0.
\end{equation}
\end{proposition}

The rest of the paper is concerned with the proof of Proposition \ref{MainProp}.
We will always work under the assumptions \eqref{bing1}-\eqref{bing2}. The proof depends on the equations derived in section \ref{Equations} and on the  improved estimates in Propositions \ref{proEE1}, \ref{proEElow}, \ref{proEEZ}, \ref{proEEZlow}, and \ref{prodecay}. The argument is provided after the statement of Proposition \ref{prodecay}.

\medskip
\section{Derivation of the main scalar equation}\label{Equations}

As in Proposition \ref{MainProp}, assume $T\geq 1$ and $(h,\phi)\in C\big([0,T]:H^{N_0+1}\times \dot{H}^{p_1,N_0+1/2}\big)$
is a real-valued solution of the system \eqref{on1} satisfying \eqref{bing1}.
Let
\begin{equation}\label{symm2}
\begin{split}
& B := \frac{G(h)\phi + h_x\phi_x}{1+h_x^2} , \qquad V := \phi_x-Bh_x,
\\
& \omega := \phi-T_BP_{\geq 1}h=\phi-T_Bh+T_BP_{\leq 0}h,
\\
& \sigma := {(1+h_x^2)}^{-3/2}-1,
\end{split}
\end{equation}
where $T$ is defined in \eqref{on3}. Using \eqref{symm2} we calculate
\begin{equation}\label{on6}
- \frac{1}{2}\phi_x^2 + \frac{(G(h)\phi+h_x\phi_x)^2}{2(1+h_x^2)} =
  \frac{1}{2}\big[B^2(1+h_x^2)-(V+Bh_x)^2\big] = \frac{1}{2}\big[B^2-2VBh_x-V^2\big].
\end{equation}
Moreover, using the formula in the second line of \eqref{on7} and standard paradifferential calculus,
\begin{equation}\label{on7.5}
\frac{\partial_x^2 h}{(1+h_x^2)^{3/2}}
  = \partial_x^2h + \partial_x\Big(\frac{\partial_xh}{(1+h_x^2)^{1/2}}-h_x\Big)
  = \partial_x^2h + \partial_xT_\sigma h_x+\partial_xE_{\geq 3,h},
\end{equation}
where $E_{\geq 3,h}$ is a more smooth cubic error (compare with \eqref{on7}), satisfying, for any $t\in[0,T]$,
\begin{equation}\label{on7.6}
\langle t\rangle^{1-p_0}\|E_{\geq 3,h}(t)\|_{H^{N_0+2}}+\langle t\rangle^{1-4p_0}\|SE_{\geq 3,h}(t)\|_{H^{N_1+2}}+\langle t\rangle^{11/10}\|E_{\geq 3,h}(t)\|_{\widetilde{W}^{N_2+2}}\lesssim \veps_1^3.
\end{equation}

We will also use the formula (see Proposition \ref{ra102})
\begin{equation}\label{on4}
G(h)\phi = |\partial_x|\omega - |\partial_x|T_BP_{\leq 0}h-\partial_x T_V h + G_2 + G_{\geq 3},
\end{equation}
where
\begin{align}\label{on4.1}
G_2 := |\partial_x|T_{|\partial_x|\phi}h - |\partial_x|(h|\partial_x|\phi )
  + \partial_xT_{\partial_x\phi}h - \partial_x(h\partial_x\phi),
\end{align}
and $G_{\geq 3}$ is a cubic error, satisfying, for any $t\in[0,T]$,
\begin{equation}\label{on4.2}
\langle t\rangle^{1-p_0}\|G_{\geq 3}(t)\|_{H^{N_0+1}}+\langle t\rangle^{1-4p_0}\|SG_{\geq 3}(t)\|_{H^{N_1+1}}+\langle t\rangle^{11/10}\|G_{\geq 3}(t)\|_{\widetilde{W}^{N_2+1}}\lesssim \e_1^3.
\end{equation}
The function $G(h)\phi$ satisfies linear estimates with derivative loss (see \eqref{bvd1} for a stronger bound)
\begin{equation}\label{on4.25}
\langle t\rangle^{-p_0}\big\|G(h)\phi\big\|_{H^{N_0-1}}+\langle t\rangle^{-4p_0}\big\|SG(h)\phi\big\|_{H^{N_1-1}}+\langle t\rangle^{1/2}\big\|G(h)\phi\big\|_{\widetilde{W}^{N_2-1}}\lesssim \e_1.
\end{equation}

For simplicity of notation, for $\alpha\in [-2,2]$ let $O_{3,\alpha}$ denote generic functions $F$ on $[0,T]$
that satisfy the ``cubic'' bounds (see also Definition \ref{Oterms})
\begin{equation}\label{on4.3}
\langle t\rangle^{1-p_0}\|F(t)\|_{H^{N_0+\alpha}}+\langle t\rangle^{1-4p_0}\|SF(t)\|_{H^{N_1+\alpha}}+\langle t\rangle^{11/10}\|F(t)\|_{\widetilde{W}^{N_2+\alpha}}\lesssim \e_1^3.
\end{equation}
In this section we proceed with the formal calculations, without proving that the various cubic errors terms that will appear
satisfy indeed the desired bounds.
All the claimed cubic bounds will follow from the assumptions \eqref{bing1} and the definitions, by elliptic estimates.
Detailed proofs are provided in Appendix \ref{aux}.

The first equation in \eqref{on1} becomes
\begin{equation*}
\partial_t h = |\partial_x|\omega - |\partial_x|T_{|\partial_x|\omega}P_{\leq 0}h-\partial_x T_V h + G_2 + G_{\geq 3},
\end{equation*}
with $G_{\geq 3} \in O_{3,1}$,
while the second equation in \eqref{on1} gives
\begin{align}
\label{dtomega1}
\begin{split}
\partial_t\omega &= \frac{\partial_x^2 h}{(1+h_x^2)^{3/2}}
 + \Big( \partial_t \phi - \frac{\partial_x^2 h}{(1+h_x^2)^{3/2}} \Big)
 - T_{\partial_tB}P_{\geq 1}h - T_B\partial_tP_{\geq 1}h\\
&= \frac{\partial_x^2 h}{(1+h_x^2)^{3/2}} + I + II + III,
\end{split}
\end{align}
where
\begin{align*}
\begin{split}
I & = (B^2-V^2)/2 - V B h_x
\\
& = T_B B + R(B,B)/2 - T_V V - R(V,V)/2 - T_V B h_x - T_{Bh_x} V + O_{3,1/2}
\\
& = T_B B - T_{Bh_x} V - T_V \phi_x + R(|\partial_x|\omega,|\partial_x|\omega)/2 - R(\partial_x\omega,\partial_x\omega)/2+O_{3,1/2},
\\
II & = T_{|\partial_x|^3 h} P_{\geq 1}h + O_{3,1/2},
\\
III & =-T_B(G(h)\phi)+T_BP_{\leq 0}(G(h)\phi)=-T_B(G(h)\phi)+T_{|\partial_x|\omega}P_{\leq 0}|\partial_x|\omega+ O_{3,1/2}.
\end{split}
\end{align*}
Notice that
\begin{equation*}
B - G(h)\phi - V h_x = B - G(h)\phi - \phi_x h_x + B h_x^2 = 0.
\end{equation*}
Therefore
\begin{align*}
\begin{split}
I+III & = T_B(Vh_x)-T_{Bh_x}V-T_V\phi_x
\\
& +T_{|\partial_x|\omega}P_{\leq 0}|\partial_x|\omega+R(|\partial_x|\omega,|\partial_x|\omega)/2-R(\partial_x\omega,\partial_x\omega)/2+O_{3,1/2}
\\
& = T_B(Vh_x)-T_{Bh_x}V-T_V(\partial_xT_BP_{\geq 1}h)-T_V\omega_x
\\
& +T_{|\partial_x|\omega}P_{\leq 0}|\partial_x|\omega+R(|\partial_x|\omega,|\partial_x|\omega)/2-R(\partial_x\omega,\partial_x\omega)/2+O_{3,1/2}.
\end{split}
\end{align*}
We show in Proposition \ref{CubicSummary} that
\begin{equation*}
T_B(Vh_x)-T_{Bh_x}V-T_V(\partial_xT_BP_{\geq 1}h) \in O_{3,1/2}.
\end{equation*}
Therefore, using also \eqref{on7.5}, the system \eqref{on1} becomes
\begin{equation}\label{on10}
\begin{cases}
&\partial_th=|\partial_x|\omega- |\partial_x|T_{|\partial_x|\omega}P_{\leq 0}h-\partial_xT_Vh + G_2 + G_{\geq 3},
\\
\\
&\partial_t\omega=\partial_x^2h+\partial_xT_\sigma h_x-T_V\omega_x + H_2 + T_{|\partial_x|^3 h} P_{\geq 1}h + \Omega_{\geq 3},
\end{cases}
\end{equation}
where
\begin{equation}\label{on11}
\begin{cases}
& G_2 := |\partial_x|T_{|\partial_x|\phi}h-|\partial_x|(h|\partial_x|\phi )+\partial_xT_{\partial_x\phi}h-\partial_x(h\partial_x\phi),
\\
\\
& H_2 := T_{|\partial_x|\omega}P_{\leq 0}|\partial_x|\omega+R(|\partial_x|\omega,|\partial_x|\omega)/2-R(\partial_x\omega,\partial_x\omega)/2,
\end{cases}
\end{equation}
and, as proved in Proposition \ref{CubicSummary},
\begin{align*}
G_{\geq 3} \in O_{3,1} , \qquad \Omega_{\geq 3} \in O_{3,1/2} .
\end{align*}

Let $\widetilde{\chi}(x,y):=1-\chi(x,y)-\chi(y,x)$,
\begin{equation}\label{on20}
\begin{split}
m_2(\xi,\eta) & := \widetilde{\chi}(\xi-\eta,\eta)[\xi(\xi-\eta)-|\xi||\xi-\eta|]-\chi(\xi-\eta,\eta)|\xi||\xi-\eta|\varphi_{\leq 0}(\eta),
\\
q_2(\xi,\eta) & := \widetilde{\chi}(\xi-\eta,\eta)[\eta(\xi-\eta)+|\eta||\xi-\eta|]/2+\chi(\xi-\eta,\eta)|\eta||\xi-\eta|\varphi_{\leq 0}(\eta),
\end{split}
\end{equation}
and recall the notation \eqref{Moutm}. To summarize, we proved the following:

\begin{proposition}\label{firstred}
Let $(h,\phi)$ be a solution of \eqref{on1} satisfying the bootstrap assumption \eqref{bing1},
and let $\o,\sigma,V$ be as in \eqref{symm2}. Then
\begin{equation}\label{symm1}
\begin{cases}
& \partial_t h = |\partial_x|\omega - \partial_x T_V h + M_2(\omega,h) + G_{\geq 3},
\\
\\
& \partial_t \omega = \partial_x^2 h + \partial_x T_\sigma h_x - T_V \omega_x + T_{|\partial_x|^3 h} P_{\geq 1}h
  + Q_2(\omega,\omega) + \Omega_{\geq 3},
\end{cases}
\end{equation}
where $M_2$ and $Q_2$ are the operators associated to the multipliers $m_2$ and $q_2$ in \eqref{on20},
and
\begin{align}
\label{firstredcubic}
G_{\geq 3} \in O_{3,1},\qquad \Omega_{\geq 3}  \in O_{3,1/2} .
\end{align}
\end{proposition}

\vskip10pt
\subsection{Symmetrization of the equations}
Recall the water waves system \eqref{symm1} for the surface elevation $h$ and Alinhac's good unknown $\o$,  and the definitions \eqref{symm2}. In this section we aim to diagonalize and symmetrize this system,
and write it as a single scalar equation for a complex valued unknown $u$.
The main result can be summarized as follows:

\begin{proposition}\label{prosymm}
We define the real-valued functions $\gamma$, $p_1$, and $p_0$ by
\begin{align}
\label{symmsymbols}
\gamma := \sqrt{1+\sigma} - 1,\qquad p_1 := \gamma,\qquad p_0 :=  -\frac{3}{4} \partial_x \gamma,
\end{align}
where $\sigma$ is as in \eqref{symm2}, and the {\it{main complex-valued unknown}}
\begin{align}
\label{symm5}
u := |\partial_x| h- i |\partial_x|^{1/2} \o+T_{p_1} P_{\geq 1}|\partial_x| h + T_{p_0} P_{\geq 1}|\partial_x|^{-1} \partial_x h .
\end{align}
Then $u$ satisfies the evolution equation
\begin{align}
\label{symm10}
\partial_t u - i |\partial_x|^{3/2} u - i \Sigma_\gamma (u) = - \partial_x T_V u + \mathcal{N}_2(h,\omega) + \widetilde{U}_{\geq 3},
\end{align}
where
\begin{align}
\label{symm10'}
\widetilde{U}_{\geq 3} \in |\partial_x|^{1/2}O_{3,1/2} ,
\end{align}
the operator $\Sigma_\gamma$ is given by
\begin{align}
\label{symm11}
\Sigma_\gamma(u) = T_\gamma P_{\geq 1}|\partial_x|^{3/2} u
  - \frac{3}{4} T_{\partial_x \gamma} P_{\geq 1}\partial_x |\partial_x|^{-1/2} u ,
\end{align}
and the quadratic terms (expressed in $h$ and $\o$) are
\begin{align}
\label{symmquadterms}
\begin{split}
\mathcal{N}_2(h,\omega) & = - [ |\partial_x| , \partial_x T_{\partial_x \o} ] h
  - i  T_{\partial_x^2 \o} |\partial_x|^{1/2} \omega + i [ |\partial_x|^{1/2}, T_{\partial_x \o} \partial_x] \omega
\\
& + |\partial_x| M_2(\omega,h) - i |\partial_x|^{1/2} Q_2(\omega,\omega) - i |\partial_x|^{1/2} T_{|\partial_x|^3 h} P_{\geq 1}h.
\end{split}
\end{align}
We will express these quadratic terms as functions of $u$ and $\bar{u}$ via \eqref{symm5} later on.
\end{proposition}

\begin{proof}
We start by calculating:
\begin{align*}
\partial_t u  & = \partial_t \big( |\partial_x| h- i |\partial_x|^{1/2} \o+ T_{p_1} P_{\geq 1}|\partial_x| h + T_{p_0} P_{\geq 1}|\partial_x|^{-1} \partial_x h  \big)
\\
& =  |\partial_x| \partial_t h- i |\partial_x|^{1/2} \partial_t \o+T_{p_1} P_{\geq 1}|\partial_x| \partial_t h + T_{p_0} P_{\geq 1}|\partial_x|^{-1} \partial_x \partial_t h + |\partial_x|^{1/2}O_{3,1/2}
\\
& =  |\partial_x|^2 \omega - |\partial_x| \partial_x T_V h + |\partial_x| M_2(\omega,h)\\
& + i |\partial_x|^{5/2} h + i |\partial_x|^{3/2} T_\sigma |\partial_x| h
+ i |\partial_x|^{1/2}T_V \partial_x \omega - i |\partial_x|^{1/2} Q_2(\omega,\omega)-i|\partial_x|^{1/2}T_{|\partial_x|^3 h} P_{\geq 1}h\\
& + T_{p_1} P_{\geq 1}\big( |\partial_x|^2 \omega - |\partial_x| \partial_x T_V h \big)
+ T_{p_0} P_{\geq 1}\big( \partial_x \omega + |\partial_x| T_V h \big)+ |\partial_x|^{1/2}O_{3,1/2}.
\end{align*}

Gathering appropriately the above terms we can write
\begin{align}
\label{symm20}
 \begin{split}
\partial_t u  - i |\partial_x|^{3/2} u
& = - |\partial_x| \partial_x T_V h - T_{p_1} P_{\geq 1}|\partial_x| \partial_x T_V h + T_{p_0} P_{\geq 1}|\partial_x| T_V h + |\partial_x| M_2(\omega,h)
\\
& + i |\partial_x|^{1/2}T_V \partial_x \omega - i |\partial_x|^{1/2} Q_2(\omega,\omega)-i|\partial_x|^{1/2}T_{|\partial_x|^3 h} P_{\geq 1}h
\\
& + T_{p_1} P_{\geq 1}|\partial_x|^2 \omega + T_{p_0} P_{\geq 1}\partial_x \omega + i |\partial_x|^{3/2} T_\sigma |\partial_x| h
\\
& - i |\partial_x|^{3/2} T_{p_1} P_{\geq 1}|\partial_x| h - i |\partial_x|^{3/2} T_{p_0} P_{\geq 1}|\partial_x|^{-1} \partial_x h + |\partial_x|^{1/2}O_{3,1/2}.
\end{split}
\end{align}
We observe that the expression in the first two lines in the right-hand side of \eqref{symm20} is equal to
\begin{align}
\begin{split}
 \label{symm25}
& - \partial_x T_V u - [ |\partial_x| , \partial_x T_V ] h - i T_{\partial_x V} |\partial_x|^{1/2}  \omega + i [ |\partial_x|^{1/2}, T_V \partial_x] \omega- [ T_{p_1} |\partial_x| P_{\geq 1}, \partial_x T_V ] h\\
  &- [T_{p_0} {|\partial_x|}^{-1} \partial_xP_{\geq 1}, \partial_x T_V] h + |\partial_x| M_2(\omega,h) - i |\partial_x|^{1/2} Q_2(\omega,\omega) -i|\partial_x|^{1/2}T_{|\partial_x|^3 h} P_{\geq 1}h
\\
& = - \partial_x T_V u - [ |\partial_x| , \partial_x T_V ] h
  - i  T_{\partial_x V} |\partial_x|^{1/2} \omega + i [ |\partial_x|^{1/2}, T_V\partial_x] \omega
\\
& + |\partial_x| M_2(\omega,h) - i |\partial_x|^{1/2} Q_2(\omega,\omega) -i|\partial_x|^{1/2}T_{|\partial_x|^3 h}P_{\geq 1} h+ |\partial_x|^{1/2} O_{3,1/2}.
\end{split}
\end{align}
Using the definition of $V$ and $\o$ in \eqref{symm2}, we
see that the above quadratic terms coincide up to $|\partial_x|^{1/2}O_{3.1/2}$ with the quadratic terms appearing in
\eqref{symm10} with \eqref{symmquadterms}.

We then look at the cubic and higher order terms in the last two lines in the right-hand side of \eqref{symm20}.
Our aim is to show that they are of the form  $i \Sigma_\gamma(u)$, see \eqref{symm11}, up to acceptable errors.
For this purpose we first use the definition of $u$ in \eqref{symm5} and write
\begin{align}
\label{symm30}
\begin{split}
i \Sigma_\gamma(u) & = i T_\gamma P_{\geq 1}|\partial_x|^{5/2} h + T_\gamma P_{\geq 1}|\partial_x|^2 \o+i T_\gamma P_{\geq 1}|\partial_x|^{3/2} T_{p_1} P_{\geq 1}|\partial_x| h
\\
& +  i T_\gamma P_{\geq 1}|\partial_x|^{3/2}  T_{p_0} P_{\geq 1}|\partial_x|^{-1} \partial_x h - \frac{3i}{4} T_{\partial_x \gamma} P_{\geq 1}\partial_x |\partial_x|^{1/2} h- \frac{3}{4} T_{\partial_x \gamma} P_{\geq 1}\partial_x \o \\
&- \frac{3i}{4} T_{\partial_x \gamma} P_{\geq 1}\partial_x |\partial_x|^{-1/2} T_{p_1} P_{\geq 1}|\partial_x| h
  -\frac{3i}{4} T_{\partial_x \gamma} P_{\geq 1}\partial_x |\partial_x|^{-1/2}  T_{p_0} P_{\geq 1}|\partial_x|^{-1} \partial_x h.
\end{split}
\end{align}
The last term on the last line is $|\partial_x|^{1/2}O_{3,1/2}$ 
so we can disregard it.

We then compare the expression in \eqref{symm30} above and the last two lines of \eqref{symm20}.
Our proposition will be proven if the identities
\begin{align}
\label{symm50}
\begin{split}
& i |\partial_x|^{3/2} T_\sigma |\partial_x| h
  - i |\partial_x|^{3/2} T_{p_1} P_{\geq 1}|\partial_x| h - i |\partial_x|^{3/2} T_{p_0} P_{\geq 1}|\partial_x|^{-1} \partial_x h
\\
& =   i T_\gamma P_{\geq 1}|\partial_x|^{5/2} h + i T_\gamma P_{\geq 1}|\partial_x|^{3/2} T_{p_1} P_{\geq 1}|\partial_x| h
  +  i T_\gamma P_{\geq 1}|\partial_x|^{3/2}  T_{p_0} P_{\geq 1}|\partial_x|^{-1} \partial_x h
\\
& - \frac{3i}{4} T_{\partial_x \gamma} P_{\geq 1}\partial_x |\partial_x|^{1/2} h
  - \frac{3i}{4} T_{\partial_x \gamma} P_{\geq 1}\partial_x |\partial_x|^{-1/2} T_{p_1} P_{\geq 1}|\partial_x| h
  + |\partial_x|^{1/2}O_{3,1/2}
\\
\end{split}
\end{align}
and
\begin{align}
\label{symm51}
\begin{split}
& T_{p_1} P_{\geq 1}|\partial_x|^2 \omega + T_{p_0} P_{\geq 1}\partial_x \omega = T_\gamma P_{\geq 1}|\partial_x|^2 \o - \frac{3}{4} T_{\partial_x \gamma} P_{\geq 1}\partial_x \o  . 
\end{split}
\end{align}
hold true. We immediately notice that the second equation \eqref{symm51} is satisfied by imposing
$p_1 = \gamma$ and $p_0 = -3\partial_x \gamma/4$, as in \eqref{symmsymbols}.
We then need to verify that \eqref{symm50} can be satisfied for an appropriate choice of the function $\gamma$.

We notice that all the multipliers $P_{\geq 1}$ in \eqref{symm50} can be dropped, at the expense of acceptable errors. Therefore \eqref{symm50} holds provided one has the following two identities for the symbols:
\begin{align}
\label{symm101}
& \sigma - p_1 = \gamma + \gamma p_1,\\
  \label{symm102}
& \frac{3}{2} \partial_x \sigma - \frac{3}{2}\partial_x p_1 + p_0
  =  \frac{3}{2} \gamma \partial_x p_1 -\gamma p_0 +\frac{3}{4} \partial_x \gamma  +\frac{3}{4} \partial_x \gamma p_1 .
\end{align}

Since $p_1 = \gamma$, \eqref{symm101} becomes $2\gamma + \gamma^2 = \sigma$, which is satisfied by imposing the first identity in \eqref{symmsymbols}.
One can then verify that the last equation \eqref{symm102} is automatically satisfied.
\end{proof}

\begin{remark}
\normalfont
We notice that the symmetrization obtained in Proposition \ref{prosymm},
and the formulas \eqref{symmsymbols} for $p_1,p_0$ and $\gamma$, are simpler than the ones of \cite{ABZ1}.
This is not only because we are considering the $1$ dimensional case,
but also because of our choice of the main variables in which we express the system, that is the energy variables
$|\partial_x|h$ and $|\partial_x|^{1/2}\o$.
\end{remark}

\subsubsection{The quadratic terms}\label{seceq}
We analyze now the quadratic terms in \eqref{symmquadterms},
\begin{align}
\label{eqquad}
\begin{split}
&\mathcal{N}_2(h,\omega) = \sum_{k=1}^6 \mathcal{N}_2^k
\\
&\mathcal{N}_2^1 := - [ |\partial_x| , \partial_x T_{\partial_x \o} ] h,\,\,\,\,\,\quad\qquad \mathcal{N}_2^2 := - i  T_{\partial_x^2 \o} |\partial_x|^{1/2} \omega, \qquad\mathcal{N}_2^3 := i [ |\partial_x|^{1/2}, T_{\partial_x \o} \partial_x] \omega,
\\
&\mathcal{N}_2^4 := - i |\partial_x|^{1/2} T_{|\partial_x|^3 h} P_{\geq 1}h,\qquad \mathcal{N}_2^5 := |\partial_x| M_2(\omega,h),\qquad\quad \mathcal{N}_2^6 := - i |\partial_x|^{1/2} Q_2(\omega,\omega).
\end{split}
\end{align}

Since $\gamma = \sqrt{1+\sigma}-1$, using \eqref{symm5} we have
\begin{align}
\label{eqhou}
h & = \frac{1}{2}  |\partial_x|^{-1} (u + \bar{u}) + P_{\geq -4}O_{3,1},\qquad \o = - \frac{1}{2i} |\partial_x|^{-1/2}(u - \bar{u}).
\end{align}
Using these relations we can express the quadratic terms \eqref{eqquad} in terms of $u$ and $\bar{u}$, i.e.
\begin{align*}
\mathcal{N}_2^1 & = \frac{1}{4i} \big[ |\partial_x| , \partial_x T_{\partial_x|\partial_x|^{-1/2} (u - \bar{u}) } \big] |\partial_x|^{-1} (u + \bar{u})+|\partial_x|^{1/2}O_{3,1/2},
\\
\mathcal{N}_2^2 & =\frac{1}{4i}  T_{|\partial_x|^{3/2} (u - \bar{u})} (u - \bar{u}),
\\
\mathcal{N}_2^3 & = \frac{1}{4i} [ |\partial_x|^{1/2}, T_{\partial_x |\partial_x|^{-1/2}(u-\bar{u})} \partial_x] |\partial_x|^{-1/2}(u-\bar{u}),
\\
\mathcal{N}_2^4 & = \frac{1}{4i} |\partial_x|^{1/2} T_{|\partial_x|^{2} (u + \bar{u})} P_{\geq 1}|\partial_x|^{-1} (u + \bar{u})+|\partial_x|^{1/2}O_{3,1/2},
\\
\mathcal{N}_2^5 & = \frac{i}{4} |\partial_x| M_2 \big( |\partial_x|^{-1/2} (u-\bar{u}), |\partial_x|^{-1} (u+\bar{u}) \big)+|\partial_x|^{1/2}O_{3,1/2},
\\
\mathcal{N}_2^6 & = \frac{i}{4} |\partial_x|^{1/2} Q_2( |\partial_x|^{-1/2} (u-\bar{u}), |\partial_x|^{-1/2} (u-\bar{u})).
\end{align*}

We divide these terms into 8 groups, by distinguishing the different types of interactions with respect to the specific pairing of $u$ and $\bar{u}$,
and the type of frequency interactions ($\mathrm{Low}\times\mathrm{High}\to\mathrm{High}$ interactions associated to the symbol $\chi(\xi-\eta,\eta)$ and $\mathrm{High}\times\mathrm{High}\to\mathrm{Low}$ interactions associated to the symbol $\widetilde{\chi}(\xi-\eta,\eta)$). Recall the formulas \eqref{on20}. We define
\begin{align}
\label{a_++}
\begin{split}
a_{++}(\xi,\eta) := \frac{1}{4i} &\chi(\xi-\eta,\eta)\Big[-\frac{\xi (\xi-\eta)}{|\xi-\eta|^{1/2}} \Big(\frac{|\xi|}{|\eta|}-1\Big)
    + |\xi-\eta|^{3/2}
    + \frac{\eta (\xi-\eta)}{|\xi-\eta|^{1/2}} \Big(1 - \frac{|\xi|^{1/2}}{|\eta|^{1/2}} \Big)
\\
  & + |\xi-\eta|^2\frac{|\xi|^{1/2}\varphi_{\geq 1}(\eta)}{|\eta|} +\frac{|\xi|^2-|\xi|^{1/2}|\eta|^{3/2}}{|\eta|}|\xi-\eta|^{1/2}\varphi_{\leq 0}(\eta)\Big] ,
\end{split}
\end{align}
\begin{align}
\label{a_+-}
\begin{split}
a_{+-}(\xi,\eta) := \frac{1}{4i} &\chi(\xi-\eta,\eta)\Big[-\frac{\xi (\xi-\eta)}{|\xi-\eta|^{1/2}} \Big(\frac{|\xi|}{|\eta|}-1\Big)
    -|\xi-\eta|^{3/2}
    -\frac{\eta (\xi-\eta)}{|\xi-\eta|^{1/2}} \Big(1 - \frac{|\xi|^{1/2}}{|\eta|^{1/2}} \Big)
\\
  & + |\xi-\eta|^2\frac{|\xi|^{1/2}\varphi_{\geq 1}(\eta)}{|\eta|} +\frac{|\xi|^2+|\xi|^{1/2}|\eta|^{3/2}}{|\eta|}|\xi-\eta|^{1/2}\varphi_{\leq 0}(\eta)\Big] ,
\end{split}
\end{align}
\begin{align}
\label{a_-+}
\begin{split}
a_{-+}(\xi,\eta) := \frac{1}{4i} &\chi(\xi-\eta,\eta)\Big[\frac{\xi (\xi-\eta)}{|\xi-\eta|^{1/2}} \Big(\frac{|\xi|}{|\eta|}-1\Big)
    -|\xi-\eta|^{3/2}
    -\frac{\eta (\xi-\eta)}{|\xi-\eta|^{1/2}} \Big(1 - \frac{|\xi|^{1/2}}{|\eta|^{1/2}} \Big)
\\
  & + |\xi-\eta|^2\frac{|\xi|^{1/2}\varphi_{\geq 1}(\eta)}{|\eta|}
  +\frac{-|\xi|^2+|\xi|^{1/2}|\eta|^{3/2}}{|\eta|}|\xi-\eta|^{1/2}\varphi_{\leq 0}(\eta)\Big] ,
\end{split}
\end{align}
\begin{align}
\label{a_--}
\begin{split}
a_{--}(\xi,\eta) := \frac{1}{4i} &\chi(\xi-\eta,\eta)\Big[\frac{\xi (\xi-\eta)}{|\xi-\eta|^{1/2}} \Big(\frac{|\xi|}{|\eta|}-1\Big)
    + |\xi-\eta|^{3/2}
    + \frac{\eta (\xi-\eta)}{|\xi-\eta|^{1/2}} \Big(1 - \frac{|\xi|^{1/2}}{|\eta|^{1/2}} \Big)
\\
  & + |\xi-\eta|^2\frac{|\xi|^{1/2}\varphi_{\geq 1}(\eta)}{|\eta|}
  +\frac{-|\xi|^2-|\xi|^{1/2}|\eta|^{3/2}}{|\eta|}|\xi-\eta|^{1/2}\varphi_{\leq 0}(\eta)\Big] ,
\end{split}
\end{align}
and
\begin{align}
\label{b++}
b_{++}(\xi,\eta) & = \frac{i}{4}\frac{|\xi| \widetilde{m}_2(\xi,\eta)}{|\xi-\eta|^{1/2} |\eta|}
  + \frac{i}{4}\frac{|\xi|^{1/2} \widetilde{q}_2(\xi,\eta)}{|\xi-\eta|^{1/2} |\eta|^{1/2}},
\end{align}
\begin{align}
\label{b+-}
\begin{split}
b_{+-}(\xi,\eta) & = \frac{i}{4}\frac{|\xi| \widetilde{m}_2(\xi,\eta)}{|\xi-\eta|^{1/2} |\eta|}
  - \frac{i}{4}\frac{|\xi|^{1/2} \widetilde{q}_2(\xi,\eta)}{|\xi-\eta|^{1/2}|\eta|^{1/2} },
\end{split}
\end{align}
\begin{align}
\label{b-+}
\begin{split}
b_{-+}(\xi,\eta) & = - \frac{i}{4}\frac{|\xi| \widetilde{m}_2(\xi,\eta)}{|\xi-\eta|^{1/2} |\eta|}
  - \frac{i}{4}\frac{|\xi|^{1/2} \widetilde{q}_2(\xi,\eta)}{ |\xi-\eta|^{1/2}|\eta|^{1/2}},
\end{split}
\end{align}
\begin{align}
\label{b--}
b_{--}(\xi,\eta) & = - \frac{i}{4}\frac{|\xi| \widetilde{m}_2(\xi,\eta)}{|\xi-\eta|^{1/2} |\eta|}
  + \frac{i}{4}\frac{|\xi|^{1/2} \widetilde{q}_2(\xi,\eta)}{|\xi-\eta|^{1/2} |\eta|^{1/2}},
\end{align}
where
\begin{align}
\label{m_2q_2}
\begin{split}
\widetilde{m}_2(\xi,\eta) & := \widetilde{\chi}(\xi-\eta,\eta)] \big[ \xi(\xi-\eta)-|\xi||\xi-\eta| \big],
\\
\widetilde{q}_2(\xi,\eta) & := \widetilde{\chi}(\xi-\eta,\eta)\frac{\eta(\xi-\eta)+|\eta||\xi-\eta|}{2}.
\end{split}
\end{align}

Using the operator-symbol notation \eqref{Moutm} and \eqref{on3}, we notice that
\begin{align}
\sum_{k=1}^6\mathcal{N}_2^k =\sum_{X\in\{A,B\}}X_{++} (u,u) + X_{+-}(u,\bar{u}) +  X_{-+}(\bar{u},u) +  X_{--}(\bar{u},\bar{u})
  + |\partial_x|^{1/2} O_{3,1/2} .
\end{align}
Let us also denote
\begin{align}
\label{notsum}
\sum_{\star} := \sum_{(\epsilon_1,\epsilon_2) \in \{ (+,+),(+,-),(-,+),(-,-)\}} .
\end{align}
For any complex-valued function $f$, we use the notation $f_+ := f$, $f_- := \bar{f}$. We summarize the above computations in the following proposition:

\begin{proposition}\label{proequ}
Let $(h,\phi)$ be a solution of \eqref{on1} satisfying the bootstrap assumption \eqref{bing1},
and let $u$ be defined as in \eqref{symm5} with $\o,\sigma,V$ given by \eqref{symm2}.
Then we have
\begin{align}
\label{equ}
\partial_t u & - i |\partial_x|^{3/2} u - i \Sigma_\gamma (u) = - \partial_x T_V u  + \mathcal{N}_u + U_{\geq 3}
\end{align}
with $U_{\geq 3}\in |\partial_x|^{1/2}O_{3,1/2}$,
\begin{align}
\label{equSigma}
\begin{split}
\Sigma_\gamma = T_\gamma P_{\geq 1}|\partial_x|^{3/2}  - \frac{3}{4} T_{\partial_x \gamma} P_{\geq 1}\partial_x |\partial_x|^{-1/2},\qquad \g = \sqrt{1+\sigma}-1 ,
\end{split}
\end{align}
and
\begin{align}
\label{equN}
\mathcal{N}_u & = \sum_{\star}\big[A_{\eps_1\eps_2}(u_{\eps_1},u_{\eps_2})+B_{\eps_1\eps_2}(u_{\eps_1},u_{\eps_2})\big].
\end{align}
The symbols of the quadratic operators are given in \eqref{a_++}--\eqref{b--}. Moreover, for any $t\in[0,T]$,
\begin{equation}\label{bing11}
\langle t\rangle^{-p_0}\mathcal{K}'_{I}(t)+\langle t\rangle^{-4p_0}\mathcal{K}'_{S}(t)+\langle t\rangle^{1/2}\sum_{k\in\mathbb{Z}}(2^{N_2k}+2^{-k/10})\big\|P_ku(t)\big\|_{L^\infty}\lesssim\e_1,
\end{equation}
where, for $\mathcal{O}\in\{I,S\}$,
\begin{equation}\label{bing21}
\big[\mathcal{K}'_{\mathcal{O}}(t)\big]^2:=\int_{\mathbb{R}}|\widehat{\mathcal{O}u}(\xi,t)|^2\cdot \big(|\xi|^{-1}+|\xi|^{2N_{\mathcal{O}}}\big)\mathcal{P}((1+t)^2|\xi|)\,d\xi.
\end{equation}
The initial data $u(0)$ satisfy the stronger bounds (recall that $0<\e_0\leq\e_1\leq\e_0^{2/3}\ll 1$)
\begin{equation}\label{bing31}
\sum_{\mathcal{O}\in\{I,S\}}\big\|\big(|\partial_x|^{-1/2+p_1}+|\partial_x|^{N_{\mathcal{O}}}\big)\mathcal{O}u(0)\big\|_{L^2}\lesssim \e_0.
\end{equation}
\end{proposition}

The bounds \eqref{bing11}--\eqref{bing31} follow from the apriori assumptions \eqref{bing1}--\eqref{bing3} and the definition \eqref{symm5} (notice that $u=|\partial_x|h-i|\partial_x|^{1/2}\phi$ at very low frequencies).

We notice that the only quasilinear quadratic contributions to the nonlinearity in \eqref{equ}
come from the term $- \partial_x T_V u$ on the right-hand side of \eqref{equ}.
All of the other quadratic contributions do not lose derivatives.
The term $- i \Sigma_\gamma (u)$ arises from the presence of surface tension. This term loses $3/2$ derivative but
it is essentially a self-adjoint operator. We will exploit this structure below to perform energy estimates.

We have thus reduced the water waves system \eqref{on1} to the equation \eqref{equ} above for a single complex valued unknown $u$. From now on we will work with $u$ as our main variable.
We will also keep $V$ as a variable and keep in mind that, in view of \eqref{symm2} and \eqref{eqhou},
\begin{align}
\label{Vu+baru0}
\begin{split}
& V = -\frac{1}{2i} \partial_x |\partial_x|^{-1/2} (u-\overline{u}) + V_2 ,
\qquad V_2 := \partial_x T_B P_{\geq 1}h - B h_x . 
\end{split}
\end{align}

\subsection{Higher order derivatives and weights}\label{operator{D}}
To implement the energy method we need to control the increment of higher order Sobolev norms of the main variable $u$. Because of the presence of the operator $\Sigma_\gamma$, which is of order $3/2$,
one cannot construct higher order energies by applying regular derivatives to the equation.
We apply instead suitably modified versions of derivatives to the equation.
The differential operator we will use, dictated by the structure of the equation, is given by $\D := |\partial_x|^{3/2} + \Sigma_\g$, see the definition \eqref{equSigma}. Let $k\in[1,2N_0/3]$ be an even integer and define
\begin{align}
\label{Wk}
W=W_k:=\D^k u,\qquad \D := |\partial_x|^{3/2} + \Sigma_\g
\end{align}
Below we derive the equation satisfied by $W$.

\begin{proposition}\label{proeqW}
Let $u$ be the solution of \eqref{equ}--\eqref{equN}, and let $W = W_k$ be defined by \eqref{Wk}. Let $N:=3k/2$.
Then we have
\begin{align}
\label{eqW}
\partial_t W - i |\partial_x|^{3/2} W - i \Sigma_\gamma (W) =  \mathcal{Q}_W + |\pa_x|^{N} \mathcal{N}_W + \mathcal{O}_W
\end{align}
where $\Sigma_\gamma$ is as in \eqref{equSigma}, and the nonlinearities are
\begin{align}
\label{eqWN}
\what{\mathcal{Q}_W}(\xi) & := \frac{1}{2\pi} \int_\R  q_{N}(\xi,\eta)\what{V}(\xi-\eta) \what{W}(\eta) \, d\eta,\qquad q_{N}(\xi,\eta):=- i\xi \frac{|\xi|^N}{|\eta|^N} \chi(\xi-\eta,\eta),
\end{align}
and
\begin{align}
\label{eqWN2}
 \mathcal{N}_W & = \sum_{\star}\big[A_{\eps_1\eps_2}\big( u_{\eps_1}, |\pa_x|^{-N}W_{\eps_2} \big)+B_{\eps_1\eps_2}\big( u_{\eps_1}, u_{\eps_2} \big)\big].
\end{align}
The cubic nonlinearity $\mathcal{O}_W$ satisfies the bounds
\begin{equation}\label{cubboundW}
\begin{split}
\langle t\rangle^{1-p_0}\|\mathcal{O}_W(t)\|_{\dot{H}^{N_0-N,-1/2}}+\langle t\rangle^{1-4p_0}\|S\mathcal{O}_W(t)\|_{\dot{H}^{N_1-N,-1/2}}&\lesssim \e_1^3,\qquad\text{ if }N\leq N_1,\\
\langle t\rangle^{1-p_0}\|\mathcal{O}_W(t)\|_{\dot{H}^{N_0-N,-1/2}}&\lesssim \e_1^3,\qquad\text{ if }N\in[N_1,N_0].
\end{split}
\end{equation}
\end{proposition}


\begin{proof}
The starting point is the equation \eqref{equ}. Applying $\D^k$ to this equation we see that
\begin{align*}
\begin{split}
& \pa_t \D ^k u -i  \D \D^k u = \big[ \pa_t, \D^k \big] u
  - \D^k\partial_x T_V  u   + \D^k \mathcal{N}_u +\D^k|\partial_x|^{1/2}O_{3,1/2}.
\end{split}
\end{align*}
Since $W = \D^k u$, we can rewrite this equation as
\begin{align}
\label{eqW5}
\begin{split}
\partial_t W - i \D W  & =
  - |\pa_x|^N \pa_x T_V |\pa_x|^{-N} W + |\partial_x|^N \mathcal{N}_W
\\
& + \big[ \pa_t, \D^k \big] u + \big[ \partial_x T_V, |\partial_x|^N \big ] (u - |\pa_x|^{-N}W)
  + \big[ \partial_x T_V, \D^k - |\partial_x|^N \big ] u
  \\
  & + (\D^k - |\partial_x|^N ) \mathcal{N}_u + |\partial_x|^N (\mathcal{N}_u - \mathcal{N}_W)+\D^k|\partial_x|^{1/2}O_{3,1/2} .
\end{split}
\end{align}
To prove the proposition it suffices to show that all the terms in the last two lines of \eqref{eqW5} satisfy the cubic bounds \eqref{cubboundW}.
This is proved in Proposition \ref{CubicSummary}.
\end{proof}

Define the weighted variable
\begin{align}
\label{defZ}
Z_k := S \D^{k} u , \qquad  \D = |\partial_x|^{3/2} + \Sigma_\g,\qquad k\in[0,2N_1/3],
\end{align}
where $S=(3/2)t\partial_t+x\partial_x$. The next lemma gives the evolution equation for the variables $Z_k$.

\begin{proposition}\label{lemeqZ}
Assume $k\in[0,2N_1/3]\cap\mathbb{Z}$, $N=3k/2$, and let $Z=Z_k$. Then we have
\begin{align}
\label{eqZ}
\pa_t Z - i |\pa_x|^{3/2} Z  - i \Sigma_{\gamma} Z = \mathcal{Q}_Z + \mathcal{N}_{Z,1}
  + \mathcal{N}_{Z,2} + \mathcal{N}_{Z,3} + \mathcal{O}_Z ,
\end{align}
where the quasilinear quadratic nonlinearity $\mathcal{Q}_Z$ is given by
\begin{align}\label{eqZN1}
\mathcal{Q}_Z := Q_{N}(V,Z) &
  = \mathcal{F}^{-1} \Big[\frac{1}{2\pi} \int_{\R} q_{N} (\xi,\eta) \what{V}(\xi-\eta) \what{Z}(\eta) \,d\eta\Big],
\end{align}
and $q_N$ is as in \eqref{eqWN}. The quadratic semilinear terms are given by
\begin{align}
\label{eqZN2}
\begin{split}
\mathcal{N}_{Z,1} & := |\pa_x|^{N} \sum_{\star} A_{\eps_1\eps_2} (u_{\eps_1}, |\pa_x|^{-N} Z_{\eps_2}) ,
\\
\mathcal{N}_{Z,2} & := (i/2) Q_{N} (\partial_x|\partial_x|^{-1/2}(Su - \bar{Su}), |\pa_x|^{N} u)
  \\
  & + |\pa_x|^{N} \sum_{\star} \big[ A_{\eps_1\eps_2} (Su_{\eps_1}, u_{\eps_2})
  + B_{\eps_1\eps_2} (S u_{\eps_1}, u_{\eps_2}) + B_{\eps_1\eps_2} (u_{\eps_1}, Su_{\eps_2}) \big] ,
\\
\mathcal{N}_{Z,3} &:= 
  (i/2) Q_{N} \big(\partial_x|\partial_x|^{-1/2} (u-\bar{u}), |\pa_x|^{N}  u\big)
\\
& + (i/2) \widetilde{Q}_{N} \big(\partial_x|\partial_x|^{-1/2} (u-\bar{u}), |\pa_x|^{N} u\big) + |\pa_x|^{N} (3/2-N)\mathcal{N}_u
\\
  & + |\pa_x|^{N} \sum_{\star} \big[ N A_{\eps_1\eps_2} (u_{\eps_1}, u_{\eps_2})
  + \widetilde{A}_{\eps_1\eps_2} (u_{\eps_1}, u_{\eps_2}) +  \widetilde{B}_{\eps_1\eps_2} (u_{\eps_1}, u_{\eps_2}) \big].
\end{split}
 \end{align}
Here we are using the definition \eqref{lemcomm}-\eqref{lemcomm2} for a bilinear operator $\widetilde{M}$
with symbol $\widetilde{m}$. The remainder term $\mathcal{O}_Z$ is cubic and satisfies
\begin{align}
\label{eqZR}
{\| \mathcal{O}_Z(t) \|}_{\dot{H}^{0,-1/2}} \lesssim \e_1^3 \langle t\rangle^{-1+4p_0} .
\end{align}
\end{proposition}

\begin{proof} Using \eqref{eqW} we have
\begin{align}
\label{NeqZ21}
\begin{split}
\partial_t W_k - i |\partial_x|^{3/2} W_k - i \Sigma_\gamma W_k
  = Q_{N} (V, W_k) + |\pa_x|^{N} \mathcal{N}_{u} + \mathcal{O}_{W_{k}} ,
\end{split}
\end{align}
with $Q_{N_1}$ and $q_{N_1}$ is in \eqref{eqZN1}, and with a remainder $\mathcal{O}_{W_{k_1}}$ satisfying \eqref{cubboundW}. Notice that $[S, \partial_t - i\Lambda] = - (3/2) (\partial_t - i \Lambda)$, where $\Lambda=|\partial_x|^{3/2}$.
Therefore, applying $S$ to \eqref{NeqZ21}, and commuting it with the left-hand side, we obtain
\begin{align}
\label{NeqZ22}
\begin{split}
\partial_t Z - i |\partial_x|^{3/2} Z - i \Sigma_\gamma Z
  &= S Q_{N} (V, \mathcal{D}^{k} u) + S |\pa_x|^{N} \mathcal{N}_u\\
  &+ (3/2) (\partial_t-i\Lambda)\mathcal{D}^{k}u+ i[S,\Sigma_{\gamma}]\mathcal{D}^{k} u+S\mathcal{O}_{W_{k}}.
\end{split}
\end{align}
Recall also the formulas \eqref{Vu+baru0} and
\begin{align*}
 S Q_{N} (V, \mathcal{D}^{k} u) = Q_{N} (V, Z) + Q_{N} (S V, \mathcal{D}^{k} u) + \widetilde{Q}_{N} (V, \mathcal{D}^{k} u) .
\end{align*}
Using also the commutation identities
\begin{equation*}
[S,\partial_x|\partial_x|^{-1/2}] = -(1/2)\partial_x|\partial_x|^{-1/2},\qquad [S,|\partial_x|^{N}] = -N|\partial_x|^{N},
\end{equation*}
it follows that
\begin{equation}\label{NeqZ32}
\begin{split}
S |\pa_x|^{N} \mathcal{N}_u &=
  - {N}|\pa_x|^{N} \mathcal{N}_u + |\pa_x|^{N} \sum_\star \big[ A_{\eps_1\eps_2}(u_{\eps_1}, Su_{\eps_2}) + A_{\eps_1\eps_2}(Su_{\eps_1}, u_{\eps_2})\\
  &+  \widetilde{A}_{\eps_1\eps_2} (u_{\eps_1}, u_{\eps_2})+  B_{\eps_1\eps_2}(Su_{\eps_1}, u_{\eps_2})
  + B_{\eps_1\eps_2}(u_{\eps_1}, Su_{\eps_2}) + \widetilde{B}_{\eps_1\eps_2} (u_{\eps_1},u_{\eps_2}) \big].
\end{split}
\end{equation}
Notice that some of these terms can all found in $\mathcal{N}_{Z,2}$ and $\mathcal{N}_{Z,3}$.

The desired formula \eqref{eqZ} follows from \eqref{NeqZ22} provided that
\begin{equation}\label{NeqZ26}
\begin{split}
\mathcal{O}_Z:&=i[S,\Sigma_{\gamma}]\mathcal{D}^{k} u+S\mathcal{O}_{W_{k}}+ \Big[\widetilde{Q}_{N} (V, \mathcal{D}^{k} u)-\frac{i}{2}\widetilde{Q}_{N}\big(\partial_x|\partial_x|^{-1/2}(u-\overline{u}),|\partial_x|^{N}u\big)\Big]\\
&+ \Big[\frac{3}{2}Q_{N} (V, \mathcal{D}^{k} u)-\frac{3i}{4}Q_{N}\big(\partial_x|\partial_x|^{-1/2}(u-\overline{u}),|\partial_x|^{N}u\big)\Big]\\
&+\Big[Q_{N} (S V, \mathcal{D}^{k} u)-\frac{i}{2}Q_{N}\big(\partial_x|\partial_x|^{-1/2}(Su-\overline{Su}),|\partial_x|^{N}u\big)\\
&+\frac{i}{4}Q_{N}\big(\partial_x|\partial_x|^{-1/2}(u-\overline{u}),|\partial_x|^{N}u\big)\Big]+\frac{3i}{2}\Sigma_\gamma\mathcal{D}^{k}u+\frac{3}{2}\Big[\mathcal{O}_{W_{k}}+|\partial_x|^N(\mathcal{N}_{W_k}-\mathcal{N}_u)\Big]\\
&+|\partial_x|^{N}\sum_{\star}\Big[A_{\eps_1\eps_2}(u_{\eps_1},Su_{\eps_2})-A_{\eps_1\eps_2}(u_{\eps_1},|\partial_x|^{-N}Z_{\eps_2})-NA_{\eps_1\eps_2}(u_{\eps_1},u_{\eps_2})\Big].
\end{split}
\end{equation}
The elliptic cubic bound \eqref{eqZR} is verified in Proposition \ref{CubicSummary}.
\end{proof}

\section{Energy estimates I: high Sobolev estimates}\label{secEE}

In this section we prove the following main proposition:

\begin{proposition}\label{proEE1}
Assume that $u$ satisfies \eqref{bing11}--\eqref{bing31}. Then
\begin{equation}\label{proEE1conc}
 \sup_{t\in[0,T]} (1+t)^{-p_0} {\|P_{\geq -20}u(t)\|}_{H^{N_0}} \lesssim \e_0 .
\end{equation}
\end{proposition}

\subsection{The higher order energy functional}
Let $W =  \D^{k_0} u$, $N_0=3k_0/2$, $\mathcal{D}=|\partial_x|^{3/2}+\Sigma_\gamma$, and recall the equations \eqref{equ} and  \eqref{eqW} for $u$ and $W$,
\begin{align}
\label{eqW10}
\begin{split}
& \partial_t u - i |\partial_x|^{3/2} u - i \Sigma_\gamma u = - \partial_x T_V u  + \mathcal{N}_u + |\partial_x|^{1/2}O_{3,1/2},
\\
& \qquad  V = -\frac{1}{2i} \partial_x |\partial_x|^{-1/2} (u-\bar{u}) + O_{2,-1/2} ,
\\
& \partial_t W - i |\partial_x|^{3/2} W - i \Sigma_\gamma W = \mathcal{Q}_W + |\pa_x|^{N_0} \mathcal{N}_W + \mathcal{O}_W
\\
& \qquad  \Sigma_\gamma = T_\gamma P_{\geq 1}|\partial_x|^{3/2}
  - \frac{3}{4} T_{\partial_x \gamma} P_{\geq 1} |\partial_x|^{-1/2} \partial_x
  \quad , \qquad \g = \sqrt{1+\sigma}-1 ,
\end{split}
\end{align}
where $\mathcal{N}_u$ is defined in \eqref{equN}, $\mathcal{Q}_W$ is in \eqref{eqWN},
$\mathcal{N}_W$ is defined in \eqref{eqWN2} together with \eqref{a_++}--\eqref{a_--}, \eqref{b++}--\eqref{b--},
and $\mathcal{O}_W$ satisfies the cubic bounds \eqref{cubboundW}. Notice that 
\begin{align}
\label{reluW}
u =  |\pa_x|^{-N_0} W + O_{3,0} ,
\end{align}
so that, using \eqref{bing11}, for any $t\in[0,T]$,
\begin{align}
\label{uWHN_0}
\begin{split}
{\| P_{\geq -20} u(t) \|}_{H^{N_0}} &\lesssim {\| W(t) \|}_{L^2} + \e_1^{3/2} ,
\\
{\| W(t) \|}_{L^2}+\|u\|_{H^{N_0}} &\lesssim \e_1(1+t)^{p_0}.
\end{split}
\end{align}

We define the quadratic energy functional associated to the second equation in \eqref{eqW10} by
\begin{align}
\label{E_N^2}
E_{N_0}^{(2)} (t) = \frac{1}{2} \int_{\R} {|W(x,t)|}^2 \,dx 
  = \frac{1}{4\pi} \int_{\R} \what{W}(\xi,t) \bar{\what{W}}(\xi,t) 
  \, d\xi .
\end{align}
Based on the equation \eqref{eqW10} we define the following cubic energy functional:
\begin{align}
\label{E^31}
\begin{split}
E^{(3)}_{m_{N_0}} (t) & := \frac{1}{4\pi^2} \Re \int_{\R \times\R} \bar{\what{W}}(\xi,t) \what{W}(\eta,t)
  m_{N_0}(\xi,\eta) \what{u}(\xi-\eta,t) \,d\xi d\eta ,
\\
m_{N_0}(\xi,\eta) & := \frac{(\xi-\eta)
  \big[ \xi |\xi|^{N_0} |\eta|^{-{N_0}} \chi(\xi-\eta,\eta) - \eta |\eta|^{N_0} |\xi|^{-{N_0}} \chi(\eta-\xi,\xi) \big]
  }{2|\xi-\eta|^{1/2}(|\xi|^{3/2}-|\xi-\eta|^{3/2}-|\eta|^{3/2})} .
\end{split}
\end{align}
Given the symbols $a_{\eps_1\eps_2}$ and $b_{\eps_1\eps_2}$ in \eqref{a_++}-\eqref{b--}, we also define the cubic functionals
\begin{align}
\label{E^3a}
\begin{split}
& E^{(3)}_{a,\epsilon_1\epsilon_2} (t) := \frac{1}{4\pi^2} \Re \int_{\R \times\R} \bar{\what{W}}(\xi,t)
  \what{W_{\epsilon_2}}(\eta,t) \what{u_{\epsilon_1}}(\xi-\eta,t) \, a_{\eps_1\eps_2}^{N_0}(\xi,\eta) \,d\xi d\eta ,
\\
& a_{\eps_1\eps_2}^{N_0}(\xi,\eta) := \frac{ -i|\xi|^{N_0}
  |\eta|^{-{N_0}} a_{\eps_1\eps_2}(\xi,\eta) }{|\xi|^{3/2} -\eps_2|\eta|^{3/2} -\eps_1|\xi-\eta|^{3/2}}
\end{split}
\end{align}
and
\begin{align}
\label{E^3b}
\begin{split}
& E^{(3)}_{b,\epsilon_1\epsilon_2}(t) := \frac{1}{4\pi^2} \Re \int_{\R \times\R} \bar{\what{W}}(\xi,t)
  \what{u_{\epsilon_2}}(\eta,t) \what{u_{\epsilon_1}}(\xi-\eta,t) \, b_{\eps_1\eps_2}^{N_0}(\xi,\eta) \,d\xi d\eta ,
\\
& b_{\eps_1\eps_2}^{N_0}(\xi,\eta) := \frac{ -i|\xi|^{N_0}
  b_{\eps_1\eps_2}(\xi,\eta) }{|\xi|^{3/2} -\eps_2|\eta|^{3/2} -\eps_1|\xi-\eta|^{3/2}} ,
\end{split}
\end{align}
where, for any function $f$, we use the notation $f_+ := f$, $f_- := \bar{f}$.

Then the cubic correction to the energy is given by
\begin{align}
\label{E_N^3}
E_{N_0}^{(3)} (t) & := E^{(3)}_{m_{N_0}} (t)
  + \sum_{\star} \Big( E^{(3)}_{a,\epsilon_1\epsilon_2}(t) + E^{(3)}_{b,\epsilon_1\epsilon_2}(t) \Big) ,
\end{align}
and the total energy is
\begin{align}
\label{E_N}
E_{N_0}(t) := E_{N_0}^{(2)} (t) + E_{N_0}^{(3)} (t) .
\end{align}

Proposition \ref{proEE1} will follow from the two lemmas below:

\begin{lemma}\label{lemEE1}
Assuming the bounds \eqref{bing11}--\eqref{bing31}, for any $t\in[0,T]$ we have
\begin{equation}\label{lemEE1conc}
| E_{N_0}^{(3)}(t) | \lesssim \e_1^3 {(1+t)}^{2p_0} . 
\end{equation}
\end{lemma}
The above lemma essentially establishes the equivalence of $E_{N_0}$ and $E_{N_0}^{(2)}$ at every fixed time slice.
The next lemma provides improved control on the increment of $E_{N_0}$.

\begin{lemma}\label{lemEE2}
Assuming the bounds \eqref{bing11}--\eqref{bing31}, for any $t\in[0,T]$ we have
\begin{equation}\label{lemEE2conc}
\frac{d}{dt} E_{N_0}(t) \lesssim \e_1^4 {(1+t)}^{-1+2p_0} .
\end{equation}
\end{lemma}

\begin{proof}[Proof of Proposition \ref{proEE1}]

Using \eqref{E_N} and \eqref{lemEE2conc}, we see that
\begin{equation*}
 \big| E^{(2)}_{N_0}(t) + E^{(3)}_{N_0}(t) \big| \leq \big| E^{(2)}_{N_0}(0) + E^{(3)}_{N_0}(0) \big|
  + \int_0^t \e_1^3 {(1+s)}^{-1+2p_0} \,ds
\end{equation*}
for any $t \in [0,T]$.
In view of \eqref{lemEE1conc} we then have
\begin{equation*}
 E^{(2)}_{N_0}(t) \lesssim E^{(2)}_{N_0}(0) + \e_1^3(1+t)^{2p_0}\lesssim \e_1^3(1+t)^{2p_0},
\end{equation*}
for any $t\in[0,T]$, and the desired conclusion follows using also \eqref{uWHN_0}.
\end{proof}

\vskip10pt
\subsection{Analysis of the symbols and proof of Lemma \ref{lemEE1}}\label{secEsym}

In order to prove Lemmas \ref{lemEE1} and \ref{lemEE2},
we need to establish bounds on the symbols of the cubic energy functionals in \eqref{E^31}, \eqref{E^3a} and \eqref{E^3b}.
With the definition \eqref{Sinfty}, inspecting the symbol $m_{N_0}$ in \eqref{E^31},
using \eqref{1/phi1}, and standard integration by parts, one can see that
\begin{align}
\label{boundm_N}
& {\| m_{N_0}^{k,k_1,k_2} \|}_{S^\infty} \lesssim 2^{k_1/2} 2^{-k/2} \mathbf{1}_{\mathcal{X}}(k,k_1,k_2)
  \mathbf{1}_{[6,\infty)}(k_2 - k_1) .
\end{align}

Looking at the definition of the symbols in \eqref{E^3a}, and using the bounds \eqref{bounda/1} in Lemma \ref{lembounda},
we see that
\begin{align}
\label{bounda^N/}
\begin{split}
{\big\| (a_{\epsilon_1 +}^{N_0})^{k,k_1,k_2}\big\|}_{S^\infty}
  & \lesssim 2^{k_1/2} 2^{-k/2} \mathbf{1}_{\mathcal{X}}(k,k_1,k_2) \mathbf{1}_{[6,\infty)}(k_2 - k_1) ,
\\
{\big\| (a_{\epsilon_1 -}^{N_0})^{k,k_1,k_2}\big\|}_{S^\infty}
  & \lesssim \big( 2^{3k_1/2} 2^{-3k/2}\mathbf{1}_{[2,\infty)}(k) + 2^{k_1/2} 2^{-k/2}\mathbf{1}_{(-\infty,1]}(k) \big)
  \\ & \qquad \times \mathbf{1}_{\mathcal{X}}(k,k_1,k_2) \mathbf{1}_{[6,\infty)}(k_2 - k_1) ,
\end{split}
\end{align}
while \eqref{boundb/1} and \eqref{E^3b} give
\begin{align}
\label{boundb^N/}
& {\big\| (b_{\epsilon_1\epsilon_2}^{N_0})^{k,k_1,k_2}\big\|}_{S^\infty}
  \lesssim 2^{(N_0+1/2)k} 2^{-k_2/2} \mathbf{1}_{\mathcal{X}}(k,k_1,k_2) \mathbf{1}_{[-15,15]}(k_2 - k_1) .
\end{align}

We now apply Lemma \ref{touse}(ii), together with the bounds established above, to prove \eqref{lemEE1conc}.
Using Lemma \ref{touse}(ii), and the notation \eqref{P'_k}, we can estimate
the term in \eqref{E^31} as follows:
\begin{align*}
\big| E_{m_{N_0}}^{(3)}(t) \big| & \lesssim
  \sum_{k,k_1,k_2 \in \mathbb{Z}} {\| m_{N_0}^{k,k_1,k_2} \|}_{S^\infty}
  {\| P_k^\prime W \|}_{L^2} {\| P_{k_2}^\prime W \|}_{L^2} {\| P_{k_1}^\prime u \|}_{L^\infty} .
\end{align*}
Using the bound \eqref{boundm_N} on $m_{N_0}$, the assumptions in \eqref{bing11}--\eqref{bing31}, and \eqref{reluW},
we obtain
\begin{align*}
\big| E_{m_{N_0}}^{(3)}(t) \big| \lesssim
  \sum_{ (k,k_1,k_2) \in \mathcal{X}, \, |k-k_2| \leq 5} 2^{k_1/2} 2^{-k/2}
  {\| P_k^\prime W \|}_{L^2} {\| P_{k_2}^\prime W \|}_{L^2}
  \e_1 2^{k_1/10} 2^{-k_1^+} \langle t\rangle^{-1/2}\lesssim \e_1^3 \langle t\rangle^{-1/3},
\end{align*}
where $x^+=\max(x,0)$. The cubic energies \eqref{E^3a} can be dealt with in an identical fashion, since the bounds \eqref{bounda^N/}
on their symbols are analogous to the one for $m_{N_0}$ in \eqref{boundm_N}.
The cubic corrections in \eqref{E^3b} can also be treated similarly.
We use Lemma \ref{touse}(ii), \eqref{boundb^N/}, the a priori assumptions in \eqref{bing11}--\eqref{bing31} and \eqref{reluW}, to obtain,
for all $\eps_1,\eps_2 \in \{+,-\}$,
\begin{align*}
\big| E_{b_{\eps_1\eps_2}}^{(3)}(t) &\big|  \lesssim
  \sum_{k,k_1,k_2 \in \mathbb{Z}} {\| b_{\eps_1\eps_2}^{k,k_1,k_2} \|}_{S^\infty}
  {\| P_k^\prime W \|}_{L^2} {\| P_{k_2}^\prime W \|}_{L^2} {\| P_{k_1}^\prime u \|}_{L^\infty} .
\\
& \lesssim \sum_{ (k,k_1,k_2) \in \mathcal{X}, \, |k_1-k_2| \leq 10} 2^{(N_0+1/2)k} 2^{-k_1/2}
  {\| P_k^\prime W \|}_{L^2} {\| P_{k_2}^\prime u \|}_{L^2}
  \e_1 2^{k_1/10} 2^{-k_1^+} \langle t\rangle^{-1/2}\\
	&\lesssim \e_1^3 \langle t\rangle^{-1/3}.
\end{align*}
This concludes the proof of \eqref{lemEE1conc}.

\vskip10pt
\subsection{Proof of Lemma \ref{lemEE2}}
Using the equation for $W$ in \eqref{eqW10} we can calculate
\begin{align*}
\begin{split}
\frac{d}{dt} E_{N_0}^{(2)}(t)
  & = \frac{1}{2\pi} \Re \int_{\R} \bar{\what{W}}(\xi,t) \partial_t \what{W}(\xi,t) \,d\xi 
  = A_1(t) + A_2(t) + A_3(t) + A_4(t) ,
\end{split}
\end{align*}
where
\begin{align}
\label{EE21}
A_1 & := \frac{1}{2\pi} \Re \int_{\R} \bar{\what{W}}(\xi) \, \what{i \Sigma_\g W}(\xi) \,d\xi,
\\
\label{EE22}
A_2 & := \frac{1}{4\pi^2} \Re \int_{\R\times\R} \bar{\what{W}}(\xi)
  \big( -i\xi |\xi|^{N_0} |\eta|^{-N_0} \what{V}(\xi-\eta) \chi(\xi-\eta,\eta) \big) \what{W}(\eta)\,d\xi d\eta ,
\\
\label{EE23}
A_3 & := \frac{1}{2\pi} \Re \int_{\R} \bar{\what{W}}(\xi) \, |\xi|^{N_0} \what{\mathcal{N}_W}(\xi) \,d\xi ,
\\
\label{EE24}
A_4 & := \frac{1}{2\pi} \Re \int_{\R} \bar{\what{W}}(\xi) \what{\mathcal{O}_W}(\xi) \,d\xi .
\end{align}

All cubic contributions coming from the above integrals are matched,
up to acceptable quartic remainder terms, with the contributions from the time evolution of $E_{N_0}^{(3)}$,
see \eqref{E_N^3} and \eqref{E^31}-\eqref{E^3b}.
This fact is established through the following series of lemmas,
which will also prove the desired estimate \eqref{lemEE2conc}.

\begin{lemma}\label{lemEE21}
Under the a priori assumptions \eqref{bing11}--\eqref{bing31}, we have
\begin{align}
\big| A_2(t) +  \frac{d}{dt} E^{(3)}_{m_{N_0}}(t) \big| \lesssim \e_1^4 {(1+t)}^{-1+2p_0} .
\end{align}
\end{lemma}

\begin{lemma}\label{lemEE22}
Under the a priori assumptions \eqref{bing11}--\eqref{bing31}, we have
\begin{align}
\Big| A_3(t) +  \frac{d}{dt} \sum_{\star}
  \big( E^{(3)}_{a,\epsilon_1\epsilon_2}(t) + E^{(3)}_{b,\epsilon_1\epsilon_2}(t) \big) \Big| \lesssim \e_1^4 {(1+t)}^{-1+2p_0} .
\end{align}
\end{lemma}

\begin{lemma}\label{lemEE23}
Under the a priori assumptions \eqref{bing11}--\eqref{bing31}, we have
\begin{align}
| A_1(t) +  A_4(t) | \lesssim \e_1^4 {(1+t)}^{-1+2p_0} .
\end{align}
\end{lemma}

The rest of this section is concerned with the proofs of these lemmas.

\vskip10pt
\subsubsection{Proof of Lemma \ref{lemEE21}} We start by symmetrizing the term $A_2$ in \eqref{EE22} using the fact that $V$ is real-valued,
\begin{align}
\label{EE211}
\begin{split}
A_2 = \frac{1}{8\pi^2} \Re \int_{\R\times\R} & \bar{\what{W}}(\xi) \what{W}(\eta) \what{V}(\xi-\eta)
  \\
  & \times \big( -i\xi |\xi|^{N_0}|\eta|^{-N_0} \chi(\xi-\eta,\eta) + i\eta |\eta|^{N_0}|\xi|^{-N_0} \chi(\eta-\xi,\xi) \big) \,d\xi d\eta .
\end{split}
\end{align}
Recall, see \eqref{Vu+baru0} and Definition \ref{Oterms}, that
\begin{align}
\label{Vu+baru1}
V 
  = -\frac{1}{2i} \partial_x |\partial_x|^{-1/2} (u-\overline{u}) + V_2 , \qquad V_2 = O_{2,-1/2} .
\end{align}
Thus, we can write $A_2 = A_{2,1} + A_{2,2}$ where
\begin{align}
\begin{split}
\label{A_2}
A_{2,1} & := \frac{1}{4\pi^2} \Re \int_{\R\times\R} \bar{\what{W}}(\xi) \what{W}(\eta) \what{u}(\xi-\eta) q_{N_0}(\xi,\eta) \,d\xi d\eta
\\
& q_{N_0}(\xi,\eta) := \frac{i(\xi-\eta)}{2|\xi-\eta|^{1/2}} \Big[ \frac{\xi|\xi|^{N_0}}{|\eta|^{N_0}} \chi(\xi-\eta,\eta)
       - \frac{ \eta|\eta|^{N_0}}{|\xi|^{N_0}} \chi(\eta-\xi,\xi) \Big] ,
\end{split}
\end{align}
and
\begin{align}
\label{A_22}
\begin{split}
A_{2,2} & := \frac{1}{8\pi^2} \Re \int_{\R\times\R}
  \bar{\what{W}}(\xi) \what{W}(\eta) \what{V_2}(\xi-\eta) a_{2,2}(\xi,\eta) \,d\xi d\eta ,
\\
& a_{2,2}(\xi,\eta) := -\frac{i\xi |\xi|^{N_0}}{|\eta|^{N_0}} \chi(\xi-\eta,\eta)
    + \frac{i\eta |\eta|^{N_0}}{|\xi|^{N_0}} \chi(\eta-\xi,\xi) .
\end{split}
\end{align}

According to \eqref{A_2}, the symbol of $E^{(3)}_{m_{N_0}}$ in \eqref{E^31} is
\begin{align}
\label{E^31again}
m_{N_0}(\xi,\eta) = \frac{q_{N_0}(\xi,\eta) }{i(|\xi|^{3/2}-|\xi-\eta|^{3/2}-|\eta|^{3/2})} .
\end{align}
Using the equation \eqref{eqW10} we can calculate
\begin{align*}
\begin{split}
\frac{d}{dt} E^{(3)}_{m_{N_0}} & = I_1 + I_2 + I_3 + I_4
\end{split}
\end{align*}
where
\begin{align}
\label{evolE311}
I_1 & := \frac{1}{4\pi^2} \Re \int_{\R \times\R} \bar{\what{W}}(\xi) \what{W}(\eta) \what{u}(\xi-\eta)
  \big[ -i|\xi|^{3/2} + i|\xi-\eta|^{3/2} + i|\eta|^{3/2} \big] m_{N_0}(\xi,\eta)  \,d\xi d\eta ,
\\
\label{evolE312}
I_2 & := \frac{1}{4\pi^2} \Re \int_{\R \times\R} \Big[ \bar{\what{i\Sigma_\g W}}(\xi) \what{W}(\eta)
  + \bar{\what{W}}(\xi) \what{i\Sigma_\g W}(\eta) \Big] \what{u}(\xi-\eta) m_{N_0}(\xi,\eta)  \,d\xi d\eta ,
\\
\label{evolE313}
I_3 & := \frac{1}{4\pi^2} \Re \int_{\R \times\R} \Big[ \bar{\what{\mathcal{Q}_W}}(\xi) \what{W}(\eta)
  + \bar{\what{W}}(\xi) \what{\mathcal{Q}_W}(\eta) \Big] \what{u}(\xi-\eta) m_{N_0}(\xi,\eta)  \,d\xi d\eta ,
\end{align}
and
\begin{align}
\label{evolE314}
\begin{split}
I_4 & := \frac{1}{4\pi^2} \Re \int_{\R \times\R} \Big[ |\xi|^{N_0} \bar{\what{\mathcal{N}_W}}(\xi) \what{W}(\eta)
  + \bar{\what{W}}(\xi) |\eta|^{N_0} \what{\mathcal{N}_W}(\eta) \Big] \what{u}(\xi-\eta) m_{N_0}(\xi,\eta)  \,d\xi d\eta
\\
& + \frac{1}{4\pi^2} \Re \int_{\R \times\R} \bar{\what{W}}(\xi) \what{W}(\eta)
  \mathcal{F}(\pa_t u - i |\pa_x|^{3/2} u)(\xi-\eta) m_{N_0}(\xi,\eta)  \,d\xi d\eta
\\
& + \frac{1}{4\pi^2} \Re \int_{\R \times\R} \Big[ \bar{\what{\mathcal{O}_W}}(\xi) \what{W}(\eta)
  + \bar{\what{W}}(\xi) \what{\mathcal{O}_W}(\eta) \Big] \what{u}(\xi-\eta) m_{N_0}(\xi,\eta)  \,d\xi d\eta  .
\end{split}
\end{align}
Using \eqref{A_2}, \eqref{E^31again} and \eqref{evolE311} we see that $A_{2,1} + I_1 = 0$, and, therefore,
\begin{align*}
A_2 + \frac{d}{dt} E_{m_{N_0}}^{(3)} = A_{2,2} + I_2 + I_3 + I_4.
\end{align*}
It then suffices to show that
\begin{align}
\label{estI_234}
| A_{2,2}(t) |+ |I_2(t)| + |I_3(t)| + |I_4(t)| \lesssim \e_1^4 {(1+t)}^{-1+2p_0}.
\end{align}

\vskip10pt
\subsubsection*{Estimate of $A_{2,2}$}
Using integration by parts one can see that the symbol in \eqref{A_22} satisfies
\begin{align*}
a_{2,2}(\xi,\eta) 
  = O \big( |\xi-\eta| \mathbf{1}_{[2^{-5},2^5]}(|\eta|/|\xi|)\big) ,
\end{align*}
see the notation \eqref{Onot} for bilinear symbols.
Then, using Lemma \ref{touse}(ii) and \eqref{Vu+baru1},
\begin{align}
 \label{estA_22}
\begin{split}
| A_{2,2} | & \lesssim
  \sum_{ k,k_1,k_2 \in \mathbb{Z}} {\| (a_{2,2})^{k,k_1,k_2} \|}_{S^\infty}
  {\| P_k^\prime W \|}_{L^2} {\| P_{k_1}^\prime V_2 \|}_{L^\infty}  {\| P_{k_2}^\prime W \|}_{L^2}
\\
& \lesssim \sum_{ (k,k_1,k_2) \in \mathcal{X}, \, |k-k_2| \leq 10} 2^{k_1}
  {\| P_k^\prime W \|}_{L^2} \e_1^2 2^{-k_1^+} \langle t\rangle^{-1} {\| P_{k_2}^\prime W \|}_{L^2}\\
	&\lesssim \e_1^4 \langle t\rangle^{-1+2p_0}.
\end{split}
\end{align}

\vskip10pt
\subsubsection*{Estimate of $I_2$}
The term $I_2$ in \eqref{evolE312} presents a potential loss of $3/2$ derivatives.
However, exploiting the structure of $\Sigma_\g$ and of the symbol $m_{N_0}$, one can recover this loss.
Recall the definition of $\Sigma_\g$ from \eqref{eqW10}. We can then estimate
\begin{equation*}
\begin{split}
|I_2|& \lesssim \Big| - \int_{\R^3} \what{\gamma}(\rho-\xi) \chi(\xi-\rho,\rho)
  |\rho|^{3/2}\Big(1+\frac{3(\xi-\rho)}{4\rho}\Big) \bar{\what{P_{\geq 1}W}}(\rho) \what{W}(\eta) \what{u}(\xi-\eta) m_{N_0}(\xi,\eta)  \,d\xi d\eta d\rho
\\
  &+ \int_{\R^3} \bar{\what{W}}(\xi) \what{\gamma}(\eta-\rho) \chi(\eta-\rho,\rho)
  |\rho|^{3/2}\Big(1+\frac{3(\eta-\rho)}{4\rho}\Big) \what{P_{\geq 1}W}(\rho) \what{u}(\xi-\eta) m_{N_0}(\xi,\eta)  \,d\xi d\eta d\rho \Big|.
\end{split}
\end{equation*}
After changes of variables, it follows that 
\begin{align}
\label{I21}
\begin{split}
|I_2|& \lesssim \Big| \int_{\R^3} \bar{\what{W}}(\xi) \what{W}(\eta) \what{u}(\xi-\rho) \what{\gamma}(\rho-\eta) m'_{2}(\xi,\rho,\eta)
  \,d\xi d\eta d\rho \Big| ,
\\
& m'_{2}(\xi,\rho,\eta) :=
  - \chi(\eta-\rho, \xi) |\xi|^{3/2} \varphi_{\geq 1}(\xi) m_{N_0}(\xi+\eta-\rho,\eta)\Big(1+\frac{3(\eta-\rho)}{4\xi}\Big)
  \\ & \hskip64pt + \chi(\rho-\eta,\eta) |\eta|^{3/2} \varphi_{\geq 1}(\eta) m_{N_0}(\xi,\rho)\Big(1+\frac{3(\rho-\eta)}{4\eta}\Big).
\end{split}
\end{align}

We then want to establish a bound for the symbol in the above expression, showing that it does not cause any derivatives loss.
Let us write
\begin{align}
\label{I_21sym}
\begin{split}
& m'_2 = n_1 + n_2 + n_3+n_4 ,
\\
& n_1(\xi,\eta,\rho) := \big[ \chi(\rho-\eta,\eta) - \chi(\eta-\rho, \xi) \big]
  |\xi|^{3/2} \varphi_{\geq 1}(\xi) m_{N_0}(\xi+\eta-\rho,\eta) ,
\\
& n_2(\xi,\eta,\rho) := \chi(\rho-\eta,\eta)
  \big[ |\eta|^{3/2}\varphi_{\geq 1}(\eta) - |\xi|^{3/2} \varphi_{\geq 1}(\xi) \big] m_{N_0}(\xi,\rho) ,
\\
& n_3(\xi,\eta,\rho) := \chi(\rho-\eta,\eta) |\xi|^{3/2} \varphi_{\geq 1}(\xi) \big[ m_{N_0}(\xi,\rho) -
  m_{N_0}(\xi+\eta-\rho,\eta) \big] ,
\\
& n_4(\xi,\eta,\rho):=\frac{3(\rho-\eta)}{4\xi}\chi(\eta-\rho, \xi) |\xi|^{3/2} \varphi_{\geq 1}(\xi) m_{N_0}(\xi+\eta-\rho,\eta),
\\
& n_5(\xi,\eta,\rho):=\frac{3(\rho-\eta)}{4\eta}\chi(\rho-\eta,\eta) |\eta|^{3/2} \varphi_{\geq 1}(\eta) m_{N_0}(\xi,\rho).
\end{split}
\end{align}

We will often use the observation
\begin{equation}\label{symprod}
\text{ if }\,\,\,f(\xi,\eta,\rho)=f_1(\xi,\rho)f_2(\eta,\rho)\,\,\,\text{ then }\,\,\,\|\mathcal{F}^{-1}f\|_{L^1(\mathbb{R}^3)}\lesssim \|\mathcal{F}^{-1}f_1\|_{L^1(\mathbb{R}^2)}\|\mathcal{F}^{-1}f_2\|_{L^1(\mathbb{R}^2)}.
\end{equation}
Moreover, since $\gamma\in O_{2,0}$, see \eqref{bvd4} and Definition \ref{Oterms}, we have, for any $l\in\mathbb{Z}$,
\begin{equation}\label{estgamma}
\|P_l\gamma\|_{L^\infty}\lesssim \e_1^2\langle t\rangle ^{-1} 2^{-3l^+},\qquad \|P_l\gamma\|_{L^2}\lesssim \e_1^2\langle t\rangle ^{-1/2+p_0} 2^{-3l^+}.
\end{equation}

Using \eqref{symprod} and the bound \eqref{boundm_N} for $m_{N_0}$, we see that
\begin{align}
\label{boundn_1}
n_1(\xi,\eta,\rho) = O \big(  (|\xi-\rho|^{3/2} +  |\rho-\eta|^{3/2})
  \mathbf{1}_{[2^2,\infty)} (|\eta|/|\rho-\eta|) \mathbf{1}_{[2^2,\infty)} (|\xi|/|\xi-\rho|) \big) .
\end{align}
Here we are using the notation \eqref{Onot2}, with \eqref{lemEE1pr1}-\eqref{lemEE1pr2}.
Then, using Lemma \ref{touse}(iii) with the bound \eqref{boundn_1},
the a priori decay assumption in \eqref{bing11} and \eqref{estgamma}, it is easy to show that
\begin{align}\label{azx1}
\begin{split}
\Big| \int_{\R^3} \bar{\what{W}}(\xi)  \what{W}(\eta) \what{u}(\xi-\rho) \what{\gamma}(\rho-\eta)
  n_1(\xi,\eta,\rho)\,d\xi d\eta d\rho \Big| \lesssim {\| W \|}_{L^2}^2 \e_1^3 {(1+t)}^{-7/6} .
\end{split}
\end{align}
Moreover, the symbols $n_2$, $n_4$, $n_5$ satisfy the same bound \eqref{boundn_1}, so their contributions can also be estimated in the same way.

Finally, we look at $n_3$ and we would like to prove the same symbol bound \eqref{boundn_1}. Recall the definition of $m_{N_0}$ and $q_{N_0}$ from \eqref{E^31again} and \eqref{A_2},
and write
\begin{align}
\label{m_Ndiff}
& m_{N_0} (\xi,\rho) - m_{N_0}(\xi+\eta-\rho,\eta) = -i r_1(\xi,\eta,\rho) - i r_2(\xi,\eta,\rho) ,
\\
\nonumber
& r_1 (\xi,\eta,\rho) := \frac{q_{N_0}(\xi,\rho) - q_{N_0}(\xi+\eta-\rho,\eta)}{|\xi+\eta-\rho|^{3/2}-|\xi-\rho|^{3/2}-|\eta|^{3/2}} ,
\\
\nonumber
& r_2 (\xi,\eta,\rho) := q_{N_0}(\xi,\rho) \Big[ \frac{1}{|\xi|^{3/2}-|\xi-\rho|^{3/2}-|\rho|^{3/2}}
  - \frac{1}{|\xi+\eta-\rho|^{3/2}-|\xi-\rho|^{3/2}-|\eta|^{3/2}} \Big] .
\end{align}
Inspecting the formula \eqref{A_2} we see that, when $|\eta-\rho|\leq 2^{-8}|\eta|$,
\begin{align}
\label{q_Nasy2}
q_{N_0}(\xi,\rho) - q_{N_0}(\xi+\eta-\rho,\eta)= O \big( |\xi-\rho|^{5/2}|\xi|^{-1} \mathbf{1}_{[2^6,\infty)} (|\xi|/|\xi-\rho|)\big),
\end{align}
and, therefore,
\begin{align*}
r_1(\xi,\eta,\rho) =O \big( |\xi-\rho|^{3/2} |\xi|^{-3/2} \mathbf{1}_{[2^2,\infty)} (|\xi|/|\xi-\rho|)\big).
\end{align*}
Moreover, one can directly verify that for $|\eta| \geq 2^6 \max(|\xi-\rho|, |\eta-\rho|)$,
\begin{align*}
\begin{split}
|\xi+\eta-\rho|^{3/2}-|\xi-\rho|^{3/2}-|\eta|^{3/2} - (|\xi|^{3/2}-|\xi-\rho|^{3/2}-|\rho|^{3/2})= O \Big( \frac{|\xi-\rho|^2 + |\rho-\eta|^2}{ |\eta|^{1/2} + |\xi|^{1/2} } \Big) .
\end{split}
\end{align*}
It follows that
\begin{align*}
 r_2(\xi,\eta,\rho) = O \Big( \frac{|\xi-\rho|^{3/2} + |\rho-\eta|^{3/2}}{ |\eta|^{3/2} + |\xi|^{3/2} }
   \mathbf{1}_{[2^2,\infty)} (|\eta|/|\xi-\rho|) \Big) ,
\end{align*}
whenever $|\eta| \geq 2^6 |\eta-\rho|$. The desired bound \eqref{boundn_1} follows for the symbol $n_3$. This shows that $| I_2 | \lesssim \e_1^4 \langle t\rangle^{-1}$,
and concludes the proof of the desired bound \eqref{estI_234} for $I_2$.

For later use, namely to estimate $I_{6,\eps_1+}$ in \eqref{I_6+} and $K_{1,1}$ in \eqref{evolE_Z11}, we record below a slighly more general result that was proved in our analysis.

\begin{lemma}\label{lemSigmabulk}
Consider the expression
\begin{align}
\label{lemSigmabulkI}
I(t) = \int_{\R \times\R} \Big[ \bar{\what{i\Sigma_{\g} F}}(\xi) \what{F}(\eta)
  + \bar{\what{F}}(\xi) \what{i\Sigma_{\g} F}(\eta) \Big] \what{u_\pm}(\xi-\eta) m(\xi,\eta)  \,d\xi d\eta ,
\end{align}
with a symbol of the form
\begin{align}
\label{lemSigmabulkm}
m (\xi,\eta) = \frac{q (\xi,\eta)}{|\xi|^{3/2} \mp |\xi-\eta|^{3/2} - |\eta|^{3/2}}
\end{align}
which is supported on a region where $2^8|\xi-\eta| \leq |\eta|$.
Here $u_+ = u$, $u_- = \bar{u}$, and $\gamma$ is in \eqref{eqW10}, \eqref{symm2}.
Assume that the following two properties hold:
\begin{align}
\label{lemSigmabulkpro1}
& q(\xi,\eta) = O \big( |\xi-\eta|^{3/2} \mathbf{1}_{[2^6,\infty)}(|\eta|/|\xi-\eta|) \big) ,
\end{align}
and, whenever $|\eta| \geq 2^6|\eta-\rho|$ and $|\xi| \geq 2^6|\xi-\rho|$,
\begin{align}
\label{lemSigmabulkpro2}
q(\xi,\rho) - q(\xi+\eta-\rho,\eta)= O \Big( \frac{|\xi-\rho|^{5/2} + {|\eta-\rho|}^{5/2} }{ |\xi| + |\eta| } \Big) .
\end{align}
Then
\begin{align}
\label{lemSigmabulkconc}
|I(t)| \lesssim {\| F \|}_{L^2}^2 \e_1^3{(1+t)}^{-7/6} .
\end{align}
\end{lemma}

\vskip10pt
\subsubsection*{Estimate of $I_3$}
Directly from the definition of $I_3$ in \eqref{evolE313} we have
\begin{align}
\label{I_30}
| I_3 | & \lesssim \Big| \int_{\R \times\R} \Big[ \bar{\what{\mathcal{Q}_W}}(\xi) \what{W}(\eta)
  + \bar{\what{W}}(\xi) \what{\mathcal{Q}_W}(\eta) \Big] \what{u}(\xi-\eta) m_{N_0}(\xi,\eta)  \,d\xi d\eta \Big|
\end{align}
where $\mathcal{Q}_W$ is defined in \eqref{eqWN}. Using \eqref{eqWN} we have
\begin{align*}
\begin{split}
| I_3 | \lesssim  \Big| \int_{\R^3} \Big[ & \xi|\xi|^{N_0} |\rho|^{-N_0} \chi(\xi-\rho,\rho) \bar{\what{V}}(\xi-\rho) \bar{\what{W}}(\rho) \what{W}(\eta)
  \\
  - & \bar{\what{W}}(\xi) \eta|\eta|^{N_0} |\rho|^{-N_0} \chi(\eta-\rho,\rho) \what{V}(\eta-\rho) \what{W}(\rho) \Big]
  \what{u}(\xi-\eta) m_{N_0}(\xi,\eta)  \,d\xi d\eta d\rho \Big| .
\end{split}
\end{align*}
Applying some changes of variables we get
\begin{align*}
|I_3| \lesssim \Big| \int_{\R^3} \bar{\what{W}}(\xi) \what{W}(\eta) \what{u}(\xi-\rho) \what{V}(\rho-\eta)
  & m'_3(\xi,\eta,\rho) \, d\xi d\eta d\rho \Big| ,
\\
m'_3(\xi,\eta,\rho) :=  (\xi+\eta-\rho) |\xi+\eta-\rho|^{N_0} & |\xi|^{-N_0} \chi(\eta-\rho,\xi) m_{N_0}(\xi+\eta-\rho,\eta)
  \\ & - \rho  |\rho|^{N_0} |\eta|^{-N_0} \chi(\rho-\eta,\eta) m_{N_0}(\xi,\rho) .
\end{align*}
Since $V\in O_{1,-1/2}$, see \eqref{bvd1}, using Lemma \ref{touse}(iii) and \eqref{boundm_N} we can estimate
\begin{align}
\label{I_31}
|I_3(t)| \lesssim I_3^\prime(t) + \e_1^2 {(1+t)}^{-1} {\| W\|}_{L^2}^2
\end{align}
where
\begin{align}
 \label{I_3'}
\begin{split}
& I_3^\prime =  \Big| \int_{\R^3} \bar{\what{W}}(\xi) \what{W}(\eta) \what{u}(\xi-\rho) \what{V}(\rho-\eta)
  m''_3(\xi,\eta,\rho) \, d\xi d\eta d\rho \Big| ,
\\
& m''_3(\xi,\eta,\rho) := \rho \big[ \chi(\eta-\rho,\xi) m_{N_0}(\xi+\eta-\rho,\eta) - \chi(\rho-\eta,\eta) m_{N_0}(\xi,\rho) \big] .
\end{split}
\end{align}
The main observation is that the symbol $m''_3$ above has a similar structure to the symbol $n_3$ in \eqref{I_21sym}.
In particular, starting from \eqref{m_Ndiff} and using the property \eqref{q_Nasy2}, it is easy to see that
\begin{align}
\label{m_3'}
 m''_3(\xi,\eta,\rho) = O \Big( \frac{|\xi-\rho|^{3/2} + |\rho-\eta|^{3/2}}{(|\xi|+|\eta|)^{1/2}}
  \mathbf{1}_{[2^2,\infty)} (|\xi|/|\xi-\rho|)  \mathbf{1}_{[2^2,\infty)} (|\eta|/|\rho-\eta|) \Big) .
\end{align}
Using this bound in combination with Lemma \ref{touse}(iii), recalling the $V\in O_{1,-1/2}$, and using the a priori bounds,
we get
\begin{align*}
| I'_3 | &\lesssim \sum_{|k_1-k_2| \leq 5, \, k_1 \geq k_3, k_2 \geq k_4}
  (2^{k_3} + 2^{k_4}) {\|P_{k_1}^\prime W \|}_{L^2} {\| P_{k_2}^\prime W \|}_{L^2}
  {\|P_{k_3}^\prime V \|}_{L^\infty} {\| P_{k_2}^\prime u \|}_{L^\infty}
\\
& \lesssim  {\| W \|}_{L^2}^2
  \sum_{k_3,k_4 \in \mathbb{Z}} (2^{k_3} + 2^{k_4}) \e_1 2^{k_3/10} 2^{-2k_3^+} \langle t\rangle^{-1/2}
  \e_1 2^{k_4/10} 2^{-2k_4^+} \langle t\rangle^{-1/2}
\\
& \lesssim \e_1^4 {\| W \|}_{L^2}^2 \langle t\rangle^{-1} .
\end{align*}
This gives $|I_3| \lesssim \e_1^4 \langle t\rangle^{-1+2p_0}$, which is the desired bound in \eqref{estI_234}.

We conclude this subsection with a more general lemma that follows from the same estimates. This lemma will be used later on to estimate terms like $I_3$, which have a potential loss of one derivative, namely, $I_{7,\eps_1+}$ in \eqref{evolE3a3} and $K_{1,2}$ in \eqref{evolE_Z12}.

\begin{lemma}\label{lemQbulk}
Let $\mathcal{Q}_F$ be defined according to \eqref{eqWN}, and let
\begin{align}
\label{lemQbulkI}
J(t) = \int_{\R \times\R} \Big[ \bar{\what{\mathcal{Q}_F}}(\xi) \what{F}(\eta)
  + \bar{\what{F}}(\xi) \what{\mathcal{Q}_F}(\eta) \Big] \what{u_\pm}(\xi-\eta,t) m(\xi,\eta)  \,d\xi d\eta ,
\end{align}
where the symbol has the form
\begin{align}
\label{lemQbulkm}
m (\xi,\eta) = \frac{q (\xi,\eta)}{|\xi|^{3/2} \mp |\xi-\eta|^{3/2} - |\eta|^{3/2}} ,
\end{align}
and is supported on a region where $2^8|\xi-\eta| \leq |\eta|$.
Assume that
\begin{align}
\label{lemQbulkpro1}
& q(\xi,\eta) = O \big( |\xi-\eta|^{3/2} \mathbf{1}_{[2^6,\infty)}(|\eta|/|\xi-\eta|) \big) ,
\end{align}
and
\begin{align}
\label{lemQbulkpro2}
q(\xi,\rho) & - q(\xi+\eta-\rho,\eta)= O \Big( \frac{|\xi-\rho|^{5/2} + {|\eta-\rho|}^{5/2} }{ |\xi| + |\eta| } \Big) ,
\end{align}
whenever $|\eta| \geq 2^6|\eta-\rho|$ and $|\xi| \geq 2^6|\xi-\rho|$.
Then
\begin{align}
\label{lemQbulkconc}
|J(t)| \lesssim {\| F \|}_{L^2}^2 \e_1^2{(1+t)}^{-1} .
\end{align}
\end{lemma}


\vskip10pt
\subsubsection*{Estimate of $I_4$}
All of the terms in \eqref{evolE314} do not lose derivatives
and are not hard to estimate, given the symbol bound \eqref{boundm_N} on $m_{N_0}$,
the estimates on the nonlinear terms in Lemma \ref{lemboundN_u}, and \eqref{cubboundW}.
We just show how to estimate the term
\begin{align*}
I_{4,1} := \int_{\R \times\R} |\xi|^{N_0} \bar{\what{\mathcal{N}_W}}(\xi) \what{W}(\eta)
  \what{u}(\xi-\eta) m_{N_0}(\xi,\eta)  \,d\xi d\eta  ,
\end{align*}
the other terms being similar or easier.
Applying Lemma \ref{touse}(ii), the symbol bound \eqref{boundm_N},
and using the a priori estimates, and the estimate on the nonlinearity \eqref{boundN_u},
we see that
\begin{align*}
| I_{4,1} | &\lesssim \sum_{k,k_1,k_2 \in \mathbb{Z}} {\| m_{N_0}^{k,k_1,k_2} \|}_{S^\infty}
  2^{N_0k} {\|P_k^\prime \mathcal{N}_W \|}_{L^2} {\| P_{k_2}^\prime W \|}_{L^2} {\| P_{k_1}^\prime u \|}_{L^\infty}
\\
& \lesssim \sum_{(k,k_1,k_2) \in \mathbb{Z}, \, |k-k_2| \leq 5}
  2^{k_1/2} 2^{-k/2} \e_1^2 2^{\min(k,0)} \langle t\rangle^{-1/2+p_0} \e_1 \langle t\rangle^{p_0} \e_1 2^{k_1/10} 2^{-2k_1^+} \langle t\rangle^{-1/2}
\\
& \lesssim \e_1^4 \langle t\rangle^{-1+2p_0}.
\end{align*}
This completes the proof of the bound \eqref{estI_234} and therefore the proof of Lemma \ref{lemEE21}.

\vskip10pt
\subsubsection{Proof of Lemma \ref{lemEE22}}

Recall the definition of $A_3$ in \eqref{EE23} and the definition of $\mathcal{N}_W$ in \eqref{eqWN2}.
Recall our definitions of the energies in $E^{(3)}_{a,\epsilon_1\epsilon_2}$ and $E^{(3)}_{a,\epsilon_1\epsilon_2} (t)$
in \eqref{E^3a}-\eqref{E^3b}, and the notation \eqref{notsum}.
Our aim is to show
\begin{align}
\label{lemEE22conc}
\Big| A_3(t) +  \frac{d}{dt} \sum_{\star}
  \big( E^{(3)}_{a,\epsilon_1\epsilon_2}(t) + E^{(3)}_{b,\epsilon_1\epsilon_2}(t) \big) \Big| \lesssim \e_1^4 {(1+t)}^{-1+2p_0} .
\end{align}

Calculating as in the previous section, using the evolution equations for $W$ in \eqref{eqW10},
we see that for each $(\epsilon_1,\epsilon_2) \in \{ (+,+),(+,-),(-,+),(-,-)\}$
\begin{align*}
\begin{split}
\frac{d}{dt} E^{(3)}_{a, \eps_1\eps_2} & = I_{5,\eps_1\eps_2} + I_{6,\eps_1\eps_2} + I_{7,\eps_1\eps_2} + I_{8,\eps_1\eps_2}
\end{split}
\end{align*}
where
\begin{equation}
\label{evolE3a1}
\begin{split}
I_{5,\eps_1\eps_2} & = \frac{1}{4\pi^2} \Re \int_{\R \times\R} \bar{\what{W}}(\xi) \what{W_{\eps_2}}(\eta) \what{u_{\eps_1}}(\xi-\eta)
  \\ & \hskip100pt \big[ -i|\xi|^{3/2} + i\eps_2|\eta|^{3/2} + i\eps_1|\xi-\eta|^{3/2} \big] a^{N_0}_{\eps_1\eps_2}(\xi,\eta)  \,d\xi d\eta,
\end{split}
\end{equation}
\begin{equation}
\label{evolE3a2}
I_{6,\eps_1\eps_2} = \frac{1}{4\pi^2} \Re \int_{\R \times\R} \Big[ \bar{\what{i\Sigma_\g W}}(\xi) \what{W_{\eps_2}}(\eta)
  + \bar{\what{W}}(\xi) \what{i\eps_2\Sigma_\g W_{\eps_2}}(\eta) \Big]
  \what{u_{\eps_1}}(\xi-\eta) a^{N_0}_{\eps_1\eps_2}(\xi,\eta)  \,d\xi d\eta ,
\end{equation}
\begin{equation}
\label{evolE3a3}
I_{7,\eps_1\eps_2} = \frac{1}{4\pi^2} \Re \int_{\R \times\R} \Big[ \bar{\what{\mathcal{Q}_W}}(\xi) \what{W_{\eps_2}}(\eta)
  + \bar{\what{W}}(\xi) \what{\mathcal{Q}^{\eps_2}_W}(\eta) \Big] \what{u_{\eps_1}}(\xi-\eta) a^{N_0}_{\eps_1\eps_2}(\xi,\eta) \,d\xi d\eta ,
\end{equation}
\begin{equation}\label{evolE3a4}
\begin{split}
I_{8,\eps_1\eps_2} & = \frac{1}{4\pi^2} \Re \int_{\R \times\R} \Big[ |\xi|^{N_0}\bar{\what{\mathcal{N}_W}}(\xi) \what{W_{\eps_2}}(\eta)
  + \bar{\what{W}}(\xi) |\eta|^{N_0}\what{\mathcal{N}_W^{\eps_2}}(\eta) \Big]
  \what{u_{\eps_1}}(\xi-\eta) a^{N_0}_{\eps_1\eps_2}(\xi,\eta) \,d\xi d\eta
\\
& + \frac{1}{4\pi^2} \Re \int_{\R \times\R} \bar{\what{W}}(\xi) \what{W_{\eps_2}}(\eta)
  \mathcal{F} \big(\pa_t u_{\eps_1} - \eps_1 i |\pa_x|^{3/2} u_{\eps_1} \big)(\xi-\eta) a^{N_0}_{\eps_1\eps_2}(\xi,\eta)  \,d\xi d\eta
\\
& + \frac{1}{4\pi^2} \Re \int_{\R \times\R} \Big[ \bar{\what{\mathcal{O}_W}}(\xi) \what{W_{\eps_2}}(\eta)
  + \bar{\what{W}}(\xi) \what{\mathcal{O}_W^{\eps_2}}(\eta) \Big] \what{u_{\eps_1}}(\xi-\eta) a^{N_0}_{\eps_1\eps_2}(\xi,\eta)  \,d\xi d\eta  .
\end{split}
\end{equation}
Here $\mathcal{Q}^+_W=\mathcal{Q}_W$, $\mathcal{Q}^-_W=\overline{\mathcal{Q}_W}$, $\mathcal{N}_W^+ = \mathcal{N}_W$, $\mathcal{N}_W^- = \overline{\mathcal{N}_W}$, $\mathcal{O}^+_W=\mathcal{O}_W$, $\mathcal{O}^-_W=\overline{\mathcal{O}_W}$.
Similarly
\begin{align*}
\begin{split}
\frac{d}{dt} E^{(3)}_{b,\eps_1\eps_2} & = I_{9,\eps_1\eps_2} + I_{10,\eps_1\eps_2}
\end{split}
\end{align*}
where
\begin{equation}
\label{evolE3b1}
\begin{split}
I_{9,\eps_1\eps_2} & = \frac{1}{4\pi^2} \Re \int_{\R \times\R} \bar{\what{W}}(\xi) \what{u_{\eps_2}}(\eta) \what{u_{\eps_1}}(\xi-\eta)
  \\ & \hskip100pt \big[ -i|\xi|^{3/2} + i\eps_2|\eta|^{3/2} + i\eps_1|\xi-\eta|^{3/2} \big] b^{N_0}_{\eps_1\eps_2}(\xi,\eta)  \,d\xi d\eta ,
\end{split}
\end{equation}
\begin{equation}\label{evolE3b2}
\begin{split}
I_{10,\eps_1\eps_2} & = \frac{1}{4\pi^2} \Re \int_{\R \times\R} \bar{\mathcal{F}\big( \pa_t W - i |\pa_x|^{3/2}W\big)}(\xi) \what{u_{\eps_2}}(\eta)
  \what{u_{\eps_1}}(\xi-\eta) b^{N_0}_{\eps_1\eps_2}(\xi,\eta) \,d\xi d\eta
\\
& + \frac{1}{4\pi^2} \Re \int_{\R \times\R} \bar{\what{W}}(\xi)
  \mathcal{F}\big(\pa_t u_{\eps_2} - \eps_2 i |\pa_x|^{3/2} u_{\eps_2} \big)(\eta) \what{u_{\eps_1}}(\xi-\eta)
  b^{N_0}_{\eps_1\eps_2}(\xi,\eta) \,d\xi d\eta
\\
& + \frac{1}{4\pi^2} \Re \int_{\R \times\R} \bar{\what{W}}(\xi) \what{u_{\eps_2}}(\eta)
  \mathcal{F}\big( \pa_t u_{\eps_1} - \eps_1 i |\pa_x|^{3/2} u_{\eps_1} \big)(\xi-\eta) b^{N_0}_{\eps_1\eps_2}(\xi,\eta) \,d\xi d\eta .
\end{split}
\end{equation}
We then see that
\begin{align*}
\begin{split}
A_3 + \frac{d}{dt} \sum_{\star} \big( E^{(3)}_{a,\eps_1\eps_2} + E^{(3)}_{b,\eps_1\eps_2} \big)
  & = \sum_{\star}
  \big( I_{6,\eps_1\eps_2} + I_{7,\eps_1\eps_2} + I_{8,\eps_1\eps_2} + I_{10,\eps_1\eps_2} \big) .
\end{split}
\end{align*}
We show below how to estimate all the terms on the right-hand side above by $C\e_1^4 \langle t\rangle^{-1+2p_0}$.

\vskip10pt
\subsubsection*{Estimate of $I_{6,\eps_1\eps_2}$}
We start by looking at the case when $\eps_2=-1$.
We see from \eqref{bounda^N/} that the symbols $a^{N_0}_{\eps_1 -}$ have a strong ellipticity. Then we use the bound
\begin{equation}\label{azx3}
\|P_k\Sigma_{\gamma}F\|_{L^2}\lesssim \e_1^2\langle t\rangle ^{-3/4}2^{3k/2}\|P'_kF\|_{L^2}\mathbf{1}_{[-4,\infty)}(k),
\end{equation}
for any $k\in\mathbb{Z}$, which is a consequence of \eqref{estgamma} and Lemma \ref{touse} (ii). The potential loss of $3/2$ derivatives coming from $\Sigma_\gamma W$ is compensated by the smoothing property of the symbols $a^{N_0}_{\eps_1-}$ in \eqref{bounda^N/}. As a consequence $|I_{6,\eps_1-}|\lesssim \e_1^4 \langle t\rangle^{-7/6}$, as desired. 

In the cases $(\eps_1,\eps_2) \in \{ (+,+),(-,+) \}$ we have
\begin{align}
\label{I_6+}
| I_{6,\eps_1 +} | & \lesssim \Big| \int_{\R \times\R} \Big[ - \bar{\what{\Sigma_\g W}}(\xi) \what{W}(\eta)
  + \bar{\what{W}}(\xi) \what{\Sigma_\g W}(\eta) \Big] \what{u_{\eps_1}}(\xi-\eta) a^{N_0}_{\eps_1 +}(\xi,\eta) \,d\xi d\eta \Big| .
\end{align}
This is of the form \eqref{lemSigmabulkI}-\eqref{lemSigmabulkm},
with $F = W$, $m = a^{N_0}_{\eps_1 +}$, $q(\xi,\eta) = |\xi|^{N_0} |\eta|^{-N_0} a_{\eps_1 +}(\xi,\eta)$.
We can then apply Lemma \ref{lemSigmabulk}, provided we verify its assumptions for $a^{N_0}_{\eps_1 +}$.
Observe that the symbols $a_{\eps_1+}(\xi,\eta)$ are supported on a region where $2^8|\xi-\eta| \leq |\eta|$.
The bound \eqref{boundaeps_1+} for the symbol $a_{\eps_1 +}(\xi,\eta)$ gives the property in \eqref{lemSigmabulkpro1}.
Moreover, the properties \eqref{boundadiff} 
show that the second assumption \eqref{lemSigmabulkpro2} is satisfied.
Applying Lemma \ref{lemSigmabulk} to $I_{6,\eps_1+}$ gives us the desired bound $| I_{6,\eps_1 +}(t) | \lesssim \e_1^4 \langle t\rangle^{-1+2p_0}$.

\vskip10pt
\subsubsection*{Estimate of $I_{7,\eps_1\eps_2}$} This is similar to the estimates on $I_{6,\eps_1\eps_2}$: in the case $\eps_2=-$ we use the simple bound $\|P_k\mathcal{Q}_W\|_{L^2}\lesssim \e_1\langle t\rangle ^{-1/2}2^k\|P'_{k}W\|_{L^2}$, which is similar to \eqref{azx3}, and the gain of $3/2$ derivatives in the symbols $a_{\eps_1-}^{N_0}$. In the case $\eps_2=+$ we use Lemma \ref{lemQbulk} instead of Lemma \ref{lemSigmabulk}. In both cases we conclude that $| I_{7,\eps_1\eps_2}(t) | \lesssim \e_1^4 {(1+t)}^{-1+2p_0}$, as desired.

\vskip10pt
\subsubsection*{Estimate of $I_{8,\eps_1\eps_2}$}

Observe that all the terms in \eqref{evolE3a4} do not lose derivatives, and moreover they are similar to the terms in \eqref{evolE314}.
By performing the same estimates above one sees that $| I_{8,\eps_1\eps_2}(t) | \lesssim \e_1^4 {(1+t)}^{-1+2p_0}$.

\vskip10pt
\subsubsection*{Estimate of $I_{10,\eps_1\eps_2}$}

We notice that the symbols $b^N_{\eps_1\eps_2}$ are smoothing and non-singular,
see the definition \eqref{E^3b} and the bound \eqref{boundb^N/}.
Therefore, there are no losses of derivatives in \eqref{evolE3b2} and
all these terms are straightforward to estimate.


\vskip10pt
\subsubsection{Proof of Lemma \ref{lemEE23}}\label{prooflemEE23}

The term $A_1$ in \eqref{EE21} can bounded as desired because $\Sigma_\g$ has been constructed as a symmetric operator up to order $-1/2$.
To see that this is indeed the case, recall the definition of $\Sigma_\g$ in \eqref{eqW10}, and write
\begin{align*}
4\pi^2 A_1 & = \Re \int_{\R} \bar{\what{W}}(\xi) i \what{P_{\geq 1}W}(\eta)
  \what{\gamma}(\xi-\eta) \, \Big(|\eta|^{3/2}+\frac{3}{4} \frac{(\xi-\eta)\eta}{|\eta|^{1/2}}\Big) \chi(\xi-\eta,\eta) \,d\eta d\xi
\\
& = \Re \int_{\R} \bar{\what{W}}(\xi) \what{W}(\eta) \what{\gamma}(\xi-\eta)
  \, i \Big[ \frac{1}{2} \Big(|\eta|^{3/2} \varphi_{\geq 1}(\eta) \chi(\xi-\eta,\eta)
  - |\xi|^{3/2} \varphi_{\geq 1}(\xi)\chi(\eta-\xi,\xi) \Big)
  \\ & \hskip200pt + \frac{3}{4}\frac{(\xi-\eta)\eta}{|\eta|^{1/2}} \varphi_{\geq 1}(\eta)\chi(\xi-\eta,\eta) \Big] \,d\eta d\xi .
\end{align*}
Notice that the symbol in the above expression is $O( (\xi-\eta)^2 (|\xi|+|\eta|)^{-1/2} \mathbf{1}_{[2^3,\infty)}(|\eta|/|\xi-\eta|) )$.
We can then proceed, as done several times before, using Lemma \ref{touse}(ii) and \eqref{estgamma}, and obtain
\begin{align}
| A_1 | & \lesssim \sum_{(k,k_1,k_2) \in \mathcal{X}, \,|k-k_2| \leq 5} 
  2^{2k_1} 2^{-k/2} {\|P_k^\prime W \|}_{L^2} {\| P_{k_2}^\prime W \|}_{L^2} {\| P_{k_1} \gamma \|}_{L^\infty}\lesssim \e_1^4 \langle t\rangle^{-1+2p_0} .
\end{align}

Using the cubic estimates on the term $\mathcal{O}_W$ in \eqref{cubboundW},
it is easy to see that also the term $A_4$ in \eqref{EE24} satisfies the desired bound.
We have then completed the proof of Lemma \ref{lemEE23}, and hence of Lemma \ref{lemEE2}. Proposition \ref{proEE1} is proved.

\section{Energy estimates II: low frequencies}\label{secEElow}

\subsection{The basic low frequency energy}
In this section we exploit the null structure of the equation to control the low frequency component of the solution $u$
which we denote by
\begin{align}
\label{defulow}
u_\low := P_{\leq -10} u .
\end{align}
With $\mathcal{P}: [0,\infty) \rightarrow [0,1]$ is as in \eqref{p0}, we define the energy
\begin{align}
\label{defE2low}
E_\low^{(2)}(t) = \frac{1}{4\pi} \int_\R {\big| \what{u_\low}(t,\xi) \big|}^2 {|\xi|}^{-1} \mathcal{P}((1+t)^2|\xi|) \, d\xi.
\end{align}

The main proposition in this section is the following:

\begin{proposition}\label{proEElow}
Assume that $u$ satisfies \eqref{bing11}--\eqref{bing31}. Then
\begin{align}
\label{proEElowconc1}
\sup_{t \in [0,T]} (1+t)^{-2p_0} E_\low^{(2)}(t) \lesssim\e_0^2.
\end{align}
\end{proposition}

\subsection{The cubic low frequency energy}
Recall from Proposition \ref{proequ} the equation \eqref{equ}-\eqref{equSigma}, from which it follows
that
\begin{align}
\label{equlow1}
\begin{split}
\partial_t u_\low - i |\partial_x|^{3/2} u_\low & = P_{\leq -10} ( - \partial_x T_V u  + \mathcal{N}_u ) + |\partial_x|^{1/2}O_{3,1/2} .
\end{split}
\end{align}
According to \eqref{equlow1},  we naturally define the cubic correction to the basic energy $E_\low^{(2)}$
as follows:
\begin{align}
\label{defE3low}
E^{(3)}_\low(t) := & \sum_\star \frac{1}{4\pi^2} \Re \int_{\R \times \R} \bar{\what{u}}(t,\xi) {|\xi|}^{-1/2}
  \what{u_{\eps_1}}(t,\xi-\eta) \what{u_{\eps_2}}(t,\eta) m_{\eps_1\eps_2}^\low(\xi,\eta) \,d\xi d\eta ,
\end{align}
where the symbol is
\begin{align}
\label{mlow}
\begin{split}
m_{\eps_1\eps_2}^\low(\xi,\eta) & := (1+\eps_1)(1+\eps_2)
  \frac{ |\xi|^{1/2} (\xi-\eta) \big[\xi |\xi|^{-1} \varphi_t(\xi) \chi(\xi-\eta,\eta) - \eta |\eta|^{-1} \varphi_t(\eta)
  \chi(\eta-\xi,\xi)\big]}{8|\xi-\eta|^{1/2} (|\xi|^{3/2} -|\xi-\eta|^{3/2} - |\eta|^{3/2})}
\\
& \hskip100pt -i  \frac{\varphi_t(\xi) \big(a_{\eps_1\eps_2}(\xi,\eta) + b_{\eps_1\eps_2}(\xi,\eta)\big) }{
  {|\xi|}^{1/2} (|\xi|^{3/2} -\epsilon_1|\xi-\eta|^{3/2} -\epsilon_2|\eta|^{3/2})} ,
\end{split}
\end{align}
and we have denoted
\begin{align*}
 \varphi_t(\xi) = \varphi_{\geq 4}((1+t)^2\xi) \varphi^2_{\leq -10}(\xi) .
\end{align*}
The first part of $m_{\eps_1\eps_2}^\low$ is non-zero only for $\eps_1=\eps_2=1$,
and takes into account the nonlinear term $-\partial_x T_V u$.
The symbol in the second line of \eqref{mlow} is needed to correct the nonlinear terms in $\mathcal{N}_u$.
Notice that no correction is needed if $(1+t)^{-2}|\xi|\leq 4$.

The total energy for the low frequency part of the solution is given by
\begin{align}
\label{defElow}
E_\low := E_\low^{(2)} + E_\low^{(3)}(t) .
\end{align}

Observe that under the a priori assumptions \eqref{bing11}--\eqref{bing31} we have
\begin{align}
\label{estlow}
\begin{split}
& {\| \varphi_{\geq 4}((1+t)^2\xi) |\xi|^{-1/2} \what{u}(t,\xi) \|}_{L^2}
 \lesssim \e_1 {(1+t)}^{p_0},
\\
& {\| \varphi_{\leq 4}((1+t)^2\xi)|\xi|^{-1/2+p_1} \what{u}(t,\xi) \|}_{L^2}
  \lesssim \e_1 {(1+t)}^{p_0-2p_1}.
\end{split}
\end{align}

As in section \ref{secEE}, Proposition \ref{proEElow} follows from two main lemmas.

\begin{lemma}\label{lemEElow1}
Under the assumptions of Proposition \ref{proEElow}, for any $t\in[0,T]$, we have
\begin{equation}
\label{lemEElow1conc}
\big| E_\low^{(3)}(t) \big| \lesssim \e_1^3{(1+t)}^{2p_0} . 
\end{equation}
\end{lemma}

\begin{lemma}\label{lemEElow2}
Under the assumptions of Proposition \ref{proEElow}, for any $t\in[0,T]$, we have
\begin{equation}
\label{lemEElow2conc}
\Big|\frac{d}{dt} E_\low(t)\Big| \leq C\e_1^3 {(1+t)}^{-1+2p_0}+40p_1E^{(2)}_\low(t)(1+t)^{-1}.
\end{equation}
\end{lemma}

\subsection{Analysis of the symbols and proof of Lemma \ref{lemEElow1}}\label{seclemEElow1}
We first show that the symbol in \eqref{mlow} satisfies the bounds
\begin{align}
\label{boundmlow}
{\| ( m_{\eps_1\eps_2}^\low )^{k,k_1,k_2} \|}_{S^\infty} \lesssim 2^{-\max(k_1,k_2)/2}
  \mathbf{1}_{\mathcal{X}}(k,k_1,k_2) \mathbf{1}_{[- 2\log_2 (2+t) - 10, 0]}(k)
\end{align}
and
\begin{align}
\label{boundmlow'}
{\| m_{\eps_1\eps_2}^\low \varphi_{k_1}(\xi-\eta) \varphi_{k_2}(\eta) \|}_{S^\infty} \lesssim 2^{-\max(k_1,k_2)/2} ,
\end{align}
for all $k,k_1,k_2 \in \mathbb{Z}$.

We start with the first component of $m_{\eps_1\eps_2}^\low$, which we denote by
\begin{align*}
n_1(\xi,\eta) :=
  \frac{ |\xi|^{1/2} (\xi-\eta) \big[\xi |\xi|^{-1} \varphi_t(\xi) \chi(\xi-\eta,\eta) - \eta |\eta|^{-1} \varphi_t(\eta)
  \chi(\eta-\xi,\xi)\big]}{2|\xi-\eta|^{1/2} (|\xi|^{3/2} -|\xi-\eta|^{3/2} - |\eta|^{3/2})} .
\end{align*}
Since $\chi(\xi-\eta,\eta)$ forces $2^3|\xi-\eta| \leq |\eta|$, through Taylor expansions and standard integration by parts, we see that
\begin{align*}
\frac{\xi|\xi|^{-1} \varphi_t(\xi) \chi(\xi-\eta,\eta) - \eta|\eta|^{-1} \varphi_t(\eta) \chi(\eta-\xi,\xi)}{
  |\xi|^{3/2} -|\xi-\eta|^{3/2} - |\eta|^{3/2}}
  = O \big( |\eta|^{-3/2} \mathbf{1}_{[2^2,\infty]}(|\eta|/|\xi-\eta|) \mathbf{1}_{[(1+t)^{-2},0]}(|\xi|) \big) ,
\end{align*}
and deduce
\begin{align*}
n_1(\xi,\eta) = O \big( |\xi-\eta|^{1/2} |\eta|^{-1} \mathbf{1}_{[2^2,\infty]}(|\eta|/|\xi-\eta|) \mathbf{1}_{[(1+t)^{-2},0]}(|\xi|) \big) .
\end{align*}

Let
\begin{align*}
n_2(\xi,\eta) := -i \frac{\varphi_t(\xi) a_{\eps_1\eps_2}(\xi,\eta) 
  }{ {|\xi|}^{1/2} (|\xi|^{3/2} -\epsilon_1|\xi-\eta|^{3/2} -\epsilon_2|\eta|^{3/2})}.
\end{align*}
Using the bound \eqref{bounda/1} in Lemma \ref{lembounda}, one can see that
\begin{align*}
n_2(\xi,\eta) = O \big( |\xi-\eta|^{1/2} |\xi|^{-1} \mathbf{1}_{[2^2,\infty)}(|\eta|/|\xi-\eta|) \mathbf{1}_{[(1+t)^{-2},0]}(|\xi|) \big) ,
\end{align*}
which is a better bound than what we need. Finally, let
\begin{align*}
n_3(\xi,\eta) := -i \frac{\varphi_t(\xi) b_{\eps_1\eps_2}(\xi,\eta) 
  }{ {|\xi|}^{1/2} (|\xi|^{3/2} -\epsilon_1|\xi-\eta|^{3/2} -\epsilon_2|\eta|^{3/2})} ,
\end{align*}
where the symbols $b_{\eps_1\eps_2}$ are defined in \eqref{b++}--\eqref{m_2q_2}.
Using the bound \eqref{boundb1} in Lemma \ref{lemboundb}, we know that 
on the support of $b_{\eps_1\eps_2}$ the sizes of $|\eta|$ and $|\xi-\eta|$ are comparable, so
\begin{align}
n_3(\xi,\eta) = O \big( |\xi-\eta|^{-1/2} \mathbf{1}_{[2^{-15},2^{15}]}(|\eta|/|\xi-\eta|) \mathbf{1}_{[(1+t)^{-2},0]}(|\xi|) \big) .
\end{align}
This suffices to obtain \eqref{boundmlow}.
Similar arguments, reexamining the formulas \eqref{b++}-\eqref{b--},
and integration by parts, also give \eqref{boundmlow'} (only the bound for $n_3$ requires an additional argument).

We now use \eqref{boundmlow} to prove \eqref{lemEElow1conc}.
Let us denote
\begin{align}
\label{min12}
k_- := \min(k_1,k_2) , \quad k_+ := \max(k_1,k_2) .
\end{align}
Using the formula \eqref{defE3low}, the bound \eqref{boundmlow} together with Lemma \eqref{touse}, we estimate
\begin{align*}
\big| E_\low^{(3)}(t) \big| & \lesssim \sum_{\star}
  \sum_{(1+t)^{-2} \leq 2^{k+10} \leq 2^{10}} {\| (m_{\eps_1\eps_2}^\low)^{k,k_1,k_2} \|}_{S^\infty}
  2^{-k/2} {\| P_k^\prime u \|}_{L^2} {\| P_{k_+}^\prime u \|}_{L^2} {\| P_{k_-}^\prime u \|}_{L^\infty} .
\end{align*}
Using  the a priori assumptions \eqref{bing11}--\eqref{bing31} (see also \eqref{estlow}) we obtain
\begin{equation}\label{cvb1}
\begin{split}
\big| E_\low^{(3)}(t) \big| \lesssim \sum_{(k,k_1,k_2) \in \mathcal{X}, \, (1+t)^{-2} \leq 2^{k+10}}
  &2^{-k/2} {\| P_k^\prime u \|}_{L^2}  2^{-k_+/2} {\| P_{k_+}^\prime u \|}_{L^2}\\
	&\times\e_1 2^{k_-/10} 2^{-N_2\max(k_-,0)} \langle t\rangle^{-1/2}\lesssim \e_1^3 \langle t\rangle^{-1/3}.
	\end{split}
\end{equation}
This completes the proof of Lemma \ref{lemEElow1}.

\vskip10pt
\subsection{Proof of Lemma \ref{lemEElow2}}\label{seclemEElow2}
Using the definitions \eqref{defE2low},\eqref{p0}, \eqref{defE3low},
the equation \eqref{equlow1}, and a symmetrization argument similar to the one performed for
the term \eqref{EE22} and leading to \eqref{A_2}, we calculate
\begin{align*}
& \frac{d}{dt} \big( E_\low^{(2)} + E_\low^{(3)} \big) =
  \frac{1}{4\pi}  J_1 + \Re \Big[  \frac{1}{2\pi} J_2
  + \frac{1}{4\pi^2} \sum_\star \big( J_{3,\eps_1\eps_2} + J_{4,\eps_1\eps_2} + J_{5,\eps_1\eps_2} + J_{6,\eps_1\eps_2} \big) \Big] + R ,
\end{align*}
\begin{align}
\label{evolElow}
\begin{split}
& J_1 = \int_\R {\big| \what{u_\low}(\xi) \big|}^2 \, \mathcal{P}^\prime ((1+t)^2|\xi|) \, 2(1+t) \, d\xi ,
\\
& J_2 = \int_\R \bar{\what{u_\low}}(\xi) \mathcal{F}\big(\partial_t u_{\low} - i|\partial_x|^{3/2}u_{\low}\big)(\xi)
  \, \varphi_{\leq 3}((1+t)^2\xi) |\xi|^{-1} \mathcal{P}((1+t)^2|\xi|)\, d\xi ,
\\
& J_{3, \eps_1\eps_2} = \int_{\R \times \R}  {|\xi|}^{-1/2} \bar{\what{u}}(\xi) 
  \, \partial_t m^\low_{\eps_1\eps_2}(\xi,\eta)  \what{u_{\eps_1}}(\xi-\eta) \what{u_{\eps_2}}(\eta) \,  d\xi d\eta ,
\\
& J_{4, \eps_1\eps_2} = \int_{\R \times \R}  {|\xi|}^{-1/2} \bar{\mathcal{F}\big(\partial_t u - i|\partial_x|^{3/2}u\big)}(\xi) 
  \, m^\low_{\eps_1\eps_2}(\xi,\eta)  \what{u_{\eps_1}}(\xi-\eta) \what{u_{\eps_2}}(\eta) \,  d\xi d\eta ,
\\
& J_{5, \eps_1\eps_2} = \int_{\R \times \R}  {|\xi|}^{-1/2} \bar{\what{u}}(\xi) 
  \, m^\low_{\eps_1\eps_2}(\xi,\eta) \mathcal{F} \big(\partial_t u_{\eps_1} - i\eps_1|\partial_x|^{3/2} u_{\eps_1} \big)(\xi-\eta) \what{u_{\eps_2}}(\eta)
  \,d\xi d\eta ,
\\
& J_{6, \eps_1\eps_2} = \int_{\R \times \R}  {|\xi|}^{-1/2} \bar{\what{u}}(\xi) 
  \, m^\low_{\eps_1\eps_2}(\xi,\eta) \what{u_{\eps_1}}(\xi-\eta) \mathcal{F}\big(\partial_t u_{\eps_2} - i\eps_2|\partial_x|^{3/2} u_{\eps_2} \big)(\eta)
  \,d\xi d\eta .
\end{split}
\end{align}
Here $R$ denotes a quartic term which includes the contribution from  $|\partial_x|^{1/2}O_{3,1/2}$ in \eqref{equlow1},
and from the quadratic part of $V$, that is $V_2$ in \eqref{Vu+baru1},
and can be easily seen to satisfy $| R(t) | \lesssim \e_1^4 {(1+t)}^{-1+2p_0}$.
The term $J_1$ comes from differentiating the cutoff $\mathcal{P}$ in the quadratic energy, whereas
$J_2$ is obtained when we differentiate $|\what{u_\low}|^2$ in the region $|\xi| \lesssim  {(1+t)}^{-2}$.
The term $J_{3, \eps_1\eps_2}$ comes from differentiating the symbol in the cubic energy functional.
The remaining three terms in \eqref{evolElow} are the net result of differentiating the quadratic energy when $|\xi| \gtrsim {(1+t)}^{-2}$,
and its cubic correction, after symmetrizations and cancellations.

\subsubsection{Estimate of $J_1$}
Using \eqref{p0}, the definition of $E^{(2)}_\low$ in \eqref{defE2low}, and the a priori assumption \eqref{bing11}--\eqref{bing31},
we immediately see that
\begin{align*}
\frac{1}{4\pi}| J_1 | & \leq \frac{1}{4\pi}\int_\R {\big| \what{u_\low}(t,\xi) \big|}^2 \frac{20p_1}{(1+t)|\xi|} \, \mathcal{P}((1+t)^2|\xi|) \, d\xi
  \leq \frac{20p_1}{1+t} E^{(2)}_\low(t),
\end{align*}
as desired.

\subsubsection{Estimate of $J_2$}
Notice that the term $J_2$ in \eqref{evolElow} is supported on a region where $|\xi| \lesssim {(1+t)}^{-2}$.
Then we use the bounds in Lemma \ref{lemboundN_u}, and the second bound in \eqref{estlow}, to see that
\begin{align*}
| J_2 | & \lesssim \sum_{2^{k-6} \leq \langle t\rangle^{-2}} 2^{(-1+2p_1)k} {\| P_k^\prime u(t) \|}_{L^2}
  {\| P_k^\prime \big( \partial_t u - i |\partial_x|^{3/2}u \big) (t) \|}_{L^2} \langle t\rangle^{4p_1}
  \\
& \lesssim  \langle t\rangle^{4p_1} \sum_{2^k \lesssim \langle t\rangle^{-2}} 2^{(-1+2p_1)k} {\| P_k^\prime u(t) \|}_{L^2}
   \, \e_1^2 2^{k/2} \langle t\rangle^{-1+p_0}
\\
& \lesssim \langle t\rangle^{-1+p_0+2p_1} \e_1^2
  \sum_{2^k \lesssim \langle t\rangle^{-2}} 2^{p_1k}\langle t\rangle^{2p_1}2^{-k/2+p_1k}{\| P_k^\prime u(t) \|}_{L^2}
\\
& \lesssim \e_1^3 \langle t\rangle^{-1+2p_0} .
\end{align*}

\subsubsection{Estimate of $J_{3,\eps_1\eps_2}$}
Recall the notation \eqref{min12}. Using the definition of $m_{\eps_1\eps_2}^\low$ in \eqref{mlow}, and the same arguments in the proof of Lemma \ref{lemEElow1}
that gave us \eqref{boundmlow}, one can see that
\begin{align*}
{\| (\partial_t m_{\eps_1\eps_2})^{k,k_1,k_2} \|}_{S^\infty}
  \lesssim \mathbf{1}_{\mathcal{X}}(k,k_1,k_2) {(1+t)}^{-1} 2^{-\max(k_1,k_2)/2} \mathbf{1}_{(-2\log_2 (2+t) - 10,0]}(k) .
\end{align*}
This is better by a factor of $(1+t)^{-1}$ than the corresponding bound on $m_{\eps_1\eps_2}$. The same argument as in the proof of Lemma \ref{lemEElow1} (see \eqref{cvb1}) shows that $|J_{3,\eps_1\eps_2}|\lesssim \varep_1^3\langle t\rangle^{-4/3}$, as desired.

\subsubsection{Estimate of $J_{4,\eps_1\eps_2}$}We use Lemma \ref{touse} to estimate, for every $\eps_1,\eps_2 = \pm$,
\begin{align*}
| J_{4,\eps_1\eps_2} | & \lesssim \sum_{k,k_1,k_2 \in \mathbb{Z}}  {\| (m_{\eps_1\eps_2}^\low)^{k,k_1,k_2} \|}_{S^\infty}
  2^{-k/2} {\| P_k^\prime (\partial_t - i|\partial_x|^{3/2}) u \|}_{L^2} {\| P_{k_+}^\prime u \|}_{L^2} {\| P_{k_-}^\prime u \|}_{L^\infty} .
\end{align*}
Using the bound \eqref{boundmlow}, Lemma \ref{lemboundN_u}, and \eqref{bing11}, we have
\begin{align*}
| J_{4,\eps_1\eps_2} | \lesssim \e_1^3 \langle t\rangle^{-1+p_0} \sum_{(k,k_1,k_2) \in \mathcal{X}, \, \langle t\rangle^{-2} \leq 2^k \leq 1}
   2^{-k_+/2} {\| P_{k_+}^\prime u \|}_{L^2} 2^{k_-/10} 2^{-N_2 \max(k_-,0)}\lesssim \e_1^4 \langle t\rangle^{-1+2p_0} .
\end{align*}

\subsubsection{Estimates of $J_{5,\eps_1\eps_2}$ and $J_{6, \eps_1\eps_2}$}
Since the terms $J_{5, \eps_1\eps_2}$ and $J_{6, \eps_1\eps_2}$ are similar, we only estimate the first one.
We begin again by using Lemma \ref{touse}, and estimate, for every $\eps_1,\eps_2$,
\begin{align*}
| J_{5,\eps_1\eps_2} | \lesssim {\| \varphi_{\geq 0}((1+t)^2\xi) |\xi|^{-1/2} \what{u} \|}_{L^2}
  \sum_{k_1,k_2} {\| m_{\eps_1\eps_2}^\low \varphi_{k_1}(\xi-\eta) \varphi_{k_2}(\eta) \|}_{S^\infty}
  \\ \times {\| P_{k_1}^\prime (\partial_t - i\Lambda)u \|}_{L^\infty} {\| P_{k_2}^\prime u \|}_{L^2} .
\end{align*}
We then use the bound \eqref{boundmlow'}, Lemma \ref{lemboundN_u}, and the a priori bounds \eqref{bing11}--\eqref{bing31}, to obtain
\begin{align*}
| J_{5,\eps_1\eps_2} | \lesssim \e_1^4 \langle t\rangle^{-1+2p_0} \sum_{k_1,k_2} 2^{-\max(k_1,k_2)/2}2^{k_1/2} 2^{-\max(k_1,0)} 2^{(1/2-p_0)k_2}2^{-\max(k_2,0)}\lesssim \e_1^4 \langle t\rangle^{-1+2p_0} .
\end{align*}
The proof of Lemma \ref{lemEElow2} is completed, and Proposition \ref{proEElow} follows. $\hfill \Box$

\section{Energy estimates III: weighted estimates for high frequencies}\label{secweighted}

In this section we want to control the high frequency component of the weighted Sobolev norm of our solution $u$. The main result of this section is the following weighted energy estimate:

\begin{proposition}\label{proEEZ}
Assume that $u$ satisfies \eqref{bing11}-\eqref{bing31}. Then
\begin{equation}\label{proEEZconc}
 \sup_{t\in[0,T]} (1+t)^{-4p_0} \|P_{\geq -20}Su(t)\|_{H^{N_1}} \lesssim \e_0 .
\end{equation}
\end{proposition}

\vskip10pt
\subsection{The weighted energy functionals}\label{secwenergy} Let $Z=S\mathcal{D}^{k_1}u$, $N_1=3k_1/2$, $\mathcal{D}=|\partial_x|^{3/2}+\Sigma_\gamma$. The quadratic energy we associate to the equation \eqref{eqZ} is
\begin{align}
\label{defE2Z}
E_Z^{(2)} (t):= \frac{1}{4\pi} \int_{\R} \bar{\what{Z}}(t,\xi) \what{Z}(t,\xi) \, \varphi_{\geq -20}^2(\xi)  \, d\xi .
\end{align}
The natural cubic energy is constructed as the sum of three energy functionals,
\begin{align}
\label{defE3Z}
E_Z^{(3)}(t) = E_{Z,1}(t) + E_{Z,2}(t) + E_{Z,3}(t) .
\end{align}
The first energy functional is is the natural correction associated to the nonlinearities $\mathcal{Q}_Z$
and $\mathcal{N}_{Z,1}$ in \eqref{eqZN1}-\eqref{eqZN2},
\begin{align}
\begin{split}
\label{E_Z1}
&E_{Z,1}(t) := \frac{1}{4\pi^2} \sum_{\star} \Re \int_{\R \times\R} \bar{\what{Z}}(\xi,t) \what{Z}_{\eps_2}(\eta,t)
  m_{1,\eps_1 \eps_2}(\xi,\eta) \what{u}_{\eps_1}(\xi-\eta,t) \,d\xi d\eta ,
\\
&m_{1,\eps_1 \eps_2} (\xi,\eta) :=- i \frac{|\xi|^{N_1}|\eta|^{-N_1} \varphi_{\geq -20}^2(\xi) a_{\eps_1\eps_2}(\xi,\eta)}{
  |\xi|^{3/2}-\eps_1|\xi-\eta|^{3/2}-\eps_2|\eta|^{3/2}}+(1+\eps_1)(1+\eps_2)\\
	&\times\frac{(\xi-\eta)
  \big[ \xi |\xi|^{N_1}|\eta|^{-N_1} \varphi_{\geq -20}^2(\xi) \chi(\xi-\eta,\eta)
  - \eta |\eta|^{N_1}|\xi|^{-N_1} \varphi_{\geq -20}^2(\eta) \chi(\eta-\xi,\xi) \big]}{
  8 |\xi-\eta|^{1/2}(|\xi|^{3/2}-|\xi-\eta|^{3/2}-|\eta|^{3/2})},
\end{split}
\end{align}
where the second part of the symbol, which is only present for $\eps_1=\eps_2=1$, is associated to $\mathcal{Q}_Z$,
while the first one corresponds to $\mathcal{N}_{Z,1}$.
The second functional takes into account $\mathcal{N}_{Z,2}$ in \eqref{eqZN2} and is defined as
\begin{align}
\label{E_Z2}
\begin{split}
E_{Z,2}(t) & := \frac{1}{4\pi^2} \sum_\star \Re \int_{\R \times\R} \bar{\what{Z}}(\xi,t) \what{u_{\eps_2}}(\eta,t)
  m_{2,\eps_1\eps_2}(\xi,\eta) \what{S u_{\eps_1}}(\xi-\eta,t) \,d\xi d\eta ,
\\
m_{2,\eps_1\eps_2}(\xi,\eta) & := \frac{\eps_1 (1+\eps_2) (\xi-\eta) \xi |\xi|^{N_1} \varphi_{\geq -20}^2(\xi) \chi(\xi-\eta,\eta)
  }{4|\xi-\eta|^{1/2}(|\xi|^{3/2}-\eps_1|\xi-\eta|^{3/2}-|\eta|^{3/2})}
  \\
  & + \frac{-i |\xi|^{N_1} \varphi_{\geq -20}^2(\xi)
  \big(a_{\eps_1\eps_2}(\xi,\eta) + b_{\eps_1\eps_2}(\xi,\eta) + b_{\eps_1\eps_2}(\xi,\xi-\eta) \big)}{
  |\xi|^{3/2}-\eps_1|\xi-\eta|^{3/2}-\eps_2|\eta|^{3/2}} .
\end{split}
\end{align}
The last functional is
\begin{align}
\label{E_Z3}
\begin{split}
E_{Z,3}(t) &:= \frac{1}{4\pi^2} \sum_\star \Re \int_{\R \times\R} \bar{\what{Z}}(\xi,t) \what{u_{\eps_2}}(\eta,t)
  m_{3,\eps_1\eps_2}(\xi,\eta) \what{u_{\eps_1}}(\xi-\eta,t) \,d\xi d\eta ,
\\
m_{3,\eps_1\eps_2}(\xi,\eta) &:= \frac{i\eps_1 (1+\eps_2) (\xi-\eta) \varphi_{\geq -20}^2(\xi)
  \big( \widetilde{q}_{N_1}(\xi,\eta) |\eta|^{N_1} +
  q_{N_1}(\xi,\eta) |\eta|^{N_1} \big)}{4|\xi-\eta|^{1/2}(|\xi|^{3/2}-\eps_1|\xi-\eta|^{3/2}-|\eta|^{3/2})}
\\
& + \frac{ -i|\xi|^{N_1} \varphi_{\geq -20}^2(\xi) (3/2 - {N_1})\big(a_{\eps_1\eps_2}(\xi,\eta) + b_{\eps_1\eps_2}(\xi,\eta) \big)}{
  |\xi|^{3/2}-\eps_1|\xi-\eta|^{3/2}-\eps_2|\eta|^{3/2}}
\\
& + \frac{ -i|\xi|^{N_1} \varphi_{\geq -20}^2(\xi) \big(N_1 a_{\eps_1\eps_2}(\xi,\eta) + \widetilde{a}_{\eps_1\eps_2}(\xi,\eta)
  + \widetilde{b}_{\eps_1\eps_2}(\xi,\eta) \big)}{|\xi|^{3/2}-\eps_1|\xi-\eta|^{3/2}-\eps_2|\eta|^{3/2}} ,
\end{split}
\end{align}
and is associated to the nonlinear term $\mathcal{N}_{Z,3}$ in \eqref{eqZN2}.
The total weighted energy is
\begin{equation}\label{defEZ}
E_Z(t) := E_Z^{(2)}(t) + E_Z^{(3)}(t) .
\end{equation}

Using \eqref{bing11}, the definitions, and Propositions \ref{proEE1} and \ref{proEElow}, we have
\begin{equation}\label{azx5}
\begin{split}
\|P_{\geq -20}Su(t)\|_{H^{N_1}}&\lesssim \|P_{\geq -20}Z(t)\|_{L^2}+\e_1^{3/2}(1+t)^{p_0},\\
\|Su(t)\|_{H^{N_1}}+\|Z(t)\|_{L^2}&\lesssim \e_1(1+t)^{4p_0}.
\end{split}
\end{equation}
Therefore, Proposition \ref{proEEZ} follows from the following two main lemmas:

\begin{lemma}\label{lemEEZ1}
Under the assumptions of Proposition \ref{proEEZ}, for any $t\in[0,T]$ we have
\begin{equation}\label{lemEEZ1conc}
| E_Z^{(3)}(t) | \lesssim \e_1^3 {(1+t)}^{8p_0} . 
\end{equation}
\end{lemma}

\begin{lemma}\label{lemEEZ2}
Under the assumptions of Proposition \ref{proEEZ}, for any $t\in[0,T]$ we have
\begin{equation}\label{lemEEZ2conc}
\frac{d}{dt} E_Z(t) \lesssim \e_1^4 {(1+t)}^{-1+8p_0} .
\end{equation}
\end{lemma}

\vskip10pt
\subsection{Analysis of the symbols and proof of Lemma \ref{lemEEZ1}}\label{secwsym}
To prove Lemmas \ref{lemEEZ1} and \ref{lemEEZ2}, we need bounds for the symbols of
the cubic energy functionals in \eqref{E_Z1}-\eqref{E_Z3}.
The first term in the symbol $m_{1,\eps_1\eps_2}$ in \eqref{E_Z1} can be estimated by using directly \eqref{bounda/1}, while the second part is similar to the symbol $m_{N_0}$ in \eqref{E^31}. It follows that 
\begin{align}
\label{boundm_1}
\begin{split}
{\| m_{1,\eps_1+}^{k,k_1,k_2} \|}_{S^\infty}
  & \lesssim 2^{k_1/2} 2^{-k/2} \mathbf{1}_{\mathcal{X}}(k,k_1,k_2) \mathbf{1}_{[6,\infty)}(k_2 - k_1) ,
\\
{\| m_{1,\eps_1-}^{k,k_1,k_2} \|}_{S^\infty}
  & \lesssim \big( 2^{3k_1/2} 2^{-3k/2} + 2^{k_1/2} 2^{-k/2} \mathbf{1}_{(-\infty,1]}(k) \big)
  \mathbf{1}_{\mathcal{X}}(k,k_1,k_2) \mathbf{1}_{[6,\infty)}(k_2 - k_1) .
\end{split}
\end{align}

Using the bounds \eqref{bounda/1} and \eqref{boundb/1} it is easy to see that
\begin{align}
\label{boundm_2}
\begin{split}
{\| m_{2,\eps_1\eps_2}^{k,k_1,k_2}  \|}_{S^\infty}
  \lesssim \mathbf{1}_{\mathcal{X}}(k,k_1,k_2) 2^{N_1k}2^{k/2-k_1/2}\mathbf{1}_{[-20,\infty)}(k_2 - k).
\end{split}
\end{align}
Moreover, for any $k,k_2\in\mathbb{Z}$,
\begin{align}
\label{boundm_2'}
{\| |\xi-\eta|^{1/2} m_{2,\eps_1\eps_2} (\xi,\eta)\varphi_k(\xi) \varphi_{k_2}(\eta) \|}_{S^\infty}
  \lesssim 2^{(N_1+1/2) k} \mathbf{1}_{[-20,\infty)}(k_2 - k) .  
\end{align}

Finally, for the symbols $m_{3,\eps_1\eps_2}$, we have the similar bounds
\begin{align}
\label{boundm_3}
{\|m_{3,\eps_1\eps_2}^{k,k_1,k_2}  \|}_{S^\infty}
  \lesssim  \mathbf{1}_{\mathcal{X}}(k,k_1,k_2) 2^{N_1k}2^{k/2-k_1/2}\mathbf{1}_{[-20,\infty)}(k_2 - k),
\end{align}
and
\begin{align}
\label{boundm_3'}
{\| |\xi-\eta|^{1/2} m_{3,\eps_1\eps_2} \varphi_k(\xi) \varphi_{k_2}(\eta) \|}_{S^\infty}
  \lesssim 2^{(N_1+1/2) k} \mathbf{1}_{[-20,\infty)}(k_2 - k) .
\end{align}

The proof of \eqref{lemEEZ1conc} is similar to the proof of \eqref{lemEE1conc} in Lemma \ref{lemEE1}
and \eqref{lemEElow1conc} in Lemma \ref{lemEElow1}, using the bounds  \eqref{boundm_1}-\eqref{boundm_3} on the symbols, and Lemma \ref{touse}(ii).

\vskip10pt
\subsection{Proof of Lemma \ref{lemEEZ2}}\label{secwlem2}
We start by computing the time evolution of $E_Z$ in \eqref{defEZ}.
Using the definitions of the energies $E^{(2)}_Z$ in  \eqref{defE2Z}, and $E^{(3)}_Z$ in \eqref{defE3Z}-\eqref{E_Z3},
and the evolution equation for $Z$ derived in Lemma \ref{lemeqZ}, see \eqref{eqZ}-\eqref{eqZR},
we can calculate:
\begin{align}
\label{evolEZ}
& \frac{d}{dt} \big( E_Z^{(2)} + E_Z^{(3)} \big) = K_0
  + \frac{1}{4\pi^2} \Re \sum_\star \big( K_{1,\eps_1\eps_2} + K_{2,\eps_1\eps_2} + K_{3,\eps_1\eps_2} \big) + R ,
\end{align}
where
\begin{align}
\label{evolE_Z0}
\begin{split}
K_0 & := \frac{1}{2\pi} \Re \int_{\R} \bar{\what{Z}}(\xi)  \, i \what{\Sigma_\gamma Z}(\xi) \,\varphi_{\geq -20}^2(\xi)\, d\xi ,
\end{split}
\end{align}
\begin{align}
\label{evolE_Z1}
\begin{split}
K_{1,\eps_1\eps_2} & :=
\int_{\R \times \R} \bar{\mathcal{F}\big(\partial_t Z - i\Lambda Z\big)}(\xi)
  \, m_{1,\eps_1\eps_2} (\xi,\eta) \,
  \what{u_{\eps_1}}(\xi-\eta) \what{Z_{\eps_2}}(\eta) \,  d\xi d\eta
\\
& + \int_{\R \times \R} \bar{\what{Z}}(\xi)
  \, m_{1,\eps_1\eps_2} (\xi,\eta) \,
  \mathcal{F}\big(\partial_t u_{\eps_1} - i\eps_1\Lambda u_{\eps_1}\big)(\xi-\eta) \what{Z_{\eps_2}}(\eta) \,  d\xi d\eta
\\
& + \int_{\R \times \R} \bar{\what{Z}}(\xi)
  \, m_{1,\eps_1\eps_2} (\xi,\eta) \,
  \what{u_{\eps_1}}(\xi-\eta) \mathcal{F}\big(\partial_t Z_{\eps_2} - i\eps_2\Lambda Z_{\eps_2}\big)(\eta) \,  d\xi d\eta ,
\end{split}
\end{align}
with $m_{1,\eps_1\eps_2}$ as in \eqref{E_Z1},
\begin{align}
\label{evolE_Z2}
\begin{split}
K_{2, \eps_1\eps_2} & = K_{2, \eps_1\eps_2}^1 + K_{2, \eps_1\eps_2}^2 + K_{2, \eps_1\eps_2}^3 ,
\\
K_{2, \eps_1\eps_2}^1 & := \int_{\R \times \R} \bar{\mathcal{F}\big(\partial_t Z - i\Lambda Z\big)}(\xi)
  \, m_{2,\eps_1\eps_2} (\xi,\eta) \,
  \what{S u_{\eps_1}}(\xi-\eta) \what{u_{\eps_2}}(\eta) \,  d\xi d\eta ,
\\
K_{2, \eps_1\eps_2}^2 & := \int_{\R \times \R} \bar{\what{Z}}(\xi)
  \, m_{2,\eps_1\eps_2} (\xi,\eta) \,
  \mathcal{F}\big(\partial_t S u_{\eps_1} - i\eps_1\Lambda S u_{\eps_1}\big)(\xi-\eta) \what{u_{\eps_2}}(\eta) \,  d\xi d\eta ,
\\
K_{2, \eps_1\eps_2}^3 & := \int_{\R \times \R} \bar{\what{Z}}(\xi)
  \, m_{2,\eps_1\eps_2} (\xi,\eta) \,
  \what{S u_{\eps_1}}(\xi-\eta) \mathcal{F}\big(\partial_t u_{\eps_2} - i\eps_2\Lambda u_{\eps_2}\big)(\eta) \,  d\xi d\eta ,
\end{split}
\end{align}
where $m_{2,\eps_1\eps_2}$ is defined in \eqref{E_Z2}, and
\begin{align}
\label{evolE_Z3}
\begin{split}
K_{3, \eps_1\eps_2} & = K_{3, \eps_1\eps_2}^1 + K_{3, \eps_1\eps_2}^2 + K_{3, \eps_1\eps_2}^3 ,
\\
K_{3, \eps_1\eps_2}^1 & := \int_{\R \times \R} \bar{\mathcal{F}\big(\partial_t Z - i\Lambda Z\big)}(\xi)
  \, m_{3,\eps_1\eps_2} (\xi,\eta) \,
  \what{u_{\eps_1}}(\xi-\eta) \what{u_{\eps_2}}(\eta) \,  d\xi d\eta
\\
K_{3, \eps_1\eps_2}^2 & := \int_{\R \times \R} \bar{\what{Z}}(\xi)
  \, m_{3,\eps_1\eps_2} (\xi,\eta) \,
  \mathcal{F}\big(\partial_t u_{\eps_1} - i\eps_1\Lambda u_{\eps_1}\big)(\xi-\eta) \what{u_{\eps_2}}(\eta) \,  d\xi d\eta
\\
K_{3, \eps_1\eps_2}^2 & := \int_{\R \times \R} \bar{\what{Z}}(\xi)
  \, m_{3,\eps_1\eps_2} (\xi,\eta) \,
  \what{u_{\eps_1}}(\xi-\eta) \mathcal{F}\big(\partial_t u_{\eps_2} - i\eps_2\Lambda u_{\eps_2}\big)(\eta) \,  d\xi d\eta ,
\end{split}
\end{align}
where $m_{3,\eps_1\eps_2}$ is defined in \eqref{E_Z3}.
The remainder $R$ comes from quartic terms involving the remainder $\mathcal{O}_Z$ in \eqref{eqZ},
and quartic terms involving the quadratic part of the function $V$ as it appears in \eqref{eqZN1}, that is $V_2$ in \eqref{Vu+baru1}.
Using the estimate \eqref{eqZR} for the first,
and arguments similar to the one used for $A_{2,2}$, see \eqref{A_22} and \eqref{estA_22}, for the second,
we see that
\begin{align*}
 | R(t) | \lesssim \e_1^4 {(1+t)}^{-1+8p_0} .
\end{align*}
We now show that all of the terms in \eqref{evolE_Z1}-\eqref{evolE_Z3} are bounded by $\e_1^4 {(1+t)}^{-1+8p_0}$ as well.

\subsubsection{Estimate of $K_0$}
The term in \eqref{evolE_Z0} is like to $A_1$ in \eqref{EE21}, with $Z$ instead of $W$.
We can then estimate it in the same way we estimated $A_1$ in subsection \ref{prooflemEE23}, $|K_0(t)| \lesssim {\| Z(t) \|}_{L^2}^2 \e_1^2\langle t\rangle^{-1}$.

\subsubsection{Estimate of $K_{1,\eps_1\eps_2}$}
We rearrange the integrals in \eqref{evolE_Z1}, using the equation \eqref{eqZ} for $Z$,
and identify the terms that could potentially lose derivatives and need additional arguments.
Let $\mathcal{N}_Z := \mathcal{N}_{Z,1} + \mathcal{N}_{Z,2} + \mathcal{N}_{Z,3}$,
see \eqref{eqZN2}, and write
\begin{align*}
& \sum_\star K_{1,\eps_1\eps_2} = K_{1,1} + K_{1,2} + K_{1,3}
\end{align*}
where
\begin{align}
\label{evolE_Z11}
K_{1,1} & := \sum_{\eps_1 = \pm } \int_{\R \times \R}
  \big[ \bar{i \what{\Sigma_\gamma Z}}(\xi) \what{Z}(\eta) + \bar{\what{Z}}(\xi) i \what{\Sigma_\gamma Z}(\eta) \big]
  m_{1,\eps_1 +} (\xi,\eta) \what{u_{\eps_1}}(\xi-\eta) \,  d\xi d\eta ,
\\
\label{evolE_Z12}
K_{1,2} & := \sum_{\eps_1 = \pm } \int_{\R \times \R}
  \big[ \bar{\what{\mathcal{Q}_Z}}(\xi) \what{Z}(\eta) + \bar{\what{Z}}(\xi) \what{\mathcal{Q}_Z}(\eta) \big]
  m_{1,\eps_1 +} (\xi,\eta) \what{u_{\eps_1}}(\xi-\eta) \,  d\xi d\eta ,
\end{align}
and
\begin{align}
\label{evolE_Z13}
\begin{split}
K_{1,3} & := \sum_{\eps_1 = \pm } \int_{\R \times \R}
  \big[ \bar{\what{\mathcal{N}_Z}}(\xi) \what{Z}(\eta) + \bar{\what{Z}}(\xi) \what{\mathcal{N}_Z}(\eta) \big]
  m_{1,\eps_1 +} (\xi,\eta) \what{u_{\eps_1}}(\xi-\eta) \,  d\xi d\eta
\\
&+\sum_{\eps_1 = \pm } \int_{\R \times \R}
  \big[ \bar{\what{\mathcal{O}_Z}}(\xi) \what{Z}(\eta) + \bar{\what{Z}}(\xi) \what{\mathcal{O}_Z}(\eta) \big]
  m_{1,\eps_1 +} (\xi,\eta) \what{u_{\eps_1}}(\xi-\eta) \,  d\xi d\eta ,
\end{split}
\end{align}
\begin{align}
\label{evolE_Z14}
\begin{split}
K_{1,4} & := \sum_\star \int_{\R \times \R} \bar{\what{Z}}(\xi)
  \, m_{1,\eps_1\eps_2} (\xi,\eta) \,
  \mathcal{F}\big(\partial_t u_{\eps_1} - i\eps_1\Lambda u_{\eps_1}\big)(\xi-\eta) \what{Z_{\eps_2}}(\eta) \,  d\xi d\eta
  \\
  & + \sum_{\eps_1 = \pm } \int_{\R \times \R} \bar{\mathcal{F}\big(\partial_t Z - i\Lambda Z\big)}(\xi)
  \, m_{1,\eps_1 -} (\xi,\eta) \,
  \what{u_{\eps_1}}(\xi-\eta) \what{Z_{-}}(\eta) \,  d\xi d\eta
  \\
  & + \sum_{\eps_1 = \pm }\int_{\R \times \R} \bar{\what{Z}}(\xi)
  \, m_{1,\eps_1 -} (\xi,\eta) \,
  \what{u_{\eps_1}}(\xi-\eta) \mathcal{F}\big(\partial_t Z_{-} + i\Lambda Z_{-}\big)(\eta) \,  d\xi d\eta.
\end{split}
\end{align}

\subsubsection*{Estimates of $K_{1,1}$ and $K_{1,2}$} These terms are estimated using Lemmas \ref{lemSigmabulk} and \ref{lemQbulk} with $F=Z$, 
\begin{equation*}
|K_{1,1}|+|K_{1,2}|\lesssim \e_1^4(1+t)^{-1+8p_0}
\end{equation*}
The symbol conditions on the multipliers $m_{1,\eps_1+}$ are satisfied because the two components of $m_{1,\eps_1+}$ are similar to the multipliers $a_{\eps_1+}^{N_0}$ and $m_{N_0}$ defined in \eqref{E^31}-\eqref{E^3a}, and the desired properties have been verified in section \ref{secEE}.

\subsubsection*{Estimates of $K_{1,3}$ and $K_{1,4}$}
In all the terms appearing in \eqref{evolE_Z13} and \eqref{evolE_Z14} there are no losses of derivatives since $\mathcal{N}_Z$ is a semilinear term,
and the symbols $m_{3,\eps_1-}$ are strongly elliptic, as we can see from the second bound in \eqref{boundm_1}. The desired estimate follows by the same arguments as before, using Lemma \ref{touse} (ii). We always estimate $Z$, $\mathcal{N}_Z$, $\mathcal{O}_Z$, and $\partial_t Z-i\Lambda Z$ in $L^2$ (using \eqref{azx5}, \eqref{estNZ}, \eqref{eqZR}, and \eqref{estLZ}) and $u$ and $\partial_t u-i\Lambda u$ in $L^\infty$ (using \eqref{bing11} and \eqref{boundLu}).

\subsubsection{Estimate of $K_{2,\eps_1\eps_2}$}
The main difficulty in estimating the terms in \eqref{evolE_Z2} comes from the singularity in the symbol $m_{2,\eps_1\eps_2}$,
see \eqref{boundm_2}. We can overcome this using the low frequencies information in \eqref{bing11}.
We begin with the first term in \eqref{evolE_Z2}, and split it into two pieces depending on the size of the frequency $\xi-\eta$,
\begin{align*}
& K_{2, \eps_1\eps_2}^1 = I + II ,
\\
& I := \int_{\R \times \R} \bar{\mathcal{F}\big(\partial_t Z - i\Lambda Z\big)}(\xi)
  \, m_{2,\eps_1\eps_2} (\xi,\eta) \,
  \varphi_{\leq 0}((1+t)^2(\xi-\eta)) \what{S u_{\eps_1}}(\xi-\eta) \what{u_{\eps_2}}(\eta) \,  d\xi d\eta ,
\\
& II := \int_{\R \times \R} \bar{\mathcal{F} \big(\partial_t Z - i\Lambda Z\big)}(\xi)
  \, m_{2,\eps_1\eps_2} (\xi,\eta) \,
  \varphi_{\geq 1}((1+t)^2(\xi-\eta)) \what{S u_{\eps_1}}(\xi-\eta) \what{u_{\eps_2}}(\eta) \,  d\xi d\eta .
\end{align*}
We observe that, in view of \eqref{bing11}, we have
\begin{align*}
{\| P_l S u \|}_{L^\infty} \lesssim 2^{l(1-p_1)} \e_1 {(1+t)}^{4p_0 - 2p_1} , \quad \mbox{for} \quad 2^l \lesssim {(1+t)}^{-2} .
\end{align*}
Then, to estimate the first term it suffices to use Lemma \ref{touse}(ii)
together with the bound on $m_{2,\eps_1\eps_2}$ in \eqref{boundm_2},
followed by \eqref{estLZ}, and the a priori assumptions \eqref{bing11}:
\begin{equation*}
\begin{split}
| I | &\lesssim \sum_{k,k_1,k_2 \in \mathbb{Z}, \, 2^{k_1} \lesssim \langle t\rangle^{-2}}
  {\| m_{2,\eps_1\eps_2}^{k,k_1,k_2} \|}_{S^\infty} {\| P_k^\prime (\partial_t - i \Lambda) Z \|}_{L^2}
  {\| P_{k_1}^\prime S u \|}_{L^\infty} {\| P_{k_2}^\prime u \|}_{L^2}
\\
&\lesssim  \sum_{(k,k_1,k_2) \in \mathcal{X}, \, 2^{k_1} \lesssim \langle t\rangle^{-2}, \, k_2 \geq k - 20}
  2^{k/2} 2^{-k_1/2}2^{N_1 k}  \e_1^2 \langle t\rangle^{-1/2+4p_0} 2^{k/2} 2^{\max(k,0)}\\
	& \times \e_1  2^{k_1(1-p_1)} \langle t\rangle^{4p_0 - 2p_1} \e_1 2^{-N_0 k_2^+}\langle t\rangle^{p_0}\lesssim \e_1^4 \langle t\rangle^{-1} .
\end{split}
\end{equation*}
To estimate $II$ we use Lemma \ref{touse}(ii) together with the symbol bound in \eqref{boundm_2'},
\eqref{estLZ}, the a priori assumption \eqref{bing11}, and the inequality $N_2\geq N_1+2$:
\begin{align*}
| II | \lesssim {\|  P_{\geq -2\log_2(1+t) - 5} |\partial_x|^{-1/2} S u \|}_{L^2}
  \sum_{k,k_2 \in \mathbb{Z}} {\| |\xi-\eta|^{1/2} m_{2,\eps_1\eps_2}(\xi,\eta) \varphi_k(\xi) \varphi_{k_2}(\eta) \|}_{S^\infty}
  \\ \times  {\| P_k^\prime (\partial_t - i \Lambda) Z \|}_{L^2}
  {\| P_{k_2}^\prime u \|}_{L^\infty}
\\
\lesssim \e_1 \langle t\rangle^{4p_0} \sum_{k,k_2 \in \mathbb{Z}, \, k_2 \geq k - 20} 2^{(N_1+1/2) k}
  \e_1^2 \langle t\rangle^{-1/2+4p_0} 2^{k/2} 2^{k^+}{\| P_{k_2}^\prime u \|}_{L^\infty}
\\
\lesssim \e_1^4 {(1+t)}^{-1+8p_0} .
\end{align*}

The terms $K_{2,\eps_1\eps_2}^2$ and $K_{2,\eps_1\eps_2}^3$ in \eqref{evolE_Z2} are similar (in fact easier). We provide the details only for $K_{2,\eps_1\eps_2}^3$.
We first write
\begin{align*}
& K_{2, \eps_1\eps_2}^3 = III + IV ,
\\
& III := \int_{\R \times \R} \bar{\what{Z}}(\xi)
  \, m_{2,\eps_1\eps_2} (\xi,\eta) \, \varphi_{\leq 0}((1+t)^2(\xi-\eta)) \what{S u_{\eps_1}}(\xi-\eta)
  \mathcal{F}\big(\partial_t u_{\eps_2} - i\eps_2\Lambda u_{\eps_2}\big)(\eta) \,  d\xi d\eta ,
\\
& IV := \int_{\R \times \R} \bar{\what{Z}}(\xi)
  \, m_{2,\eps_1\eps_2} (\xi,\eta) \, \varphi_{\geq 1}((1+t)^2(\xi-\eta)) \what{S u_{\eps_1}}(\xi-\eta)
  \mathcal{F}\big(\partial_t u_{\eps_2} - i\eps_2\Lambda u_{\eps_2}\big)(\eta) \,  d\xi d\eta .
\end{align*}
Then we use Lemma \ref{touse}(ii), the symbol bound \eqref{boundm_2},
the estimates \eqref{boundLu}, and the a priori assumptions \eqref{bing11} to obtain
\begin{equation*}
\begin{split}
| III | &\lesssim \sum_{k,k_1,k_2 \in \mathbb{Z}, \, 2^{k_1} \leq {(1+t)}^{-2}}
  {\| m_{2,\eps_1\eps_2}^{k,k_1,k_2} \|}_{S^\infty} {\| P_k^\prime Z \|}_{L^2}
  {\| P_{k_1}^\prime S u \|}_{L^\infty} {\| P_{k_2}^\prime  (\partial_t - i \Lambda)u \|}_{L^2}
\\
&\lesssim  \e_1^4 \langle t\rangle^{-1/2+9p_0} \sum_{(k,k_1,k_2) \in \mathcal{X}, \, 2^{k_1} \leq \langle t\rangle^{-2}, \, k_2 \geq k - 20}
  2^{N_1 k} 2^{k/2} 2^{-k_1/2}\cdot 2^{k_1(1-p_1)} 2^{k_2/2} 2^{-(N_0 -2)k_2^+}
\\
&\lesssim \e_1^4 \langle t\rangle^{-1} .
\end{split}
\end{equation*}
To estimate $IV$ we use again Lemma \ref{touse}(ii), this time together with the bound \eqref{boundm_2'},
the low frequency assumption in \eqref{bing11}, and the estimates \eqref{boundLu}, and see that
\begin{align*}
| IV | \lesssim {\|  P_{\geq -2\log_2(1+t) - 5} |\partial_x|^{-1/2} S u \|}_{L^2} \sum_{k,k_2 \in \mathbb{Z}}
  {\| |\xi-\eta|^{1/2} m_{2,\eps_1\eps_2}(\xi,\eta) \varphi_k(\xi) \varphi_{k_2}(\eta) \|}_{S^\infty}\\
	\times {\| P_k^\prime Z \|}_{L^2}{\| (\partial_t - i \Lambda) P_{k_2}^\prime u \|}_{L^\infty}
\\
\lesssim  \e_1 \langle t\rangle^{4p_0} \sum_{k,k_2 \in \mathbb{Z}, \, k_2 \geq k - 20} 2^{(N_1 + 1/2) k} \e_1 \langle t\rangle^{4p_0}{\| (\partial_t - i \Lambda) P_{k_2}^\prime u \|}_{L^\infty}
   \\
\lesssim \e_1^4 \langle t\rangle^{-1+8p_0} .
\end{align*}
This concludes the desired estimate for the integrals in \eqref{evolE_Z2}.

\subsubsection{Estimate of $K_{3,\eps_1\eps_2}$}
We observe that the symbols $m_{3,\eps_1\eps_2}$ in \eqref{evolE_Z3}
satisfy the same bounds as the symbols $m_{2,\eps_1\eps_2}$, see \eqref{boundm_2}-\eqref{boundm_2'} and \eqref{boundm_3}-\eqref{boundm_3'}.
Moreover, the terms $K_{3,\eps_1\eps_2}^j$, $j=1,2,3$, in \eqref{evolE_Z3} are trilinear expressions of $(Z,u,u)$,
while the terms $K_{2,\eps_1\eps_2}^j$, $j=1,2,3$, in \eqref{evolE_Z2} are trilinear expressions of $(Z,u,Su)$.
Thus, it is clear that estimating the terms in \eqref{evolE_Z3} is easier than estimating those in \eqref{evolE_Z2},
because the bounds we assume on $u$ are stronger than those we have on $Su$.
This concludes the proof of Proposition \ref{proEEZ}.

\section{Energy estimates IV: weighted estimates for low frequencies}\label{secweightedlow}

In this section we improve our control on the low frequency component of $S u$. The basic quadratic energy is
\begin{align}
\label{defE2Zlow}
E_{Su_\low}^{(2)}(t) = \frac{1}{4\pi} \int_\R {\big| \what{S u}(t,\xi) \big|}^2 {|\xi|}^{-1}
  \mathcal{P}((1+t)^2|\xi|) \varphi^2_{\leq -10}(\xi) \, d\xi ,
\end{align}
where $\mathcal{P}$ is defined in \eqref{p0}. We will prove the following:

\begin{proposition}\label{proEEZlow}
Assume that $u$ satisfies \eqref{bing11}--\eqref{bing31}. Then
\begin{equation}
\label{proEEZlowconc}
 \sup_{t\in[0,T]} (1+t)^{-8p_0} E_{Su_\low}^{(2)}(t) \lesssim \e_0^2.
\end{equation}
\end{proposition}

\subsection{The cubic low frequency weighted energy} We start from Proposition \ref{lemeqZ} with $k=0$ and apply the projection operator $P_{\leq -10}$. It follows that
\begin{align}
 \label{eqSu}
P_{\leq -10}(\partial_t - i\Lambda) (S u) = P_{\leq -10} Q_0(V, Su) + P_{\leq -10} \mathcal{N}_{Su,1} +
  P_{\leq -10} \mathcal{N}_{Su,2} + P_{\leq -10}\mathcal{O}_{Su} .
\end{align}
Here the symbol of $Q_0$ is $q_0(\xi,\eta) = -i\xi\chi(\xi-\eta,\eta)$, the nonlinear terms are
\begin{align}
\label{eqSuN}
\begin{split}
&\mathcal{N}_{Su,1}  := 
(i/2) Q_0\big( \partial_x |\partial_x|^{-1/2} (Su - \bar{Su}), u\big)
\\
& + \sum_\star \big[A_{\eps_1\eps_2}(Su_{\eps_1},u_{\eps_2}) + A_{\eps_1\eps_2}(u_{\eps_1},S u_{\eps_2})
  + B_{\eps_1\eps_2}(Su_{\eps_1},u_{\eps_2}) + B_{\eps_1\eps_2}(u_{\eps_1},S u_{\eps_2})\big] ,
\\
&\mathcal{N}_{Su,2}  := 
(i/2) \widetilde{Q}_0( \partial_x |\partial_x|^{-1/2} (u-\bar{u}, u) )
+ (i/2) Q_0\big( \partial_x |\partial_x|^{-1/2} (u-\bar{u}), u\big)
\\
& + \sum_\star \big[\widetilde{A}_{\eps_1\eps_2}(u_{\eps_1},u_{\eps_2}) + (3/2) A_{\eps_1\eps_2}(u_{\eps_1},u_{\eps_2})
 + \widetilde{B}_{\eps_1\eps_2} (u_{\eps_1},u_{\eps_2}) + (3/2) B_{\eps_1\eps_2}(u_{\eps_1},u_{\eps_2})\big],
\end{split}
\end{align}
and the remainder satisfies
\begin{align}
\label{estRS}
{\| |\partial_x|^{-1/2}\mathcal{O}_{Su}(t) \|}_{L^2} \lesssim \e_1^3 {(1+t)}^{-1+4p_0} .
\end{align}


%

 The natural cubic energy associated to \eqref{eqSu}-\eqref{eqSuN} is given by
\begin{align}
\label{defE3Zlow}
E^{(3)}_{Su_\low} = E_{Su_\low,1} + E_{Su_\low,2} ,
\end{align}
where, with $\varphi_t(\xi) = \varphi_{\geq 4}((1+t)^2\xi) \varphi^2_{\leq -10}(\xi)$,\begin{align}
\begin{split}
\label{E_Z1low}
E_{Su_\low,1}(t) := & \frac{1}{4\pi^2} \sum_\star \Re \int_{\R \times\R} |\xi|^{-1/2} \bar{\what{Su}}(\xi,t) \what{Su_{\eps_2}}(\eta,t)
  \what{u_{\eps_1}}(\xi-\eta,t) q_{\eps_1 \eps_2}(\xi,\eta) \,d\xi d\eta ,
\\
q_{\eps_1 \eps_2}(\xi,\eta) & := (1+\eps_1)(1+\eps_2)\frac{|\xi|^{1/2}(\xi-\eta)
  \big[ \xi|\xi|^{-1} \varphi_t(\xi) \chi(\xi-\eta,\eta) - \eta|\eta|^{-1}  \varphi_t(\eta)\chi(\eta-\xi,\xi) \big]}{
  8 |\xi-\eta|^{1/2}(|\xi|^{3/2}-\eps_1|\xi-\eta|^{3/2}-\eps_2|\eta|^{3/2})}
\\
& +  \frac{ i (1+\eps_1) \eps_2\varphi_t(\xi) \,\eta \, q_0(\xi,\xi-\eta)}{
  4|\eta|^{1/2} |\xi|^{1/2} (|\xi|^{3/2} -\eps_1|\xi-\eta|^{3/2} -\eps_2|\eta|^{3/2})}
\\
& - i \frac{\varphi_t(\xi) \big( a_{\eps_1\eps_2}(\xi,\xi-\eta) + a_{\eps_1\eps_2}(\xi,\eta) +
   b_{\eps_1\eps_2}(\xi,\xi-\eta) + b_{\eps_1\eps_2}(\xi,\eta) \big) }{
  |\xi|^{1/2} (|\xi|^{3/2} -\epsilon_1|\xi-\eta|^{3/2} -\epsilon_2|\eta|^{3/2})} ,
\end{split}
\end{align}
and
\begin{align}
\begin{split}
\label{E_Z2low}
E_{Su_\low,2}(t) & := \frac{1}{4\pi^2} \sum_\star \Re \int_{\R \times\R} |\xi|^{-1/2} \bar{\what{Su}}(\xi,t) \what{u_{\eps_2}}(\eta,t)
  r_{\eps_1 \eps_2}(\xi,\eta) \what{u_{\eps_1}}(\xi-\eta,t) \,d\xi d\eta ,
\\
r_{\eps_1 \eps_2}(\xi,\eta) & := \frac{i \eps_1 (1+\eps_2)\varphi_t(\xi)  (\xi-\eta)[\widetilde{q_0}(\xi,\eta) + q_0(\xi,\eta)] }{
  4|\xi-\eta|^{1/2} |\xi|^{1/2} (|\xi|^{3/2} -\epsilon_1|\xi-\eta|^{3/2} -\epsilon_2|\eta|^{3/2})}
\\
& + \frac{-i \varphi_t(\xi) \big(\widetilde{a}_{\eps_1\eps_2}(\xi,\eta) + (3/2)a_{\eps_1\eps_2}(\xi,\eta)
  + \widetilde{b}_{\eps_1\eps_2}(\xi,\eta) + (3/2)b_{\eps_1\eps_2}(\xi,\eta) \big) }{
  |\xi|^{1/2} (|\xi|^{3/2} -\epsilon_1|\xi-\eta|^{3/2} -\epsilon_2|\eta|^{3/2})} ,
\end{split}
\end{align}
where we are using the notation \eqref{lemcomm2}.
The first part of the symbol $q_{\eps_1\eps_2}$ in \eqref{E_Z1low} is needed to correct the quadratic interaction
$Q_0(V, Su)$ in \eqref{eqSu}, after the proper symmetrization.
This is similar to the symbols in \eqref{E^31}, \eqref{mlow}, and \eqref{E_Z1}.
The rest of the symbol $q_{\eps_1\eps_2}$ takes into account the nonlinear term $\mathcal{N}_{Su_\low,1}$ in \eqref{eqSuN}.
The symbol $r_{\eps_1\eps_2}$ is naturally associated to the nonlinear term $\mathcal{N}_{Su_\low,2}$ in \eqref{eqSuN}.

The estimate \eqref{proEEZlowconc} follows from the two lemmas below.

\begin{lemma}\label{lemEEZlow1}
Under the assumptions of Proposition \ref{proEEZlow}, for any $t\in[0,T]$ we have
\begin{equation}
\label{lemEEZlow1conc}
| E_{Su_\low}^{(3)}(t) | \lesssim \e_1^3 {(1+t)}^{8p_0} .
\end{equation}
\end{lemma}

\begin{lemma}
\label{lemEEZlow2}
Under the assumptions of Proposition \ref{proEEZlow}, for any $t\in[0,T]$ we have
\begin{equation}\label{lemEEZlow2conc}
\frac{d}{dt} E_{Su_\low}(t) \leq C\e_1^3 {(1+t)}^{-1+8p_0} +40p_1E_{Su_{low}}^{(2)}(t){(1+t)}^{-1}.
\end{equation}
\end{lemma}

\vskip10pt
\subsection{Analysis of the symbols and proof of Lemma \ref{lemEEZlow1}}\label{secwlowsym}
To prove Lemma \ref{lemEEZlow1} and \ref{lemEEZlow2}
we need appropriate bounds for the symbols in \eqref{E_Z1low} and \eqref{E_Z2low}.
These can be obtained as in the previous three sections.
In particular, using the bounds \eqref{bounda/1} and \eqref{boundb/1}, 
it is not hard to verify that
\begin{align}
\label{boundqlow}
\begin{split}
{\| q_{\eps_1,\eps_2}^{k,k_1,k_2} \|}_{S^\infty} & \lesssim 2^{-k_2/2}
  \mathbf{1}_{\mathcal{X}}(k,k_1,k_2)  \mathbf{1}_{[-2\log_2(2+t)-10,0]}(k) .
\end{split}
\end{align}
This is a somewhat rougher bound than what actually holds true, but it will be sufficient for our estimates.
Furthermore, using also
\begin{align*}
{\Big\| \frac{|\xi|^{1/2} \eta  \chi(\eta,\xi-\eta) }{
  |\xi|^{3/2} - |\xi-\eta|^{3/2}-\eps_2|\eta|^{3/2}} \varphi_{k}(\xi) \varphi_{k_1}(\xi-\eta) \Big\|}_{S^\infty} \lesssim 1 ,
\end{align*}
one can verify that
\begin{align}
\label{boundqlowkk_1}
& {\| |\eta|^{1/2} q_{\eps_1,\eps_2}(\xi,\eta) \varphi_{k}(\xi)\varphi_{k_1}(\xi-\eta) \|}_{S^\infty} \lesssim \mathbf{1}_{[-2\log_2(2+t)-10,0]}(k) .
\end{align}
We also have
\begin{align}
\label{boundqlowk_1k_2}
& {\| |\eta|^{1/2} q_{\eps_1\eps_2}(\xi,\eta) \varphi_{k_1}(\xi-\eta)\varphi_{k_2}(\eta) \|}_{S^\infty} \lesssim 1
\end{align}



Similarly, using again \eqref{bounda/1} and \eqref{boundb/1} one can see that
\begin{align}
\label{boundrlow}
\begin{split}
{\| r_{\eps_1,\eps_2}^{k,k_1,k_2} \|}_{S^\infty} & \lesssim 2^{-k_1/2} \mathbf{1}_{\mathcal{X}}(k,k_1,k_2)
  \mathbf{1}_{[-2\log_2(2+t)-10,0]}(k) .
\end{split}
\end{align}
and
\begin{align}
\label{boundrlow'}
\begin{split}
{\| |\xi-\eta|^{1/2} r_{\eps_1,\eps_2}\varphi_{k_1}(\xi-\eta)\varphi_{k_2}(\eta) \|}_{S^\infty} & \lesssim 1.
\end{split}
\end{align}

The proof of \eqref{lemEEZlow1conc} can then be done in a similar fashion to what was done before in section \ref{secEElow}
in the proof of Lemma \ref{lemEElow1}, by using the bounds \eqref{boundqlow}-\eqref{boundrlow'} above.

\subsection{Proof of Lemma \ref{lemEEZlow2}}
As in section \ref{seclemEElow2}, we use the definition of the quadratic energy \eqref{defE2Zlow},
the equation \eqref{eqSu}-\eqref{eqSuN},
the formulas for the cubic energies \eqref{defE3Zlow}-\eqref{E_Z2low},
and a symmetrization argument like the one performed for the term \eqref{EE22} and leading to \eqref{A_2}, to calculate
\begin{align*}
\frac{d}{dt} \big( E_{Su_\low}^{(2)} + E_{Su_\low}^{(3)} \big) =
  \frac{1}{2\pi}  L_1 &+ \frac{1}{2\pi} \Re L_2
  + \frac{1}{4\pi^2} \Re \sum_\star \big( L_{3,\eps_1\eps_2} + L_{4,\eps_1\eps_2} + L_{5,\eps_1\eps_2} + L_{6,\eps_1\eps_2} \big)
\\
& + \frac{1}{4\pi^2} \Re \sum_\star \big( L_{7,\eps_1\eps_2} + L_{8,\eps_1\eps_2} + L_{9,\eps_1\eps_2} + L_{10,\eps_1\eps_2} \big)
+ R ,
\end{align*}
where
\begin{align*}
L_1 & = \int_\R {\big| \what{S u}(\xi) \big|}^2 \, (1+t) \mathcal{P}^\prime ((1+t)^2|\xi|) \varphi^2_{\leq -10}(\xi) \, d\xi ,
\\
L_2 & = \int_\R \bar{\what{Su}}(\xi) \mathcal{F}\big(\partial_t S u - i\Lambda S u \big)(\xi)
  \, \varphi_{\leq 3}((1+t)^2\xi)|\xi|^{-1}\mathcal{P}((1+t)^2|\xi|) \varphi^2_{\leq -10}(\xi) \, d\xi ,
\end{align*}
\begin{align}
\label{evolEZlow}
\begin{split}
L_{3, \eps_1\eps_2} & = \int_{\R \times \R}  {|\xi|}^{-1/2} \bar{\what{S u}}(\xi)
  \, \partial_t q_{\eps_1\eps_2}(\xi,\eta)  \what{u_{\eps_1}}(\xi-\eta) \what{S u_{\eps_2}}(\eta) \,  d\xi d\eta ,
\\
L_{4, \eps_1\eps_2} & = \int_{\R \times \R}  {|\xi|}^{-1/2} \bar{\mathcal{F}\big(\partial_t S u - i\Lambda S u\big)}(\xi)
  \, q_{\eps_1\eps_2}(\xi,\eta)  \what{u_{\eps_1}}(\xi-\eta) \what{S u_{\eps_2}}(\eta) \,  d\xi d\eta ,
\\
L_{5, \eps_1\eps_2} & = \int_{\R \times \R}  {|\xi|}^{-1/2} \bar{\what{S u}}(\xi)
  \, q_{\eps_1\eps_2}(\xi,\eta) \mathcal{F} \big(\partial_t u_{\eps_1} - i\eps_1\Lambda u_{\eps_1} \big)(\xi-\eta) \what{Su_{\eps_2}}(\eta)
  \,d\xi d\eta ,
\\
L_{6, \eps_1\eps_2} & = \int_{\R \times \R}  {|\xi|}^{-1/2} \bar{\what{Su}}(\xi)
  \, q_{\eps_1\eps_2}(\xi,\eta) \what{u_{\eps_1}}(\xi-\eta)
  \mathcal{F}\big(\partial_t S u_{\eps_2} - i\eps_2\Lambda S u_{\eps_2} \big)(\eta) \,d\xi d\eta ,
\end{split}
\end{align}
\begin{align}
\label{evolEZlow2}
\begin{split}
L_{7, \eps_1\eps_2} & = \int_{\R \times \R}  {|\xi|}^{-1/2} \bar{\what{S u}}(\xi)
  \, \partial_t r_{\eps_1\eps_2}(\xi,\eta)  \what{u_{\eps_1}}(\xi-\eta) \what{u_{\eps_2}}(\eta) \,  d\xi d\eta ,
\\
L_{8, \eps_1\eps_2} & = \int_{\R \times \R}  {|\xi|}^{-1/2} \bar{\mathcal{F}\big(\partial_t S u - i\Lambda S u\big)}(\xi)
  \, r_{\eps_1\eps_2}(\xi,\eta)  \what{u_{\eps_1}}(\xi-\eta) \what{u_{\eps_2}}(\eta) \,  d\xi d\eta ,
\\
L_{9, \eps_1\eps_2} & = \int_{\R \times \R}  {|\xi|}^{-1/2} \bar{\what{S u}}(\xi)
  \, r_{\eps_1\eps_2}(\xi,\eta) \mathcal{F} \big(\partial_t u_{\eps_1} - i\eps_1\Lambda u_{\eps_1} \big)(\xi-\eta) \what{u_{\eps_2}}(\eta)
  \,d\xi d\eta ,
\\
L_{10, \eps_1\eps_2} & = \int_{\R \times \R}  {|\xi|}^{-1/2} \bar{\what{Su}}(\xi)
  \, r_{\eps_1\eps_2}(\xi,\eta) \what{u_{\eps_1}}(\xi-\eta)
  \mathcal{F}\big(\partial_t u_{\eps_2} - i\eps_2\Lambda u_{\eps_2} \big)(\eta) \,d\xi d\eta.
\end{split}
\end{align}
The remainder $R$ satisfies $| R(t) | \lesssim \e_1^4 {(1+t)}^{-1+8p_0}$, in view of \eqref{estRS}.

Recall that under the assumption \eqref{bing11} we have
\begin{align}
\label{estlow1}
\begin{split}
& {\| \varphi_{\geq 4}((1+t)^2\xi) |\xi|^{-1/2} \what{u}(t,\xi) \|}_{L^2} \lesssim \e_1 {(1+t)}^{p_0} ,
\\
& {\| \varphi_{\leq 4}((1+t)^2\xi) |\xi|^{-1/2+p_1} \what{u} \|}_{L^2} \lesssim \e_1  {(1+t)}^{p_0-2p_1} ,
\end{split}
\end{align}
and
\begin{align}
\label{estZlow1}
\begin{split}
& {\| \varphi_{\geq 4}((1+t)^2\xi) |\xi|^{-1/2} \what{Su}(t,\xi) \|}_{L^2}\lesssim \e_1 {(1+t)}^{4p_0} ,
\\
& {\| \varphi_{\leq 4}((1+t)^2\xi) |\xi|^{-1/2+p_1} \what{Su}(t,\xi) \|}_{L^2} \lesssim \e_1  {(1+t)}^{4p_0-2p_1} .
\end{split}
\end{align}

\subsubsection{Estimate of $L_1$}
Using \eqref{p0} 
the definition of $E^{(2)}_{Su_\low}$ in \eqref{defE2Zlow}, 
we see that
\begin{align*}
\frac{1}{2\pi}| L_1 | & \leq \frac{1}{4\pi}\int_\R {\big| \what{S u}(t,\xi) \big|}^2 \frac{20p_1}{(1+t)|\xi|} \, \mathcal{P}((1+t)^2|\xi|) \varphi^2_{\leq -10}(\xi)\, d\xi
  \leq 20p_1E^{(2)}_{Su_\low}(t)(1+t)^{-1},
\end{align*}
as desired.

\subsubsection{Estimate of $L_2$}
Notice that the integrand in $L_2$ is supported on a region where $|\xi| \lesssim {(1+t)}^{-2}$,
so that we can use the bound in \eqref{estLSu}, and \eqref{estZlow1}, to obtain
\begin{align*}
| L_2 |  &\lesssim \sum_{k\in\mathbb{Z}, \, 2^{k-10} \leq (1+t)^{-2}} 2^{ k(-1+2p_1)} {\| P_k^\prime S u(t) \|}_{L^2}
  {\| P_k^\prime \big( \partial_t - i \Lambda) S u (t) \|}_{L^2} {(1+t)}^{4p_1}
  \\
&\lesssim \sum_{k\in\mathbb{Z}, \, 2^{k-10} \leq (1+t)^{-2}} 2^{ k(-1+2p_1)} {\| P_k^\prime S u(t) \|}_{L^2}
   \e_1^2 2^{k/2} {(1+t)}^{-1+4p_0+4p_1}
\\
&\lesssim  \e_1 {(1+t)}^{-1+8p_0+2p_1} \e_1^2 \sum_{k\in\mathbb{Z}, \, 2^{k-10} \leq (1+t)^{-2}} 2^{p_1k}\\
&\lesssim \e_1^3 {(1+t)}^{-1+8p_0} .
\end{align*}

\subsubsection{Estimate of $L_{3,\eps_1\eps_2}$}
Looking at the definition of $q_{\eps_1\eps_2}$ in \eqref{E_Z1low}, and using \eqref{boundqlow},
one can see that
\begin{align*}
{\| (\partial_t q_{\eps_1\eps_2})^{k,k_1,k_2} \|}_{S^\infty}
  \lesssim {(1+t)}^{-1} 2^{-k_2/2} \mathbf{1}_{\mathcal{X}}(k,k_1,k_2) \mathbf{1}_{[-2\log_2 (2+t)-10, -2\log_2 (2+t) + 10]}(k) .
\end{align*}
We can then use Lemma \ref{touse} to estimate, for every $\eps_1,\eps_2 \in\{+,-\}$,
\begin{align*}
& | L_{3,\eps_1\eps_2} | \lesssim L_{3,1} + L_{3,2} ,
\\
& L_{3,1} :=  \langle t\rangle^{-1} \sum_{ (k,k_1,k_2) \in \mathcal{X}, \, k_1 \leq k_2+10, \, 2^k \approx \langle t\rangle^{-2}}
  2^{-k_2/2} 2^{-k/2} {\| P_k^\prime S u \|}_{L^2} {\| P_{k_1}^\prime u \|}_{L^\infty} {\| P_{k_2}^\prime S u \|}_{L^2} ,
\\
& L_{3,2} :=  \langle t\rangle^{-1} \sum_{ (k,k_1,k_2) \in \mathcal{X}, \, k_1 \geq k_2+10, \, 2^k \approx \langle t\rangle^{-2}}
  2^{-k_2/2} 2^{-k/2} {\| P_k^\prime S u \|}_{L^2} {\| P_{k_1}^\prime u \|}_{L^2} {\| P_{k_2}^\prime S u \|}_{L^\infty} .
\end{align*}
Using \eqref{estZlow1} and the $L^\infty$ bound in \eqref{bing11} we see that
\begin{align*}
L_{3,1} \lesssim \e_1^3 \langle t\rangle^{-1+4p_0}
  \sum_{k_1,k_2 \in \mathbb{Z}, \, 2^{k_2+20} \geq \langle t\rangle^{-2}}
  2^{k_1/10} 2^{-N_2 \max(k_1,0)} \langle t\rangle^{-1/2} 2^{-k_2/2} {\| P_{k_2}^\prime S u \|}_{L^2}\lesssim \e_1^3 \langle t\rangle^{-1} .
\end{align*}
For the second term we can use \eqref{estZlow1} and the inequality
\begin{align}
\label{Bern}
& {\| P_l Su \|}_{L^\infty} \lesssim 2^{(1-p_1) l} \e_1 {(1+t)}^{4p_0 - 2p_1} , \qquad \mbox{if } 2^l \lesssim {(1+t)}^{-2} ,
\end{align}
to obtain
\begin{align*}
L_{3,2}  \lesssim \e_1^3 \langle t\rangle^{-1+10p_0} \sum_{k_1,k_2 \in \mathbb{Z}, \, 2^{k_2} \lesssim \langle t\rangle^{-2}}
  2^{-k_2/2}  \min \big(2^{(1/2-p_1)k_1}, 2^{-N_0 \max(k_1,0)}) 2^{(1-p_1) k_2} 
  \lesssim \e_1^3 \langle t\rangle^{-4/3} .
\end{align*}
The desired bound $|L_{3,\eps_1\eps_2}|\lesssim \e_1^3(1+t)^{-1+8p_0}$ follows.

\subsubsection{Estimate of $L_{4,\eps_1\eps_2}$}
To deal with the term $L_{4,\eps_1\eps_2}$ we first use Lemma \ref{touse}(ii) to estimate
\begin{align*}
& | L_{4,\eps_1\eps_2} | \lesssim  L_{4,1} + L_{4,2} + L_{4,3} ,
\\
\\
\begin{split}
& L_{4,1} := \sum_{(k,k_1,k_2) \in \mathcal{X}, \, 2^{k_2+10} \leq (1+t)^{-2}}  {\| q_{\eps_1\eps_2}^{k,k_1,k_2} \|}_{S^\infty}
  2^{-k/2} {\| P_k^\prime (\partial_t - i\Lambda) S u \|}_{L^2} {\| P_{k_1}^\prime u \|}_{L^2} {\| P_{k_2}^\prime Su \|}_{L^\infty} ,
\\
\\
& L_{4,2} := {\| \varphi_{\geq -20}((1+t)^2\eta) |\eta|^{-1/2} \what{Su}(\eta) \|}_{L^2}
  \sum_{k,k_1 \in \mathbb{Z}, \, |k-k_1| \leq 10}
  {\| |\eta|^{1/2} q_{\eps_1\eps_2}(\xi,\eta) \varphi_{k}(\xi) \varphi_{k_1}(\xi-\eta) \|}_{S^\infty}
  \\ & \hskip200pt \times 2^{-k/2} {\| P_k^\prime (\partial_t - i\Lambda)Su \|}_{L^2} {\| P_{k_1}^\prime u \|}_{L^\infty} ,
\end{split}
\\
\\
& L_{4,3} :=
  \sum_{(k,k_1,k_2) \in \mathcal{X}, \, k_2 + 10 \geq k_1}  {\| q_{\eps_1\eps_2}^{k,k_1,k_2} \|}_{S^\infty}
  2^{-k/2} {\| P_k^\prime (\partial_t - i\Lambda) S u \|}_{L^2}
  {\| P_{k_1}^\prime u \|}_{L^\infty} {\| P_{k_2}^\prime Su \|}_{L^2} .
\end{align*}

Using the bounds \eqref{boundqlow}, \eqref{estLSu}, \eqref{bing11}, and \eqref{Bern}, we see that
\begin{align*}
L_{4,1} \lesssim \sum_{(k,k_1,k_2)\in \mathcal{X}, \, 2^{k_2} \leq \langle t\rangle^{-2}, \, |k-k_1| \leq 10}
  2^{-k_2/2} \e_1^2 \langle t\rangle^{-1/2+4p_0}  \e_1 2^{(1/2-p_1)k_1}2^{-N_0\max(k_1,0)} \langle t\rangle^{p_0}
  \\ \times \e_1 2^{(1-p_1)k_2} \langle t\rangle^{4p_0}
\lesssim \e_1^4 \langle t\rangle^{-1}.
\end{align*}
Using \eqref{estZlow1}, the symbol bound \eqref{boundqlowkk_1}, and \eqref{estLSu}, we obtain
\begin{align*}
L_{4,2} \lesssim \e_1 \langle t\rangle^{4p_0} \sum_{k,k_1 \in \mathbb{Z}, \, |k-k_1| \leq 10}
  \e_1^2 \langle t\rangle^{-1/2+4p_0}  \e_1 2^{k_1/10} 2^{-\max(k_1,0)} \langle t\rangle^{-1/2}\lesssim \e_1^4 \langle t\rangle^{-1+8p_0}.
\end{align*}
To estimate the last term we use \eqref{estZlow1}, \eqref{boundqlow}, and \eqref{estLSu}, and see that
\begin{align*}
L_{4,3} \lesssim \sum_{(k,k_1,k_2) \in \mathcal{X}, \, \langle t\rangle^{-2} \leq 2^k \leq 2^{10} , \, k_2 + 10 \geq k_1}
  \e_1^2 \langle t\rangle^{-1/2+4p_0} \big(2^{k/2} + \langle t\rangle^{-1/2})
\\ \times \e_1 2^{k_1/10} 2^{-\max(k_1,0)} \langle t\rangle^{-1/2} \e_1 \langle t\rangle^{4p_0}\lesssim \e_1^4 \langle t\rangle^{-1+8p_0}.
\end{align*}
The desired bound $|L_{4,\eps_1\eps_2}|\lesssim \e_1^3(1+t)^{-1+8p_0}$ follows.

\subsubsection{Estimate of $L_{5,\eps_1\eps_2}$}
We first use Lemma \ref{touse}(ii) to bound
\begin{align*}
& | L_{5,\eps_1\eps_2} | \lesssim  L_{5,1} + L_{5,2} + L_{5,3} ,
\\
\begin{split}
& L_{5,1} := \sum_{(k,k_1,k_2) \in \mathcal{X}, \, 2^{k_2+10} \leq (1+t)^{-2}}  {\| q_{\eps_1\eps_2}^{k,k_1,k_2} \|}_{S^\infty}
  2^{-k/2} {\| P_k^\prime Su \|}_{L^2} {\| P_{k_1}^\prime (\partial_t - i\Lambda) u \|}_{L^2} {\| P_{k_2}^\prime Su \|}_{L^\infty} ,
\\
& L_{5,2} := {\| \varphi_{\geq -20}((1+t)^2\eta) |\eta|^{-1/2} \what{Su}(\eta)\|}_{L^2}
  \sum_{k,k_1 \in \mathbb{Z}, \, |k-k_1| \leq 10}
  {\| |\eta|^{1/2} q_{\eps_1\eps_2}(\xi,\eta) \varphi_{k}(\xi) \varphi_{k_1}(\xi-\eta) \|}_{S^\infty}
  \\ & \hskip200pt \times 2^{-k/2} {\| P_k^\prime Su \|}_{L^2} {\| P_{k_1}^\prime (\partial_t - i\Lambda) u \|}_{L^\infty} ,
\end{split}
\\
& L_{5,3} := {\| \varphi_{\geq 0}((1+t)^2\xi) |\xi|^{-1/2} \what{Su}(\xi)\|}_{L^2}
  \sum_{k_1,k_2 \in \mathbb{Z}, \, k_2 + 10 \geq k_1}
  {\| |\eta|^{1/2} q_{\eps_1\eps_2}(\xi,\eta) \varphi_{k_1}(\xi-\eta) \varphi_{k_2}(\eta) \|}_{S^\infty}
  \\ & \hskip200pt \times {\| P_{k_1}^\prime (\partial_t - i\Lambda) u \|}_{L^\infty} 2^{-k_2/2} {\| P_{k_2}^\prime Su \|}_{L^2} .
\end{align*}

Using the symbol bound \eqref{boundqlow}, \eqref{estLSu},
\eqref{bing11}, and \eqref{Bern}, we get:
\begin{align*}
L_{5,1} \lesssim \sum_{k,k_1,k_2\in \mathbb{Z}, \, 2^{k_2+10} \leq \langle t\rangle^{-2}, \, |k-k_1| \leq 10}
  2^{-k_2/2} \e_1 \langle t\rangle^{4p_0}  \e_1^2 2^{k_1/2} 2^{-\max(k_1,0)} \langle t\rangle^{-1/2+p_0}
  \\ \times \e_1 2^{(1-p_1)k_2} \langle t\rangle^{4p_0-2p_1}
\lesssim \e_1^4 \langle t\rangle^{-1}.
\end{align*}
Using \eqref{estZlow1}, the symbol bound \eqref{boundqlowkk_1}, and \eqref{boundLu}, we obtain
\begin{align*}
L_{5,2} \lesssim \e_1 \langle t\rangle^{4p_0} \sum_{k,k_1 \in \mathbb{Z}, \, |k-k_1| \leq 10}
  \e_1 \langle t\rangle^{4p_0}  \e_1^2 2^{k_1/2} 2^{-\max(k_1,0)} \langle t\rangle^{-1}\lesssim \e_1^4 \langle t\rangle^{-1+8p_0}.
\end{align*}
Similarly, using \eqref{estZlow1}, \eqref{boundqlowk_1k_2}, and \eqref{boundLu}, we see that
\begin{align*}
L_{5,3} \lesssim \e_1 \langle t\rangle^{4p_0} \sum_{k_1,k_2 \in \mathbb{Z}, \, k_2 + 10 \geq k_1}
  \e_1^2 2^{k_1/2} 2^{-\max(k_1,0)} \langle t\rangle^{-1} \e_1 2^{-p_1 k_2} \langle t\rangle^{4p_0}\lesssim \e_1^4 \langle t\rangle^{-1+8p_0}.
\end{align*}
The desired bound $|L_{5,\eps_1\eps_2}|\lesssim \e_1^3(1+t)^{-1+8p_0}$ follows.

\subsubsection{Estimate of $L_{6,\eps_1\eps_2}$}Using Lemma \ref{touse}(ii) we can bound
\begin{align*}
& | L_{6,\eps_1\eps_2} | \lesssim  L_{6,1} + L_{6,2},
\\
& L_{6,1} := \sum_{k,k_1,k_2 \in \mathbb{Z}, \, 2^{k_2} \leq \langle t\rangle^{-2}}  {\| q_{\eps_1\eps_2}^{k,k_1,k_2} \|}_{S^\infty}
  2^{-k/2} {\| P_k^\prime Su \|}_{L^2} {\| P_{k_1}^\prime u \|}_{L^2} {\| P_{k_2}^\prime (\partial_t - i\Lambda) Su \|}_{L^\infty} ,
\\
& L_{6,2} := {\| \varphi_{\geq 0}((1+t)^2\xi) |\xi|^{-1/2} \what{Su}(\xi)\|}_{L^2}
  \sum_{k_1,k_2 \in \mathbb{Z}, \, 2^{k_2} \geq \langle t\rangle^{-2}}
  {\| |\eta|^{1/2} q_{\eps_1\eps_2}(\xi,\eta) \varphi_{k_1}(\xi-\eta) \varphi_{k_2}(\eta) \|}_{S^\infty}
  \\ & \hskip200pt \times {\| P_{k_1}^\prime u \|}_{L^\infty} 2^{-k_2/2} {\| P_{k_2}^\prime (\partial_t - i\Lambda)Su \|}_{L^2} .
\end{align*}

Using \eqref{boundqlow}, \eqref{estlow1}--\eqref{estZlow1}, and \eqref{estLSu}, we have
\begin{align*}
L_{6,1} \lesssim \sum_{k,k_1,k_2\in \mathbb{Z}, \, 2^{k_2} \leq \langle t\rangle^{-2} , \, k_2-10 \leq k \leq 10}
  2^{-k_2/2} \e_1 \langle t\rangle^{4p_0} \e_1 2^{(1/2-p_1)k_1} 2^{-\max(k_1,0)} \langle t\rangle^{p_0}
  \\ \times 2^{k_2/2} \e_1^2 2^{k_2/2} \langle t\rangle^{-1/2+4p_0}\lesssim \e_1^4 \langle t\rangle^{-1}.
\end{align*}
Using \eqref{estZlow1}, \eqref{boundqlowk_1k_2}, \eqref{bing11}, and \eqref{estLSu}, we have
\begin{align*}
L_{6,2} &\lesssim \e_1 \langle t\rangle^{4p_0} \sum_{k_1,k_2 \in \mathbb{Z}, \, 2^{k_2} \geq \langle t\rangle^{-2}}
  \e_1^2 2^{k_1/10} 2^{-\max(k_1,0)} \langle t\rangle^{-1/2}
  \\
&\times \e_1^2 \langle t\rangle^{-1/2+4p_0} ( 2^{k_2/2} + \langle t\rangle^{-1/2} ) 2^{-\max(k_2,0)}\lesssim \e_1^4 \langle t\rangle^{-1+8p_0}.
\end{align*}
The desired bound $|L_{6,\eps_1\eps_2}|\lesssim \e_1^3(1+t)^{-1+8p_0}$ follows.

\subsubsection{Estimate of $L_{j,\eps_1\eps_2}$, $j=7,\dots,10$}
Observe that the terms $L_{j, \eps_1\eps_2}$, for $j=7, \dots, 10$ in \eqref{evolEZlow2}
are easier to estimate than the terms $L_{j, \eps_1\eps_2}$, for $j=3, \dots, 6$ in \eqref{evolEZlow}.
This is because the bounds for the symbols are essentially the same,
see \eqref{boundqlow}-\eqref{boundrlow'}, but we have stronger information on $u$ than on $Su$.
Therefore, the integrals $L_{j, \eps_1\eps_2}$, for $j=7, \dots, 10$ can be treated similarly to the ones
that we have just estimated.
We can then conclude the desired bound of $\e_1^3 {(1+t)}^{-1+8p_0}$ for the evolution of $E_{Su_\low}$.
This gives Lemma \ref{lemEEZlow2}, and hence Proposition \ref{proEEZlow}.


\section{Decay estimates}\label{secdecay}

To prove decay we return to the Eulerian variables $(h,\phi)$ and prove the following:

\begin{proposition}\label{prodecay}
With $U=|\partial_x|h-i|\partial_x|^{1/2}\phi$ as in Proposition \ref{MainProp}, we have
\begin{equation}\label{prodecayconc}
 \sup_{t\in[0,T]}(1+t)^{1/2} \sum_{k \in \mathbb{Z}} \big(2^{-k/10} + 2^{N_2k} \big) {\|P_k U(t)\|}_{L^\infty}  \lesssim \e_0.
\end{equation}
\end{proposition}

Assuming this proposition, we can complete the proof of the main Proposition \ref{MainProp}.

\begin{proof}[Proof of Proposition \ref{MainProp}] It follows from Propositions \ref{proEE1}, \ref{proEElow}, \ref{proEEZ}, and \ref{proEEZlow} that
\begin{equation*}
\langle t\rangle^{-p_0}\mathcal{K}'_I(t)+\langle t\rangle^{-4p_0}\mathcal{K}'_S(t)+\lesssim \e_0
\end{equation*}
for any $t\in[0,T]$. It follows from \eqref{symm5} that $u-U\in O_{2,0}$ is a quadratic expression that does not lose derivatives and vanishes at frequencies $\leq 2^{-10}$. The desired conclusion \eqref{bing4} follows. 
\end{proof}

The proof of Proposition \ref{prodecay} will be given through a series of steps below.
The strategy follows the general approach of \cite{IoPu2,IoPu3}.

\subsection{Set up}
Recall that for any suitable multiplier $m:\mathbb{R}^d \rightarrow \mathbb{C}$
we define the associated bilinear and trilinear operators $M$ by the formulas
\begin{align*}
\mathcal{F} \big[M(f,g)\big](\xi) & = \frac{1}{2\pi} \int_{\R} m(\xi,\eta) \what{f}(\xi-\eta) \what{g}(\eta)\,d\eta ,
\\
\mathcal{F} \big[M(f,g,h)\big](\xi) & = \frac{1}{4\pi^2} \int_{\R \times \R} m(\xi,\eta,\sigma) \what{f}(\xi-\eta) \what{g}(\eta-\sigma)
  \what{h}(\sigma) \,d\eta d\sigma .
\end{align*}

In view of Lemma \ref{lemcubicfinal}, the variable $U$ satisfies the equation
\begin{align}
\label{equ100}
(\partial_t-i\Lambda)U= \mathcal{Q}_U + \mathcal{C}_U + \mathcal{R}_{\geq 4},\qquad \Lambda = |\partial_x|^{3/2},
\end{align}
where:

\setlength{\leftmargini}{1.5em}
\begin{itemize}

\smallskip
\item The quadratic nonlinear terms are
\begin{equation}\label{equN100}
\mathcal{Q}_U=\sum_{(\eps_1\eps_2)\in\{(++),(+-),(--)\}}Q_{\eps_1\eps_2}(U_{\eps_1},U_{\eps_2}),
\end{equation}
where $U_+=U$, $U_-=\overline{U}$, and the operators $Q_{++},Q_{+-},Q_{--}$ are defined by the symbols
\begin{equation}\label{equq_0}
\begin{split}
q_{++}(\xi,\eta)&:=\frac{i|\xi|(\xi\eta-|\xi||\eta|)}{8|\eta|^{1/2}|\xi-\eta|}+\frac{i|\xi|(\xi(\xi-\eta)-|\xi||\xi-\eta|)}{8|\eta||\xi-\eta|^{1/2}}+\frac{i|\xi|^{1/2}(\eta(\xi-\eta)+|\eta||\xi-\eta|)}{8|\eta|^{1/2}|\xi-\eta|^{1/2}},\\
q_{+-}(\xi,\eta)&:=-\frac{i|\xi|(\xi\eta-|\xi||\eta|)}{4|\eta|^{1/2}|\xi-\eta|}+\frac{i|\xi|(\xi(\xi-\eta)-|\xi||\xi-\eta|)}{4|\eta||\xi-\eta|^{1/2}}-\frac{i|\xi|^{1/2}(\eta(\xi-\eta)+|\eta||\xi-\eta|)}{4|\eta|^{1/2}|\xi-\eta|^{1/2}},\\
q_{--}(\xi,\eta)&:=-\frac{i|\xi|(\xi\eta-|\xi||\eta|)}{8|\eta|^{1/2}|\xi-\eta|}-\frac{i|\xi|(\xi(\xi-\eta)-|\xi||\xi-\eta|)}{8|\eta||\xi-\eta|^{1/2}}+\frac{i|\xi|^{1/2}(\eta(\xi-\eta)+|\eta||\xi-\eta|)}{8|\eta|^{1/2}|\xi-\eta|^{1/2}}.
\end{split}
\end{equation}

\smallskip
\item  The cubic terms have the form
\begin{align}
\label{equcubic100}
\mathcal{C}_U := M_{++-}(U, U, \bar{U}) + M_{+++}(U,U,U) + M_{--+}(\bar{U}, \bar{U}, U) + M_{---}(\bar{U}, \bar{U}, \bar{U}) ,
\end{align}
with purely imaginary symbols $m_{\iota_1\iota_2\iota_3}$ such that
\begin{align}
\label{equcubicsym}
\begin{split}
& {\big\| \mathcal{F}^{-1}\big[m_{\iota_1\iota_2\iota_3}(\xi,\eta,\sigma)\cdot 
  \varphi_k(\xi)\varphi_{k_1}(\xi-\eta)\varphi_{k_2}(\eta-\sigma)\varphi_{k_3}(\sigma)\big] \big\|}_{L^1} \lesssim 2^{k/2} 2^{\max(k_1,k_2,k_3)}
\end{split}
\end{align}
for all $(\iota_1\iota_2\iota_3) \in \{ (++-),(--+),(+++),(---) \}$.
Moreover, with $d_1 = 1/8$,
\begin{align}
\label{val00}
m_{++-} (\xi,0,-\xi) = id_1|\xi|^{3/2}.
\end{align}

\smallskip
\item $\mathcal{R}_{\geq 4}$ is a quartic remainder satisfying
\begin{align}
\label{R_4}
{\| \mathcal{R}_{\geq 4} \|}_{L^2} +{\| S \mathcal{R}_{\geq 4} \|}_{L^2} \lesssim \e_1^4 \langle t \rangle^{-5/4}. 
\end{align}
Moreover
\begin{align}
 \label{equRR4}
\mathcal{C}_U+\mathcal{R}_{\geq 4}\in |\partial_x|^{1/2} O_{3,-1} .
\end{align}
\end{itemize}

\subsection{The ``semilinear'' normal form transformation}\label{secnf}
We follow the classical normal form approach of Shatah \cite{shatahKGE} 
to define a modified variable which is a quadratic perturbation of $u$ and satisfies a cubic equation.
Let
\begin{equation}
\label{defv}
v := U + M_{++}(U,U) + M_{+-}(U,\bar{U})+ M_{--}(\bar{U}, \bar{U}) ,
\end{equation}
where, for any $(\eps_1,\eps_2) \in \{(++),(+-),(--)\}$, the bilinear operators $M_{\eps_1\eps_2}$ are defined by the multipliers
\begin{equation}\label{nf}
m_{\eps_1\eps_2}(\xi,\eta) := -i\frac{q_{\eps_1\eps_2}(\xi,\eta)}{|\xi|^{3/2}-\eps_1|\xi-\eta|^{3/2}-\eps_2|\eta|^{3/2}}.
\end{equation}
A direct computation shows that $v$ solves the equation 
\begin{align}
\label{eqv}
(\partial_t-i\Lambda)v ={\sum_{\ast}}^\prime\big[M_{\eps_1\eps_2}( (\mathcal{L} U)_{\eps_1}, U_{\eps_2})+ M_{\eps_1\eps_2}(U_{\eps_1}, (\mathcal{L}U)_{\eps_2})\big]
  + \mathcal{C}_U+ \mathcal{R}_{\geq 4}
\end{align}
where ${\sum_{\ast}}^\prime:=\sum_{(\eps_1,\eps_2) \in \{(++),(+-),(--)\}}$ and $\mathcal{L}:=(\partial_t-i\Lambda)$.

We now prove several bounds on the new variable $v$.
\begin{lemma}\label{lemv}
Let $v$ be defined by \eqref{defv}-\eqref{nf}. Then for any $t\in[0,T]$ and $k \in \mathbb{Z}$ we have
\begin{align}
\label{lemvinfty}
{\|P_k(U(t)-v(t))\|}_{L^\infty}&\lesssim \e_1^2 \min \big( 2^{k/2}, 2^{-(N_2-1/2) k} \big) \langle t \rangle^{-3/4+2p_0},
\\
\label{lemvL2}
{\|P_k(U(t) - v(t))\|}_{L^2} &\lesssim \e_1^2 \min(2^{k/2},2^{-(N_0-1/2)k}) \langle t \rangle^{-1/4+3p_0} ,
\\
\label{lemvS}
{\|P_k S (U(t)-v(t)) \|}_{L^2} &\lesssim \e_1^2 \min\big( 2^{k/2}, 2^{-(N_1-1/2)k} \big)  \langle t \rangle^{-1/4+6p_0} .
\end{align}
Furthermore, we have
\begin{align}
\label{lemvL22}
& {\|P_k(U(t) - v(t))\|}_{L^2} \lesssim \e_1^2 \langle t \rangle^{-1/2+3p_0} \min \big( 2^{k/2}, 2^{-(N_2-1/2)k} \big) .
\end{align}
\end{lemma}

The bounds \eqref{lemvinfty}-\eqref{lemvS} show that $v$ and $u$ have the same relevant norms,
and that all the a priori assumptions on $u$ transfer without significant losses 
to $v$. The bound \eqref{lemvL22} is a variant of the $L^2$ bound \eqref{lemvL2} which provides more decay in time,
but less decay at high frequencies. This will be used later on to bound quartic remainder terms.

\begin{remark}
\normalfont
Observe that \eqref{lemvinfty}, \eqref{lemvL2} and Sobolev embedding, imply
\begin{align}
\label{lemvinfty2}
& {\|P_k(U(t)-v(t))\|}_{L^\infty} \lesssim \e_1^2  \langle t \rangle^{-1/2}  \min\big(  2^{-(N_2+1) k}, 2^{k/8} \big) .
\end{align}
This shows that in order to obtain \eqref{prodecayconc} is suffices to show
\begin{align}
\label{prodecayconcv}
\sup_{t\in[0,T]}\langle t \rangle^{1/2} \sum_{k \in \mathbb{Z}} \big(2^{-k/10} + 2^{N_2k} \big) {\|P_k v(t)\|}_{L^\infty} \lesssim \e_0 .
\end{align}
\end{remark}

\begin{proof}[Proof of Lemma \ref{lemv}]
In view of \eqref{defv}, we see that we have to estimate the bilinear terms $M_{\eps_1\eps_2}$.
Notice that the formulas \eqref{equq_0} and \eqref{nf} show easily that
\begin{align}
\label{lemvsym}
\begin{split}
{\| q_{\eps_1\eps_2}^{k,k_1,k_2} \|}_{S^\infty}
  & \lesssim
  2^{k} 2^{\min(k,k_1,k_2)/2} \mathbf{1}_{\mathcal{X}}(k,k_1,k_2),\\
{\| m_{\eps_1\eps_2}^{k,k_1,k_2} \|}_{S^\infty}
  & \lesssim
  2^{k/2} 2^{-\min(k_1,k_2)/2} \mathbf{1}_{\mathcal{X}}(k,k_1,k_2).
\end{split}
\end{align}

The a priori bounds \eqref{bing1} show that, for any $l\in\mathbb{Z}$ and $t\in[0,T]$,
\begin{equation}
\label{lemvboundu}
\begin{split}
{\|P_lU(t)\|}_{L^2} & \lesssim \e_1 \langle t \rangle^{p_0} \min \big( 2^{l(1/2-p_1)}, 2^{-N_0l} \big),
\\
{\|P_lU(t) \|}_{L^\infty} & \lesssim \e_1 \langle t \rangle^{-1/2} \min( 2^{l/10}, 2^{-N_2l} \big),
\\
{\|P_l SU(t) \|}_{L^2} & \lesssim \e_1 \langle t \rangle^{4p_0} \min \big( 2^{(1/2-p_1)l}, 2^{-N_1l} \big) .
\end{split}
\end{equation}
Therefore, using Lemma \ref{touse}(ii), \eqref{lemvsym}, and \eqref{lemvboundu}, we estimate, for any $k\in\mathbb{Z}$ and $t\in[0,T]$,
\begin{align*}
& {\| P_k M_{\eps_1\eps_2} ( U_{\eps_1}, U_{\eps_2}) \|}_{L^\infty}
  \lesssim \sum_{k_1\leq k_2} {\| m_{\eps_1\eps_2}^{k,k_1,k_2}\|}_{S^\infty}
  {\| P_{k_1}^\prime U \|}_{L^\infty} {\| P_{k_2}^\prime U\|}_{L^\infty}
\\
& \lesssim \e_1^2 \langle t \rangle^{-1} \sum_{k_1\leq k_2,\, 2^{k_1} \geq \langle t \rangle^{-1/2}}
  \mathbf{1}_\mathcal{X}(k,k_1,k_2) 2^{(k-k_1)/2} \min(2^{-N_2k_1}, 2^{k_1/10}) \min(2^{-N_2k_2}, 2^{k_2/10})
\\
& + \e_1^2 \langle t \rangle^{-1/2+p_0} \sum_{k_1\leq k_2,\, 2^{k_1}\leq \langle t \rangle^{-1/2}}
  \mathbf{1}_\mathcal{X}(k,k_1,k_2) 2^{(k-k_1)/2} 2^{k_1/2} 2^{(1/2-p_1)k_1} \min(2^{-N_2k_2}, 2^{k_2/10})
\\
& \lesssim \e_1^2 \min \big( 2^{k/2}, 2^{-(N_2-1/2)k}\big) \langle t \rangle^{-3/4+2p_0},
\end{align*}
which suffices. We proceed similarly to obtain \eqref{lemvL2}:
\begin{align*}
& {\| P_k M_{\eps_1\eps_2} ( U_{\eps_1}, U_{\eps_2} ) \|}_{L^2}
  \lesssim \sum_{k_1\leq k_2} {\| m_{\eps_1\eps_2}^{k,k_1,k_2}\|}_{S^\infty}
  {\| P_{k_1}^\prime U \|}_{L^\infty} {\| P_{k_2}^\prime U\|}_{L^2}
\\
& \lesssim \e_1^2 \langle t \rangle^{-1/2+p_0} \sum_{k_1\leq k_2, \, 2^{k_1} \geq \langle t \rangle^{-1/2}}
  \mathbf{1}_\mathcal{X}(k,k_1,k_2) 2^{(k-k_1)/2} \min( 2^{-N_0k_2}, 2^{(1/2-p_1)k_2} )
\\
& + \e_1^2 \langle t \rangle^{2p_0} \sum_{k_1\leq k_2,\, 2^{k_1}\leq \langle t \rangle^{-1/2}}
  \mathbf{1}_\mathcal{X}(k,k_1,k_2) 2^{(k-k_1)/2} 2^{(1-p_1)k_1} \min(2^{-N_0k_2}, 2^{(1/2-p_1)k_2})
\\
& \lesssim \e_1^2 \min \big( 2^{k/2}, 2^{-(N_0-1/2)k}\big) \langle t\rangle^{-1/4+3p_0}.
\end{align*}

Using \eqref{lemcomm}-\eqref{lemcomm2} we have
\begin{align*}
S M_{\eps_1\eps_2}(f,g) = M_{\eps_1\eps_2}(Sf,g) + M_{\eps_1\eps_2}(f,Sg) + \widetilde{M}_{\eps_1\eps_2}(f,g)
\end{align*}
where $\widetilde{M}_{\eps_1\eps_2}$ is the operator associated to
$\widetilde{m}_{\eps_1,\eps_2}(\xi,\eta) = -(\xi\partial_\xi +\eta \partial_\eta ) m_{\eps_1\eps_2}(\xi,\eta)$.
It is not hard to verify that the symbols $\widetilde{m}_{\eps_1,\eps_2}$ satisfy the same bounds \eqref{lemvsym} as
the symbols $m_{\eps_1,\eps_2}$ and, therefore, estimating as above
\begin{align*}
{\| \widetilde{M}_{\eps_1\eps_2}(U_{\eps_1}, U_{\eps_2}) \|}_{L^2}
  \lesssim \e_1^2 \min \big( 2^{k/2}, 2^{-(N_0-1/2)k}\big) \langle t \rangle^{-1/4+3p_0}.
\end{align*}
Thus, to prove \eqref{lemvS} it suffices to estimate
$M_{\eps_1\eps_2}(S U_{\eps_1}, U_{\eps_2})$ and $M_{\eps_1\eps_2}(U_{\eps_1},S U_{\eps_2})$.
As in the proof of \eqref{lemvL2}, we use Lemma \ref{touse}(ii), followed by \eqref{lemvsym},
and \eqref{lemvboundu}, to obtain
\begin{equation}\label{cxz30}
\begin{split}
& {\| P_k M_{\eps_1\eps_2} ( SU_{\eps_1}, U_{\eps_2} ) \|}_{L^2}\\
&\lesssim \sum_{k_1,k_2\in\mathbb{Z}} {\| m_{\eps_1\eps_2}^{k,k_1,k_2}\|}_{S^\infty}
  \min \big( {\| P_{k_1}^\prime S U \|}_{L^2} {\| P_{k_2}^\prime U\|}_{L^\infty},
  {\| P_{k_1}^\prime S U \|}_{L^\infty} {\| P_{k_2}^\prime U\|}_{L^2} \big).
	\end{split}
\end{equation}
We use the bound ${\| P_{k_1}^\prime S U \|}_{L^2} {\| P_{k_2}^\prime U\|}_{L^\infty}$ when $k_2\leq k_1$ or when $k_1\leq k_2$ and $2^{k_1}\geq \langle t\rangle^{1/2}$. We use the bound ${\| P_{k_1}^\prime S U \|}_{L^\infty} {\| P_{k_2}^\prime U\|}_{L^2}$ when $k_1\leq k_2$ and $2^{k_1}\leq \langle t\rangle^{1/2}$. It follows that the left-hand side of \eqref{cxz30} is bounded by 
\begin{equation*}
\begin{split}
&C\e_1^2 \langle t \rangle^{-1/2+4p_0} \sum_{k_2\leq k_1,\,2^{k_2}\geq\langle t\rangle^{-1/2}}
  \mathbf{1}_\mathcal{X}(k,k_1,k_2) 2^{(k-k_2)/2} 2^{-N_2k_2^+} \min(2^{-N_1k_1}, 2^{(1/2-p_1)k_1})\\
	&+C\e_1^2 \langle t \rangle^{5p_0} \sum_{k_2\leq k_1,\,2^{k_2}\leq\langle t\rangle^{-1/2}}
  \mathbf{1}_\mathcal{X}(k,k_1,k_2) 2^{(k-k_2)/2} 2^{(1-p_1)k_2}\min(2^{-N_1k_1}, 2^{(1/2-p_1)k_1})\\
	&+C\e_1^2 \langle t \rangle^{-1/2+4p_0} \sum_{k_1\leq k_2,\,2^{k_1}\geq\langle t\rangle^{-1/2}}
  \mathbf{1}_\mathcal{X}(k,k_1,k_2) 2^{(k-k_1)/2} 2^{-N_2k_2^+}\min(2^{-N_1k_1}, 2^{(1/2-p_1)k_1})\\
	&+C\e_1^2 \langle t \rangle^{5p_0} \sum_{k_1\leq k_2,\,2^{k_1}\leq\langle t\rangle^{-1/2}}
  \mathbf{1}_\mathcal{X}(k,k_1,k_2) 2^{(k-k_1)/2} 2^{-N_2k_2^+}2^{(1-p_1)k_1}\\
	&\lesssim \e_1^2 \min \big( 2^{k/2}, 2^{-(N_1-1/2)k}\big) \langle t \rangle^{-1/4+6p_0}.
\end{split}
\end{equation*}
This completes the proof for ${\| P_k M_{\eps_1\eps_2} ( SU_{\eps_1}, U_{\eps_2} ) \|}_{L^2}$. The proof for ${\| P_k M_{\eps_1\eps_2} ( U_{\eps_1}, SU_{\eps_2} ) \|}_{L^2}$ is similar, due to the symmetric bounds on the symbols $m_{\eps_1\eps_2}$. 

To prove \eqref{lemvL22}, which gives better time decay than \eqref{lemvL2}, we use again Lemma \ref{touse}(ii), the symbol bounds \eqref{lemvsym},
and the a priori bounds on $u$ in \eqref{lemvboundu},
\begin{align*}
& {\| P_k M_{\eps_1\eps_2} ( U_{\eps_1}, U_{\eps_2} ) \|}_{L^2}\lesssim \sum_{k_1\leq k_2} {\| m_{\eps_1\eps_2}^{k,k_1,k_2}\|}_{S^\infty}
  \min \big( {\| P_{k_1}^\prime U\|}_{L^2} {\| P_{k_2}^\prime U\|}_{L^\infty},
  {\| P_{k_1}^\prime U \|}_{L^\infty} {\| P_{k_2}^\prime U\|}_{L^2} \big)
\\
& \lesssim \e_1^2 \langle t \rangle^{-1/2+p_0} \sum_{k_1\leq k_2,\,2^{k_1} \geq \langle t \rangle^{-4}}
  \mathbf{1}_\mathcal{X}(k,k_1,k_2) 2^{(k-k_1)/2} \min(2^{(1/2-p_1)k_1},2^{-N_0k_1}) \min( 2^{-N_2k_2}, 2^{k_2/10} )
\\
& + \e_1^2 \langle t \rangle^{2p_0} \sum_{k_1\leq k_2,\, 2^{k_1}\leq \langle t \rangle^{-4}}
  \mathbf{1}_\mathcal{X}(k,k_1,k_2) 2^{(k-k_1)/2} 2^{(1-p_1)k_1} \min(2^{-N_0k_2}, 2^{(1/2-p_1)k_2})
\\
& \lesssim \e_1^2 \min \big( 2^{k/2}, 2^{-(N_2-1/2)k}\big) \langle t \rangle^{-1/2+3p_0}.
\end{align*}
This completes the proof of the lemma.
\end{proof}

\subsection{The profile $f$}

For $t\in[0,T]$ we define the profile of the solution of \eqref{eqv} as
\begin{equation}
\label{prof}
f(t) := e^{-it\Lambda}v(t).
\end{equation}
In the next proposition we summarize the main properties of the function $f$.

\begin{proposition}[Bounds for the profile]\label{proprof}
With $v$ and $f$ are defined as above, we have
\begin{align}
\label{eqprof}
\begin{split}
& e^{it\Lambda} \partial_t f = (\partial_t - i\Lambda) v = \mathcal{N}^\prime
\\
& \mathcal{N}^\prime := {\sum_\star}^\prime M_{\eps_1\eps_2}( (\mathcal{L}U)_{\eps_1}, U_{\eps_2}) + M_{\eps_1\eps_2}(U_{\eps_1}, (\mathcal{L}U)_{\eps_2})
  + \mathcal{C}_U + \mathcal{R}_{\geq 4},
\end{split}
\end{align}
where the bilinear operators $M_{\eps_1\eps_2}$ are defined via \eqref{nf}, and $\mathcal{C}_U$ and $\mathcal{R}_{\geq 4}$ satisfy 
\eqref{equcubic100}--\eqref{R_4}.

Moreover, for any $t\in[0,T]$ and $k\in\mathbb{Z}$, we have the estimates
\begin{align}
\label{profLinfty}
{\big\| P_k ( e^{it\Lambda}f(t) ) \big\|}_{L^\infty} & \lesssim \e_1 \min( 2^{k/10}, 2^{-N_2k} ) \langle t \rangle^{-1/2},
\\
\label{profL2}
{\|P_k f(t)\|}_{L^2} & \lesssim \e_0 \langle t \rangle^{6p_0} \min \big(2^{(1/2-p_1)k}, 2^{-(N_0-1/2)k} \big),
\\
\label{profS}
{\|P_k(x\partial_x f(t)) \|}_{L^2} & \lesssim \e_0 \langle t \rangle^{6p_0}\min\big(2^{(1/2-p_1)k}, 2^{-(N_1-1/2)k} \big).
\end{align}

\end{proposition}

\begin{proof}
The equation \eqref{eqprof} follows from the definition \eqref{prof} and the equation \eqref{eqv}.
The $L^\infty$ bound \eqref{profLinfty} follows from \eqref{lemvinfty2} and \eqref{lemvboundu}.
The $L^2$ bound \eqref{profL2} follows from the energy estimates in Propositions \ref{proEE1} and \ref{proEElow}, and the bounds \eqref{lemvL2}.

To prove \eqref{profS} we start from the identity
\begin{equation*}
Sv = e^{it\Lambda}(x\partial_xf) + (3/2)te^{it\Lambda}(\partial_tf),
\end{equation*}
which is a consequence of the commutation identity $[S,e^{it\Lambda}]=0$.
Therefore, for any $t\in[0,T]$ and $k\in\mathbb{Z}$,
\begin{equation}
\label{P_kxdxf}
{\|P_k(x\partial_xf(t))\|}_{L^2} \lesssim {\|P_k(Sv(t))\|}_{L^2} + (1+t) {\|P_k (\partial_t - i\Lambda) v(t)\|}_{L^2}.
\end{equation}
Using Proposition \ref{proEEZ} and Proposition \ref{proEEZlow},
we know, in particular, that
\begin{align}
\label{P_kSu}
{\| P_k SU(t)\|}_{L^2} \lesssim \e_0 \min(2^{(1/2-p_1)k }, 2^{-N_1k}) \langle t \rangle^{4p_0} ,
\end{align}
for all $k \in \mathbb{Z}$.
Together with \eqref{lemvS} and \eqref{P_kxdxf}, this gives
\begin{align*}
{\| P_k ( x\partial_xf(t) ) \|}_{L^2}
  \lesssim \langle t \rangle{\| P_k (\partial_t - i \Lambda) v(t) \|}_{L^2} + \e_1^2 \big( 2^{k/2}, 2^{-(N_1-1/2)k} \big)  \langle t \rangle^{-1/4+6p_0}
\\
 + \e_0 \min(2^{(1/2-p_1)k }, 2^{-N_1k}) \langle t \rangle^{4p_0} .
\end{align*}
It then suffices to show
\begin{equation}
\label{estN'}
{\| P_k (\partial_t - i \Lambda) v(t) \|}_{L^2} \lesssim \langle t \rangle^{-1+6p_0} \e_1^2 \big( 2^{(1/2-p_1)k}, 2^{-(N_1-1/2)k} \big) .
\end{equation}

Since $\mathcal{C}_U+\mathcal{R}_{\geq 4}\in |\partial_x|^{1/2}O_{3,-1}$, see \eqref{equRR4}, its contribution is already bounded by the right-hand side above,
according to the definition of $O_{3,\alpha}$, see \eqref{cubb}.
To complete the proof of \eqref{estN'} it suffices to estimate
$M_{\eps_1\eps_2}(U_{\eps_1}, (\mathcal{L}U)_{\eps_2})$ and $M_{\eps_1\eps_2}( (\mathcal{L}U)_{\eps_1}, U_{\eps_2})$.

It follows from \eqref{equRR4} and the symbol bounds on $q_{\eps_1\eps_2}$ in \eqref{lemvsym} that
\begin{align}
\label{lemprofLu}
\begin{split}
& {\| P_l \mathcal{L}U(t) \|}_{L^2} \lesssim \e_1^2 \min(2^{l/2}, 2^{-(N_0-2)l}) \langle t \rangle^{-1/2+p_0},
\\
& {\| P_l \mathcal{L}U(t) \|}_{L^\infty} \lesssim \e_1^2 \min(2^{l/2}, 2^{-(N_2-2)l}) \langle t \rangle^{-1}.
\end{split}
\end{align}
Then, using Lemma \ref{touse}(ii) with the symbol bounds \eqref{lemvsym}, \eqref{lemvboundu} and \eqref{lemprofLu}, we can estimate
\begin{equation}\label{cxz31}
\begin{split}
& {\| P_k M_{\eps_1\eps_2} ( U_{\eps_1}(t), (\mathcal{L}U)_{\eps_2}(t) ) \|}_{L^2}
\\
& \lesssim \sum_{k_1,k_2\in\mathbb{Z}} {\| m_{\eps_1\eps_2}^{k,k_1,k_2}\|}_{S^\infty}
  \min \big(  {\| P_{k_1}^\prime U(t) \|}_{L^2} {\| P_{k_2}^\prime \mathcal{L}U(t)\|}_{L^\infty},
  {\| P_{k_1}^\prime U(t) \|}_{L^\infty} {\| P_{k_2}^\prime \mathcal{L}U(t)\|}_{L^2} \big).
	\end{split}
\end{equation}
We use the bound ${\| P_{k_1}^\prime U(t) \|}_{L^2} {\| P_{k_2}^\prime \mathcal{L}U(t)\|}_{L^\infty}$ when ($k_1\leq k_2$ and $2^{k_1}\geq \langle t\rangle ^{-4}$) or when ($k_2\leq k_1$ and $2^{k_2}\leq \langle t\rangle ^{-4}$). We use the bound ${\| P_{k_1}^\prime U(t) \|}_{L^\infty} {\| P_{k_2}^\prime \mathcal{L}U(t)\|}_{L^2}$ when ($k_1\leq k_2$ and $2^{k_1}\leq \langle t\rangle ^{-4}$) or when ($k_2\leq k_1$ and $2^{k_2}\geq \langle t\rangle ^{-4}$). It follows that the left-hand side of \eqref{cxz31} s bounded by
\begin{equation*}
\begin{split}
&C\e_1^2 \langle t\rangle^{-1+p_0} \sum_{k_1\leq k_2,\,2^{k_1} \geq \langle t\rangle^{-4}}
  \mathbf{1}_\mathcal{X}(k,k_1,k_2) 2^{(k-k_1)/2} \min(2^{(1/2-p_1)k_1},2^{-N_0k_1}) \min( 2^{k_2/2}, 2^{-(N_2-2)k_2})\\
	& + C\e_1^2\langle t\rangle^{-1/2+2p_0} \sum_{k_1\leq k_2,\,
  2^{k_1}\leq \langle t\rangle^{-4}}
  \mathbf{1}_\mathcal{X}(k,k_1,k_2) 2^{(k-k_1)/2} 2^{(1-p_1)k_1} \min(2^{k_2/2}, 2^{-(N_0-2)k_2})\\
	&+C\e_1^2 \langle t\rangle^{-1+p_0} \sum_{k_2\leq k_1,\,2^{k_2} \geq \langle t\rangle^{-4}}
  \mathbf{1}_\mathcal{X}(k,k_1,k_2) 2^{(k-k_2)/2} 2^{-N_2k_1^+} \min( 2^{k_2/2}, 2^{-k_2})\\
	& + C\e_1^2\langle t\rangle^{-1/2+2p_0} \sum_{k_2\leq k_1,\,
  2^{k_2}\leq \langle t\rangle^{-4}}
  \mathbf{1}_\mathcal{X}(k,k_1,k_2) 2^{(k-k_2)/2} 2^{-N_2k_1^+}2^{k_2}\\
	&\lesssim \langle t \rangle^{-1+6p_0} \e_1^2 \big( 2^{(1/2-p_1)k}, 2^{-(N_1-1/2)k} \big) .
\end{split}
\end{equation*}
The bound on ${\| P_k M_{\eps_1\eps_2} ( (\mathcal{L}U)_{\eps_1}(t), U_{\eps_2}(t) ) \|}_{L^2}$ is similar. This completes the proof of \eqref{estN'} and the proposition.
\end{proof}

\subsection{The $Z$-norm and proof of Proposition \ref{prodecay}}
For any function $h \in L^2(\R)$ let
\begin{equation}
\label{Znorm}
{\| h \|}_Z := {\big\| \big(|\xi|^{1/10} + |\xi|^{N_2+1/2} \big) \what{h}(\xi) \big\|}_{L^\infty_\xi} .
\end{equation}

\begin{proposition}
\label{proZ}
Let $f$ be defined as in \eqref{prof} and assume that for $T'\in[0,T]$
\begin{equation}
\label{proZapr}
\sup_{t\in[0,T']}\|f(t)\|_Z\leq \e_1 .
\end{equation}
Then
\begin{equation}
\label{proZconc}
\sup_{t\in[0,T']} {\|f(t)\|}_Z \lesssim \e_0.
\end{equation}
\end{proposition}

We now show how to prove Proposition \ref{prodecay} using Proposition \ref{proZ} above.

\begin{proof}[Proof of Proposition \ref{prodecay}]
Define $z(t):= {\|f(t)\|}_Z$ and notice that $z:[0,T]\to\mathbb{R}_+$ is a continuous function. We show first that
\begin{equation}
\label{z_0}
 z(0)\lesssim \e_0 .
\end{equation}
Indeed, using the definitions and Lemma \ref{interpolation} we get
\begin{equation*}
\begin{split}
 z(0) & \lesssim \sup_{k \in \mathbb{Z} } \,( 2^{k/10} + 2^{(N_2+1/2)k}) {\| \what{P_kv(0)} \|}_{L^\infty}
\\
& \lesssim \sup_{k\in\mathbb{Z}} \,(2^{k/10}+2^{(N_2+1/2)k})2^{-k/2}
  {\| \what{P'_kv(0)} \|}_{L^2}^{1/2} \big( 2^k {\|\partial \widehat{P^\prime_kv(0)}\|}_{L^2}
  + {\|\widehat{P'_kv(0)}\|}_{L^2} \big)^{1/2}.
\end{split}
\end{equation*}
Thanks to \eqref{lemvL2}-\eqref{lemvL22} with $t=0$, and the initial data assumptions \eqref{bing31}, we have
\begin{align*}
& {\| \what{P'_k v(0)} \|}_{L^2} \lesssim \e_0 \min (2^{(1/2-p_1)k}, 2^{-(N_0-1/2)k}) ,
\\
& 2^k {\|\partial \widehat{P^\prime_k v(0)}\|}_{L^2} \lesssim \e_0 \min (2^{(1/2-p_1)k}, 2^{-(N_1-1/2)k}) ,
\end{align*}
so that \eqref{z_0} follows, using also $(N_0+N_1)/2 \geq N_2+1$, see \eqref{constants}.

We apply now Proposition \ref{proZ}.
By continuity, $z(t) \lesssim  \e_0$ for any $t\in[0,T]$, provided that
$\e_0$ is  sufficiently small and  $\e_0 \ll \e_1 \leq \e_0^{2/3} \ll 1$ as in \eqref{constants}.
Therefore, for any $k \in \mathbb{Z}$ and $t\in[0,T]$,
\begin{equation}
\label{prodecay1}
(2^{k/10} + 2^{(N_2+1/2)k}) {\| \what{P_kf}(t) \|}_{L^\infty} \lesssim \e_0.
\end{equation}

Recall that we aim to prove, for all $t\in[1,T]$, the decay bound
\begin{align}
\label{prodecay0}
\sup_{t\in[0,T]} \langle t \rangle^{1/2} \sum_{k \in \mathbb{Z}} \big(2^{-k/10} + 2^{N_2k} \big) {\|P_k v(t)\|}_{L^\infty} \lesssim \e_0 ,
\end{align}
which, as already observed, implies \eqref{prodecayconc} via \eqref{lemvinfty2} (the bound for $t\in[0,1]$ is a consequence of Propositions \ref{proEE1} and \ref{proEElow}).
Also, observe that Lemma \ref{dispersive} applied to $v = e^{it\Lambda}f$ gives
\begin{align}
\label{prodecay2}
{\| P_k v(t) \|}_{L^\infty} \lesssim t^{-1/2}2^{k/4} {\|\what{P_k^\prime f}(t)\|}_{L^\infty}
   + t^{-3/5} 2^{-2k/5} \big[ 2^k {\|\partial \what{P_k^\prime f}(t)\|}_{L^2} + {\|\what{P_k^\prime f}(t)\|}_{L^2} \big]
\end{align}
and
\begin{equation}
\label{prodecay3}
{\| P_k v(t) \|}_{L^\infty} \lesssim t^{-1/2} 2^{k/4} {\| P_k^\prime f(t) \|}_{L^1}.
\end{equation}

Recall that from \eqref{profL2} and \eqref{profS} we have
\begin{align}
 \label{prodecay4}
2^k {\|\partial \what{P_k^\prime f}(t)\|}_{L^2} + {\|\what{P_k^\prime f}(t)\|}_{L^2}
  \lesssim \e_0 \langle t \rangle^{6p_0} 2^{(1/2-p_1)k}
\end{align}
We can then use \eqref{prodecay1} and \eqref{prodecay4} in \eqref{prodecay2}, to obtain
\begin{align*}
{\| P_k v(t) \|}_{L^\infty} \lesssim t^{-1/2}2^{k/4} \frac{\e_0}{2^{k/10} + 2^{(N_2+1/2)k}}
   + t^{-3/5} 2^{-2k/5} \e_0 \langle t \rangle^{6p_0} 2^{(1/2-p_1)k} ,
\end{align*}
which is enough to show that
\begin{align}
\label{prodecay20}
\sum_{k\in \mathbb{Z}, \, \langle t \rangle^{-100p_0} \leq 2^k \leq \langle t \rangle^{100p_0}} \big(2^{-k/10} + 2^{N_2k} \big) {\|P_k v(t)\|}_{L^\infty}
\lesssim \e_0 \langle t \rangle^{-1/2}.
\end{align}

Combining \eqref{prodecay3}, with Lemma \ref{interpolation}, and \eqref{profL2}-\eqref{profS} we see that
\begin{equation}
{\| P_k v(t) \|}_{L^\infty} \lesssim t^{-1/2} 2^{k/4} \e_0 {(1+t)}^{6p_0} \min(2^{-kp_1}, 2^{-(N_0+N_1)k/2}).
\end{equation}
Using also \eqref{constants} it follows that
\begin{align}
\label{prodecay30}
\sum_{k\in \mathbb{Z}, \, 2^k \leq \langle t \rangle^{-100p_0}} 2^{-k/10} {\|P_k v(t)\|}_{L^\infty}
  + \sum_{k\in \mathbb{Z}, \, 2^k \geq \langle t \rangle^{100p_0}} 2^{N_2k} {\|P_k v(t)\|}_{L^\infty}
  \lesssim \e_0 {(1+t)}^{-1/2} .
\end{align}
The estimates \eqref{prodecay20} and \eqref{prodecay30} give us \eqref{prodecay0}, and conclude the proof of Proposition \ref{prodecay}.
\end{proof}

\subsection{The equation for $v$ and proof of Proposition \ref{proZ}}
We now derive an equation for $v$ with cubic terms that only involve $v$ itself.
Recall that $u$ solves \eqref{equ100}-\eqref{R_4}, 
$v$ solves \eqref{eqprof}, and they are related via \eqref{defv}-\eqref{nf}.

According to \eqref{equN100} and \eqref{equq_0} we define
\begin{align}
\label{Q_v}
& \mathcal{Q}_v := {\sum_{\ast}}'Q_{\eps_1\eps_2}(v_{\eps_1},v_{\eps_2})=\sum_{(\eps_1\eps_2)\in\{(++),(+-),(--)\}}Q_{\eps_1\eps_2}(v_{\eps_1},v_{\eps_2}),
\end{align}
where the operators $Q_{\eps_1,\eps_2}$ are defined by the symbols $q_{\eps_1\eps_2}$. We use the notation
\begin{align*}
\sum_{\star\star} := \sum_{(\eps_1\eps_2\eps_3) \in \{ (++-), (+++), (--+), (---)\} }
\end{align*}
and rewrite \eqref{eqprof} in the form
\begin{align}
\label{eqprof10}
 \partial_t f = e^{-it\Lambda} (\mathcal{N}^{''} + \mathcal{R}_{\geq 4}^{''})
\end{align}
where
\begin{align}
\label{eqprof11}
\begin{split}
\mathcal{N}^{''} := &
  {\sum_\star}' M_{\eps_1\eps_2}( (\mathcal{Q}_v)_{\eps_1}, v_{\eps_2}) + M_{\eps_1\eps_2}(v_{\eps_1}, (\mathcal{Q}_v)_{\eps_2})
  + \sum_{\star\star}  M_{\eps_1\eps_2\eps_3}(v_{\eps_1},v_{\eps_2}, v_{\eps_3})
\end{split}
\end{align}
and
\begin{align}
\begin{split}
\label{eqprof12}
\mathcal{R}_{\geq 4}^{''} := &
  {\sum_\star}'  \big[ M_{\eps_1\eps_2}\big( (\mathcal{Q}_U)_{\eps_1}, U_{\eps_2} \big)
  - M_{\eps_1\eps_2} \big((\mathcal{Q}_v)_{\eps_1}, v_{\eps_2}\big) \big]
  \\
+ & {\sum_\star}' \big[ M_{\eps_1\eps_2}\big(U_{\eps_1}, (\mathcal{Q}_U)_{\eps_2}\big)
  - M_{\eps_1\eps_2}\big(v_{\eps_1}, (\mathcal{Q}_v)_{\eps_2}\big) \big]
\\
+ & \sum_{\star\star} \big[ M_{\eps_1\eps_2\eps_3}(U_{\eps_1}, U_{\eps_2}, U_{\eps_3})
  - M_{\eps_1\eps_2\eps_3}(v_{\eps_1},v_{\eps_2}, v_{\eps_3}) \big] + \mathcal{R}_{\geq 4}
\\ + & {\sum_\star}' M_{\eps_1\eps_2} \big( (\mathcal{C}_U+\mathcal{R}_{\geq 4})_{\eps_1}, U_{\eps_2} \big)
   + M_{\eps_1\eps_2} \big(U_{\eps_1}, (\mathcal{C}_U+\mathcal{R}_{\geq 4})_{\eps_2} \big) .
\end{split}
\end{align}
The point of the above decomposition is to identify $\mathcal{N}^{''}$
as the main ``cubic'' part of the nonlinearity, which can be
expressed only in terms of $v(t)=e^{it\Lambda}f(t)$. $\mathcal{R}_{\geq 4}^{''}$ can be thought of as a quartic
remainder, due to the quadratic nature of $u-v$, see Lemma \ref{lemv}.

To analyze the equation \eqref{eqprof10}, and identify the asymptotic logarithmic phase correction,
we need to distinguish among different types of interactions in the nonlinearity $\mathcal{N}^{''}$. We write
\begin{align}
\label{N''}
\mathcal{N}^{''} & = \mathcal{C}^{++-} + \mathcal{C}^{+++} + \mathcal{C}^{--+} + \mathcal{C}^{---},
\end{align}
where, recalling that the operators $Q_{++},M_{++},Q_{--},M_{--}$ are symmetric, 
\begin{align}
\label{C++-}
\begin{split}
\mathcal{C}^{++-} & = \mathcal{C}^{++-}(v,v,\bar{v})
  := 2M_{++}( v, Q_{+-}(v,\bar{v}))+ M_{+-}( Q_{++}(v,v), \bar{v})\\
	&+ M_{+-}(v, \bar{Q_{+-}(v,\bar{v})})+ 2M_{--}( \bar{Q_{--}(\bar{v},\bar{v})}, \bar{v})
  + M_{++-}(v,v, \bar{v})  ,
\end{split}
\end{align}
\begin{align*}
\begin{split}
\mathcal{C}^{+++} &= \mathcal{C}^{+++}(v,v,v)
  := 2M_{++}(v, Q_{++}(v,v))+ M_{+-}(v, \bar{Q_{--}(\bar{v},\bar{v})}) + M_{+++}(v,v,v) ,
\end{split}
\\
\begin{split}
\mathcal{C}^{--+} &= \mathcal{C}^{--+}(\bar{v},\bar{v},v)
  := 2M_{++}( v, Q_{--}(\bar{v},\bar{v}))+ M_{+-}( Q_{+-}(v,\bar{v}), \bar{v})+ M_{+-}(v, \bar{Q_{++}(v,v)})
  \\ & + 2M_{--}( \bar{Q_{+-}(v,\bar{v})}, \bar{v}) + M_{--+}(\bar{v},\bar{v}, v)  ,
\end{split}
\\
\begin{split}
\mathcal{C}^{---} &= \mathcal{C}^{---}(\bar{v},\bar{v},\bar{v})
  := M_{+-}( Q_{--}(\bar{v},\bar{v}), \bar{v})+ 2M_{--}( \bar{v}, \bar{Q_{++}(v,v)} )  + M_{---}(\bar{v},\bar{v},\bar{v}).
\end{split}
\end{align*}

Notice that
\begin{align}
 \bar{Q_{\eps_1\eps_2}(g_1, g_2)} = - Q_{\eps_1\eps_2}(\bar{g_1}, \bar{g_2}).
\end{align}
Letting $v^+ = v, v^- = \bar{v}$, we expand
\begin{equation}
\label{Chat}
\what{\mathcal{C}^{\iota_1\iota_2\iota_3}}(\xi)=\frac{i}{4\pi^2} \int_{\R\times\R} c^{\iota_1\iota_2\iota_3}(\xi,\eta,\sigma)
  \what{v^{\iota_1}}(\xi-\eta) \what{v^{\iota_2}}(\eta-\sigma)\what{v^{\iota_3}}(\sigma)\,d\eta d\sigma
\end{equation}
for $(\iota_1,\iota_2,\iota_3)\in\{(++-),(+++),(--+),(---)\}$, where
\begin{equation}
\label{csymbols}
\begin{split}
ic^{++-}(\xi,\eta,\sigma) :&=2m_{++}(\xi,\eta) q_{+-}(\eta,\sigma)+ m_{+-}(\xi,\sigma) q_{++}(\xi-\sigma, \xi-\eta)\\
& - m_{+-}(\xi,\eta) q_{+-}(\eta, \eta-\sigma)- 2m_{--}(\xi,\sigma) q_{--}(\xi-\sigma, \xi-\eta)+ m_{++-}(\xi,\eta,\sigma),
\\
ic^{+++}(\xi,\eta,\sigma) :&=2m_{++}(\xi,\eta) q_{++}(\eta,\sigma)- m_{+-}(\xi,\eta) q_{--}(\eta,\sigma)+ m_{+++}(\xi,\eta,\sigma) ,\\
ic^{--+}(\xi,\eta,\sigma) :&=2m_{++}(\xi,\sigma) q_{--}(\xi-\sigma,\xi-\eta)+ m_{+-}(\xi,\xi-\eta) q_{+-}(\eta,\eta-\sigma)\\
&- m_{+-}(\xi,\xi-\sigma) q_{++}(\xi-\sigma, \xi-\eta)- 2m_{--}(\xi,\eta) q_{+-}(\eta,\sigma)+  m_{--+}(\xi,\eta,\sigma) ,
\\
ic^{---}(\xi,\eta,\sigma) :&= m_{+-}(\xi,\xi-\eta) q_{--}(\eta,\sigma)- 2m_{--}(\xi,\eta) q_{++}(\eta,\sigma)  + m_{---}(\xi,\eta,\sigma) .
\end{split}
\end{equation}

Using the definitions of the quadratic symbols \eqref{nf} and \eqref{equq_0}, we see that the cubic symbols $c^{\iota_1\iota_2\iota_3}$ are real-valued.
Recalling the formulas
\begin{align*}
\what{v^+}(\xi,t) = \what{f}(\xi,t)e^{it|\xi|^{3/2}} , \qquad \what{v^-}(\xi,t) = \what{\bar{f}}(\xi,t)e^{-it|\xi|^{3/2}} ,
\end{align*}
we can rewrite
\begin{equation}\label{nf46}
\mathcal{F} \big(e^{-it\Lambda} \mathcal{N}^{''}(t)\big)(\xi,t)
 = \frac{i}{4\pi^2} \big[ I^{++-}(\xi,t) + I^{+++}(\xi,t) + I^{--+}(\xi,t) + I^{---}(\xi,t) \big],
\end{equation}
where
\begin{equation}\label{nf47}
\begin{split}
I^{\iota_1\iota_2\iota_3}(\xi):=\int_{\R\times\R}
  & e^{it(-|\xi|^{3/2}+\iota_1|\xi-\eta|^{3/2}+\iota_2|\eta-\sigma|^{3/2}+\iota_3|\sigma|^{3/2})}
\\
  & \times c^{\iota_1\iota_2\iota_3}(\xi,\eta,\sigma)
  \what{f^{\iota_1}}(\xi-\eta) \what{f^{\iota_2}}(\eta-\sigma) \what{f^{\iota_3}}(\sigma) \,d\eta d\sigma
\end{split}
\end{equation}
for $(\iota_1,\iota_2,\iota_3)\in\{(++-),(+++),(--+),(---)\}$.
The formulas \eqref{eqprof10}-\eqref{eqprof12} become
\begin{equation}\label{nf48}
(\partial_t\what{f})(\xi,t)=\frac{i}{4\pi^2}\big[I^{++-}(\xi,t) + I^{+++}(\xi,t) + I^{--+}(\xi,t) + I^{---}(\xi,t) \big]
+ e^{-it|\xi|^{3/2}}\what{\mathcal{R}_{\geq 4}^{''}}(\xi,t).
\end{equation}

In analyzing the formula \eqref{nf48}, the main contribution comes from the stationary points of the phase functions $(t,\eta,\sigma)\to t\Psi^{\iota_1\iota_2\iota_3}(\xi,\eta,\sigma)$, where
\begin{equation}\label{nf49}
\Psi^{\iota_1\iota_2\iota_3}(\xi,\eta,\sigma)
  := -|\xi|^{3/2}+\iota_1|\xi-\eta|^{3/2}+\iota_2|\eta-\sigma|^{3/2}+\iota_3|\sigma|^{3/2}.
\end{equation}
More precisely, one needs to understand the contribution of the {\it{spacetime resonances}},
i.e., the points where
\begin{align*}
 \Psi^{\iota_1\iota_2\iota_3}(\xi,\eta,\sigma) = (\partial_\eta\Psi^{\iota_1\iota_2\iota_3})(\xi,\eta,\sigma)
  = (\partial_\sigma\Psi^{\iota_1\iota_2\iota_3})(\xi,\eta,\sigma) = 0 .
\end{align*}
In our case, it can be easily verified that the only spacetime resonances (except for $(\xi,\eta,\sigma)=(0,0,0)$) correspond to $(\iota_1\iota_2\iota_3)=(++-)$
and $(\xi,\eta,\sigma) = (\xi,0,-\xi)$.
Moreover, the contribution from these points is not absolutely integrable in time,
and we have to identify and eliminate its leading order term using a suitable logarithmic phase correction.
More precisely, see also \eqref{c++-value}, we define, with $d_2=-1/16$, see \eqref{c++-value},
\begin{equation}\label{nf50}
\begin{split}
& \widetilde{c}(\xi) := -\frac{8\pi|\xi|^{1/2}}{3} c^{++-}(\xi,0,-\xi) = -\frac{8\pi d_2|\xi|^2}{3}=\frac{\pi|\xi|^2}{6}  ,
\\
& L(\xi,t) := \frac{\widetilde{c}(\xi)}{4\pi^2}\int_0^t {|\what{f}(\xi,s)|}^2 \frac{1}{s+1}\,ds,
\\
& g(\xi,t) := e^{iL(\xi,t)} \what{f}(\xi,t).
\end{split}
\end{equation}
The formula \eqref{nf48} then becomes
\begin{equation}\label{nf51}
\begin{split}
(\partial_tg)(\xi,t) &= \frac{i}{4\pi^2} e^{iL(\xi,t)}
  \Big[I^{++-}(\xi,t) + \widetilde{c}(\xi)\frac{|\what{f}(\xi,t)|^2}{t+1} \what{f}(\xi,t) \Big]
\\
&+ \frac{i}{4\pi^2} e^{iL(\xi,t)} \big[I^{+++}(\xi,t) + I^{--+}(\xi,t) + I^{---}(\xi,t) \big]
\\ &+ \, e^{-it|\xi|^{3/2}} e^{iL(\xi,t)} \what{\mathcal{R}_{\geq 4}^{''}}(\xi,t).
\end{split}
\end{equation}
Notice that the phase $L$ is real-valued.
Therefore, to complete the proof of Proposition \ref{proZ}, it suffices to prove the following main lemma:

\begin{lemma}\label{mainlem}
Recall the bounds \eqref{profL2}-\eqref{profS} and the
apriori assumption \eqref{proZapr}. Then, for any $m \in \{1,2,\ldots\}$ and any $t_1 \leq t_2 \in [2^{m}-2,2^{m+1}] \cap [0,T']$,
we have
\begin{equation}
\label{Zcontrolconc}
{\big\| (|\xi|^{1/10} + |\xi|^{N_2+1/2}) [g(\xi,t_2)-g(\xi,t_1)] \big\|}_{L^\infty_\xi} \lesssim \e_0 2^{-p_0m}.
\end{equation}
\end{lemma}

\section{Proof of Lemma \ref{mainlem}}\label{secprmainlem}
In this section we provide the proof of Lemma \ref{mainlem}.
We first notice that the desired conclusion can be easily proved for large and small enough frequencies.
Indeed, for any $t \in [2^{m}-2,2^{m+1}] \cap[0,T']$, and any $|\xi| \approx 2^k$ with $k \in \mathbb{Z}$ and
\begin{align*}
k \in (-\infty, -80p_0m] \cup [20p_0m,\infty) ,
\end{align*}
we can use the interpolation inequality \eqref{interp1} and the bounds \eqref{profL2}-\eqref{profS} to obtain
\begin{equation*}
\begin{split}
\big( |\xi|^{1/10} + |\xi|^{N_2+1/2} \big) |g(\xi,t)| & \lesssim (2^{k/10} + 2^{(N_2+1/2)k}) \big[ 2^{-k} {\|\what{P_kf}\|}_{L^2} \big(2^k {\|\partial \what{P_kf} \|}_{L^2}
  + {\|\what{P_kf}\|}_{L^2} \big) \big]^{1/2}
  \\
& \lesssim \e_0 \langle t \rangle^{6p_0} \min \big( 2^{(1/10-p_0)k} , 2^{-k/2} \big) ,
\\
& \lesssim \e_0\langle t \rangle^{-p_0},
\end{split}
\end{equation*}
having also used $(N_0+N_1)/2 \geq N_2+1$, see \eqref{constants}.

It remains to prove \eqref{Zcontrolconc} in the intermediate range $|\xi| \in [(1+t)^{-80p_0},(1+t)^{20p_0}]$.
For $k\in\mathbb{Z}$ let $f_k^+ := P_k f$ and $f_k^- := P_k \bar{f}$ and, for any $k_1,k_2,k_3 \in \mathbb{Z}$,  let
\begin{align}
\label{rak0}
\begin{split}
I^{\iota_1\iota_2\iota_3}_{k_1,k_2,k_3}(\xi,t):=\int_{\R\times\R} &
  e^{it(-|\xi|^{3/2} + \iota_1|\xi-\eta|^{3/2} + \iota_2|\eta-\sigma|^{3/2} + \iota_3|\sigma|^{3/2})}
\\
& \times c^{\iota_1\iota_2\iota_3}(\xi,\eta,\sigma) \what{f^{\iota_1}_{k_1}}(\xi-\eta)
  \what{f^{\iota_2}_{k_2}}(\eta-\sigma) \what{f^{\iota_3}_{k_3}}(\sigma)\,d\eta d\sigma.
\end{split}
\end{align}
Using \eqref{profL2}-\eqref{profS}, \eqref{proZapr} and Lemma \ref{dispersive} we know that for any $t\in[0,T']$ and $l \leq 0$
\begin{align}
\label{boundsflow}
\begin{split}
{\|\what{f_l^{\pm}(t)}\|}_{L^2} + 2^l {\|\partial\what{f_l^{\pm}(t)} \|}_{L^2}
  & \lesssim \e_1 2^{(1/2-p_0)l} \langle t \rangle^{6p_0},
\\
{\|e^{\pm it\Lambda}f_l^{\pm}(t)\|}_{L^\infty} & \lesssim \e_1 2^{l/10} \langle t \rangle^{-1/2},
\\
{\|\what{f_l^{\pm}(t)} \|}_{L^\infty} & \lesssim \e_1 2^{-l/10} ,
\end{split}
\end{align}
whereas, for $l\geq 0$, using also Lemma \ref{interpolation} and \eqref{disperseEa},
\begin{align}
\label{boundsfhigh}
\begin{split}
{\|\what{f_l^{\pm}(t)}\|}_{L^2} & \lesssim \e_1 2^{-(N_0-1/2)l} \langle t \rangle^{6p_0},
\\
{\|\what{f_l^{\pm}(t)}\|}_{L^2} + 2^l {\|\partial\what{f_l^{\pm}(t)} \|}_{L^2}
  & \lesssim \e_1 2^{-(N_1-1/2)l} \langle t \rangle^{6p_0},
\\
{\|e^{\pm it\Lambda}f_l^{\pm}(t)\|}_{L^\infty} & \lesssim \e_1 2^{-(N_2-1/2)l} \langle t \rangle^{-1/2},
\\
{\|\what{f_l^{\pm}(t)} \|}_{L^\infty} & \lesssim \e_1 2^{-(N_2+1/2) l} .
\end{split}
\end{align}

Using \eqref{csymbols} and the symbol bounds \eqref{equcubicsym} and \eqref{lemvsym}, it is not hard to see that
\begin{equation}
\label{boundcsym0}
\begin{split}
\big| c^{\iota_1\iota_2\iota_3} &(\xi,\eta,\sigma) \cdot \varphi'_{k_1}(\xi-\eta)\varphi'_{k_2}(\eta-\sigma)\varphi'_{k_3}(\sigma)\big|
\\
&\lesssim 2^{3\max(k_1,k_2,k_3)/2} 2^{[\mathrm{med}(k_1,k_2,k_3)-\min(k_1,k_2,k_3)]/2} . 
\end{split}
\end{equation}
Using this one can decompose the integrals $I^{\iota_1\iota_2\iota_3}$ into sums of the integrals
$I_{k_1,k_2,k_3}^{\iota_1\iota_2\iota_3}$,  and then estimate the terms corresponding to
large frequencies, and the terms corresponding to small frequencies (relative to $m$),
using only the bounds \eqref{boundsfhigh}-\eqref{boundcsym0}.
As in 
\cite[Section 6]{IoPu2}, we can then reduce matters to proving the following:

\begin{lemma}\label{mainlemma2}
Assume that $k \in [-80p_0m, 20p_0m]$, $|\xi| \in [2^k,2^{k+1}] \cap \mathbb{Z}$,
$m \geq 1$, $t_1\leq t_2 \in [2^{m}-2,2^{m+1}] \cap [0,T']$, and $k_1,k_2,k_3$ are integers satisfying
\begin{equation}
\label{mainlemma2cond}
\begin{split}
& k_1,k_2,k_3\in [-3m, 3m/N_0 - 1000],
\\
& \min(k_1,k_2,k_3)/2 + 3\mathrm{med}(k_1,k_2,k_3)/2 \geq -m(1+10p_0).
\end{split}
\end{equation}
Then
\begin{equation}
\label{mainlemma2conc1}
\Big|\int_{t_1}^{t_2} e^{iL(\xi,s)} \Big[ I^{++-}_{k_1,k_2,k_3}(\xi,s)
  + \widetilde{c}(\xi)\frac{\widehat{f^+_{k_1}}(\xi,s)\widehat{f^+_{k_2}}(\xi,s)\widehat{f^-_{k_3}}(-\xi,s)}{s+1}\Big]\,ds\Big|
  \lesssim \e_1^32^{-120p_0m},
\end{equation}
and, for $(\iota_1,\iota_2,\iota_3)\in\{(+++),(--+),(---)\}$,
\begin{equation}
\label{mainlemma2conc2}
\Big| \int_{t_1}^{t_2}e^{iL(\xi,s)} I^{\iota_1\iota_2\iota_3}_{k_1,k_2,k_3}(\xi,s) \,ds\Big| \lesssim \e_1^3 2^{-120p_0m}.
\end{equation}
Moreover
\begin{equation}
\label{mainlemma2conc3}
\Big|\int_{t_1}^{t_2}e^{iL(\xi,s)}e^{-is|\xi|^{3/2}} \what{\mathcal{R}_{\geq 4}^{''}}(\xi,s)\,ds\Big| \lesssim \eps_1^3 2^{-120p_0m}.
\end{equation}
\end{lemma}

The rest of this section is concerned with the proof of this lemma.
We will often use the alternative formulas
\begin{equation}
\label{I}
I^{\iota_1\iota_2\iota_3}_{k_1,k_2,k_3}(\xi,t) = \int_{\R \times \R}
  e^{it\Phi^{\iota_1\iota_2\iota_3}(\xi,\eta,\sigma)}
  c^{\ast,\iota_1\iota_2\iota_3}_{\xi;k_1,k_2,k_3}(\eta,\sigma) \what{f^{\iota_1}_{k_1}}(\xi+\eta)
  \what{f^{\iota_2}_{k_2}}(\xi+\sigma) \what{f^{\iota_3}_{k_3}}(-\xi-\eta-\sigma)\,d\eta d\sigma,
\end{equation}
where
\begin{equation}
\label{I'}
\begin{split}
& \Phi^{\iota_1\iota_2\iota_3}(\xi,x,y) := -\Lambda(\xi) + \iota_1\Lambda(\xi+x) + \iota_2\Lambda(\xi+y) + \iota_3\Lambda(\xi+x+y),
\\
& c^{\ast,\iota_1\iota_2\iota_3}_{\xi;k_1,k_2,k_3}(x,y) := c^{\iota_1\iota_2\iota_3}(\xi,-x,-\xi-x-y)
  \cdot \varphi'_{k_1}(\xi+x) \varphi'_{k_2}(\xi+y)\varphi'_{k_3}(\xi+x+y).
\end{split}
\end{equation}
These formulas follow from \eqref{rak0} via simple changes of variables.

We now recall that the symbols $c^{\iota_1\iota_2\iota_3}$ are given by the explicit formulas \eqref{csymbols},
together with the formulas \eqref{nf} for $m_{\eps_1\eps_2}$, and \eqref{equq_0} for $q_{\eps_1\eps_2}$.
Using the symbol bounds \eqref{lemvsym} and the explicit formulas \eqref{cxz40} one can verify that the symbols $c^{\ast,\iota_1\iota_2\iota_3}_{\xi;k_1,k_2,k_3}$ satisfy the $S^\infty$ estimates
\begin{equation}
\label{cbounds}
{\big\| \mathcal{F}^{-1}\big(c^{\ast,\iota_1\iota_2\iota_3}_{\xi;k_1,k_2,k_3}\big) \big \|}_{L^1(\R^2) }
  \lesssim 2^{3k_{\max}/2}2^{(k_{\mathrm{med}}-k_{\min})/2} ,
\end{equation}
and, with $\psi_{\xi;l_1,l_2,l_3}(x,y)=\varphi_{l_1}(x)\varphi_{l_2}(y)\varphi_{l_3}(2\xi+x+y)$,
\begin{equation}
\label{cbounds'}
\begin{split}
& {\big\|\mathcal{F}^{-1}\big[(\partial_x c^{\ast,\iota_1\iota_2\iota_3}_{\xi;k_1,k_2,k_3})(x,y)\cdot \psi_{\xi;l_1,l_2,l_3}(x,y)\big]\big\|}_{L^1(\R^2)}
  \lesssim 2^{3k_{\max}/2}2^{(k_{\mathrm{med}}-k_{\min})/2}2^{-\min(k_1,k_3)},
\\
& {\big\|\mathcal{F}^{-1}\big[(\partial_y c^{\ast,\iota_1\iota_2\iota_3}_{\xi;k_1,k_2,k_3})(x,y)\cdot \psi_{\xi;l_1,l_2,l_3}(x,y)\big]\big\|}_{L^1(\R^2)}
  \lesssim 2^{3k_{\max}/2}2^{(k_{\mathrm{med}}-k_{\min})/2}2^{-\min(k_2,k_3)},
\end{split}
\end{equation}
for any $\xi \in \R$ and $k_1,k_2,k_3,l_1,l_2,l_3\in\mathbb{Z}$,
where $k_{\max} := \max(k_1,k_2,k_3)$, $k_{\mathrm{med}} := \mathrm{med}(k_1,k_2,k_3)$, $k_{\min} := \min(k_1,k_2,k_3)$. In fact the contributions of the cubic symbols $m_{\iota_1\iota_2\iota_3}$ satisfy stronger bounds, with a factor of $|\xi|^{1/2}2^{k_{\max}}$ instead of $2^{3k_{\max}/2}2^{(k_{\mathrm{med}}-k_{\min})/2}$, see \eqref{cxz40}. 

The cutoff factors $\psi_{\xi;l_1,l_2,l_3}(x,y)$ are needed in order for the bounds \eqref{cbounds'} to hold, due to the presence of factors such as $|x|$ in the symbols. These cutoffs do not play an important role in the proof of Lemma \ref{mainlemma2}. Inserting these cutoffs only leads to an additional logarithmic loss in $\langle t\rangle$ in some of the estimates, which we can easily tolerate.

\subsection{Proof of \eqref{mainlemma2conc1}}\label{prooftech1}

We divide the proof of the bound \eqref{mainlemma2conc1} into several lemmas.
Since in this subsection we will only deal with interactions of the type $++-$,
for simplicity of notation we denote $\Phi := \Phi^{++-}$ and $c^{\ast}_{\mathbf{k}} := c^{\ast,++-}_{\xi;k_1,k_2,k_3}$.

\begin{lemma}\label{lemma1}
The bound \eqref{mainlemma2conc1} holds provided that \eqref{mainlemma2cond} holds, and
\begin{equation*}
\max \big(|k-k_1|, |k-k_2|, |k-k_3|\big)\leq 20 .
\end{equation*}
\end{lemma}

\begin{proof}
This is the case which gives the precise form of the correction.
However, the proof is similar to the proof of Lemma 6.4 in \cite{IoPu2}.
The only difference is that in the present case $\Lambda(\xi)=|\xi|^{3/2}$ instead of $|\xi|^{1/2}$.
Therefore one has the expansion
\begin{equation*}
\Big| \Lambda(\xi) - \Lambda(\xi+\eta) - \Lambda(\xi+\sigma) + \Lambda(\xi+\eta+\sigma)
  - \frac{3\eta\sigma}{4|\xi|^{1/2}} \Big| \lesssim 2^{-3k/2}(|\eta|^3+|\sigma|^3).
\end{equation*}
One can then follow the same argument in \cite[Lemma 6.4]{IoPu2} to obtain the desired bound.
\end{proof}

\begin{lemma}\label{lemma2}
The bound \eqref{mainlemma2conc1} holds provided that \eqref{mainlemma2cond} holds and, in addition,
\begin{align}
\label{lemma2cond}
\begin{split}
& \max (|k-k_1|, |k- k_2|, |k-k_3|) \geq 21,
\\
& \max (|k_1 - k_3| , |k_2-k_3|) \geq 5 \quad \mbox{and} \quad \min (k_1,k_2,k_3) \geq -\frac{49}{100}m .
\end{split}
\end{align}
\end{lemma}

\begin{proof}
In this case we will show the stronger bound
\begin{equation}
\label{lemma2conc}
|I^{++-}_{k_1,k_2,k_3}(\xi,s)|\lesssim \e_1^3 2^{-m} 2^{-200p_0m}.
\end{equation}
Without loss of generality, by symmetry we can assume that $|k_1-k_3| \geq 5$ and $k_2 \leq \max(k_1,k_3)+5$.
Under the assumptions \eqref{lemma2cond} we have
\begin{align}
\label{lowerbound1}
 |(\partial_\eta\Phi)(\xi,\eta,\sigma)| = |\Lambda'(\xi+\eta) - \Lambda'(\xi+\eta+\sigma)| \gtrsim 2^{k_{\max}/2}.
\end{align}
Therefore we can integrate by parts in $\eta$ in the integral expression \eqref{I} for $I_{k_1,k_2,k_3}^{++-}$.
This gives
\begin{align*}
|I_{k_1,k_2,k_3}^{++-}(\xi,s)| & \lesssim  |K_1(\xi,s)| + |K_2(\xi,s)| + |K_3(\xi,s)| + |K_4(\xi,s)|,
\end{align*}
where
\begin{align}
\label{IBPeta}
\begin{split}
K_1(\xi) & := \int_{\mathbb{R}\times\mathbb{R}} e^{is\Phi(\xi,\eta,\sigma)}
   m_1(\eta,\sigma) c^\ast_{\mathbf{k}}(\eta,\sigma)
  (\partial\what{f_{k_1}^{+}})(\xi+\eta) \what{f_{k_2}^{+}}(\xi+\sigma) \what{f_{k_3}^{-}}(-\xi-\eta-\sigma) \,d\eta d\sigma,
\\
K_2(\xi) & := \int_{\mathbb{R}\times\mathbb{R}} e^{is\Phi(\xi,\eta,\sigma)}
   m_1(\eta,\sigma) c^\ast_{\mathbf{k}}(\eta,\sigma)
  \what{f_{k_1}^{+}}(\xi+\eta) \what{f_{k_2}^{+}}(\xi+\sigma) (\partial\what{f_{k_3}^{-}})(-\xi-\eta-\sigma) \,d\eta d\sigma,
\\
K_3(\xi) & := \int_{\mathbb{R}\times\mathbb{R}} e^{is\Phi(\xi,\eta,\sigma)}
  (\partial_\eta  m_1)(\eta,\sigma) c^\ast_{\mathbf{k}}(\eta,\sigma)
  \what{f_{k_1}^{+}}(\xi+\eta) \what{f_{k_2}^{+}}(\xi+\sigma) \what{f_{k_3}^{-}}(-\xi-\eta-\sigma) \,d\eta d\sigma,
\\
K_4(\xi) & := \int_{\mathbb{R}\times\mathbb{R}} e^{is\Phi(\xi,\eta,\sigma)}
   m_1(\eta,\sigma) (\partial_\eta c^\ast_{\mathbf{k}})(\eta,\sigma)
  \what{f_{k_1}^{+}}(\xi+\eta) \what{f_{k_2}^{+}}(\xi+\sigma) \what{f_{k_3}^{-}}(-\xi-\eta-\sigma) \,d\eta d\sigma,
\end{split}
\end{align}
having denoted
\begin{align}
\label{symIBPeta}
 m_1(\eta,\sigma) := \frac{1}{s \partial_\eta \Phi(\xi,\eta,\sigma)} \varphi'_{k_1}(\xi+\eta)\varphi'_{k_3}(\xi+\eta+\sigma).
\end{align}

Observe that, under the restrictions \eqref{lemma2cond}, we have
\begin{align}
\label{symIBPetaest1}
{\| m_1(\eta,\sigma) \|}_{S^\infty} \lesssim 2^{-m} 2^{-k_{\max}/2}.
\end{align}
We can then estimate $K_1$ using Lemma \ref{touse}(ii), the estimate on $c^\ast_{\mathbf{k}}$ in \eqref{cbounds},
the bounds \eqref{boundsflow}-\eqref{boundsfhigh}, and the last constraint in \eqref{lemma2cond}.
More precisely, if $k_1\leq k_3$ (so that $2^{k_3} \approx 2^{k_{\max}}$) then we estimate
\begin{align*}
|K_1&(\xi,s)| \lesssim {\| m_1(\eta,\sigma) \|}_{S^\infty}  {\| c^\ast_{\mathbf{k}}(\eta,\sigma) \|}_{S^\infty}
  {\| \partial \what{f_{k_1}^{+}}(s) \|}_{L^2} {\| e^{is\Lambda} f_{k_2}^{+}(s) \|}_{L^\infty} {\| \what{f_{k_3}^{-}}(s) \|}_{L^2}
\\
& \lesssim 2^{-m} 2^{-k_3/2} \cdot 2^{3 k_3^+} 2^{-k_{\min}/2}\cdot
  \e_1 2^{-k_1(1/2+p_0)} 2^{6mp_0} \cdot \e_1 2^{-m/2} \cdot \e_1 2^{6mp_0}2^{-(N_0-1)k_3^+}
\\
& \lesssim \e_1^3 2^{-3m/2} 2^{12p_0m} 2^{-k_{\min}/2}2^{-k_1(1/2+p_0)}
\\
&\lesssim \e_1^3 2^{15p_0m}2^{-3m/2}2^{-k_{\min}}.
\end{align*}
This suffices to prove \eqref{lemma2conc} because of the last inequality in \eqref{lemma2cond}.
On the other hand, if $k_1\geq k_3$ (so that $2^{k_1}\approx 2^{k_{\max}} \gtrsim 2^k$) then
\begin{align*}
|K_1(\xi,s)| &\lesssim {\| m_1(\eta,\sigma) \|}_{S^\infty}  {\| c^\ast_{\mathbf{k}}(\eta,\sigma) \|}_{S^\infty}
  {\| \partial \what{f_{k_1}^{+}}(s) \|}_{L^2}{\| e^{is\Lambda} f_{k_2}^{+}(s) \|}_{L^\infty} {\| \what{f_{k_3}^{-}}(s) \|}_{L^2}
\\
& \lesssim 2^{-m} 2^{-k_1/2} \cdot 2^{2k_1} 2^{-k_{\min}/2} \cdot \e_1 2^{-k_1-k_1^+}2^{6mp_0}\cdot \e_12^{-m/2} \cdot \e_12^{6p_0m}
\\
& \lesssim \e_1^3 2^{-3m/2} 2^{14p_0m} 2^{-k_{\min}/2}.
\end{align*}
This suffices to prove \eqref{lemma2conc}.
The estimate for $|K_2(\xi,s)|$ is identical. 
The estimates for $|K_3(\xi,s)|$ and $|K_4(\xi,)|$ are also similar, because differentiation with respect 
to $\eta$ is equivalent\footnote{Some care is needed when the derivative hits $c^\ast_{\mathbf{k}}$, 
because of the slightly weaker bounds in \eqref{cbounds'}. 
In this case we insert cutoff functions localizing the variables $\eta,\sigma,2\xi+\eta+\sigma$ 
to dyadic intervals and estimate in the same way. The final bound is multiplied by $Cm^3$, which is acceptable.} 
to multiplication by a factor of $(2^{-k_1}+2^{-k_3})$, and one can use similar estimates as before 
\end{proof}

\begin{lemma}\label{lemma3}
The bound \eqref{mainlemma2conc1} holds provided that \eqref{mainlemma2cond} holds and, in addition,
\begin{align}
\label{lemma3cond}
\begin{split}
& \max (|k-k_1|, |k- k_2|, |k-k_3|) \geq 21,
\\
& \max (|k_1 - k_3| , |k_2-k_3|) \geq 5 \quad \mbox{and} \quad \min (k_1,k_2) \leq -\frac{48m}{100}.
\end{split}
\end{align}
\end{lemma}

\begin{proof}
By symmetry we may assume that $k_2 = \min(k_1,k_2)$.
The main observation is that we still have the strong lower bound
\begin{align*}
 |(\partial_\eta \Phi)(\xi,\eta,\sigma)| = |\Lambda'(\xi+\eta) - \Lambda'(\xi+\eta+\sigma)| \gtrsim 2^{-k_{\max}/2} 2^k.
\end{align*}
This is easy to verify since $|\xi|\geq 2^{-80p_0m}$, $|\xi+\sigma|\leq 2^{k_2+1} \leq 2^{-48m/100 + 1}$.
We can then integrate by parts in $\eta$, and estimate the resulting integrals as in Lemma \ref{lemma2},
by placing $\what{f_{k_2}}$ in $L^2$, and using the restriction $k_{\min} + 3k_{\mathrm{med}} \geq -2m(1+10p_0)$
in \eqref{mainlemma2cond}.
\end{proof}

\begin{lemma}\label{lemma4}
The desired bound \eqref{mainlemma2conc1} holds provided that \eqref{mainlemma2cond} holds and,
in addition,
\begin{align}
\label{lemma4cond}
\begin{split}
& \max (|k-k_1|, |k- k_2|, |k-k_3|) \geq 21,
\\
& \max (|k_1 - k_3| , |k_2-k_3|) \geq 5 \quad \mbox{and} \quad
  \min (k_1,k_2) \geq -\frac{48m}{100}\quad \mbox{and} \quad k_3 \leq -\frac{49m}{100}.
\end{split}
\end{align}
\end{lemma}

\begin{proof}
In this case we need to integrate by parts in time.
Without loss of generality, we may again assume $k_2 = \min (k_1,k_2)$.
Therefore, also in view of \eqref{mainlemma2cond}, we have
\begin{align}
\label{conf10}
k_3 \leq -49m/100 \leq -48m/100 \leq k_2 \leq k_1,\qquad k_1 \geq k-10 \geq -80p_0m -10.
\end{align}

Recall that
\begin{align*}
\Phi(\xi,\eta,\sigma) = -\Lambda(\xi) + \Lambda(\xi+\eta) + \Lambda(\xi+\sigma) - \Lambda(\xi+\eta+\sigma).
\end{align*}
For $|\xi+\eta| \approx 2^{k_1}$, $|\xi+\sigma| \approx |\eta| \approx 2^{k_2}$, $|\xi+\eta+\sigma| \approx 2^{k_3}$,
with $k_1,k_2,k_3$ satisfying \eqref{conf10}, one can use \eqref{phiab} to show that
\begin{align}
\label{IBPs0}
| \Phi(\xi,\eta,\sigma) | & \geq |- \Lambda(\xi) + \Lambda(\xi+\eta) + \Lambda(\eta) | - 2^{10}|\xi+\eta+\sigma| 2^{k_2/2}
  \gtrsim 2^{k_2} 2^{k/2}.
\end{align}
Thanks to this lower bound we can integrate by parts in $s$ to obtain
\begin{align}
\label{IBPs}
\begin{split}
\Big| \int_{t_1}^{t_2} & e^{iL(\xi,s)} I_{k_1,k_2,k_3}^{++-}(\xi,s) \, ds \Big|
  \lesssim  | N_1 (\xi,t_1) | + | N_1 (\xi,t_2) |
  \\ & + \int_{t_1}^{t_2} | N_2(\xi,s)|+ | N_3(\xi,s) | + | N_4(\xi,s)| + |(\partial_sL)(\xi,s)|| N_1(\xi,s) | \, ds,
\end{split}
\end{align}
where
\begin{align}
\label{IBPs2}
\begin{split}
& N_1(\xi) := \int_{\R\times\R} e^{is\Phi(\xi,\eta,\sigma)} \frac{c^\ast_{\mathbf{k}} (\eta,\sigma)}{i \Phi(\xi,\eta,\sigma)}
  \what{f_{k_1}^{+}}(\xi+\eta) \what{f_{k_2}^{+}}(\xi+\sigma) \what{f_{k_3}^{-}}(-\xi-\eta-\sigma) \,d\eta d\sigma ,
\\
& N_2(\xi) := \int_{\R\times\R} e^{is\Phi(\xi,\eta,\sigma)} \frac{c^\ast_{\mathbf{k}}(\eta,\sigma)}{i \Phi(\xi,\eta,\sigma)}
 \, (\partial_s\what{f_{k_1}^{+}})(\xi+\eta) \what{f_{k_2}^{+}}(\xi+\sigma) \what{f_{k_3}^{-}}(-\xi-\eta-\sigma) \,d\eta d\sigma ,
\\
& N_3(\xi) := \int_{\R\times\R} e^{is\Phi(\xi,\eta,\sigma)} \frac{c^\ast_{\mathbf{k}}(\eta,\sigma)}{i \Phi(\xi,\eta,\sigma)}
  \what{f_{k_1}^{+}}(\xi+\eta)(\partial_s\what{f_{k_2}^{+}})(\xi+\sigma) \what{f_{k_3}^{-}}(-\xi-\eta-\sigma) \,d\eta d\sigma ,
\\
& N_4(\xi) := \int_{\R\times\R} e^{is\Phi(\xi,\eta,\sigma)} \frac{c^\ast_{\mathbf{k}}(\eta,\sigma)}{i \Phi(\xi,\eta,\sigma)}
  \what{f_{k_1}^{+}}(\xi+\eta) \what{f_{k_2}^{+}}(\xi+\sigma) (\partial_s\what{f_{k_3}^{-}}) (-\xi-\eta-\sigma) \,d\eta d\sigma .
\end{split}
\end{align}

To estimate the first term in \eqref{IBPs} we first notice that we have the pointwise bound
\begin{align}
\label{IBPssym}
\Big| \frac{c^\ast_{\mathbf{k}} (\eta,\sigma)}{ \Phi(\xi,\eta,\sigma)} \Big| \lesssim 2^{2k_1^+} 2^{-k/2} 2^{-k_2/2} 2^{-k_3/2},
\end{align}
see \eqref{IBPs0} and \eqref{cbounds}.
Using also \eqref{boundsflow}-\eqref{boundsfhigh} we can obtain, for any $s\in[t_1,t_2]$,
\begin{align*}
|N_1(\xi,s)| & \lesssim 2^{2k_1^+} 2^{-k/2} 2^{-k_2/2} 2^{-k_3/2}
  {\| \what{f_{k_1}^{+}}(s) \|}_{L^\infty}
  2^{k_2/2} {\| \what{f_{k_2}^{+}}(s) \|}_{L^2}  2^{k_3/2} {\| \what{f_{k_3}^{-}}(s) \|}_{L^2}
\\
& \lesssim \e_1^3 2^{k_3(1/2-p_0)} 2^{100p_0m}
\\
&\lesssim \e_1^3 2^{-m/10}.
\end{align*}
Moreover, the definition of $L$ in \eqref{nf50}, and the a priori assumptions \eqref{boundsflow}-\eqref{boundsfhigh}, show that
\begin{align*}
\big|(\partial_sL)(\xi,s) \big| \lesssim \e_1^2 2^{-m}.
\end{align*}
Therefore
\begin{align}
\label{N1bound}
| N_1 (\xi,t_1) | + | N_1 (\xi,t_2) | + \int_{t_1}^{t_2} |(\partial_sL)(\xi,s)|| N_1(\xi,s) | \, ds\lesssim \e_1^3 2^{-m/10}.
\end{align}

Using the equation for $\partial_t f$ in \eqref{eqprof}, and the estimate \eqref{estN'}, we have
\begin{equation}
\label{prs}
{\big\| (\partial_s\what{f_{l}^{\pm}})(s) \big\|}_{L^2}\lesssim \e_1^2 \min(2^{(1/2-p_0)l}, 2^{-2l}) 2^{-m+6p_0m} .
\end{equation}
Using this $L^2$ bound, and the pointwise bound \eqref{IBPssym} on the symbol, the term $N_2$ can be estimated as follows:
 \begin{align*}
\begin{split}
|N_2(\xi,s)| & \lesssim 2^{2k_1^+} 2^{-k/2} 2^{-k_2/2} 2^{-k_3/2}
    {\big\|(\partial_s\what{f_{k_1}^{+}})(s) \big\|}_{L^2}
    {\| \what{f_{k_2}^{+}}(s) \|}_{L^2} {\| \what{f_{k_3}^{-}}(s) \|}_{L^2} 2^{k_3/2}
\\
& \lesssim \e_1^3 2^{2\max(k_1,0)} 2^{-m + 100p_0m} 2^{-4k_2^+}2^{k_3(1/2-p_0)}
\\
& \lesssim \e_1^3 2^{-m} 2^{-m/10}.
\end{split}
\end{align*}
The integral $|N_3(\xi,s)|$ can be estimated in the same way.
To deal with the last term in \eqref{IBPs2} we use again \eqref{IBPssym}, \eqref{prs}, and
the a priori bounds \eqref{boundsflow}-\eqref{boundsfhigh} to obtain:
 \begin{align*}
\begin{split}
|N_4(\xi,s)| & \lesssim 2^{2\max(k_1,0)} 2^{-k/2} 2^{-k_2/2} 2^{-k_3/2}
    {\big\|\what{f_{k_1}^{+}}(s) \big\|}_{L^2}
    {\| \what{f_{k_2}^{+}}(s) \|}_{L^2} {\| (\partial_s\what{f_{k_3}^{-}})(s) \|}_{L^2} 2^{k_3/2}
\\
& \lesssim \e_1^3 2^{-m + 100p_0m} 2^{k_3(1/2-p_0)}
\\
& \lesssim \e_1^3 2^{-m}2^{-m/10}.
\end{split}
\end{align*}

We deduce that
\begin{align*}
\int_{t_1}^{t_2} | N_2(\xi,s)| + |N_3(\xi,s)| + |N_4(\xi,s)| \, ds \lesssim \e_1^3 2^{-m/10} ,
\end{align*}
and the lemma follows from \eqref{IBPs} and \eqref{N1bound}.
\end{proof}

\begin{lemma}\label{lemma5}
The bound \eqref{mainlemma2conc1} holds provided that \eqref{mainlemma2cond} holds and, in addition,
\begin{equation}
\label{lemma5cond}
\max (|k-k_1|, |k- k_2|, |k-k_3|) \geq 21 \qquad \mbox{and} \qquad \max (|k_1 - k_3|, |k_2-k_3|) \leq 4.
\end{equation}
\end{lemma}

\begin{proof}
In this case we have $\min (k_1,k_2,k_3) \geq  k + 10$, so that, in particular, $|\sigma|\approx 2^{k_2}$.
Since all input frequencies are comparable in view of the second assumption in \eqref{lemma5cond},
we can see that
\begin{align*}
|(\partial_\eta \Phi)(\xi,\eta,\sigma)| = |\Lambda'(\xi+\eta)-\Lambda'(\xi+\eta+\sigma)| 
  \gtrsim 2^{k_{\max}/2} ,
\end{align*}
which is the same lower bound as in \eqref{lowerbound1}.
We can then integrate by parts in $\eta$ similarly to what was done before in Lemma \ref{lemma2}.
This gives
\begin{align*}
|I_{k_1,k_2,k_3}^{++-}(\xi,s)| & \lesssim  |K_1(\xi,s)| + |K_2(\xi,s)| + |K_3(\xi,s)| + |K_4(\xi,s)|
\end{align*}
where the term $K_j$, $j=1,\dots 4$ are defined in \eqref{IBPeta}-\eqref{symIBPeta},
and the bound \eqref{symIBPeta} is satisfied.
Then, the same estimates that followed \eqref{symIBPetaest1}
show that $|I^{++-}_{k_1,k_2,k_3}(\xi,s)|\lesssim \eps_1^32^{-m}2^{-200p_0m}$, which suffices to prove the lemma.
\end{proof}

\subsection{Proof of \eqref{mainlemma2conc2}}\label{prooftech2}

We divide the proof of the bound \eqref{mainlemma2conc2} into several lemmas.
We only consider in detail the case $(\iota_1\iota_2\iota_3)=(--+)$,
since the cases $(\iota_1\iota_2\iota_3) = (+++)$ or $(---)$ are very similar.
In the rest of this subsection we let $\Phi := \Phi^{--+}$ and $c^{\ast}_{\mathbf{k}} := c^{\ast,--+}_{\xi;k_1,k_2,k_3}$.

\begin{lemma}\label{lemma11}
The bound \eqref{mainlemma2conc2} holds provided that \eqref{mainlemma2cond} holds and, in addition,
\begin{align}
\label{lemma11cond}
\max (|k_1 - k_3| , |k_2-k_3|) \geq 5 \quad \mbox{and} \quad \min (k_1,k_2,k_3) \geq -\frac{49}{100}m.
\end{align}
\end{lemma}

\begin{proof}
This case is similar to the proof of Lemma \ref{lemma2}.
Without loss of generality, by symmetry we can assume that $|k_1-k_3| \geq 5$ and $k_2\leq \max(k_1,k_3) + 5$.
Under the assumptions \eqref{lemma11cond} we still have the strong lower bound
\begin{align*}
 |(\partial_\eta \Phi)(\xi,\eta,\sigma)| = | \Lambda'(\xi+\eta)-\Lambda'(\xi+\eta+\sigma) | \gtrsim 2^{k_{\max}/2} .
\end{align*}
The proof can then proceed exactly as in Lemma \ref{lemma2}, using integration by parts in $\eta$.
\end{proof}

\begin{lemma}\label{lemma12}
The bound \eqref{mainlemma2conc2} holds provided that \eqref{mainlemma2cond} holds and, in addition,
\begin{align}
\label{lemma12cond}
\max (|k_1 - k_3| , |k_2-k_3|) \geq 5 \quad \mbox{and} \quad \mathrm{med} (k_1,k_2,k_3) \leq -48m/100.
\end{align}
\end{lemma}

\begin{proof}
This is similar to the situation in Lemma \ref{lemma3}.
By symmetry we may assume that $k_2 = \min(k_1,k_2)$. The main observation is that in this case we must have
$|k_1-k_3 |\geq 5$, because of the second assumption in \eqref{lemma12cond} and $k \geq -80p_0 m$.
We then have the lower bound
\begin{align*}
  |(\partial_\eta \Phi)(\xi,\eta,\sigma)| = |\Lambda'(\xi+\eta) - \Lambda'(\xi+\eta+\sigma)| \gtrsim 2^{k_{\max}/2}.
\end{align*}
Thus, we integrate by parts in $\eta$ and estimate the resulting integrals as in Lemma \ref{lemma2},
by placing $\what{f_{k_2}^+}$ in $L^2$, and using the restriction $k_{\min} + 3k_{\mathrm{med}}\geq -2m(1+10p_0)$,
see \eqref{mainlemma2cond}.
\end{proof}

\begin{lemma}\label{lemma13}
The bound \eqref{mainlemma2conc2} holds provided that \eqref{mainlemma2cond} holds and, in addition,
\begin{align*}
\max (|k_1 - k_3| , |k_2-k_3|) \geq 5 \quad \mbox{and} \quad k_{\mathrm{med}}\geq -\frac{48m}{100}
  \quad \mbox{and} \quad k_{\min}\leq -\frac{49m}{100}.
\end{align*}
\end{lemma}

\begin{proof}
This situation is similar to the one in Lemma \ref{lemma4}. The main observation is that we have the lower bound
\begin{align*}
| \Phi(\xi,\eta,\sigma) | \gtrsim 2^{\min(k_{\mathrm{med}},k)} 2^{k/2},
\end{align*}
so that we can integrate by parts in time and estimate the resulting integrals as in the proof of Lemma \ref{lemma4}.
\end{proof}

\begin{lemma}\label{lemma14}
The bound \eqref{mainlemma2conc2} holds provided that \eqref{mainlemma2cond} holds and, in addition,
\begin{equation*}
\max (|k-k_1|, |k- k_2|, |k-k_3|) \geq 21 \qquad \mbox{and} \qquad \max (|k_1 - k_3|, |k_2-k_3|) \leq 4.
\end{equation*}
\end{lemma}

\begin{proof}
This case is similar to that of Lemma \ref{lemma5}.
Observing that $k+10 \leq \min(k_1,k_2,k_3)$, we see that
\begin{align*}
 |(\partial_\eta \Phi)(\xi,\eta,\sigma)| = |\Lambda'(\xi+\eta)-\Lambda'(\xi+\eta+\sigma)| \gtrsim 2^{k_{\max}/2} .
\end{align*}
We can then use again integration by parts in $\eta$ to obtain the desired bound.
\end{proof}

\begin{lemma}\label{lemma15}
The bound \eqref{mainlemma2conc2} holds provided that \eqref{mainlemma2cond} holds and, in addition,
\begin{align}
\label{lemma15cond}
\max (|k-k_1|, |k- k_2|, |k-k_3|) \leq 20.
\end{align}
\end{lemma}

\begin{proof}
This is the main case where there is a substantial difference between the integrals $I_{k_1,k_2,k_3}^{++-}$ and $I_{k_1,k_2,k_3}^{--+}$.
The main point is that the phase function $\Phi^{--+}$ does not have any spacetime resonances,
i.e. there are no $(\eta,\sigma)$ solutions of the equations
\begin{equation*}
\Phi^{--+}(\xi,\eta,\sigma) = (\partial_\eta\Phi^{--+})(\xi,\eta,\sigma) = (\partial_\sigma\Phi^{--+})(\xi,\eta,\sigma) = 0.
\end{equation*}

For any $l,j \in \mathbb{Z}$ satisfying $l \leq j$ define
\begin{equation*}
\varphi_j^{(l)} :=
\begin{cases}
\varphi_j & \text{ if } j \geq l+1 ,
\\
\varphi_{\leq l} & \text{ if } j=l .
\end{cases}
\end{equation*}
Let $\bar{l} := k - 20$ and decompose
\begin{align*}
I^{--+}_{k_1,k_2,k_3} & = \sum_{l_1,l_2 \in [\bar{l},k+40] } J_{l_1,l_2},
\\
J_{l_1,l_2}(\xi,t) & := \int_{\R \times \R} e^{it\Phi(\xi,\eta,\sigma)} c^{\ast}_{\mathbf{k}}(\eta,\sigma)
  \varphi_{l_1}^{(\bar{l})}(\eta) \varphi_{l_2}^{(\bar{l})}(\sigma)
  \what{f^{-}_{k_1}}(\xi+\eta) \what{f^{-}_{k_2}}(\xi+\sigma) \what{f^{+}_{k_3}}(-\xi-\eta-\sigma) \, d\eta d\sigma.
\end{align*}

The contributions of the integrals $J_{l_1,l_2}$ for $(l_1,l_2) \neq (\bar{l},\bar{l})$
can be estimated by integration by parts either in $\eta$ or in $\sigma$
(depending on the relative sizes of $l_1$ and $l_2$), since the $(\eta,\sigma)$ gradient of the phase function
$\Phi$ is bounded from below by $c 2^{k/2}$ in the support of these integrals.

On the other hand, to estimate the contribution of the integral $J_{\overline{l},\overline{l}}$
we notice that
\begin{align*}
|\Phi(\xi,\eta,\sigma)| \gtrsim 2^{3k/2}
\end{align*}
in the support of the integral, so that we can integrate by parts in $s$. This gives
\begin{align}
\begin{split}
\label{sum1}
\Big| \int_{t_1}^{t_2} &e^{iL(\xi,s)} J_{\overline{l},\overline{l}}(\xi,s) \, ds \Big|
  \lesssim  | L_4 (\xi,t_1) | + | L_4 (\xi,t_2) | \\
& + \int_{t_1}^{t_2} | L_1(\xi,s)|+ | L_2(\xi,s) | + |L_3(\xi,s)| + |(\partial_sL)(\xi,s)| | L_4(\xi,s) | \, ds,
\end{split}
\end{align}
where
\begin{align*}
\begin{split}
& L_1(\xi) := \int_{\R\times\R} e^{is\Phi(\xi,\eta,\sigma)}
  \frac{c^\ast_{\mathbf{k}}(\eta,\sigma)\varphi_{\leq \overline{l}}(\eta)\varphi_{\leq \overline{l}}(\sigma)}{i \Phi(\xi,\eta,\sigma)}
  \,(\partial_s\what{f_{k_1}^{-}})(\xi+\eta) \what{f_{k_2}^{-}}(\xi+\sigma) \what{f_{k_3}^{+}}(-\xi-\eta-\sigma) \,d\eta d\sigma ,
\\
& L_2(\xi) := \int_{\R\times\R} e^{is\Phi(\xi,\eta,\sigma)}
  \frac{c^\ast_{\mathbf{k}}(\eta,\sigma)\varphi_{\leq \overline{l}}(\eta)\varphi_{\leq \overline{l}}(\sigma)}{i \Phi(\xi,\eta,\sigma)}
  \what{f_{k_1}^{-}}(\xi+\eta)(\partial_s\what{f_{k_2}^{-}})(\xi+\sigma) \what{f_{k_3}^{+}}(-\xi-\eta-\sigma) \,d\eta d\sigma,
  \\
  & L_3(\xi) := \int_{\R\times\R} e^{is\Phi(\xi,\eta,\sigma)}
  \frac{c^\ast_{\mathbf{k}}(\eta,\sigma)\varphi_{\leq \overline{l}}(\eta)\varphi_{\leq \overline{l}}(\sigma)}{i \Phi(\xi,\eta,\sigma)}
  \what{f_{k_1}^{-}}(\xi+\eta)\what{f_{k_2}^{-}}(\xi+\sigma) (\partial_s\what{f_{k_3}^{+}})(-\xi-\eta-\sigma) \,d\eta d\sigma,
  \\
  & L_4(\xi) := \int_{\R\times\R} e^{is\Phi(\xi,\eta,\sigma)}
  \frac{c^\ast_{\mathbf{k}} (\eta,\sigma)\varphi_{\leq \overline{l}}(\eta)\varphi_{\leq \overline{l}}(\sigma)}{i \Phi(\xi,\eta,\sigma)}
  \what{f_{k_1}^{-}}(\xi+\eta) \what{f_{k_2}^{-}}(\xi+\sigma) \what{f_{k_3}^{+}}(-\xi-\eta-\sigma) \,d\eta d\sigma.
\end{split}
\end{align*}

To estimate the integrals $L_1, L_2, L_3, L_4$ we notice that using \eqref{cbounds} and \eqref{lemma15cond} we have
\begin{equation*}
{\Big\|\frac{c^\ast_{\mathbf{k}} (\eta,\sigma)
  \varphi_{\leq \overline{l}}(\eta) \varphi_{\leq \overline{l}}(\sigma)}{i \Phi(\xi,\eta,\sigma)}\Big\|}_{S^\infty}
  \lesssim 2^{3k^+} .
\end{equation*}
Therefore, using Lemma \ref{touse}(ii) and the a priori bounds \eqref{boundsflow}-\eqref{boundsfhigh} we see that
\begin{align*}
|L_4(\xi,s)|\lesssim {\Big\|\frac{c^\ast_{\mathbf{k}} (\eta,\sigma)
  \varphi_{\leq \overline{l}}(\eta)\varphi_{\leq \overline{l}}(\sigma)}{i\Phi(\xi,\eta,\sigma)}\Big\|}_{S^\infty}
  {\| \what{f_{k_1}^{-}}(s) \|}_{L^2} {\| e^{is\Lambda} f_{k_2}^{-}(s) \|}_{L^\infty} {\| \what{f_{k_3}^{+}}(s) \|}_{L^2}
  \lesssim \e_1^3 2^{-m/4}.
\end{align*}
Using also \eqref{prs} we obtain
\begin{align*}
|L_1(\xi,s)| \lesssim {\Big\| \frac{c^\ast_{\mathbf{k}} (\eta,\sigma)
  \varphi_{\leq \overline{l}}(\eta)\varphi_{\leq \overline{l}}(\sigma)}{i\Phi(\xi,\eta,\sigma)}\Big\|}_{S^\infty}
  {\| (\partial_s\what{f_{k_1}^{-}})(s) \|}_{L^2} {\| e^{is\Lambda} f_{k_2}^{-}(s) \|}_{L^\infty}
  {\| \what{f_{k_3}^{+}}(s) \|}_{L^2} \lesssim \e_1^3 2^{-5m/4}.
\end{align*}
The bounds on $|L_2(\xi,s)|$ and $|L_3(\xi,s)|$ are similar.
Recalling also the bound $|(\partial_sL)(\xi,s) | \lesssim \e_1^2 2^{-m}$, see the definition \eqref{nf50},
it follows that the right-hand side of \eqref{sum1} is dominated by $C\e_1^3 2^{-m/10}$,
which completes the proof of the lemma.
\end{proof}

\subsection{Proof of \eqref{mainlemma2conc3}}\label{prooftech3}
We consider now the quartic remainder term $\mathcal{R}^{''}_{\geq 4}$ defined in \eqref{eqprof12}. Our main aim is to show the following $L^2$ estimates:
\begin{lemma}\label{lemmaestR}
For any $t\in[0,T^\prime]$ and $l \in \mathbb{Z}$ we have
\begin{align}
\label{estR}
\begin{split}
{\| P_l \mathcal{R}_{\geq 4}^{''} (t) \|}_{L^2} & \lesssim \e_1^4 \langle t \rangle^{-9/8+10p_0}, 
  \qquad {\| P_l S\mathcal{R}_{\geq 4}^{''} (t) \|}_{L^2} \lesssim \e_1^4 \langle t \rangle^{-1+20p_0} .
\end{split}
\end{align}
\end{lemma}
The desired conclusion \eqref{mainlemma2conc3} can then obtained using also the interpolation inequality in Lemma \ref{interpolation}; the precise argument is given after the proof of Lemma \ref{lemmaestR}.

To prove \eqref{estR} we first recall that from the a priori assumptions on $u$ and Lemma \ref{lemv} we have
the linear bounds
\begin{align}
\label{estuv}
\begin{split}
{\|P_lU(t)\|}_{L^2} + {\| P_lv(t)\|}_{L^2} & \lesssim \e_1 \min \big(2^{(1/2-p_0)l}, 2^{-(N_0-1/2)l}\big) \langle t \rangle^{p_0} ,
\\
{\|P_lU(t)\|}_{L^\infty} + {\|P_lv(t)\|}_{L^\infty} & \lesssim \e_1 \min \big( 2^{l/10}, 2^{-(N_2-1/2)l} \big) \langle t \rangle^{-1/2},
\\
{\|P_l S U(t) \|}_{L^2} + {\|P_l S v(t) \|}_{L^2} & \lesssim \e_1 \min\big( 2^{(1/2-p_0)l}, 2^{-(N_1-1/2)l} \big) \langle t \rangle^{4p_0} ,
\end{split}
\end{align}
and the quadratic bounds
\begin{align}
\label{estu-v}
\begin{split}
{\|P_l(U(t) - v(t))\|}_{L^2} & \lesssim \e_1^2 \min \big(2^{l/2}, 2^{-(N_2-1/2) l} \big) \langle t \rangle^{-1/2 + 6p_0} ,
\\
{\|P_l(U(t)-v(t))\|}_{L^\infty} & \lesssim \e_1^2 \min \big( 2^{l/2}, 2^{-(N_2-1/2) l} \big) \langle t \rangle^{-3/4 + 2p_0},
\\
{\|P_l S (U(t)-v(t)) \|}_{L^2} & \lesssim \e_1^2 \min \big( 2^{l/2}, 2^{-(N_1-1/2)l} \big) \langle t \rangle^{-1/4+6p_0} .
\end{split}
\end{align}

Using these we now prove a few more nonlinear estimates.
\begin{lemma}\label{lemnonlinear}
For any $t\in[0,T']$ and $l \in \mathbb{Z}$ we have the quadratic-type bounds
\begin{equation}
\label{estnonlinear1}
\begin{split}
{\| P_l\mathcal{Q}_U(t) \|}_{L^2} + {\|P_l\mathcal{Q}_v(t)\|}_{L^2} & \lesssim
  \e_1^2 \min \big( 2^{l/2}, 2^{-(N_0-3/2)l} \big) \langle t \rangle^{-1/2 + p_0} ,
\\
{\| P_l\mathcal{Q}_U(t) \|}_{L^\infty} + {\|P_l \mathcal{Q}_v(t) \|}_{L^\infty} & \lesssim
  \e_1^2 \min \big( 2^{l/2}, 2^{-3l} \big) \langle t \rangle^{-1} ,
\\
{\|P_l S\mathcal{Q}_U(t) \|}_{L^2} + {\|P_l S\mathcal{Q}_v(t)\|}_{L^2} & \lesssim
  \e_1^2 \min \big(2^{l/2} , 2^{-l} \big) \langle t \rangle^{-1/2 + 4p_0} ,
\end{split}
\end{equation}
and the cubic-type bounds
\begin{align}
\label{estnonlinear2}
\begin{split}
{\|P_l(\mathcal{Q}_U(t) - \mathcal{Q}_v(t)) \|}_{L^2} & \lesssim \e_1^3 \langle t \rangle^{-7/8+6p_0} \min(2^{l/2}, 2^{-3l}),
\\
{\|P_l(\mathcal{Q}_U(t) - \mathcal{Q}_v(t)) \|}_{L^\infty} & \lesssim \e_1^3 \langle t \rangle^{-5/4+2p_0} \min(2^{l/2}, 2^{-3l}),
\\
{\|P_lS(\mathcal{Q}_U(t) - \mathcal{Q}_v(t)) \|}_{L^2} & \lesssim \e_1^3 \langle t \rangle^{-3/4+8p_0} \min(2^{l/2}, 2^{-l}).
\end{split}
\end{align}
\end{lemma}

\begin{proof}
To obtain \eqref{estnonlinear1} is suffices to recall the definition of $\mathcal{Q}_U$ and $\mathcal{Q}_v$,
in \eqref{equN100} and \eqref{Q_v}, and use the symbol bounds \eqref{lemvsym},
\begin{align}
\label{boundq_0}
{\| q_{\eps_1\eps_2}^{k,k_1,k_2} \|}_{S^\infty}
  \lesssim \mathbf{1}_\mathcal{X}(k,k_1,k_2) 2^k 2^{\min(k,k_1,k_2)/2} 
\end{align}
the commutation identity \eqref{lemcomm}-\eqref{lemcomm2}, and the linear bounds \eqref{estuv} (see the proof of Lemma \ref{OtermsProd} for a similar argument).

We now prove the inequalities \eqref{estnonlinear2}. In view of the definitions \eqref{equN100} and \eqref{Q_v} we have
\begin{align}
\label{Q_u-Q_v}
\mathcal{Q}_U - \mathcal{Q}_v & ={\sum_\star}'
  \big[Q_{\eps_1\eps_2}(U_{\eps_1}, U_{\eps_2} - v_{\eps_2}) + Q_{\eps_1\eps_2}(U_{\eps_1}-v_{\eps_1}, v_{\eps_2})\big].
\end{align}

Using Lemma \ref{touse}(ii) with \eqref{boundq_0}, the $L^\infty$ estimate in \eqref{estuv} and
the first $L^2$ estimate in \eqref{estu-v}, we get
\begin{equation*}
\begin{split}
\|P_k&Q_{\eps_1\eps_2}(U_{\eps_1}, U_{\eps_2} - v_{\eps_2})(t)\|_{L^2}\lesssim \sum_{k_1,k_2\in\mathbb{Z}}{\| q_{\eps_1\eps_2}^{k,k_1,k_2} \|}_{S^\infty}{\| P_{k_1}^\prime U(t)\|}_{L^\infty} {\| P_{k_2}^\prime (U-v)(t) \|}_{L^2}\\
&\lesssim \e_1^3\langle t\rangle^{-1+6p_0}\sum_{k_1,k_2\in\mathbb{Z}}\mathbf{1}_\mathcal{X}(k,k_1,k_2) 2^k 2^{\min(k_1,k_2)/2}2^{-(N_2-1/2)k_1^+}2^{-(N_2-1/2)k_2^+}\\
&\lesssim \e_1^3\langle t\rangle^{-1+6p_0}\min(2^{k/2},2^{-(N_2-3/2)k}).
\end{split}
\end{equation*}

The first bound in \eqref{estnonlinear2} follows,
after estimating similarly the $L^2$ norm of the second term in \eqref{Q_u-Q_v}.
The $L^\infty$ bound in \eqref{estnonlinear2} can also be obtained similarly using
the identity \eqref{Q_u-Q_v} and the $L^\infty$ bounds in \eqref{estuv}-\eqref{estu-v}. For the last bound in \eqref{estnonlinear2} we first use \eqref{Q_u-Q_v} and \eqref{lemcomm}-\eqref{lemcomm2} (notice that the symbols $q_{\eps_1\eps_2}$ are homogeneous). Then we estimate the $L^2$ norm in the same way, placing $SU$ and $S(U-v)$ in $L^2$ and $U-v$ and $U$ in $L^\infty$.
\end{proof}

We are now ready to prove \eqref{estR}.

\begin{proof}[Proof of Lemma \ref{lemmaestR}]
We examine the formula \eqref{eqprof12} for $\mathcal{R}_{\geq 4}^{''}$, and begin by looking at the terms in
the second line:
\begin{align}
\label{prestR1}
\begin{split}
& M_{\eps_1\eps_2}(U_{\eps_1}, (\mathcal{Q}_U)_{\eps_2}) - M_{\eps_1\eps_2}(v_{\eps_1}, (\mathcal{Q}_v)_{\eps_2})
  \\
  & = M_{\eps_1\eps_2}(U_{\eps_1}-v_{\eps_1}, (\mathcal{Q}_U)_{\eps_2})
  + M_{\eps_1\eps_2}(v_{\eps_1}, (\mathcal{Q}_U)_{\eps_2} - (\mathcal{Q}_v)_{\eps_2}) .
\end{split}
\end{align}
We recall the bound \eqref{lemvsym},
\begin{align}
\label{boundmpmpm}
{\| m_{\eps_1\eps_2}^{k,k_1,k_2} \|}_{S^\infty} \lesssim
  2^{k/2} 2^{-\min(k_1,k_2)/2} \mathbf{1}_{\mathcal{X}}(k,k_1,k_2),
\end{align}
and remark that the difficulty in estimating the terms in \eqref{prestR1} comes from the
low frequency singularity in this estimate.
Using Lemma \ref{touse}(ii), \eqref{estu-v}, and \eqref{estnonlinear1} we have
\begin{align*}
\begin{split}
&\big\| P_l M_{\eps_1\eps_2}(U_{\eps_1}-v_{\eps_1}, (\mathcal{Q}_U)_{\eps_2}) (t)\big\|_{L^2}\lesssim I+II,\\
&I:=\sum_{k_1,k_2 \in \mathbb{Z}, \, 2^{\min(k_1,k_2)} \leq \langle t \rangle^{-2}} {\|m_{\eps_1\eps_2}^{l,k_1,k_2}\|}_{S^\infty}2^{\min(k_1,k_2)/2}
  {\| P_{k_1}^\prime(U-v)(t)\|}_{L^2} {\| P_{k_2}^\prime \mathcal{Q}_U(t) \|}_{L^2},\\
	&II:=\sum_{k_1,k_2 \in \mathbb{Z}, \, 2^{\min(k_1,k_2)} \geq \langle t \rangle^{-2}} {\|m_{\eps_1\eps_2}^{l,k_1,k_2}\|}_{S^\infty}
  {\| P_{k_1}^\prime(U-v)(t)\|}_{L^2} {\| P_{k_2}^\prime \mathcal{Q}_U(t) \|}_{L^\infty}.
\end{split}
\end{align*}
Then we estimate 
\begin{align*}
\begin{split}
I \lesssim \sum_{2^{\min(k_1,k_2)} \leq \langle t\rangle^{-2}}
  \e_1^2 \langle t\rangle^{-1/2+6p_0} \e_1^2 \langle t\rangle^{-1/2+p_0} 2^{\min(k_1,k_2)/2}2^{-3\max(k_1,k_2,0)}
  \lesssim \e_1^4 
  \langle t\rangle^{-3/2}
\end{split}
\end{align*}
and
\begin{align*}
\begin{split}
II \lesssim \sum_{2^{\min(k_1,k_2)} \geq \langle t\rangle^{-2}}
  \e_1^2 \langle t\rangle^{-1/2+6p_0} \e_1^2 \langle t\rangle^{-1}2^{-2\max(k_1,k_2,0)}
  \lesssim \e_1^4 
  \langle t\rangle^{-5/4} .
\end{split}
\end{align*}

Similarly, using \eqref{estuv} (only the $L^2$ estimate in the first line) and \eqref{estnonlinear2} instead of \eqref{estu-v} and \eqref{estnonlinear1}, we estimate
\begin{align*}
{\big\| P_l M_{\eps_1\eps_2} (v_{\eps_1}, (\mathcal{Q}_U)_{\eps_2} - (\mathcal{Q}_v)_{\eps_2} ) (t)\big\|}_{L^2} \lesssim \e_1^4 
  \langle t\rangle^{-5/4+10p_0}.
\end{align*}
We have obtained, for any $l\in\mathbb{Z}$ and $t\in[0,T']$,
\begin{align}
 \label{prestR5}
{\big\| P_l \big[ M_{\eps_1\eps_2}(U_{\eps_1}, (\mathcal{Q}_U)_{\eps_2})
  - M_{\eps_1\eps_2}(v_{\eps_1}, (\mathcal{Q}_v)_{\eps_2})\big] (t) \big\|}_{L^2} \lesssim \e_1^4 
  \langle t\rangle^{-9/8} .
\end{align}

Using \eqref{prestR1} and the commutation identity \eqref{lemcomm}-\eqref{lemcomm2} we compute
\begin{align}
\label{prestR2}
\begin{split}
& S \big( M_{\eps_1\eps_2}(U_{\eps_1}, (\mathcal{Q}_U)_{\eps_2}) - M_{\eps_1\eps_2}(v_{\eps_1}, (\mathcal{Q}_v)_{\eps_2}) \big)
  \\
  & = M_{\eps_1\eps_2}( S(U_{\eps_1}-v_{\eps_1}), (\mathcal{Q}_U)_{\eps_2})
  + M_{\eps_1\eps_2}(U_{\eps_1}-v_{\eps_1}, (S \mathcal{Q}_U)_{\eps_2})
  + \widetilde{M}_{\eps_1\eps_2}(U_{\eps_1}-v_{\eps_1}, (\mathcal{Q}_U)_{\eps_2})
\\
  & + M_{\eps_1\eps_2}(S v_{\eps_1}, (\mathcal{Q}_U - \mathcal{Q}_v)_{\eps_2})
  + M_{\eps_1\eps_2}(v_{\eps_1}, S ( \mathcal{Q}_U - \mathcal{Q}_v)_{\eps_2})
  + \widetilde{M}_{\eps_1\eps_2}(v_{\eps_1}, (\mathcal{Q}_U - \mathcal{Q}_v)_{\eps_2}) .
\end{split}
\end{align}
By homogeneity, the symbols $\widetilde{m}_{\eps_1\eps_2}$
satisfy the same bounds as the symbols $m_{\eps_1\eps_2}$. Therefore
\begin{align*}
{\big\| \widetilde{M}_{\eps_1\eps_2}(U_{\eps_1}, (\mathcal{Q}_U)_{\eps_2})(t)
  - \widetilde{M}_{\eps_1\eps_2}(v_{\eps_1}, (\mathcal{Q}_v)_{\eps_2})(t) \big\|}_{L^2} \lesssim \e_1^4 \langle t \rangle^{-9/8} .
\end{align*}

The terms containing the vector-field $S$ can also be estimated in the same way, using the bounds \eqref{estuv}--\eqref{estnonlinear2} and \eqref{boundmpmpm}. We always place the factor containing $S$ in $L^2$. For example, for the most difficult term, we can estimate 
\begin{equation*}
\begin{split}
&{\big\| P_lM_{\eps_1\eps_2} (v_{\eps_1}, S (\mathcal{Q}_U- \mathcal{Q}_v)_{\eps_2})(t) \big\|}_{L^2}\lesssim I_S+II_S,\\
&I_S=\sum_{2^{\min(k_1,k_2)}\leq \langle t\rangle^{-1/2}} 2^{l/2}\mathbf{1}_{\mathcal{X}}(l,k_1,k_2)
  {\| P_{k_1}^\prime v(t) \|}_{L^2} {\| P_{k_2}^\prime S (\mathcal{Q}_U - \mathcal{Q}_v)(t)\big\|}_{L^2},
\\
&II_S:=\sum_{2^{\min(k_1,k_2)}\geq \langle t\rangle^{-1/2}} 2^{l/2}2^{-\min(k_1,k_2)/2}\mathbf{1}_{\mathcal{X}}(l,k_1,k_2)
  {\| P_{k_1}^\prime v(t) \|}_{L^\infty} {\| P_{k_2}^\prime S (\mathcal{Q}_U - \mathcal{Q}_v)(t)\big\|}_{L^2}. 
\end{split}
\end{equation*}
Then we estimate 
\begin{align*}
\begin{split}
I_S \lesssim \sum_{2^{\min(k_1,k_2)} \leq \langle t\rangle^{-1/2}}
  \e_1 \langle t\rangle^{p_0} \e_1^3 \langle t\rangle^{-3/4+8p_0} 2^{\min(k_1,k_2)(1/2-p_0)}2^{-\max(k_1,k_2,0)/2}
  \lesssim \e_1^4 
  \langle t\rangle^{-1+20p_0}
\end{split}
\end{align*}
and
\begin{align*}
\begin{split}
II_S \lesssim \sum_{2^{\min(k_1,k_2)} \geq \langle t\rangle^{-1/2}}
  \e_1 \langle t\rangle^{-1/2} \e_1^3 \langle t\rangle^{-3/4+8p_0}2^{-\min(k_1,k_2)/2}2^{-\max(k_1,k_2,0)/2}
  \lesssim \e_1^4 
  \langle t\rangle^{-1+20p_0} .
\end{split}
\end{align*}
The bound \eqref{estR} for the terms in the second line of \eqref{eqprof12} follows, using also \eqref{prestR5} and \eqref{prestR2}.

The terms in the first line of \eqref{eqprof12} are similar to the terms in the second line, since the bounds satisfied by the symbols of the operators $M_{\eps_1\eps_2}$ are symmetric in $k_1$ and $k_2$. Next, we look at the terms in the third line of \eqref{eqprof12}:
\begin{align}
\label{prestR10}
\begin{split}
M_{\eps_1\eps_2\eps_3}&(U_{\eps_1}, U_{\eps_2}, U_{\eps_3})
  - M_{\eps_1\eps_2\eps_3}(v_{\eps_1},v_{\eps_2}, v_{\eps_3})= M_{\eps_1\eps_2\eps_3}\big( {(U-v)}_{\eps_1}, U_{\eps_2}, U_{\eps_3} \big)\\
  &+ M_{\eps_1\eps_2\eps_3}\big( v_{\eps_1}, {(U-v)}_{\eps_2}, U_{\eps_3} \big)
  + M_{\eps_1\eps_2\eps_3}\big( v_{\eps_1}, v_{\eps_2}, {(U-v)}_{\eps_3} \big) .
\end{split}
\end{align}
Recall that the symbols $m_{\eps_1\eps_2\eps_3}$ satisfy the strong (non-singular) bounds \eqref{equcubicsym}.
Using 
\eqref{estuv} and \eqref{estu-v},
one can estimate 
\begin{align*}
 {\big\| P_l \big[ M_{\eps_1\eps_2\eps_3}(U_{\eps_1}, U_{\eps_2}, U_{\eps_3})
  - M_{\eps_1\eps_2\eps_3}(v_{\eps_1},v_{\eps_2}, v_{\eps_3})\big](t) \big\|}_{L^2} \lesssim 
  \e_1^4 \langle t \rangle^{-9/8} .
\end{align*}

The analogue of \eqref{lemcomm} for symbols of three variables is
\begin{align}
 \label{lemcomm3}
\begin{split}
SM(f,g,h) & = M(Sf,g,h) + M(f,Sg,h) + M(f,g,Sh) + \widetilde{M}(f,g,h)
\\
\widetilde{m}(\xi,\eta,\sigma) & = -(\xi\partial_\xi +\eta \partial_\eta +\sigma\partial_\sigma) m(\xi,\eta,\sigma).
\end{split}
\end{align}
By homogeneity, the symbols $\widetilde{m}_{\eps_1\eps_2\eps_3}$ satisfy the same bounds \eqref{equcubicsym}.
Then, applying $S$ to the identity \eqref{prestR10} above, and using \eqref{estuv} and \eqref{estu-v},
we can obtain
\begin{align*}
 {\big\| P_lS \big[ M_{\eps_1\eps_2\eps_3}(U_{\eps_1}, U_{\eps_2}, U_{\eps_3})
  - M_{\eps_1\eps_2\eps_3}(v_{\eps_1},v_{\eps_2}, v_{\eps_3}) \big](t) \big\|}_{L^2} \lesssim \e_1^4 \langle t \rangle^{-9/8} .
\end{align*}

Since the term $\mathcal{R}_{\geq 4}$ in \eqref{eqprof12} already satisfies the desired bounds, see \eqref{R_4}, in view of \eqref{equRR4} it remains to estimate operators of the form
\begin{align}
\label{prestR20}
M_{\eps_1\eps_2}( F_{\eps_1}, U_{\eps_2}) \quad \mbox{and} \quad M_{\eps_1\eps_2}(U_{\eps_1}, F_{\eps_2}) ,
\qquad F\in |\partial_x|^{1/2}O_{3,-1} .
\end{align}
The $O_{3,\alpha}$ notation in \eqref{cubb} implies that we have the following estimate
\begin{align*}
\begin{split}
{\| P_l F(t) \|}_{L^2} & \lesssim \e_1^3 \min\big( 2^{l/2}, 2^{-(N_0-2)l} \big) \langle t \rangle^{-1+p_0} ,
\\
{\| P_l S F(t) \|}_{L^2} & \lesssim \e_1^3 \min\big( 2^{l/2}, 2^{-(N_1-2)l} \big) \langle t \rangle^{-1+4p_0} .
\end{split}
\end{align*}
Using these estimates, Lemma \ref{touse}(ii), the symbol bound \eqref{boundmpmpm}, and \eqref{estuv}, we can bound
\begin{equation*}
\begin{split}
\big\| P_l M_{\eps_1\eps_2}&(F_{\eps_1}, U_{\eps_2})(t) \big\|_{L^2}
\lesssim 
  \sum_{k_1,k_2 \in \mathbb{Z}, \, 2^{\min(k_1,k_2)} \leq \langle t \rangle^{-1/2}} 2^{\max(k_1,k_2)/2}
  {\| P_{k_1}^\prime F(t)\|}_{L^2} {\| P_{k_2}^\prime U(t) \|}_{L^2}
\\
&+ \sum_{k_1,k_2 \in \mathbb{Z}, \, 2^{\min(k_1,k_2)} \geq \langle t \rangle^{-1/2}} 2^{\max(k_1,k_2)/2} 2^{-\min(k_1,k_2)/2}
  {\| P_{k_1}^\prime F(t)\|}_{L^2} {\| P_{k_2}^\prime U(t) \|}_{L^\infty}\\
	&\lesssim \e_1^4 \langle t \rangle^{-9/8} .
\end{split}
\end{equation*}
The same estimate also holds for $P_l \widetilde{M}_{\eps_1\eps_2}(F_{\eps_1}, U_{\eps_2})$.
Similarly, we have
\begin{align*}
\begin{split}
{\big\| P_l M_{\eps_1\eps_2}(S F_{\eps_1}, U_{\eps_2})(t) \big\|}_{L^2}
  \lesssim \sum_{k_1,k_2 \in \mathbb{Z}} 2^{\max(k_1,k_2)/2}
  {\| P_{k_1}^\prime S F(t)\|}_{L^2} {\| P_{k_2}^\prime U(t) \|}_{L^2} \lesssim \e_1^4 \langle t \rangle^{-1+10p_0} ,
\end{split}
\end{align*}
\begin{align*}
\begin{split}
{\big\| P_l M_{\eps_1\eps_2}(F_{\eps_1}, S U_{\eps_2})(t) \big\|}_{L^2} 
  \lesssim \sum_{k_1,k_2 \in \mathbb{Z}} 2^{\max(k_1,k_2)/2}
  {\| P_{k_1}^\prime F(t)\|}_{L^2} {\| P_{k_2}^\prime S U(t) \|}_{L^2} \lesssim \e_1^4 \langle t \rangle^{-1+10p_0} .
\end{split}
\end{align*}
The estimates for the terms $M_{\eps_1\eps_2}(U_{\eps_1}, F_{\eps_2})$ are similar. This completes the proof of \eqref{estR}.
\end{proof}

We can now complete the proof of the estimate \eqref{mainlemma2conc3}.

\begin{proof}[Proof of \eqref{mainlemma2conc3}] 
Assume that $k\in[-80p_0m,20p_0m]$, $|\xi_0|\in [2^k,2^{k+1}]$, $m\geq 1$, $t_1\leq t_2\in[2^{m}-2,2^{m+1}]\cap[0,T']$. We would like to prove that
\begin{equation}\label{lmj20}
\Big|\varphi_k(\xi_0)\int_{t_1}^{t_2}e^{iL(\xi_0,s)}e^{-is|\xi_0|^{3/2}}\widehat{\mathcal{R}''_{\geq 4}}(\xi_0,s)\,ds\Big|\lesssim \eps_1^32^{-200p_0m}.
\end{equation}

Let
\begin{equation}\label{lmj21}
F(\xi):=\varphi_k(\xi)\int_{t_1}^{t_2}e^{iL(\xi_0,s)}e^{-is|\xi|^{3/2}}\widehat{\mathcal{R}''_{\geq 4}}(\xi,s)\,ds.
\end{equation}
In view of Lemma \ref{interpolation}, it suffices to prove that 
\begin{equation*}
2^{-k}\|F\|_{L^2}\big[2^k\|\partial F\|_{L^2}+\|F\|_{L^2}\big]\lesssim \eps_1^62^{-400p_0m}.
\end{equation*}
Since $\|F\|_{L^2}\lesssim \eps_1^42^{-m/8+10p_0m}$, see the first inequality in \eqref{estR}, it suffices to prove that
\begin{equation}\label{lmj22}
2^k\|\partial F\|_{L^2}\lesssim \eps_1^42^{k}2^{(1/8-500p_0)m}.
\end{equation}

To prove \eqref{lmj22} we write
\begin{equation*}
|\xi\partial_\xi F(\xi)|\leq |F_1(\xi)|+|F_2(\xi)|+|F_3(\xi)|,
\end{equation*}
where
\begin{equation*}
\begin{split}
&F_1(\xi):=\xi(\partial_\xi\varphi_k)(\xi)\int_{t_1}^{t_2}e^{iL(\xi_0,s)}\big[e^{-is|\xi|^{3/2}}\widehat{\mathcal{R}''_{\geq 4}}(\xi,s)\big]\,ds,\\
&F_2(\xi):=\varphi_k(\xi)\int_{t_1}^{t_2}e^{iL(\xi_0,s)}\big[\xi\partial_\xi-(3/2)s\partial_s\big]\big[e^{-is|\xi|^{3/2}}\widehat{\mathcal{R}''_{\geq 4}}(\xi,s)\big]\,ds,\\
&F_3(\xi):=\frac{3}{2}\varphi_k(\xi)\int_{t_1}^{t_2}e^{iL(\xi_0,s)}s\partial_s\big[e^{-is|\xi|^{3/2}}\widehat{\mathcal{R}''_{\geq 4}}(\xi,s)\big]\,ds.
\end{split}
\end{equation*}
Using \eqref{estR} and the commutation identity $\big[\big[\xi\partial_\xi-(3/2)s\partial_s\big], e^{-is|\xi|^{3/2}}\big]=0$, we have
\begin{equation*}
\|F_1\|_{L^2}+\|F_2\|_{L^2}\lesssim \eps_1^42^{30p_0m}.
\end{equation*}
Moreover, using integration by parts in $s$ and the bound $\big|\partial_s\big[e^{iL(\xi_0,s)}\big]\big|\lesssim 2^{-m}$, see the definition \eqref{nf50}, we can also estimate $\|F_3\|_{L^2}\lesssim \eps_1^42^{30p_0m}$. The desired bound \eqref{lmj22} follows.
\end{proof}

\medskip
\section{Modified scattering}\label{Sca1}

In this section we provide a precise description of the asymptotic behavior of solutions. Our main result is the following:

\begin{theorem}[Modified Scattering] \label{ScaThm}
Assume that $N_0,N_1,N_2, p_0, p_1$ are as in Theorem \ref{MainTheo}, and $(h,\phi)$ is the global solution of the system \eqref{CPW}. Let 
\begin{equation*}
U=|\partial_x|\phi-i|\partial_x|^{1/2}\phi\in C\big([0,\infty) : \dot{H}^{N_0,p_1-1/2}). 
\end{equation*}

\setlength{\leftmargini}{1.8em}
\begin{itemize}

\item[(i)] (Modified scattering in the Fourier space) Let
\begin{equation}\label{Sca1.5}
L'(\xi,t) = \frac{|\xi|^2}{24\pi}\int_0^t {|\what{U}(\xi,s)|}^2 \frac{1}{s+1}\,ds,
\end{equation}
compare with \eqref{nf50}. Then there is $w_\infty$ with $\|(1+|\xi|^{N_2})w_\infty(\xi)\|_{L^2_\xi}\lesssim \e_0$ such that
\begin{equation*}
e^{iL'(\xi,t)}e^{-it|\xi|^{3/2}}\widehat{U}(\xi,t)\,\,\,\text{ converges to }\,\,\, w_\infty(\xi)\,\,\,\text{ as }\,\,\, t\to\infty\,,
\end{equation*}
more precisely
\begin{equation}\label{Sca1.6}
(1+t)^{p_0/2}\big\|(|\xi|^{-1/4}+|\xi|^{N_2})[e^{iL'(\xi,t)}e^{-it|\xi|^{3/2}}\widehat{U}(\xi,t)-w_\infty(\xi)]\big\|_{L^2_\xi}\lesssim\e_0.
\end{equation}

\item[(ii)] (Modified scattering in the physical space) There exists a unique asymptotic profile $f_\infty$, with $\| (|\xi|^{-1/10} + |\xi|^4)  f_\infty \|_{L^\infty_\xi} \lesssim \e_0$,
such that for all $t \geq 1$
\begin{equation}\label{proasymconc}
\begin{split}
\Big\| U(x,t)
  - \frac{e^{- i t (4/27) |x/t|^3 }}{\sqrt{1+t}}  f_\infty (x/t)  e^{-id_0
  |x/t|^3 {| f_\infty (x/t) |}^2 \log(1+t)} \Big\|_{L^\infty_x}\lesssim \e_0 \langle t\rangle^{-1/2-p_0/2} \, ,
\end{split}
\end{equation}
where $d_0=1/54$.

\end{itemize}

\end{theorem}

We remark that the first statement of modified scattering, in the Fourier space, is stronger because of the stronger norm of convergence. 
One can, in fact, make this convergence even stronger, by changing the norm in \eqref{Sca1.6} to an $L^\infty_\xi$ based norm. 

Modified scattering in the physical space follows by an argument similar to the one used by Hayashi--Naumkin \cite{HN} (see also \cite{IoPunote}).  

To prove Theorem \ref{ScaThm} we need a more precise scale invariant linear dispersive estimate:

\begin{lemma}\label{ScaDecay}

\setlength{\leftmargini}{1.8em}
\begin{itemize}

\item[(i)] For any $t\in\mathbb{R}\setminus\{0\}$, $k\in\mathbb{Z}$, and $f\in L^2(\mathbb{R})$ we have
\begin{equation}
\label{Sca2}
 \|e^{i t \Lambda}P_kf\|_{L^\infty}\lesssim |t|^{-1/2}2^{k/4}\|\widehat{f}\|_{L^\infty}
  + |t|^{-3/5}2^{-2k/5}\big[\|\widehat{f}\|_{L^2}+2^k\|\partial \widehat{f}\|_{L^2}\big]
\end{equation}
and
\begin{equation}\label{Sca3}
 \|e^{i t \Lambda}P_kf\|_{L^\infty}\lesssim |t|^{-1/2}2^{k/4}\|f\|_{L^1}.
\end{equation}

\item[(ii)] Moreover, if $\xi_0 := -4x^2/(9t^2) \,\mathrm{sgn}(x/t)$ and $\widehat{f}$ is continuous then
\begin{align}
\label{disperse2}
 \begin{split}
\Big| e^{i t\Lambda} f (x) - \sqrt{\frac{2i}{3\pi t}} e^{- it (4/27) |x/t|^3 } {|\xi_0|}^{1/4} \what{f} (\xi_0 ) \Big|
    \lesssim |t|^{-3/5}
    \big({\| |\xi|^{-2/5} \what{f} \|}_{L^2}+{\| |\xi|^{3/5} \partial\what{f} \|}_{L^2}  \big).
 \end{split}
\end{align}

\end{itemize}

\end{lemma}

\begin{proof} (i) This is similar to the proof of Lemma 2.3 in \cite{IoPu1}. The bound \eqref{Sca3} is a standard dispersive estimate. Also, \eqref{Sca2} is a consequence of \eqref{disperse2}, so we focus on this last bound.

We start from the identity
\begin{equation*}
e^{i t\Lambda} f (x)=\frac{1}{2\pi}\int_{\mathbb{R}}e^{i(t|\xi|^{3/2}+x\xi)}\widehat{f}(\xi)\,d\xi.
\end{equation*}
We would like to use the standard stationary phase approximate formula
\begin{equation}\label{Sca4}
\int_{\mathbb{R}}e^{i\Phi(x)}\Psi(x)\,dx\approx e^{i\Phi(x_0)}\Psi(x_0)\sqrt\frac{2\pi}{-i\Phi''(x_0)} + \mathrm{error},
\end{equation}
in a neighborhood of a stationary point $x_0$. To justify this we make the change of variables
\begin{equation*}
\xi=b\eta,\qquad b:= -\frac{4x^2}{9t^2}\,\mathrm{sgn}(x/t),\qquad t' := t\cdot \frac{4|x|^3}{9|t|^3}.
\end{equation*}
and notice that
\begin{equation*}
t|\xi|^{3/2}+x\xi=t'(2|\eta|^{3/2}/3-\eta).
\end{equation*}
Therefore, letting $\widehat{f_b}(\eta):=\widehat{f}(b\eta)$, we have
\begin{equation}\label{Sca8}
e^{i t\Lambda} f (x)=\frac{1}{2\pi}\frac{4x^2}{9t^2}\int_{\mathbb{R}}e^{it'(2|\eta|^{3/2}/3-\eta)}\widehat{f_b}(\eta)\,d\eta.
\end{equation}

Notice that
\begin{equation}\label{Sca9}
\||\eta|^{-2/5}\widehat{f_b}(\eta)\|_{L^2_\eta}+\||\eta|^{3/5}(\partial\widehat{f_b})(\eta)\|_{L^2_\eta}=|b|^{-1/10}\big(\||\xi|^{-2/5}\widehat{f}(\xi)\|_{L^2_\xi}+\||\xi|^{3/5}(\partial\widehat{f})(\xi)\|_{L^2_\xi}\big).
\end{equation}
For \eqref{disperse2} it suffices to prove that
\begin{equation}\label{Sca10}
\Big|\int_{\mathbb{R}}e^{it'(2|\eta|^{3/2}/3-\eta)}\widehat{f_b}(\eta)\,d\eta-e^{-it'/3}\widehat{f_b}(1)\sqrt{\frac{4\pi i}{t'}}\Big|\lesssim |t'|^{-3/5}
    \big({\| |\eta|^{-2/5} \what{f_b} \|}_{L^2}+{\| |\eta|^{3/5} \partial\what{f_b} \|}_{L^2}  \big).
\end{equation}

Let $g:=\widehat{f_b}$ and $g_k:=g\cdot\varphi_k$, $k\in\mathbb{Z}$. Let 
\begin{equation*}
a_k:=2^{-2k/5}\|g_k\|_{L^2}+2^{3k/5}\|\partial g_k\|_{L^2}.
\end{equation*}
Let
\begin{equation}\label{Sca11}
\Psi(\eta):=2|\eta|^{3/2}/3-\eta,\qquad \Psi'(\eta)=\eta|\eta|^{-1/2}-1,\qquad \Psi''(\eta)=|\eta|^{-1/2}/2.
\end{equation}
Clearly, for any $k\in\mathbb{Z}$,
\begin{equation*}
\Big|\int_{\mathbb{R}}e^{it'\Psi(\eta)}g_k(\eta)\,d\eta\Big|\lesssim \|g_k\|_{L^1}\lesssim 2^{9k/10}a_k.
\end{equation*}
Moreover, if $|k|\geq 3$ then we can integrate by parts in $\eta$ to estimate
\begin{equation*}
\Big|\int_{\mathbb{R}}e^{it'\Psi(\eta)}g_k(\eta)\,d\eta\Big|\lesssim \frac{\|g_k\|_{L^1}}{|t'|2^{3k/2}}+\frac{\|\partial g_k\|_{L^1}}{|t'|2^{k/2}}\lesssim \frac{a_k2^{-3k/5}}{|t'|}.
\end{equation*}
These last two inequalities show that
\begin{equation*}
\sum_{|k|\geq 3}\Big|\int_{\mathbb{R}}e^{it'\Psi(\eta)}g_k(\eta)\,d\eta\Big|\lesssim |t'|^{-3/5}\sup_{k\in\mathbb{Z}}a_k.
\end{equation*}
One can also estimate, in the same way, the contribution of the negative frequencies of the functions $g_k$, $|k|\leq 2$. For \eqref{Sca10} it remains to show that
\begin{equation}\label{Sca15}
\Big|\int_{\mathbb{R}}e^{it'\Psi(\eta)}g(\eta)\,d\eta-e^{-it'/3}g(1)\sqrt{\frac{4\pi i}{t'}}\Big|\lesssim |t'|^{-3/5}
    \big({\| g\|}_{L^2}+{\| \partial g\|}_{L^2}  \big).
\end{equation}
for any continuous function $g$ supported in $[1/10,10]$.

In proving \eqref{Sca15} we may assume that $|t'|$ is large, say $|t'|\geq 2^{10}$. Let $\chi(\eta)=\varphi_{\leq -10}(\eta-1)$ denote a cutoff function supported around 1. For \eqref{Sca15} it suffices to prove that
\begin{equation}\label{Sca16}
\Big|\int_{\mathbb{R}}e^{it'\Psi(\eta)}[g(\eta)-g(1)\chi_1(\eta)]\,d\eta\Big|\lesssim |t'|^{-3/5}
    \big({\| g\|}_{L^2}+{\| \partial g\|}_{L^2}  \big)
\end{equation}
and, using also Lemma \ref{interpolation},
\begin{equation}\label{Sca17}
\Big|\int_{\mathbb{R}}e^{it'\Psi(\eta)}\chi_1(\eta)\,d\eta-e^{-it'/3}\sqrt{\frac{4\pi i}{t'}}\Big|\lesssim |t'|^{-3/5}.
\end{equation}

The bound \eqref{Sca17} follows easily by stationary phase and \eqref{Sca11} (compare with \eqref{Sca4}). To prove \eqref{Sca16} we define
\begin{equation*}
J_l:=\int_{\mathbb{R}}e^{it'\Psi(\eta)}\varphi_l(\eta-1)[g(\eta)-g(1)\chi_1(\eta)]\,d\eta,
\end{equation*}
for $l\leq 10$. By linearity we may assume that ${\| g\|}_{L^2}+{\| \partial g\|}_{L^2}=1$. Notice that 
\begin{equation}\label{Sca18}
\text{ if }\quad|\eta-1|\lesssim 2^l\quad\text{ then }\quad|g(\eta)-g(1)\chi_1(\eta)|\lesssim |\eta-1|^{1/2}\lesssim 2^{l/2}.
\end{equation}
Therefore $|J_l|\lesssim 2^{3l/2}$. Moreover, if $2^l\geq |t'|^{-2/5}$ then we integrate by parts in $\eta$ and estimate
\begin{equation*}
\begin{split}
|J_l|&\lesssim \int_{\mathbb{R}}\Big|\frac{d}{d\eta}\frac{\varphi_l(\eta-1)[g(\eta)-g(1)\chi_1(\eta)]}{|t'|\Psi'(\eta)}\Big|\,d\eta\\
&\lesssim \int_{|\eta-1|\leq 2^{l+1}}\frac{2^{-l}|g(\eta)-g(1)\chi_1(\eta)|+|g'(\eta)|+1}{2^{l}|t'|}\,d\eta.
\end{split}
\end{equation*}
Using \eqref{Sca18} it follows that $|J_l|\lesssim |t'|^{-1}2^{-l/2}$. The desired bound \eqref{Sca16} follows, which completes the proof.
\end{proof}

We can complete now the proof of Theorem \ref{ScaThm}:

\begin{proof}[Proof of Theorem \ref{ScaThm}] (i) Recall the formulas \eqref{nf50}, 
\begin{equation*}
L(\xi,t) = \frac{|\xi|^2}{24\pi}\int_0^t {|\what{v}(\xi,s)|}^2 \frac{1}{s+1}\,ds,\qquad g(\xi,t)=e^{iL(\xi,t)}e^{-it|\xi|^{3/2}}\widehat{v}(\xi,t).
\end{equation*}
Assume $1\leq t_1\leq t_2$, let $H(t):=e^{iL'(\xi,t)}e^{-it|\xi|^{3/2}}\widehat{U}(\xi,t)$, and estimate
\begin{equation}\label{Sca20}
\begin{split}
\big|[H(\xi,t_2)&-H(\xi,t_1)]-[g(\xi,t_2)-g(\xi,t_1)]e^{i(L'(\xi,t_1)-L(\xi,t_1))}\big|\\
=\big|&[e^{i(L'(\xi,t_2)-L'(\xi,t_1))}e^{-it_2|\xi|^{3/2}}\widehat{U}(\xi,t_2)-e^{-it_1|\xi|^{3/2}}\widehat{U}(\xi,t_1)]\\
&-[e^{i(L(\xi,t_2)-L(\xi,t_1))}e^{-it_2|\xi|^{3/2}}\widehat{v}(\xi,t_2)-e^{-it_1|\xi|^{3/2}}\widehat{v}(\xi,t_1)]\big|\\
\lesssim |&\widehat{U}(\xi,t_2)-\widehat{v}(\xi,t_2)|+|\widehat{U}(\xi,t_1)-\widehat{v}(\xi,t_1)|\\
+\,|&\widehat{v}(\xi,t_2)||(L'(\xi,t_2)-L'(\xi,t_1))-(L(\xi,t_2)-L(\xi,t_1))|.
\end{split}
\end{equation}

It follows from Lemma \ref{mainlem} that $(|\xi|^{1/10}+|\xi|^{N_2+1/2})|\widehat{v}(\xi,t)|\lesssim \e_0$ uniformly in $t$. Therefore
\begin{equation*}
|\widehat{v}(\xi,t_2)||(L'(\xi,t_2)-L'(\xi,t_1))-(L(\xi,t_2)-L(\xi,t_1))|\lesssim \int_{t_1}^{t_2}|\what{v}(\xi,s)-\widehat{U}(\xi,s)| \frac{1}{s+1}\,ds.
\end{equation*}
Using \eqref{lemvL2} and \eqref{Sca20} it follows that
\begin{equation*}
\big\|(|\xi|^{-1/4}+|\xi|^{N_2})\{[H(\xi,t_2)-H(\xi,t_1)]-[g(\xi,t_2)-g(\xi,t_1)]e^{i(L'(\xi,t_1)-L(\xi,t_1))}\}\big\|_{L^2}\lesssim \e_0\langle t_1\rangle^{-1/6}.
\end{equation*} 
Using now Lemma \ref{mainlem} and the bound \eqref{profL2} it follows that
\begin{equation*}
\big\|(|\xi|^{-1/4}+|\xi|^{N_2})[H(\xi,t_2)-H(\xi,t_1)]\big\|_{L^2}\lesssim \e_0\langle t_1\rangle^{-2p_0/3},
\end{equation*}
and the desired conclusion \eqref{Sca1.6} follows.

(ii) In view of \eqref{estu-v},
\begin{equation}\label{Sca30}
\|U(t)-v(t)\|_{L^\infty}\lesssim \eps_0(1+t)^{-2/3}.
\end{equation}
Moreover, since $\widehat{v}(\xi,t)=\widehat{f}(\xi,t)e^{it|\xi|^{3/2}}=e^{it|\xi|^{3/2}}e^{-iL(\xi,t)}g(\xi,t)$, and using also \eqref{profL2}--\eqref{profS} and \eqref{disperse2}, it follows that, for any $x\in\mathbb{R}$ and $t\geq 1$,
\begin{equation}\label{Sca30.5}
\Big|\widehat{v}(x,t)- \sqrt{\frac{2i}{3\pi t}} e^{- it (4/27) |x/t|^3 }  e^{-iL(\xi_0,t)}{|\xi_0|}^{1/4}g (\xi_0,t) \Big|
    \lesssim \langle t\rangle^{-1/2-10p_0},
\end{equation}
where $\xi_0 := -4x^2/(9t^2) \,\mathrm{sgn}(x/t)$. 

Let
\begin{equation}\label{Sca31}
g_\infty(\xi):=\lim_{t\to\infty}{|\xi|}^{1/4}g (\xi,t),
\end{equation}
where the limit exists due to Lemma \ref{mainlem}; more precisely for any $t\geq 1$
\begin{equation}\label{Sca32}
\Big\|g_\infty(\xi)-{|\xi|}^{1/4}g (\xi,t)\Big\|_{L^\infty_\xi}\lesssim \e_0(1+t)^{-p_0},\qquad \|{|\xi|}^{1/4}g (\xi,t)\|_{L^\infty_\xi}\lesssim \e_0.
\end{equation}
Then, for any $\xi\in\mathbb{R}$,
\begin{equation*}
L(\xi,t)-\frac{|\xi|^{3/2}}{24\pi}|g_\infty(\xi)|^2\ln(t+1)=\frac{|\xi|^{3/2}}{24\pi}\int_0^t\big[{|\xi|}^{1/2}|g (\xi,t)|^2-|g_\infty(\xi)|^2\big]\frac{1}{s+1}\,ds
\end{equation*}
Using \eqref{Sca2} it follows that there is a real-valued function $A_{\infty}$ such that
\begin{equation}\label{Sca33}
\Big\|L(\xi,t)-\frac{|\xi|^{3/2}}{24\pi}|g_\infty(\xi)|^2\ln(t+1)-A_{\infty}(\xi)\Big\|_{L^\infty_\xi}\lesssim \e_0(1+t)^{-2p_0/3}.
\end{equation}
Using \eqref{Sca30.5}, \eqref{Sca32}, and \eqref{Sca33}, it follows that
\begin{equation*}
\Big|\widehat{v}(x,t)- \sqrt{\frac{2i}{3\pi t}} e^{- it (4/27) |x/t|^3 }  e^{i[\frac{|\xi_0|^{3/2}}{24\pi}|g_\infty(\xi_0)|^2\ln(t+1)+A_{\infty}(\xi_0)]}g_{\infty}(\xi_0) \Big|
    \lesssim \langle t\rangle^{-1/2-2p_0/3}.
\end{equation*}
Then we define 
\begin{equation*}
f_{\infty}(y):=\sqrt{\frac{2i}{3\pi}}g_{\infty}(\xi(y)) e^{iA_{\infty}(\xi(y))}, \qquad \xi(y) := -\frac{4}{9}y^2\mathrm{sgn}(y).
\end{equation*}
The desired conclusion \eqref{proasymconc} follows using also \eqref{Sca30}.
\end{proof}

\appendix

\medskip
\section{Analysis of symbols}\label{appsym}

\subsection{Notation}
Recall the definition of the class of symbols
\begin{equation}
\label{Sinfty1}
S^\infty := \{m: \R^d \to \mathbb{C} : \,m \text{ continuous and } {\| m \|}_{S^\infty} := {\|\mathcal{F}^{-1}(m)\|}_{L^1} < \infty \} ,
\end{equation}
and the notation
\begin{equation}
\label{defm^k}
m^{k,k_1,k_2}(\xi,\eta) := m(\xi,\eta) \varphi_k(\xi) \varphi_{k_1}(\xi-\eta) \varphi_{k_2}(\eta).
\end{equation}
Recall that $\mathcal{X} = \{(k,k_1,k_2) \in \mathbb{Z}^3 : \max(k,k_1,k_2) - \mathrm{med}(k,k_1,k_2) \leq 6\}$. Recall also the notation
\begin{align}
\label{Onot1}
m(\xi,\eta) = O\big( f(|\xi|,|\xi-\eta|,|\eta|) \big)
\, \Longleftrightarrow \, {\| m^{k,k_1,k_2}(\xi,\eta) \|}_{S^\infty} \lesssim f(2^k,2^{k_1},2^{k_2}) \mathbf{1}_{\mathcal{X}}(k,k_1,k_2).
\end{align}

\subsection{Quadratic symbols} In the following lemma we collect several estimates on the symbols $a_{\pm\pm}$ that are used throughout the paper.

\begin{lemma}\label{lembounda}
Let $\eps_1,\eps_2 \in \{+,-\}$, and $a_{\eps_1\eps_2}$ be the symbols in \eqref{a_++}--\eqref{a_--}, then
\begin{align}
\label{bounda0}
\begin{split}
{\| a_{\eps_1\eps_2}^{k,k_1,k_2} \|}_{S^\infty} & \lesssim 2^{3k_1/2}
  \mathbf{1}_{\mathcal{X}}(k,k_1,k_2) \mathbf{1}_{[6,\infty)}(k_2-k_1)
  \\
& + 2^{k_1/2} 2^k \mathbf{1}_{(-\infty,1]}(k) \mathbf{1}_{\mathcal{X}}(k,k_1,k_2) \mathbf{1}_{[6,\infty)}(k_2-k_1) .
\end{split}
\end{align}
Moreover, for $\eps_1 \in \{+,-\}$, the following bounds holds
\begin{align}
\label{boundaeps_1+}
& {\| a_{\eps_1+}^{k,k_1,k_2} \|}_{S^\infty} \lesssim 2^{3k_1/2}
  \mathbf{1}_{\mathcal{X}}(k,k_1,k_2) \mathbf{1}_{[6,\infty)}(k_2-k_1) ,
\\ \label{boundaeps_1->}
& {\| a_{\eps_1-}^{k,k_1,k_2} \|}_{S^\infty} \lesssim
  \big( 2^{3k_1/2} \mathbf{1}_{[2,\infty)}(k)+ 2^{k_1/2} 2^k \mathbf{1}_{(-\infty,1]}(k) \big)
  \mathbf{1}_{\mathcal{X}}(k,k_1,k_2) \mathbf{1}_{[6,\infty)}(k_2-k_1) .
\end{align}
As a consequence, we have
\begin{align}
\label{bounda/1}
\begin{split}
{\Big\| \frac{ a_{\eps_1\eps_2}^{k,k_1,k_2}(\xi,\eta)  }{|\xi|^{3/2} -\eps_1|\xi-\eta|^{3/2} - |\eta|^{3/2}} \Big\|}_{S^\infty}
  & \lesssim 2^{k_1/2} 2^{-k/2} \mathbf{1}_{\mathcal{X}}(k,k_1,k_2) \mathbf{1}_{[6,\infty)}(k_2 - k_1) ,
\\
{\Big\| \frac{ a_{\epsilon_1 -}^{k,k_1,k_2}(\xi,\eta)}{
  |\xi|^{3/2} -\eps_1|\xi-\eta|^{3/2} + |\eta|^{3/2}} \Big\|}_{S^\infty}
  & \lesssim \big( 2^{3k_1/2} 2^{-3k/2} \mathbf{1}_{[2,\infty)}(k) + 2^{k_1/2} 2^{-k/2} \mathbf{1}_{(-\infty,1]}(k) \big)
  \\ & \qquad \times \mathbf{1}_{\mathcal{X}}(k,k_1,k_2) \mathbf{1}_{[6,\infty)}(k_2 - k_1) .
\end{split}
\end{align}

Furthermore, for $\e_1 \in \{+,-\}$ let us define the symbol
\begin{align}
\label{adiff}
\alpha_{\eps_1 +}(\xi,\eta,\rho) := a_{\eps_1 +}(\xi,\rho) - a_{\eps_1 +}(\xi+\eta-\rho,\eta) .
\end{align}
Then we have the following: if $k_3 \leq k_1-4$, $k_4\leq k_2-4$, and $k_1 \geq 8$, then
\begin{align}
\begin{split}
\label{boundadiff}
{\| \alpha_{\eps_1 +}(\xi,\eta,\rho)\cdot\varphi_{k_1}(\xi)\varphi_{k_2}(\eta)\varphi_{k_3}(\rho-\xi)\varphi_{k_4}(\rho-\eta)\|}_{S^\infty} \lesssim
  (2^{5k_3/2} + 2^{5k_4/2}) {(2^{k_1} + 2^{k_2})}^{-1}.
\end{split}
\end{align}
\end{lemma}


\begin{proof}
We begin by recalling that the cutoff $\chi$, see \eqref{Tab}, is supported on a region where $2^3|\xi-\eta| \leq |\eta|$.
Using integration by parts, one can verify that 
\begin{align}
\label{lema1}
\begin{split}
&  \frac{\xi (\xi-\eta)}{|\xi-\eta|^{1/2}} \Big(\frac{|\xi|}{|\eta|} - 1 \Big) \chi(\xi-\eta,\eta)
  = O \big( {|\xi-\eta|}^{3/2} \mathbf{1}_{[2^8,\infty)} (|\eta|/|\xi-\eta|)  \big) ,
\\
& \frac{\eta (\xi-\eta)}{|\xi-\eta|^{1/2}} \Big(1 - \frac{|\xi|^{1/2}}{|\eta|^{1/2}} \Big) \chi(\xi-\eta,\eta)
  = O \big( |\xi-\eta|^{3/2} \mathbf{1}_{[2^8,\infty)} (|\eta|/|\xi-\eta|)  \big) ,
\\
& \frac{|\xi-\eta|^2 |\xi|^{1/2}}{|\eta|} \chi(\xi-\eta,\eta)
  = O \big( {|\xi-\eta|}^2 |\eta|^{-1/2} \mathbf{1}_{[2^8,\infty)} (|\eta|/|\xi-\eta|)  \big) ,
\\
&  \frac{|\xi|^2-|\xi|^{1/2}|\eta|^{3/2}}{|\eta|} |\xi-\eta|^{1/2} \chi(\xi-\eta,\eta)
  = O \big( |\xi-\eta|^{3/2} \mathbf{1}_{[2^8,\infty)} (|\eta|/|\xi-\eta|)  \big)
\\
&  \frac{|\xi|^2 + |\xi|^{1/2}|\eta|^{3/2}}{|\eta|} |\xi-\eta|^{1/2} \varphi_{\leq 0}(\eta) \chi(\xi-\eta,\eta)
  = O \big( |\xi| |\xi-\eta|^{1/2} \mathbf{1}_{[2^8,\infty)} (|\eta|/|\xi-\eta|)  \mathbf{1}_{(0,2^3]} (|\xi|) \big) ,
\end{split}
\end{align}
where we are using the notation \eqref{Sinfty1}-\eqref{Onot1}.
Since the bound for $|\xi-\eta|^{3/2}$ is obvious,
using \eqref{lema1}, and inspecting the formulas \eqref{a_++}-\eqref{a_--}, one immediately obtains \eqref{bounda0}.
The bound \eqref{boundaeps_1+} follows from the first four identities in \eqref{lema1}.
\eqref{boundaeps_1->} follows directly from \eqref{bounda0}.

To prove \eqref{bounda/1} we notice first that
\begin{equation}
\label{phiab}
(a+b)^{3/2} - b^{3/2} - a^{3/2} \in [ab^{1/2}/4, 4ab^{1/2}] \qquad \text{if} \quad  0 \leq a \leq b.
\end{equation}
Therefore, using standard integration by parts, we see that
\begin{align}
\label{1/phi0}
& {\Big\| \frac{\varphi_k(\xi)\varphi_{k_1}(\xi-\eta)\varphi_{k_2}(\eta)}{|\xi|^{3/2} -|\xi-\eta|^{3/2} -|\eta|^{3/2}}
   \Big\|}_{S^\infty}
  \lesssim \frac{1}{2^{\min(k_1,k_2)} 2^{\max(k,k_2)/2}} , 
\end{align}
for all $k,k_1,k_2 \in \mathbb{Z}$.
In particular, whenever $k \geq k_1 + 3$, we have
\begin{align}
\label{1/phi1}
& {\Big\| \frac{\varphi_k(\xi)\varphi_{k_1}(\xi-\eta)\varphi_{k_2}(\eta)}{|\xi|^{3/2} -\eps_1|\xi-\eta|^{3/2} -|\eta|^{3/2}}
   \Big\|}_{S^\infty}
  \lesssim \frac{1}{2^{k_1} 2^{\max(k,k_2)/2}} , 
\\
\label{1/phi2}
& {\Big\| \frac{\varphi_k(\xi)\varphi_{k_1}(\xi-\eta)\varphi_{k_2}(\eta)}{|\xi|^{3/2} -\eps_1|\xi-\eta|^{3/2} + |\eta|^{3/2}}
  \Big\|}_{S^\infty}
  \lesssim \frac{1}{2^{3\max(k,k_2)/2}} .
\end{align}
Thus, \eqref{boundaeps_1+} and \eqref{1/phi1} give the first inequality in \eqref{bounda/1}. 
The second inequality in \eqref{bounda/1} is a consequence of \eqref{boundaeps_1->} and \eqref{1/phi2}.

To prove \eqref{boundadiff} we write down the four components of the symbols,
\begin{align}
\label{lema5}
\begin{split}
a_1(\xi,\eta) & := \frac{\xi (\xi-\eta)}{|\xi-\eta|^{1/2}} \Big(\frac{|\xi|}{|\eta|} - 1 \Big) ,
\qquad
a_2(\xi,\eta) := |\xi-\eta|^{3/2} ,
\\
a_3(\xi,\eta) & := \frac{\eta (\xi-\eta)}{|\xi-\eta|^{1/2}} \Big(1 - \frac{|\xi|^{1/2}}{|\eta|^{1/2}} \Big) ,
\qquad
a_4(\xi,\eta) := \frac{|\xi-\eta|^2 |\xi|^{1/2}}{|\eta|} .
\end{split}
\end{align}
Notice that the last component in the formulas \eqref{a_++} and \eqref{a_-+} has been disregarded, since
we are only interested in the case $k_1 \geq 5$ in \eqref{boundadiff}.
It suffices to prove that
\begin{align}
\label{lema7}
\begin{split}
{\| \alpha_j(\xi,\eta,\rho)\cdot \varphi_{k_1}(\xi)\varphi_{k_2}(\eta)\varphi_{k_3}(\rho-\xi)\varphi_{k_4}(\rho-\eta)\|}_{S^\infty} \lesssim
  (2^{5k_3/2} + 2^{5k_4/2}) {(2^{k_1} + 2^{k_2})}^{-1},
\\
\mbox{for} \quad \alpha_j(\xi,\eta,\rho) := a_j(\xi,\rho) - a_j(\xi+\eta-\rho,\eta) , \qquad j=1,\dots,4 ,
\end{split}
\end{align}
whenever $k_3 \leq k_1-4$ and $k_4\leq k_2-4$. By Taylor expansion, one easily sees that for $4|\xi-\eta| \leq |\eta|$
\begin{align*}
a_1(\xi,\eta) = |\xi-\eta|^{3/2}+ O \big( |\xi-\eta|^{5/2} |\eta|^{-1} \big), \quad a_2(\xi,\eta) = |\xi-\eta|^{3/2}\\
a_3(\xi,\eta) = -\frac{1}{2} |\xi-\eta|^{3/2} + O \big( |\xi-\eta|^{5/2} |\eta|^{-1} \big),\quad a_4(\xi,\eta) = |\xi-\eta|^{2} |\eta|^{-1/2} + O \big( |\xi-\eta|^{3} |\eta|^{-3/2} \big).
\end{align*}
The desired conclusions in \eqref{lema7} follow.
\end{proof}

We consider now the symbols $b_{\pm\pm}$. Recall that $\widetilde{\chi}(x,y) = 1 - \chi(x,y) - \chi(y,x)$.

\begin{lemma}\label{lemboundb}
With $b_{\eps_1\eps_2}$ as \eqref{b++}--\eqref{b--} and \eqref{m_2q_2}, we have
\begin{align}
\label{boundb1}
& {\| b_{\eps_1\eps_2}^{k,k_1,k_2}  \|}_{S^\infty}
  \lesssim 2^{3k/2} \mathbf{1}_{\mathcal{X}}(k,k_1,k_2) \mathbf{1}_{[-15,15]}(k_1 - k_2) ,
\end{align}
and
\begin{align}
\label{boundb/1}
& {\Big\| \frac{ b_{\eps_1\eps_2}^{k,k_1,k_2} }{ |\xi|^{3/2} - \eps_1|\xi-\eta|^{3/2} - \eps_2|\eta|^{3/2} } \Big\|}_{S^\infty}
  \lesssim 2^{k/2} 2^{-k_1/2} \mathbf{1}_{\mathcal{X}}(k,k_1,k_2) \mathbf{1}_{[-15,15]}(k_1 - k_2) ,
\end{align}
\end{lemma}


\begin{proof}
Inspecting the formula \eqref{m_2q_2} it is easy to see that
\begin{align*}
\widetilde{m}_2(\xi,\eta) & = O\big( |\xi| |\xi-\eta|
  \mathbf{1}_{[2^{-13},2^{13}]}(|\xi-\eta|/|\eta|) \big),
\end{align*}
and therefore
\begin{align*}
\frac{|\xi| \widetilde{m}_2(\xi,\eta)}{|\xi-\eta|^{1/2} |\eta|}
  & = O\big( |\xi|^{3/2} \mathbf{1}_{[2^{-13},2^{13}]}(|\xi-\eta|/|\eta|) \big),
\end{align*}
which is consistent with the bound in \eqref{boundb1}. Similarly
\begin{align*}
\frac{|\xi|^{1/2} \widetilde{q}_2(\xi,\eta)}{|\xi-\eta|^{1/2} |\eta|^{1/2}}
& = O\big( |\xi|^{3/2} \mathbf{1}_{[2^{-13},2^{13}]}(|\xi-\eta|/|\eta|) \big) .
\end{align*}
The desired conclusion \eqref{boundb1} follows from these bounds and the formulas \eqref{b++}-\eqref{b--}.
The bounds in \eqref{boundb/1} follow from \eqref{boundb1} and \eqref{1/phi0}.
\end{proof}

\section{The Dirichlet-Neumann operator}\label{secDN}

Recall the spaces $\mathcal{C}_0$, $\dot{H}^{N,b}$, $\dot{W}^{N,b}$,
 and $\widetilde{W}^N$ defined in \eqref{normC} and \eqref{norms0}. Assume in this section that $h\in C\big([0,T]:\mathcal{C}_0\cap {\dot{H}^{N_0+1,1/2+p_1}}\big)$ satisfies the bounds
\begin{equation}\label{ra1}
\begin{split}
\big\|h(t)\big\|_{\dot{H}^{N_0+1,1/2+p_1,}}&\lesssim\e_1\langle t\rangle^{p_0},\\
\big\|h(t)\big\|_{\dot{W}^{N_2+1,9/10}}&\lesssim\e_1\langle t\rangle^{-1/2},\\
\big\|Sh(t)\big\|_{\dot{H}^{N_1+1,1/2+p_1}}&\lesssim\e_1\langle t\rangle^{4p_0},
\end{split}
\end{equation}
for any $t\in[0,T]$, which follow from the bootstrap assumption \eqref{bing1}. For $\alpha\in[-1,1]$ let
\begin{equation}\label{ra28}
\mathcal{E}^\alpha_{w,T}:=\Big\{f\in C([0,T]:\dot{H}^{N_0+\alpha,p_1}):\,\|f\|_{\mathcal{E}^\alpha_{w,T}}:=\sup_{t\in[0,T]}\|f(t)\|_{\mathcal{E}^\alpha_{w}}\Big\},
\end{equation}
where
\begin{equation}\label{ra28.5}
\begin{split}
\|f(t)\|_{\mathcal{E}^\alpha_{w}}:&=\langle t\rangle^{-p_0}\|f(t)\|_{\dot{H}^{N_0+\alpha,p_1}}+\langle t\rangle^{1/2}\big\|f(t)\big\|_{\dot{W}^{N_2+\alpha,2/5}}+\langle t\rangle^{-4p_0}\|Sf(t)\|_{\dot{H}^{N_1+\alpha,p_1}}.
\end{split}
\end{equation}

Let $G(h)$ denote the Dirichlet-Neumann operator defined in \eqref{defG0}. Our main result in this section is the following paralinearization of the operator $G(h)$:

\begin{proposition}\label{ra102}
Assume $\phi\in \mathcal{E}^{1/2}_{w,T}$ and $h$ satisfies \eqref{ra1}, and define 
\begin{equation}\label{zxc1}
B := \frac{G(h)\phi + h_x\phi_x}{1+h_x^2} , \qquad V :=\frac{\phi_x-h_xG(h)\phi}{1+h_x^2}.
\end{equation}
Then, for any $f\in\{\phi_x,G(h)\phi\}$ and $t\in[0,T]$,
\begin{equation}\label{ra43}
\langle t\rangle^{-p_0}\|f\|_{\dot{H}^{N_0-1/2,-1/10}}+\langle t\rangle^{1/2}\|f\|_{\dot{W}^{N_2-1/2,-1/10}} 
  + \langle t\rangle^{-4p_0}\|Sf\|_{\dot{H}^{N_1-1/2,-1/10}}\lesssim \|\phi\|_{\mathcal{E}^{1/2}_{w,T}}.
\end{equation}
Moreover, we have the decomposition
\begin{equation}\label{zxc2}
G(h)\phi=|\partial_x|\phi-|\partial_x|T_Bh-\partial_xT_Vh+G_2(h,\phi)+G_{\geq 3},
\end{equation}
where
\begin{equation}\label{na72}
\widehat{G_2(h,\phi)}(\xi)=\frac{1}{2\pi}\int_{\mathbb{R}}\widehat{h}(\eta)\widehat{\phi}(\xi-\eta)\big[1-\chi(\xi-\eta,\eta)\big]\big[\xi(\xi-\eta)
  -|\xi||\xi-\eta|\big]\,d\eta,
\end{equation}
and, for any $t\in[0,T]$,
\begin{equation}\label{ra73}
\langle t\rangle^{1-p_0}\|G_{\geq 3}(t)\|_{H^{N_0+1}}+\langle t\rangle^{11/10}\|G_{\geq 3}\|_{\widetilde{W}^{N_2+1}}
  +\langle t\rangle^{1-4p_0}\|SG_{\geq 3}(t)\|_{H^{N_1+1}}\lesssim \e_1^2\|\phi\|_{\mathcal{E}^{1/2}_{w,T}}.
\end{equation}
\end{proposition}

The rest of this section is concerned with the proof of this proposition. We will need several intermediate results. To estimate products we often use the following simple general lemma:

\begin{lemma}\label{algebra}
Assume $a_0,a_2\in[1/100,100]$, $A_0,A_2\in(0,\infty)$, and $f,g\in L^2(\mathbb{R})$ satisfy
\begin{equation}\label{na60}
A_0^{-1}\big(\|f\|_{H^{a_0}}+\|g\|_{H^{a_0}}\big)+A_2^{-1}\big(\|f\|_{\widetilde{W}^{a_2}}+\|g\|_{\widetilde{W}^{a_2}}\big)\leq 1.
\end{equation}
Then
\begin{equation}\label{na61}
A_0^{-1}\|fg\|_{H^{a_0}}+A_2^{-1}\|fg\|_{\widetilde{W}^{a_2}}\lesssim A_2.
\end{equation}
\end{lemma}

\begin{proof}
Clearly,
\begin{equation*}
\|fg\|_{L^\infty}\lesssim A_2^2,\qquad \|fg\|_{L^2}\lesssim A_0A_2.
\end{equation*}
Moreover, for any $k\geq 0$ and $p\in\{2,\infty\}$
\begin{equation*}
\|P_k(fg)\|_{L^p}\lesssim A_2\sum_{k'\geq k-4}\big[\|P_{k'}f\|_{L^p}+\|P_{k'}g\|_{L^p}\big].
\end{equation*}
The desired estimate follows.
\end{proof}

\subsection{The perturbed Hilbert transform and proof of Proposition \ref{ra102}} With $h$ as in \eqref{ra1} let
\begin{equation}\label{na2}
\gamma(x):=x+ih(x)\qquad\text{ and }\qquad\Omega:=\{x+iy\in\mathbb{C}:y\leq h(x)\}.
\end{equation}
Let $D:=\|h\|_{L^\infty}$. The finiteness of the (large) constant $D$ is used to justify the convergence of integrals and some
identities, but $D$ itself does not appear in the main quantitative bounds.

For any $f\in L^2(\mathbb{R})$ we define the {\it{perturbed Hilbert transform}}
\begin{equation}\label{na3}
\begin{split}
&(\mathcal{H}_\gamma f)(\alpha):=\frac{1}{\pi i}\mathrm{p.v.} \int_{\mathbb{R}}
\frac{f(\beta)\gamma'(\beta)}{\gamma(\alpha)-\gamma(\beta)}\,d\beta:=\lim_{\eps\to 0}(\mathcal{H}^\eps_\gamma f)(\alpha),\\
&(\mathcal{H}^\eps_\gamma f)(\alpha):=\frac{1}{\pi i}\int_{|\beta-\alpha|\geq\eps}
\frac{f(\beta)\gamma'(\beta)}{\gamma(\alpha)-\gamma(\beta)}\,d\beta.
\end{split}
\end{equation}
For any $z\in\Omega$ and $f\in L^2(\mathbb{R})$ we define
\begin{equation}\label{na4}
F_f(z):=\frac{1}{2\pi i}\int_{\mathbb{R}}
\frac{f(\beta)\gamma'(\beta)}{z-\gamma(\beta)}\,d\beta.
\end{equation}
Clearly, $F_f$ is an analytic function in $\Omega$. For $\eps>0$ let
\begin{equation}\label{na5}
(\mathcal{T}_\gamma^\eps f)(\alpha):=F_f(\gamma(\alpha)-i\eps)=\frac{1}{2\pi i}\int_{\mathbb{R}}
\frac{f(\beta)\gamma'(\beta)}{\gamma(\alpha)-i\eps-\gamma(\beta)}\,d\beta.
\end{equation}
By comparing with the unperturbed case $h=0$, it is easy to verify that, for any $p\in (1,\infty)$,
\begin{equation}\label{na6}
\big\|\sup_{\eps\in(0,\infty)}|(\mathcal{H}^\eps_\gamma f)(\alpha)|\big\|_{L^p_\alpha}
+\big\|\sup_{\eps\in(0,\infty)}|(\mathcal{T}^\eps_\gamma f)(\alpha)|\big\|_{L^p_\alpha}\lesssim_{D,p} \|f\|_{L^p},
\end{equation}
and
\begin{equation}\label{na7}
\lim_{\eps\to 0}\mathcal{H}_\gamma^\eps f=\mathcal{H}_\gamma f,
\qquad \lim_{\eps\to 0}\mathcal{T}_\gamma^\eps f=\frac{1}{2}(I+\mathcal{H}_\gamma)f
\end{equation}
in $L^p$ for any $f\in L^p(\mathbb{R})$, $p\in(1,\infty)$.

The perturbed Hilbert transform can be used to derive explicit formulas for $G(h)\phi$ (see Lemma \ref{na101}). For this we will need a technical lemma:

\begin{lemma}\label{Wboundedness}
Given $h$ as in \eqref{ra1} and $n\in\mathbb{Z}_+$, we define the real operators
\begin{equation}\label{na8}
 (R_nf)(\alpha):=\frac{1}{\pi}\int_{\mathbb{R}}
\frac{h(\alpha)-h(\beta)-h'(\beta)(\alpha-\beta)}{\alpha-\beta}\Big(\frac{h(\alpha)-h(\beta)}{\alpha-\beta}\Big)^n\frac{f(\beta)}{\alpha-\beta}\,d\beta.
\end{equation}
Assume that $f\in \mathcal{E}^{-1}_{w,T}$ and $g\in\dot{W}^{N_2-1,b}$. There is a constant $C'\geq 1$ such that
\begin{equation}\label{ra10}
\begin{split}
&\langle t\rangle^{1/2-p_0}\|R_0f\|_{H^{N_0+1}}+\langle t\rangle\|R_0 f\|_{\dot{W}^{N_2+1,2/5}} 
  + \langle t\rangle^{1/2-4p_0} \|SR_0f\|_{H^{N_1+1}}\leq C'\e_1\|f\|_{\mathcal{E}^{-1}_{w,T}},
\\
& \langle t\rangle\|R_0 g\|_{\dot{W}^{N_2+1,b}}\leq C'\e_1\langle t\rangle^{1/2}\|g\|_{\dot{W}^{N_2-1,b}},\qquad b\in[1/100,2/5],
\end{split}
\end{equation}
for any $t\in[0,T]$. Moreover
\begin{equation}\label{ra11}
\begin{split}
\big\|R_1f\big\|_{H^{N_0+1}}&\leq (C'\e_1)^{2}\langle t\rangle^{p_0-1}\|f\|_{\mathcal{E}^{-1}_{w,T}},\\
\big\|R_1f\big\|_{\dot{W}^{N_2+1,-1/10}}&\leq (C'\e_1)^{2}\langle t\rangle^{-11/10}\|f\|_{\mathcal{E}^{-1}_{w,T}},\\
\big\|SR_1f\big\|_{H^{N_1+1}}&\leq (C'\e_1)^{2}\langle t\rangle^{4p_0-1}\|f\|_{\mathcal{E}^{-1}_{w,T}}.
\end{split}
\end{equation}
and, for any $n\geq 2$,
\begin{equation}\label{ra11.5}
\big\|R_nf\big\|_{H^{N_0+1}}+\big\|SR_nf\big\|_{H^{N_1+1}}\leq (C'\e_1)^{n+1}\langle t\rangle^{-5/4}\|f\|_{\mathcal{E}^{-1}_{w,T}}.
\end{equation}
\end{lemma}

This is proved in subsection \ref{zxc100} below. In rest of this subsection we show how to use it to prove Proposition \ref{ra102}. We consider first a suitable decomposition of the operators $\mathcal{H}_\gamma$.

\begin{lemma}\label{boundedness2}
Let
\begin{equation}\label{na20}
(H_0 f)(\alpha):=\frac{1}{\pi i}\mathrm{p.v.} \int_{\mathbb{R}}\frac{f(\beta)}{\alpha-\beta}\,d\beta,
\end{equation}
denote the unperturbed Hilbert transform, and consider the operators $T_1$ and $T_2$ defined by
\begin{equation}\label{na21}
\begin{split}
&(T_1 f)(\alpha):=\frac{1}{\pi}\mathrm{p.v.}\int_{\mathbb{R}}
\frac{h(\alpha)-h(\beta)-h'(\beta)(\alpha-\beta)}{|\gamma(\alpha)-\gamma(\beta)|^2}f(\beta)\,d\beta,\\
&(T_2 f)(\alpha):=\frac{1}{\pi}\mathrm{p.v.}\int_{\mathbb{R}}
\frac{h(\alpha)-h(\beta)-h'(\beta)(\alpha-\beta)}{|\gamma(\alpha)-\gamma(\beta)|^2}\frac{h(\alpha)-h(\beta)}{\alpha-\beta}f(\beta)\,d\beta.
\end{split}
\end{equation}
Then
\begin{equation}\label{na22}
\mathcal{H}_\gamma=H_0-T_1+iT_2,
\end{equation}
and
\begin{equation}\label{na23}
T_1=\sum_{n\geq 0}(-1)^{n}R_{2n},\qquad T_2=\sum_{n\geq 0}(-1)^{n}R_{2n+1}.
\end{equation}

Moreover, if $f\in \mathcal{E}^{-1}_{w,T}$ and $t\in[0,T]$ then
\begin{equation}\label{ra25}
\begin{split}
\big\|(T_1-R_0)f\big\|_{H^{N_0+1}}+\big\|T_2f\big\|_{H^{N_0+1}}&\lesssim \e_1^2\langle t\rangle^{p_0-1}\|f\|_{\mathcal{E}^{-1}_{w,T}},\\
\big\|(T_1-R_0)f\big\|_{\dot{W}^{N_2+1,-1/10}}+\big\|T_2f\big\|_{\dot{W}^{N_2+1,-1/10}}&\lesssim \e_1^2\langle t\rangle^{-11/10}\|f\|_{\mathcal{E}^{-1}_{w,T}},\\
\big\|S(T_1-R_0)f\big\|_{H^{N_1+1}}+\big\|ST_2f\big\|_{H^{N_1+1}}&\lesssim \e_1^2\langle t\rangle^{4p_0-1}\|f\|_{\mathcal{E}^{-1}_{w,T}}.
\end{split}
\end{equation}
In particular
\begin{equation}\label{ra29}
\|\mathcal{H}_\gamma f\|_{\mathcal{E}^{-1}_w}\lesssim \|f\|_{\mathcal{E}^{-1}_{w,T}},
\end{equation}
\begin{equation}\label{ra30.5}
\|R_0 f\|_{\mathcal{E}^1_w}+\|T_1 f\|_{\mathcal{E}^1_w}+\|T_2f\|_{\mathcal{E}^1_w}\lesssim \e_1\langle t\rangle^{-1/2}\|f\|_{\mathcal{E}^{-1}_{w,T}}.
\end{equation}
\end{lemma}

\begin{proof}
The identities \eqref{na22} and \eqref{na23} follow directly from definitions. The bounds \eqref{ra25} follow from \eqref{ra11}--\eqref{ra11.5}. Notice that
\begin{equation}\label{na27}
\widehat{H_0g}(\xi)=-\sgn(\xi)\widehat{g}(\xi).
\end{equation}
The bounds \eqref{ra29}--\eqref{ra30.5} follow using \eqref{na22} and the bounds \eqref{ra25} and \eqref{ra10}.
\end{proof}

We are now ready to define the conjugate pair $(\phi,\psi)$.

\begin{lemma}\label{na100}
(i) We have
\begin{equation}\label{na31}
\mathcal{H}_\gamma^2=I\text{ on }\mathcal{E}^{-1}_{w,T}.
\end{equation}
Moreover, if $\phi\in \mathcal{E}^{-1}_{w,T}$ is a real-valued function then there is a unique real-valued
function $\psi\in \mathcal{E}^{-1}_{w,T}$ with the property that
\begin{equation}\label{na32}
(I-\mathcal{H}_\gamma)(\phi+i\psi)=0.
\end{equation}

(ii) The function $F:\Omega\to\mathbb{C}$,
\begin{equation}\label{na33}
F(z):=\frac{1}{2\pi i}\int_{\mathbb{R}}
\frac{(\phi+i\psi)(\beta)\gamma'(\beta)}{z-\gamma(\beta)}\,d\beta
\end{equation}
in a bounded analytic function in $\Omega$, which extends to a $C^1$ function in $\overline{\Omega}$
with the property that $F(x+ih(x))=(\phi+i\psi)(x)$
for any $x\in\mathbb{R}$.
\end{lemma}

\begin{proof}
The identity \eqref{na31} is a standard consequence of the Cauchy integral formula applied to the analytic function $F_f$ defined in \eqref{na4},
and the second limit in \eqref{na7}.

The uniqueness of $\psi$ satisfying \eqref{na32} follows from Lemma \ref{boundedness2}: if $\psi_1,\psi_2$ are real-valued solutions of \eqref{na32} then
\begin{equation*}
(I-\mathcal{H}_\gamma)(\psi_1-\psi_2)=0.
\end{equation*}
Using the formula \eqref{na32} and taking the real part, it follows that $(I+T_1)(\psi_1-\psi_2)=0$. Since $\|T_1\|_{\mathcal{E}^{-1}_w\to\mathcal{E}^{-1}_w}\lesssim\e_1$ (see \eqref{ra30.5}), this shows that $\psi_1-\psi_2=0$.

To prove existence we use that $T_1$ is a contraction on $\mathcal{E}^{-1}_{w,T}$ and define $\psi$ such that
\begin{equation}\label{na35}
(I+T_1)\psi=(-iH_0+T_2)\phi.
\end{equation}
Let $P:=(I-\mathcal{H}_\gamma)(-i\phi+\psi)$ and notice that $(I+\mathcal{H}_\gamma)P=0$ (as a consequence
of \eqref{na31}). Moreover, $\Re P=0$, as a consequence of \eqref{na35}. Therefore $P=0$.
This completes the proof of part (i). The claims in (ii) follow from the second identity in \eqref{na7}.
\end{proof}

Assume $\phi,\psi$ and $F=G+iH$ are as in Lemma \ref{na100}. Let
\begin{equation*}
v_1:=\partial_xG=\partial_yH,\qquad v_2:=\partial_yG=-\partial_xH.
\end{equation*}
Notice that
\begin{equation*}
G(x+ih(x))=\phi(x),\qquad H(x+ih(x))=\psi(x).
\end{equation*}
Therefore
\begin{equation*}
-\psi_x=\partial_yG(x+ih(x))-ih'(x)\partial_xG(x+ih(x))=G(h)\phi.
\end{equation*}
Let
\begin{equation*}
\widetilde{V}(x):=v_1(x+ih(x)),\qquad \widetilde{B}(x):=v_2(x+ih(x)),\qquad x\in\mathbb{R}.
\end{equation*}

Taking derivatives, we have $\phi_x=\widetilde{V}+h_x\widetilde{B}$ and $-\psi_x=\widetilde{B}-h_x\widetilde{V}$. Therefore
\begin{equation}\label{na37}
\widetilde{V}=\frac{\phi_x+h_x\psi_x}{1+h_x^2}=V,\qquad \widetilde{B}=\frac{h_x\phi_x-\psi_x}{1+h_x^2}=B,
\end{equation}
where $V,B$ are defined in \eqref{zxc1}. Notice that
\begin{equation*}
 \partial_x\mathcal{H}_\gamma f=\gamma_x\mathcal{H}_\gamma(f_x/\gamma_x).
\end{equation*}
Therefore the identity \eqref{na32} gives
\begin{equation*}
 (I-\mathcal{H}_\gamma)\Big(\frac{\phi_x+i\psi_x}{\gamma_x}\Big)=0.
\end{equation*}
Notice also that $\phi_x+i\psi_x=(1+ih_x)(V-iB)$, as a consequence of \eqref{na37}. Therefore
\begin{equation*}
 (I-\mathcal{H}_\gamma)\big(V-iB\big)=0.
\end{equation*}

To summarize, we have the following lemma:
\begin{lemma}\label{na101}
Assume $\phi\in\mathcal{E}^{1/2}_{w,T}$ and define $\psi$ as in Lemma \ref{na100}. Then $\psi\in\mathcal{E}^{1/2}_{w,T}$ and
\begin{equation}\label{na40}
\|\psi\|_{\mathcal{E}^{1/2}_{w,T}}\lesssim \|\phi\|_{\mathcal{E}^{1/2}_{w,T}}.
\end{equation}
Moreover,
\begin{equation}\label{na41}
G(h)\phi=-\psi_x,\qquad V=\frac{\phi_x+h_x\psi_x}{1+h_x^2},\qquad B=\frac{h_x\phi_x-\psi_x}{1+h_x^2}.
\end{equation}
and
\begin{equation}\label{na42}
\phi_x=V+h_xB,\qquad -\psi_x=B-h_xV,\qquad (I-\mathcal{H}_\gamma)\big(V-iB\big)=0.
\end{equation}
In addition, for any $t\in[0,T]$ and any $g\in\{\phi_x,\psi_x\}$,
\begin{equation}\label{na42.5}
\langle t\rangle^{-p_0}\|g\|_{\dot{H}^{N_0-1/2,-1/10}}+\langle t\rangle^{1/2}\|g\|_{\dot{W}^{N_2-1/2,-1/10}}+\langle t\rangle^{-4p_0}\|Sg\|_{\dot{H}^{N_1-1/2,-1/10}}\lesssim \|\phi\|_{\mathcal{E}^{1/2}_{w,T}},
\end{equation}
and, for any $f\in\{\phi_x,\psi_x,V,B\}$,
\begin{equation}\label{na43}
\langle t\rangle^{-p_0}\|f\|_{H^{N_0-1/2}}+\langle t\rangle^{1/2}\|f\|_{\widetilde{W}^{N_2-1/2}}+\langle t\rangle^{-4p_0}\|Sf\|_{H^{N_1-1/2}}\lesssim \|\phi\|_{\mathcal{E}^{1/2}_{w,T}}.
\end{equation}
\end{lemma}

\begin{proof}
The inequality \eqref{na40} follows from \eqref{ra30.5} and \eqref{na35}. The formulas \eqref{na42} were derived earlier.
 The bounds \eqref{na42.5} follow as a consequence of \eqref{na40} and the definition. The function $h_x$ satisfies similar bounds, for any $t\in[0,T]$,
\begin{equation}\label{na43.5}
\langle t\rangle^{-p_0}\|h_x\|_{\dot{H}^{N_0,-1/2+p_1}}+\langle t\rangle^{1/2}\|h_x\|_{\dot{W}^{N_2,-1/10}}+\langle t\rangle^{-4p_0}\|Sh_x\|_{\dot{H}^{N_1,-1/2+p_1}}\lesssim \e_1,
\end{equation}
see \eqref{ra1}. The desired bounds \eqref{na43} for $V,B$ follow from \eqref{na41} and Lemma \ref{algebra}.
\end{proof}

We can now complete the proof of Proposition \ref{ra102}.

\begin{proof}[Proof of Proposition \ref{ra102}]
We may assume $\|\phi\|_{\mathcal{E}^{1/2}_{w,T}}=1$. The bound \eqref{ra43} was already proved, see \eqref{na43}. So we define $G_{\geq 3}$ according to \eqref{zxc2}, and we need to prove the bounds \eqref{ra73}. 

The last identity in \eqref{na42} gives
\begin{equation*}
(I-H_0)(V-iB)+(T_1-iT_2)(V-iB)=0
\end{equation*}
Taking real and imaginary parts we have
\begin{equation}\label{na74}
V+iH_0B=-T_1V+T_2B,\qquad -B+iH_0V=T_2V+T_1B.
\end{equation}
Using also the formulas $-\psi_x=B-h_xV$, see \eqref{na42}, and $|\partial_x|\phi=iH_0(\phi_x)$, we have
\begin{equation}\label{na75}
\begin{split}
|\partial_x|\phi-G(h)\phi&=iH_0(\phi_x)-B+h_xV\\
&=T_2V+T_1B+h_xV+iH_0(h_xB).
\end{split}
\end{equation}

Let
\begin{equation}\label{na76}
D:=T_2V+T_1B=iH_0V-B.
\end{equation}
Therefore
\begin{equation}\label{na77}
\begin{split}
-G_{\geq 3}&=|\partial_x|\phi-G(h)\phi-|\partial_x|T_Bh-\partial_xT_Vh+G_2=I+II,\\
I:&=h_xV-H_0(h_xH_0V)-|\partial_x|T_{iH_0V}h-\partial_xT_Vh+G_2+R_0(iH_0V),\\
II:&=D-R_0(iH_0V)-iH_0(h_xD)+|\partial_x|T_Dh.
\end{split}
\end{equation}

Using the definitions and \eqref{na16}, we write
\begin{equation*}
\widehat{I}(\xi)=\frac{1}{2\pi}\int_{\mathbb{R}}\widehat{h}(\eta)\widehat{V}(\xi-\eta)p(\xi,\eta)\,d\eta+\widehat{G_2}(\xi)
\end{equation*}
where
\begin{equation*}
\begin{split}
p(\xi,\eta):&=i\eta-i\eta\sgn(\xi)\sgn(\xi-\eta)+i|\xi|\sgn(\xi-\eta)\chi(\xi-\eta,\eta)-i\xi\chi(\xi-\eta,\eta)\\
&+i(\xi-\eta)[1-\sgn(\xi)\sgn(\xi-\eta)]\\
&=i\xi[1-\chi(\xi-\eta,\eta)][1-\sgn(\xi)\sgn(\xi-\eta)].
\end{split}
\end{equation*}
Using also the formulas \eqref{na72} and \eqref{na42}, we have
\begin{equation*}
\widehat{I}(\xi)=\frac{1}{2\pi}\int_{\mathbb{R}}
  \widehat{h}(\eta)\mathcal{F}(V-\phi_x)(\xi-\eta)p(\xi,\eta)\,d\eta
  =\frac{-1}{2\pi}\int_{\mathbb{R}}\widehat{h}(\eta)\mathcal{F}(h_xB)(\xi-\eta)p(\xi,\eta)\,d\eta.
\end{equation*}
In view of \eqref{na43}--\eqref{na43.5}, and Lemma \ref{algebra},
\begin{equation*}
\langle t\rangle^{1/2-p_0}\|h_xB\|_{H^{N_0-1/2}}+\langle t\rangle\|h_xB\|_{\widetilde{W}^{N_2-1/2}}+\langle t\rangle^{1/2-4p_0}\|S(h_xB)\|_{H^{N_1-1/2}}\lesssim \varep_1.
\end{equation*}
In other words, $\varep_1 h_xB\in O_{2,-1/2}$, see Definition \ref{Oterms}, and $h_x\in O_{1,0}$, see \eqref{na43.5}. In view of Lemma \ref{OtermsProd} (applied to $m_2(\xi,\eta)=
p(\xi,\eta)/\eta$), $\varep_1I\in O_{3,1}$, which gives, for any $t\in[0,T]$,
\begin{equation}\label{na78}
\langle t\rangle^{1-p_0}\|I\|_{H^{N_0+1}}+\langle t\rangle^{11/10}\|I\|_{\widetilde{W}^{N_2+1}}+\langle t\rangle^{1-4p_0}\|SI\|_{H^{N_1+1}}\lesssim \e_1^2.
\end{equation}

We consider now the contribution of $II$. It follows from \eqref{ra25} and \eqref{na43} that
\begin{equation*}
\begin{split}
\big\|(T_1-R_0)B\big\|_{H^{N_0+1}}+\big\|T_2V\big\|_{H^{N_0+1}}&\lesssim \e_1^2\langle t\rangle^{p_0-1},\\
\big\|(T_1-R_0)B\big\|_{\dot{W}^{N_2+1,-1/10}}+\big\|T_2V\big\|_{\dot{W}^{N_2+1,-1/10}}&\lesssim \e_1^2\langle t\rangle^{-11/10},\\
\big\|S(T_1-R_0)B\big\|_{H^{N_1+1}}+\big\|ST_2V\big\|_{H^{N_1+1}}&\lesssim \e_1^2\langle t\rangle^{4p_0-1}.
\end{split}
\end{equation*}
Therefore, using the formula \eqref{na76},
\begin{equation}\label{na79}
\langle t\rangle^{1-p_0}\|D-R_0B\|_{H^{N_0+1}} +\langle t\rangle^{11/10}\|D-R_0B\big\|_{\widetilde{W}^{N_2+1}}+\langle t\rangle^{1-4p_0}\|S(D-R_0B)\|_{H^{N_1+1}}\lesssim \e_1^2.
\end{equation}
Therefore, using also \eqref{ra10},
\begin{equation}\label{na80}
\langle t\rangle^{1/2-p_0}\|D\|_{H^{N_0+1}}+\langle t\rangle\|D\|_{\dot{W}^{N_2+1,1/100}}+\langle t\rangle^{1/2-4p_0}\|SD\|_{H^{N_1+1}}\lesssim \e_1.
\end{equation}
Using \eqref{na76}, in the form $R_0D=R_0(iH_0V-B)$, and the bounds \eqref{ra10} and \eqref{na80},
\begin{equation}\label{na81}
\begin{split}
\langle t\rangle^{1-p_0}\|R_0B-R_0(iH_0V)\|_{H^{N_0+1}}&+\langle t\rangle^{3/2}\|R_0B-R_0(iH_0V)\|_{\dot{W}^{N_2+1,1/100}}\\
&+\langle t\rangle^{1-4p_0}\|S[R_0B-R_0(iH_0V)]\|_{H^{N_1+1}}\lesssim\e_1^2.
\end{split}
\end{equation}
Moreover,
\begin{equation*}
\mathcal{F}\big[-iH_0(h_xD)+|\partial_x|T_Dh\big](\xi)=\frac{1}{2\pi}\int_{\mathbb{R}}\widehat{D}(\xi-\eta)\widehat{h}(\eta)\sgn(\xi)[\xi\chi(\xi-\eta,\eta)-\eta]\,d\eta.
\end{equation*}
Using \eqref{ra1} and \eqref{na80}, it follows from Lemma \ref{OtermsProd}
\begin{equation}\label{na82}
\begin{split}
\big\|-iH_0(h_xD)+|\partial_x|T_Dh\big\|_{H^{N_0+1}}&\lesssim \e_1^2\langle t\rangle^{-1+p_0},\\
\big\|-iH_0(h_xD)+|\partial_x|T_Dh\big\|_{\widetilde{W}^{N_2+1}}&\lesssim \e_1^2\langle t\rangle^{-11/10},\\
\big\|S[-iH_0(h_xD)+|\partial_x|T_Dh]\big\|_{H^{N_1+1}}&\lesssim \e_1^2\langle t\rangle^{-1+4p_0}.
\end{split}
\end{equation}
The desired bound \eqref{ra73} follows from \eqref{na77} and the bounds \eqref{na78}--\eqref{na82}.
\end{proof}

\subsection{Proof of Lemma \ref{Wboundedness}}\label{zxc100}We rewrite first the operators $R_n$.
We take the Fourier transform in $\alpha$ and make the change of variables $\alpha\to\beta+\rho$ to write
\begin{equation*}
\mathcal{F}\big[R_nf\big](\xi)=\frac{1}{\pi}\int_{\mathbb{R}^2}
\frac{f(\beta)}{\rho}\Big(\frac{h(\beta+\rho)-h(\beta)}{\rho}\Big)^n\frac{h(\beta+\rho)-h(\beta)-\rho h'(\beta)}{\rho}e^{-i\xi\beta}e^{-i\xi\rho}\,d\beta d\rho.
\end{equation*}
Notice that
\begin{equation*}
\frac{h(\beta+\rho)-h(\beta)}{\rho}=\frac{1}{2\pi}\int_{\mathbb{R}}\widehat{h}(\eta)e^{i\eta\beta}\frac{e^{i\eta\rho}-1}{\rho}\,d\eta.
\end{equation*}
Therefore
\begin{equation*}
\begin{split}
\mathcal{F}\big[R_nf\big](\xi)=&\frac{1}{\pi(2\pi)^{n+1}}\int_{\mathbb{R}^2\times \mathbb{R}^n}
\frac{f(\beta)}{\rho}e^{-i\xi\beta}e^{-i\xi\rho}e^{i(\eta_1+\ldots+\eta_{n+1})\beta}\\
&\times\widehat{h}(\eta_1)\cdot\ldots\cdot\widehat{h}(\eta_{n+1})\frac{e^{i\eta_{n+1}\rho}-1-i\eta_{n+1}\rho}{\rho}\prod_{l=1}^n\frac{e^{i\eta_l\rho}-1}{\rho}\,d\beta d\rho d\eta_1\ldots d\eta_n.
\end{split}
\end{equation*}
This can be rewritten in the form
\begin{equation}\label{na12}
\begin{split}
\mathcal{F}\big[R_nf\big](\xi)=\frac{1}{\pi(2\pi)^{n+1}}\int_{\mathbb{R}^{n+1}}
\widetilde{M}_{n+1}&(\xi;\eta_1,\ldots,\eta_{n+1})\widehat{f}(\xi-\eta_1-\ldots-\eta_{n+1})\\
&\times\widehat{h}(\eta_1)\cdot\ldots\cdot\widehat{h}(\eta_{n+1})\, d\eta_1\ldots d\eta_{n+1},
\end{split}
\end{equation}
where
\begin{equation}\label{na13}
\begin{split}
\widetilde{M}_{n+1}(\xi;\eta_1,\ldots,\eta_{n+1}):=\int_{\mathbb{R}}\frac{e^{-i\xi\rho}}{\rho}\frac{e^{i\eta_{n+1}\rho}-1-i\eta_{n+1}\rho}{\rho}\prod_{l=1}^n\frac{e^{i\eta_l\rho}-1}{\rho}\,d\rho.
\end{split}
\end{equation}

Using the formula
\begin{equation*}
\frac{d}{d\rho}\frac{e^{-i\eta\rho}-1}{\rho}=\frac{1-e^{-i\eta\rho}-i\eta\rho e^{-i\eta\rho}}{\rho^2},
\end{equation*}
and integration by parts in $\rho$ in \eqref{na13}, we have
\begin{equation}\label{na13.1}
\begin{split}
\widetilde{M}_{n+1}(\xi;\eta_1,\ldots,\eta_{n+1})&=-\int_{\mathbb{R}}\frac{e^{-i\eta_{n+1}\rho}-1}{\rho}\frac{d}{d\rho}\Big[e^{-i(\xi-{\eta_{n+1}})\rho}\prod_{l=1}^n\frac{e^{i\eta_l\rho}-1}{\rho}\Big]\,d\rho\\
&=-i(\xi-\eta_{n+1})\int_{\mathbb{R}}e^{-i\xi\rho}\prod_{l=1}^{n+1}\frac{e^{i\eta_l\rho}-1}{\rho}\,d\rho\\
&+\sum_{j=1}^n\int_{\mathbb{R}}e^{-i\xi\rho}\frac{i\eta_j\rho e^{i\eta_j\rho}-e^{i\eta_j\rho}+1}{\rho^2}\prod_{l=1,\,l\neq j}^{n+1}\frac{e^{i\eta_l\rho}-1}{\rho}\,d\rho.
\end{split}
\end{equation}
For $j\in\{1,\ldots,n\}$ let $\pi_j:\mathbb{R}^{n+1}\to\mathbb{R}^{n+1}$ denote the map that permutes the variables $\eta_j$ and $\eta_{n+1}$,
\begin{equation*}
\pi_j(\eta_1,\ldots,\eta_j,\ldots,\eta_n,\eta_{n+1}):=(\eta_1,\ldots,\eta_{n+1},\ldots,\eta_n,\eta_j).
\end{equation*}
The formulas \eqref{na13} and \eqref{na13.1} show that
\begin{equation*}
\begin{split}
\widetilde{M}_{n+1}(\xi;\underline{\eta})+\sum_{j=1}^n\widetilde{M}_{n+1}(\xi;\pi_j(\underline{\eta}))&=-i(\xi-\eta_{n+1})\int_{\mathbb{R}}e^{-i\xi\rho}\prod_{l=1}^{n+1}\frac{e^{i\eta_l\rho}-1}{\rho}\,d\rho\\
&+\sum_{j=1}^n\int_{\mathbb{R}}e^{-i\xi\rho}\frac{i\eta_j\rho e^{i\eta_j\rho}-e^{i\eta_j\rho}+1}{\rho^2}\prod_{l=1,\,l\neq j}^{n+1}\frac{e^{i\eta_l\rho}-1}{\rho}\,d\rho\\
&+\sum_{j=1}^n\int_{\mathbb{R}}\frac{e^{-i\xi\rho}}{\rho}\frac{e^{i\eta_{j}\rho}-1-i\eta_{j}\rho}{\rho}\prod_{l=1,\,l\neq j}^{n+1}\frac{e^{i\eta_l\rho}-1}{\rho}\,d\rho\\
&=-i(\xi-\eta_1-\ldots-\eta_{n+1})\int_{\mathbb{R}}e^{-i\xi\rho}\prod_{l=1}^{n+1}\frac{e^{i\eta_l\rho}-1}{\rho}\,d\rho,
\end{split}
\end{equation*}
where $\underline{\eta}:=(\eta_1,\ldots,\eta_{n+1})$. Letting
\begin{equation}\label{na13.2}
\begin{split}
M_{n+1}(\xi;\underline{\eta}):&=\frac{1}{n+1}\big[\widetilde{M}_{n+1}(\xi;\underline{\eta})+\sum_{j=1}^n\widetilde{M}_{n+1}(\xi;\pi_j(\underline{\eta}))\big]\\
&=\frac{-i(\xi-\eta_1-\ldots-\eta_{n+1})}{n+1}\int_{\mathbb{R}}e^{-i\xi\rho}\prod_{l=1}^{n+1}\frac{e^{i\eta_l\rho}-1}{\rho}\,d\rho,
\end{split}
\end{equation}
it follows by symmetrization from \eqref{na12} that
\begin{equation}\label{na13.3}
\begin{split}
\mathcal{F}\big[R_nf\big](\xi)=\frac{1}{\pi(2\pi)^{n+1}}\int_{\mathbb{R}^{n+1}}M_{n+1}(\xi;\underline{\eta})&\widehat{f}(\xi-\eta_1-\ldots-\eta_{n+1})\\
&\times\widehat{h}(\eta_1)\cdot\ldots\cdot\widehat{h}(\eta_{n+1})\, d\underline{\eta}.
\end{split}
\end{equation}

\subsubsection{The operator $R_0$} We prove now the bounds  \eqref{ra10}. Recall the formulas
\begin{equation}\label{na14}
\int_{\mathbb{R}}e^{-i\xi x}\,dx=2\pi\delta_0(\xi),\qquad \int_{\mathbb{R}}e^{-i\xi x}\frac{1}{x}\,dx=-i\pi\,\mathrm{sgn}(\xi),\qquad\int_{\mathbb{R}}\frac{e^{-i\xi x}-1}{x^2}\,dx=-\pi\,|\xi|,
\end{equation}
for any $\xi\in\mathbb{R}$. Using these formulas, the symbol $M_1(\xi;\eta_1)$ can be calculated easily,
\begin{equation}\label{na15}
M_{1}(\xi;\eta_{1})=\pi(\xi-\eta_1)\big[\,\mathrm{sgn}(\xi)-\,\mathrm{sgn}(\xi-\eta_1)\big].
\end{equation}
Therefore
\begin{equation}\label{na16}
\mathcal{F}\big[R_0f\big](\xi)=\frac{1}{2\pi}\int_{\mathbb{R}}(\xi-\eta_1)\big[\,\mathrm{sgn}(\xi)-\,\mathrm{sgn}(\xi-\eta_1)\big]\widehat{f}(\xi-\eta_1)\widehat{h}(\eta_1)\, d\eta_1.
\end{equation}

The bounds \eqref{ra10} can be proved easily. We may assume that $\|f\|_{\mathcal{E}^{-1}_{w,T}}\lesssim 1$. It follows from \eqref{na16} that
\begin{equation*}
\begin{split}
\|P_kR_0f\|_{L^2}&\lesssim \sum_{k_2+10\geq \max(k,k_1)}2^{k_1}\|P'_{k_1}f\|_{L^\infty}\|P'_{k_2}h\|_{L^2}\lesssim \langle t\rangle^{-1/2}\sum_{k_2+10\geq k}2^{3\min(k_2,0)/5}\|P'_{k_2}h\|_{L^2},
\end{split}
\end{equation*}
for any $k\in\mathbb{Z}$. Moreover,
\begin{equation*}
\|R_0f\|_{L^2}\lesssim \sum_{k_2+10\geq k_1}2^{k_1}\|P'_{k_1}f\|_{L^\infty}\|P'_{k_2}h\|_{L^2}\lesssim \e_1\langle t\rangle^{p_0-1/2}.
\end{equation*}
The $L^2$ bound on $R_0f$ in \eqref{ra10} follows using the assumptions \eqref{ra1}. The $L^\infty$ bound on $R_0g$ in the second line of \eqref{ra10} follows in a similar way, and this implies the $L^\infty$ bound on $R_0f(t)$ in the first line of \eqref{ra10} (by setting $b=2/5$). The weighted bound in \eqref{ra10} also follows in the same way, using \eqref{lemcomm} and the homogeneity of the symbol defining $R_0$.

\subsubsection{The operator $R_1$} We prove now the bounds \eqref{ra11}. Using \eqref{na14} again
\begin{equation}\label{na19}
M_{2}(\xi;\eta_{1},\eta_2):= \frac{i\pi}{2} (\xi-\eta_1-\eta_{2}) \big[|\xi-\eta_1-\eta_2|+|\xi|-|\xi-\eta_1|-|\xi-\eta_2|\big].
\end{equation}
The main observation is that $M_{2}(\xi;\eta_{1},\eta_2)=0$ if $|\eta_1|+|\eta_2|\leq|\xi|$. It is easy to see that
\begin{equation*}
\Big\|\mathcal{F}^{-1}\big[M_{2}(\xi;\eta_{1},\eta_2)\varphi_{k_1}(\eta_1)\varphi_{k_2}(\eta_2)\varphi_{k_3}(\xi-\eta_1-\eta_2)\big]\Big\|_{L^1(\mathbb{R}^3)}\lesssim 2^{k_3}2^{\min(k_1,k_2)}.
\end{equation*}
Therefore, using Lemma \ref{touse}, the assumption $\|f\|_{\mathcal{E}^{-1}_{w,T}}\lesssim 1$, and the bounds \eqref{ra1}
\begin{equation}\label{na19.5}
\begin{split}
\|P_kR_1f\|_{L^p}&\lesssim \sum_{k_2+10\geq\max(k,k_1,k_3)}2^{k_1+k_3}\|P'_{k_1}h\|_{L^\infty}\|P'_{k_2}h\|_{L^p}\|P'_{k_3}f\|_{L^\infty}\\
&\lesssim \e_1\langle t\rangle^{-1}\sum_{k_2+10\geq k}2^{3\min(k_2,0)/5}\|P'_{k_2}h\|_{L^p},
\end{split}
\end{equation}
for $p\in\{2,\infty\}$ and any $k\in\mathbb{Z}$. In particular,
\begin{equation}\label{na19.6}
\|P_{\geq 0}R_1f\|_{H^{N_0+1}}\lesssim \e_1^2\langle t\rangle^{-1+p_0}\quad\text{ and }\quad\|P_{\geq 0}R_1f\|_{\dot{W}^{N_2+1,-1/10}}\lesssim\e_1^2\langle t\rangle^{-3/2}.
\end{equation}
Moreover
\begin{equation}\label{na19.7}
\begin{split}
\|R_1f\|_{L^2}&\lesssim \sum_{k_2+10\geq\max(k_1,k_3)}2^{k_1+k_3}\|P'_{k_1}h\|_{L^\infty}\|P'_{k_2}h\|_{L^2}\|P'_{k_3}f\|_{L^\infty}\lesssim \e_1^2\langle t\rangle^{p_0-1}.
\end{split}
\end{equation}
Finally, as a consequence of \eqref{ra1}, \eqref{na19.5}, and Sobolev embedding,
\begin{equation}\label{na19.8}
\begin{split}
&2^{-k/10}\|P_kR_1f\|_{L^\infty}\lesssim 2^{2k/5}\|P_kR_1f\|_{L^2}\lesssim 2^{2k/5}\e_1^2\langle t\rangle^{p_0-1},\\
&2^{-k/10}\|P_kR_1f\|_{L^\infty}\lesssim 2^{-k/10}\e_1\langle t\rangle^{-1}\cdot 2^{-2k/5}\e_1\langle t\rangle^{-1/2}\lesssim \e_1^2\langle t\rangle^{-3/2}2^{-k/2},
\end{split}
\end{equation}
for any integer $k\leq 0$. The bounds in the first two lines of \eqref{ra11} follow from \eqref{na19.6}--\eqref{na19.8}. The weighted bound in the last line follows in a similar way, using an identity similar to \eqref{lemcomm} and the homogeneity of the symbol.

\subsubsection{The operator $R_n$, $n\geq 2$} To prove \eqref{ra11.5} we would like to use induction over $n$. For this we need to prove slightly stronger bounds. We start with a lemma:

\begin{lemma}\label{na149}
For any $\underline{k}=(k_1,\ldots,k_{n+1})\in\mathbb{Z}^{n+1}$ and $n\geq 1$ let $F_{n;\underline{k}}(f_{\underline{k}})$ be defined by
\begin{equation}\label{na150}
\begin{split}
\mathcal{F}\big(F_{n;\underline{k}}(f_{\underline{k}})\big)(\xi):=\int_{\mathbb{R}^{n+1}}
&M'_{n+1}(\xi;\underline{\eta})\varphi_{k_1}(\eta_1)\ldots \varphi_{k_{n+1}}(\eta_{n+1})\\
&\times\widehat{f_{\underline{k}}}(\xi-\eta_1-\ldots-\eta_{n+1})\cdot\widehat{h}(\eta_1)\cdot\ldots\cdot\widehat{h}(\eta_{n+1})\, d\underline{\eta},
\end{split}
\end{equation}
where $\underline{\eta}=(\eta_1,\ldots,\eta_{n+1})$ and
\begin{equation}\label{na150.1}
M'_{n+1}(\xi;\underline{\eta}):=\int_{\mathbb{R}}e^{-i\xi\rho}\prod_{l=1}^{n+1}\frac{e^{i\eta_l\rho}-1}{\rho}\,d\rho.
\end{equation}
Assume that the functions $f_{\underline{k}}$ satisfy the uniform bounds
\begin{equation}\label{na151}
\|f_{\underline{k}}\|_{\dot{W}^{1/5,0}}\leq\langle t\rangle^{-1/2}.
\end{equation}
Then there is a constant $C_0\geq 1$ such that if $n\geq 2$ then
\begin{equation}\label{na152}
\Big\|\sum_{\underline{k}\in\mathbb{Z}^{n+1}}F_{n;\underline{k}}(f_{\underline{k}})\Big\|_{H^{N_0+1}}\leq (C_0\e_1)^{n+1}\langle t\rangle^{-5/4}.
\end{equation}
\end{lemma}

\begin{proof}
Notice that
\begin{equation}\label{na201}
M'_{n+1}(\xi,\eta_1,\ldots,\eta_{n+1})=0\qquad\text{ if }\qquad \eta_1\cdot\ldots\cdot\eta_{n+1}=0.
\end{equation}
Taking partial derivatives in $\eta$ it follows that
\begin{equation}\label{na202}
M'_{n+1}(\xi,\eta_1,\ldots,\eta_{n+1})=0\qquad\text{ if }\qquad |\eta_1|+\ldots+|\eta_{n+1}|\leq|\xi|.
\end{equation}

Let $\chi_0:=\mathcal{F}^{-1}(\varphi_0)$ and let $\chi'_0$ denote the derivative of $\chi_0$. Let $\widetilde{k}_1\leq\ldots\leq\widetilde{k}_{n+1}$ denote the increasing rearrangement of the integers $k_1,\ldots,k_{n+1}$. In view of \eqref{na202},
\begin{equation}\label{na202.5}
F_{n,\underline{k}}(f_{\underline{k}})=F_{n,\underline{k}}(P_{\leq \widetilde{k}_{n+1}+5n} f_{\underline{k}}).
\end{equation}

Using the formula
\begin{equation}\label{na206}
\frac{1}{2\pi}\int_{\mathbb{R}}\varphi_k(\mu)e^{iy\mu}\frac{e^{i\mu\rho}-1}{\rho}\,d\mu=\frac{2^k\big[\chi_0(2^k(y+\rho))-\chi_0(2^ky)\big]}{\rho}=:K_k(y,\rho),
\end{equation}
and the definition \eqref{na150}, we rewrite
\begin{equation}\label{na205}
F_{n,\underline{k}}(f_{\underline{k}})(x)=(2\pi)^{n+1}\int_{\mathbb{R}^{n+2}}f_{\underline{k}}(x-\rho)\prod_{l=1}^{n+1}P'_{k_l}h(x-\rho-y_l)\cdot \prod_{l=1}^{n+1}K_{k_l}(y_l,\rho)\,d\rho d\underline{y}.
\end{equation}

We notice that
\begin{equation}\label{na207}
\int_{\mathbb{R}}\big|K_k(y,\rho)\big|\,dy\lesssim \min\big(|\rho|^{-1},2^k\big).
\end{equation}
 The inequalities \eqref{na207} show that
\begin{equation}\label{na208}
\int_{\mathbb{R}^{n+2}}\Big|\prod_{l=1}^{n+1}K_{k_l}(y_l,\rho)\Big|\,d\rho d\underline{y}\leq C^{n+1} 2^{\widetilde{k}_1+\ldots+\widetilde{k}_n}(1+|\widetilde{k}_{n+1}-\widetilde{k}_n|),
\end{equation}
for some constant $C\geq 1$ (which is allowed to change from now on from line to line). Therefore, using \eqref{na205},
\begin{equation}\label{na209}
\big\|F_{n,\underline{k}}(f_{\underline{k}})\big\|_{L^2}\leq C^{n+1}(1+|\widetilde{k}_{n+1}-\widetilde{k}_n|)\cdot\|f_{\underline{k}}\|_{L^\infty}\|P'_{\widetilde{k}_{n+1}}h\|_{L^2}\prod_{l=1}^n2^{\widetilde{k}_l}\|P'_{\widetilde{k}_{l}}h\|_{L^\infty}.
\end{equation}

Therefore
\begin{equation}\label{na212}
\begin{split}
\sum_{(k_1,\ldots,k_{n+1})\in J_1}\big\|F_{n,\underline{k}}(f_{\underline{k}})\big\|_{H^{N_0+1}}&\leq (C\e_1)^{n+1}\langle t\rangle^{-(n+1)/2+p_0},\\
\sum_{(k_1,\ldots,k_{n+1})\in J_2}\big\|F_{n,\underline{k}}(P_{\geq \widetilde{k}_{n+1}-10n+1}f_{\underline{k}})\big\|_{H^{N_0+1}}&\leq (C\e_1)^{n+1}\langle t\rangle^{-(n+1)/2+p_0},
\end{split}
\end{equation}
where
\begin{equation}\label{na214}
\begin{split}
&J_1:=\big\{(k_1,\ldots,k_{n+1})\in\mathbb{Z}^{n+1}:\,\widetilde{k}_{n+1}\leq\max(0,\widetilde{k}_n)+10n\big\},\\
&J_2:=\big\{(k_1,\ldots,k_{n+1})\in\mathbb{Z}^{n+1}:\,k_1=\widetilde{k}_{n+1}\geq\max(0,\widetilde{k}_n)+10n\big\}.
\end{split}
\end{equation}

It remains to prove that if $n\geq 1$ then
\begin{equation}\label{na213}
\Big\|\sum_{(k_1,\ldots,k_{n+1})\in J_2}F_{n,\underline{k}}(P_{\leq k_1-10n}f_{\underline{k}})\Big\|_{H^{N_0+1}}\leq (C_1\e_1)^{n+1}\langle t\rangle^{-e_n}
\end{equation}
for some constant $C_1\geq 1$, where $e_1=1-p_0$ and $e_n=5/4$ if $n\geq 2$.

We prove the inequalities in \eqref{na213} using induction over $n$ (the case $n=1$ follows from our discussion of the operator $R_1$ before, noticing that only $L^\infty$ norms on $\partial_xf$ were used in \eqref{na19.5}--\eqref{na19.7}). We decompose
\begin{equation*}
K_{k_{n+1}}(y_{n+1},\rho)=\widetilde{K}_{k_{n+1}}(y_{n+1},\rho)+2^{2k_{n+1}}\chi'_0(2^{k_{n+1}}y_{n+1})
\end{equation*}
where
\begin{equation*}
\widetilde{K}_k(y,\rho):=\frac{2^{k}\big[\chi_0(2^{k}(y+\rho))-\chi_0(2^{k}y)-2^k\rho\chi'_0(2^ky)\big]}{\rho}.
\end{equation*}
Then we decompose
\begin{equation*}
F_{n,\underline{k}}(P_{\leq k_1-10n}f_{\underline{k}})=F^1_{n,\underline{k}}(P_{\leq k_1-10n}f_{\underline{k}})
+F^2_{n,\underline{k}}(P_{\leq k_1-10n}f_{\underline{k}}),
\end{equation*}
where
\begin{equation}\label{na216}
F^1_{n,\underline{k}}(g)(x):=C^{n+1}\int_{\mathbb{R}^{n+2}}g(x-\rho)\prod_{l=1}^{n+1}P'_{k_l}h(x-\rho-y_l)\cdot \widetilde{K}_{k_{n+1}}(y_{n+1},\rho)\prod_{l=1}^{n}K_{k_l}(y_l,\rho)\,d\rho d\underline{y},
\end{equation}
and
\begin{equation}\label{na217}
\begin{split}
F^2_{n,\underline{k}}(g)(x):=C^{n+1}&\int_{\mathbb{R}^{n+2}}g(x-\rho)\prod_{l=1}^{n+1}P'_{k_l}h(x-\rho-y_l)\\
&\times 2^{2k_{n+1}}\chi'_0(2^{k_{n+1}}y_{n+1})\prod_{l=1}^{n}K_{k_l}(y_l,\rho)\,d\rho d\underline{y}.
\end{split}
\end{equation}

We notice that
\begin{equation*}
\int_{\mathbb{R}}\big|\widetilde{K}_k(y,\rho)\big|\,dy\lesssim 2^k\rho\min\big(|\rho|^{-1},2^k\big).
\end{equation*}
Using also \eqref{na207} it follows that if $(k_1,\ldots,k_{n+1})\in J_2$ then
\begin{equation*}
\int_{\mathbb{R}^{n+2}}\Big|\widetilde{K}_{k_{n+1}}(y_{n+1},\rho)\prod_{l=1}^{n}K_{k_l}(y_l,\rho)\Big|\,d\rho d\underline{y}\leq C^{n+1}(1+|\widetilde{k}_n-\widetilde{k}_{n-1}|)2^{\widetilde{k}_1+\ldots+\widetilde{k}_n},
\end{equation*}
which is slightly stronger than the inequality \eqref{na208}. Therefore
\begin{equation*}
\Big\|F^1_{n,\underline{k}}(P_{\leq k_1-10n}f_{\underline{k}})\Big\|_{L^2}\leq C^{n+1}(1+|\widetilde{k}_n-\widetilde{k}_{n-1}|)\cdot\|f_{\underline{k}}\|_{L^\infty}\|P'_{k_1}h\|_{L^2}\prod_{l=1}^n2^{\widetilde{k}_l}\|P'_{\widetilde{k}_{l}}h\|_{L^\infty}.
\end{equation*}
Therefore, for any $l\geq 0$,
\begin{equation*}
\Big\|P_l\Big[\sum_{\underline{k}\in J_2}F^1_{n,\underline{k}}(P_{\leq k_1-10n}f_{\underline{k}})\Big]\Big\|_{L^2}\leq C^{n+1}(\e_1\langle t\rangle^{-1/2})^n\langle t\rangle^{-1/2}\sum_{|k_1-l|\leq 10}\|P'_{k_1}h\|_{L^2}.
\end{equation*}
Therefore
\begin{equation}\label{na218}
\Big\|\sum_{\underline{k}\in J_2}F^1_{n,\underline{k}}(P_{\leq k_1-10n}f_{\underline{k}})\Big\|_{H^{N_0+1}}\leq (C\e_1)^{n+1}\langle t\rangle^{-5/4}.
\end{equation}

We estimate now the contributions of the terms $F^2_{n,\underline{k}}(P_{\leq k_1-10n}f_{\underline{k}})$. We integrate the variable $y_{n+1}$ in the defining formula \eqref{na217}. Let
\begin{equation*}
g_{(k_1,\ldots,k_{n})}(z):=\sum_{k_{n+1}\leq k_1-10n}f_{(k_1,\ldots,k_{n},k_{n+1})}(z)\int_{\mathbb{R}}P'_{k_{n+1}}h(z-y_{n+1})2^{2k_{n+1}}\chi'_0(2^{k_{n+1}}y_{n+1})\,dy_{n+1}.
\end{equation*}
Using the assumptions \eqref{na151} and \eqref{ra1}, it is easy to see that
\begin{equation*}
\sum_{l\in\mathbb{Z}}(2^{l/5}+1)\|P_lg_{(k_1,\ldots,k_n)}\|_{L^\infty}\lesssim\langle t\rangle^{-1/2}\cdot\e_1\langle t\rangle^{-2/5}.
\end{equation*}
The induction hypothesis shows that
\begin{equation*}
\Big\|\sum_{\underline{k}\in J_2}F^2_{n,\underline{k}}(P_{\leq \widetilde{k}_{n+1}-10n}f_{\underline{k}})\Big\|_{H^{N_0+1}}\leq (C_1\e_1)^{n}\langle t\rangle^{-1+p_0}\cdot C\e_1\langle t\rangle^{-2/5}.
\end{equation*}
The desired conclusion follows, using also \eqref{na218} provided that $C_1$ is sufficiently large.
\end{proof}

The inequality on the first term in \eqref{ra11.5} follows from this lemma with $f_{\underline{k}}=\partial_xf$. To prove the bound on the second term we start from the formula \eqref{na13.3} and write, as in \eqref{lemcomm},
\begin{equation}\label{ra13}
\begin{split}
\mathcal{F}\big[SR_nf\big]&(\xi)=\frac{c_n}{(2\pi)^{n+1}}\int_{\mathbb{R}^{n+1}}M'_{n+1}(\xi;\underline{\eta})\widehat{S\partial_xf}(\xi-\eta_1-\ldots-\eta_{n+1})\widehat{h}(\eta_1)\ldots\widehat{h}(\eta_{n+1})\, d\underline{\eta}\\
&+\frac{(n+1)c_n}{(2\pi)^{n+1}}\int_{\mathbb{R}^{n+1}}M'_{n+1}(\xi;\underline{\eta})\widehat{\partial_xf}(\xi-\eta_1-\ldots-\eta_{n+1})\widehat{h}(\eta_1)\ldots\widehat{Sh}(\eta_{n+1})\, d\underline{\eta}\\
&+\frac{c_n}{(2\pi)^{n+1}}\int_{\mathbb{R}^{n+1}}\widetilde{M}'_{n+1}(\xi;\underline{\eta})\widehat{\partial_xf}(\xi-\eta_1-\ldots-\eta_{n+1})\widehat{h}(\eta_1)\ldots\widehat{h}(\eta_{n+1})\, d\underline{\eta},
\end{split}
\end{equation}
where $c_n:=-1/(\pi(n+1))$, $M'_{n+1}(\xi;\underline{\eta})$ is as in \eqref{na150.1}, and
\begin{equation*}
\widetilde{M}'_{n+1}(\xi;\underline{\eta}):=-\Big(\xi\partial_\xi+\sum_{j=1}^{n+1}\eta_j\partial_{\eta_j}\Big)\big(M'_{n+1}\big)(\xi;\underline{\eta}).
\end{equation*}
We notice that $M'_{n+1}(\lambda\xi;\lambda\underline{\eta})=\lambda^nM'_{n+1}(\xi;\underline{\eta})$, if $\lambda>0$. Taking the $\lambda$ derivative it follows that
\begin{equation}\label{ra15}
\widetilde{M}'_{n+1}(\xi;\underline{\eta})=-nM'_{n+1}(\xi;\underline{\eta}).
\end{equation}

The estimate on $SR_nf$ in \eqref{ra11.5} follows by the same argument as in the proof of Lemma \ref{na149}, using dyadic decompositions and the bounds \eqref{ra1}. As a general rule, we always estimate the factor that has the vector-field $S$ in $L^2$ and the remaining factors in $L^\infty$. This completes the proof of Lemma \ref{Wboundedness}.

\section{Elliptic bounds}\label{aux}

In this appendix we prove elliptic-type bounds on several multilinear expressions that appear
in the course of the proofs, mainly in the derivation of the equations done in section \ref{Equations},
and in the energy estimates in sections \ref{secEE}-\ref{secweightedlow}. 

\subsection{The spaces $O_{m,\alpha}$} We start with a definition (recall \eqref{norms0}).

\begin{definition}\label{Oterms}
Assume $\alpha\in[-2,2]$ and let $b:=-1/10$. Let $O_{1,\alpha}$ denote the set of functions $f_1\in C([0,T]:L^2)$ that satisfy the ``linear'' bounds, for any $t\in[0,T]$,
\begin{equation}\label{linb}
\langle t\rangle^{-p_0}\big\|f_1(t)\big\|_{\dot{H}^{N_0+\alpha,b}}+\langle t\rangle^{-4p_0}\big\|Sf_1(t)\big\|_{\dot{H}^{N_1+\alpha,b}}+\langle t\rangle^{1/2}\|f_1(t)\|_{\dot{W}^{N_2+\alpha,b}}\lesssim \e_1.
\end{equation}
Let $O_{2,\alpha}$ denote the set of functions $f_2\in C([0,T]:L^2)$ that satisfy the ``quadratic'' bounds
\begin{equation}\label{quab}
\langle t\rangle^{1/2-p_0}\big\|f_2(t)\big\|_{H^{N_0+\alpha}}+\langle t\rangle^{1/2-4p_0}\big\|Sf_2(t)\big\|_{H^{N_1+\alpha}}+\langle t\rangle\|f_2(t)\|_{\widetilde{W}^{N_2+\alpha}}\lesssim \e_1^2.
\end{equation}
Let $O_{3,\alpha}$ denote the set of functions $f_3\in C([0,T]:L^2)$ that satisfy the ``cubic'' bounds
\begin{equation}\label{cubb}
\langle t\rangle^{1-p_0}\big\|f_3(t)\big\|_{H^{N_0+\alpha}}+\langle t\rangle^{1-4p_0}\big\|Sf_3(t)\big\|_{H^{N_1+\alpha}}+\langle t\rangle^{11/10}\|f_3(t)\|_{\widetilde{W}^{N_2+\alpha}}\lesssim \e_1^3.
\end{equation}
\end{definition}

In other words, the generic notation $O_{m,\alpha}$ measures (1) the degree of the function
(linear, quadratic, or cubic) represented by the exponent $m$, and
(2) the number of derivatives under control, relative to the Hamiltonian variables
$|\partial_x|h,|\partial_x|^{1/2}\phi$ which correspond to $\alpha=0$.
Notice that $O_{3,\alpha}\subseteq O_{2,\alpha}$;
in the linear case $O_{1,\alpha}$ we make slightly stronger assumptions on the low frequency part of the $L^2$ norms.

We will often use the following lemma to estimate products and paraproducts of functions.

\begin{lemma}\label{OtermsProd}
Assume that $m_1,m_2:\mathbb{R}\times\mathbb{R}\to\mathbb{C}$ are continuous functions with
\begin{equation}\label{prods0}
\begin{split}
&\mathrm{supp}\,m_1\subseteq\{(\xi,\eta)\in\mathbb{R}^2:\,|\xi-\eta|\leq |\eta|/2^4\},\\
&\mathrm{supp}\,m_2\subseteq\{(\xi,\eta)\in\mathbb{R}^2:\,|\xi-\eta|/|\eta|\in[2^{-20},2^{20}]\}.
\end{split}
\end{equation}
Assume that
\begin{equation}\label{prods0.5}
\big\|\mathcal{F}^{-1}[m_1(\xi,\eta)\varphi_k(\xi)]\big\|_{L^1}\lesssim 1,\qquad \big\|\mathcal{F}^{-1}[m_2(\xi,\eta)\varphi_{k_1}(\xi-\eta)\varphi_{k_2}(\eta)]\big\|_{L^1}\lesssim 1,
\end{equation}
for any $k,k_1,k_2\in\mathbb{Z}$. Let $\widetilde{m}_l:=-(\xi\partial_\xi+\eta\partial_\eta)m_l$, see \eqref{lemcomm2}, and assume that the multipliers $\widetilde{m}_1,\widetilde{m}_2$ satisfy the bounds \eqref{prods0.5} as well. Let $M_1$ and $M_2$ denote the bilinear  operators associated to $m_1$ and $m_2$, see \eqref{Moutm}. Assume that $m,n\in\{1,2,3\}$. Then, for any $\alpha\in[-2,2]$,
\begin{equation}\label{prods1}
\text{ if }\,f\in O_{m,-2},\quad g\in O_{n,\alpha}\qquad\text{ then }\qquad M_1(f,g)\in O_{\min(m+n,3),\alpha}\cap O_{1,\alpha},
\end{equation}
and, for any $\alpha\in[-2,0]$,
\begin{equation}\label{prods1.1}
\text{ if }\,f\in O_{m,\alpha},\quad g\in O_{n,\alpha}\qquad\text{ then }\qquad M_2(f,g)\in O_{\min(m+n,3),\alpha+2}\cap O_{1,\alpha+2}.
\end{equation}
\end{lemma}

\begin{proof}
We only show in detail how to prove the bounds \eqref{prods1} and \eqref{prods1.1} when $m=n=1$, since the other cases are similar. We assume that $t\in[0,T]$ is fixed and sometimes drop it from the notation. Using Lemma \ref{touse}, the definition \eqref{linb}, and  the assumption \eqref{prods0.5},
\begin{equation*}
\big\|P_kM_1(f,g)\big\|_{L^p}\lesssim \|P'_kg\|_{L^p}\|P_{\leq k}f\|_{L^\infty}\lesssim \e_1\langle t\rangle^{-1/2}\|P'_kg\|_{L^p}
\end{equation*}
for $p\in\{2,\infty\}$ and for any $k\in\mathbb{Z}$. Therefore
\begin{equation}\label{prods5}
\begin{split}
\big\|M_1(f,g)\big\|_{\dot{H}^{N_0+\alpha,b}}&\lesssim \e_1^2\langle t\rangle^{-1/2+p_0},\\
\big\|M_1(f,g)\big\|_{\dot{W}^{N_2+\alpha,b}}&\lesssim \e_1^2\langle t\rangle^{-1}.
\end{split}
\end{equation}
Similarly, using Lemma \ref{touse}, the definition \eqref{linb}, and  the assumption \eqref{prods0.5},
\begin{equation*}
\big\|P_kM_1(f,Sg)\big\|_{L^2}\lesssim \|P'_kSg\|_{L^2}\|P_{\leq k}f\|_{L^\infty}\lesssim \e_1\langle t\rangle^{-1/2}\|P'_kSg\|_{L^2},
\end{equation*}
and
\begin{equation*}
\big\|P_kM_1(Sf,g)\big\|_{L^2}\lesssim \|P'_kg\|_{L^\infty}\|P_{\leq k}Sf\|_{L^2}\lesssim \e_1\langle t\rangle^{4p_0}\e_1\langle t\rangle^{-1/2}\min(2^{-2bk},2^{-k(N_2+\alpha)}).
\end{equation*}
Therefore
\begin{equation}\label{prods7}
\big\|M_1(f,Sg)\big\|_{\dot{H}^{N_1+\alpha,b}}+\big\|M_1(Sf,g)\big\|_{\dot{H}^{N_1+\alpha,b}}\lesssim \e_1^2\langle t\rangle^{-1/2+4p_0}.
\end{equation}
Moreover, since $\widetilde{m}_1$ satisfies the same bounds as $m_1$, we have as in the proof of \eqref{prods5},
\begin{equation}\label{prods8}
\big\|\widetilde{M}_1(f,g)\big\|_{\dot{H}^{N_1+\alpha,b}}\lesssim\e_1^2\langle t\rangle^{-1/2+4p_0}.
\end{equation}
The desired identities \eqref{prods1} follow from the bounds \eqref{prods5}--\eqref{prods8} and the identity \eqref{lemcomm}.

We consider now the operator $M_2$ and prove \eqref{prods1.1} when $m=n=1$. We estimate first, using as before Lemma \ref{touse}, the definition \eqref{linb}, and  the assumption \eqref{prods0.5},
\begin{equation}\label{prods10}
\big\|M_2(f,g)\big\|_{L^2}\lesssim \sum_{|k_1-k_2|\leq 30}\|P_{k_1}f\|_{L^2}\|P_{k_2}g\|_{L^\infty}\lesssim\e_1^2\langle t\rangle^{-1/2+p_0}
\end{equation}
and, similarly,
\begin{equation}\label{prods11}
\big\|M_2(f,g)\big\|_{L^\infty}\lesssim \e_1^2\langle t\rangle^{-1},
\end{equation}
\begin{equation}\label{prods11.5}
\big\|M_2(f,g)\big\|_{L^1}\lesssim \e_1^2\langle t\rangle^{2p_0},
\end{equation}
\begin{equation}\label{prods12}
\big\|M_2(Sf,g)\big\|_{L^2}+\big\|M_2(f,Sg)\big\|_{L^2}+\big\|\widetilde{M}_2(f,g)\big\|_{L^2}\lesssim \e_1^2\langle t\rangle^{-1/2+4p_0},
\end{equation}
\begin{equation}\label{prods12.5}
\big\|M_2(Sf,g)\big\|_{L^1}+\big\|M_2(f,Sg)\big\|_{L^1}+\big\|\widetilde{M}_2(f,g)\big\|_{L^1}\lesssim \e_1^2\langle t\rangle^{8p_0}.
\end{equation}
These estimates and Sobolev embedding (in the form $\|P_kh\|_{L^2}\lesssim 2^{k/2}\|P_kh\|_{L^1}$) provide the desired estimates on low frequencies,
\begin{equation*}
\begin{split}
\langle t\rangle^{1/2-p_0}\big\|P_{\leq 4}M_2(f,g)\big\|_{H^{N_0+2}}\negmedspace+\langle t\rangle^{1/2-4p_0}\big\|P_{\leq 4}SM_2(f,g)\big\|_{H^{N_1+2}}\negmedspace+\langle t\rangle\|P_{\leq 4}M_2(f,g)\|_{\widetilde{W}^{N_2+2}}\lesssim \e_1^2
\end{split}
\end{equation*}
and
\begin{equation*}
\begin{split}
\langle t\rangle^{-p_0}\big\|P_{\leq 4}M_2(f,g)\big\|_{\dot{H}^{N_0+2,b}}&+\langle t\rangle^{-4p_0}\big\|P_{\leq 4}SM_2(f,g)\big\|_{\dot{H}^{N_1+2,b}}\\
&+\langle t\rangle^{1/2}\|P_{\leq 4}M_2(f,g)\|_{\dot{W}^{N_2+2,b}}\lesssim \e_1^2.
\end{split}
\end{equation*}

To estimate the high frequencies we notice that, for $k\geq 0$ and $p\in\{2,\infty\}$
\begin{equation*}
\big\|P_kM_2(f,g)\big\|_{L^p}\lesssim \sum_{k_1\geq k-30}\big\|P_{k_1}f\big\|_{L^p}\cdot\e_1\langle t\rangle^{-1/2}2^{-k_1(N_2+\alpha)}.
\end{equation*}
Therefore
\begin{equation}\label{prods13}
\begin{split}
\big\|P_{\geq 0}M_2(f,g)\big\|_{H^{N_0+\alpha+2}}&\lesssim \e_1^2\langle t\rangle^{-1/2+p_0},\\
\sum_{k\geq 0}2^{(N_2+\alpha+2)k}\big\|P_kM_2(f,g)\big\|_{L^\infty}&\lesssim \e_1^2\langle t\rangle^{-1}.
\end{split}
\end{equation}
Similarly
\begin{equation*}
\begin{split}
\big\|P_{\geq 0}M_2(Sf,g)\big\|_{H^{N_1+\alpha+2}}&+\big\|P_{\geq 0}M_2(f,Sg)\big\|_{H^{N_1+\alpha+2}}+\big\|P_{\geq 0}\widetilde{M}_2(f,g)\big\|_{H^{N_1+\alpha+2}}\lesssim \e_1^2\langle t\rangle^{-1/2+4p_0}.
\end{split}
\end{equation*}
The desired conclusions in \eqref{prods1.1} follow.
\end{proof}

\subsection{Linear, quadratic, and cubic bounds}\label{CubicBounds0}

In this subsection we prove the elliptic bounds on the quadratic and the cubic terms used implicitly to justify the calculations in section \ref{Equations}. 
Recall the formulas derived in section \ref{Equations},
\begin{equation}\label{bvd3}
\begin{split}
& \sigma=(1+h_x^2)^{-3/2}-1,\qquad \gamma=(1+h_x^2)^{-3/4}-1,\qquad p_1=\gamma,\qquad p_0=-(3/4)\gamma_x,
\\
& B= \frac{G(h)\phi + h_x\phi_x}{(1+h_x^2)}, \qquad V = \phi_x-Bh_x,\qquad \omega=\phi-T_BP_{\geq 1}h.
\end{split}
\end{equation}
Recall the bounds \eqref{bing1}.
With the notation in Definition \ref{Oterms}, and using also Proposition \ref{ra102} and Lemma \ref{OtermsProd}, we have
\begin{equation}\label{bvd1}
h_x\in O_{1,0},\qquad \phi_x\in O_{1,-1/2},\qquad G(h)\phi\in O_{1,-1/2},\qquad B\in O_{1,-1/2},\qquad V\in O_{1,-1/2}.
\end{equation}
Using \eqref{bvd1} and Lemma \ref{OtermsProd}, it follows that
\begin{equation}\label{bvd4}
\begin{split}
&\sigma\in O_{2,0},\quad \gamma\in O_{2,0},\quad p_1\in O_{2,0},\quad p_0=O_{2,-1},\\
&\omega_x-\phi_x\in O_{2,0},\quad |\partial_x|\omega-|\partial_x|\phi\in O_{2,0},\quad V-\phi_x\in O_{2,-1/2},\quad B-G(h)\phi\in O_{2,-1/2}.
\end{split}
\end{equation}
Using Proposition \ref{ra102} it follows that
\begin{equation}\label{bvd4.5}
G(h)\phi-|\partial_x|\phi\in O_{2,0},\qquad B - |\partial_x|\omega \in O_{2,-1/2} ,\qquad G_2(h,\phi)\in O_{2,1},\qquad G_{\geq 3}\in O_{3,1}.
\end{equation}
The main evolution system \eqref{on1} shows that
\begin{equation*}
\partial_th\in O_{1,-1/2},\qquad \partial_t\phi\in O_{1,-1}.
\end{equation*}
The formula \eqref{na35}, that is $(I+T_1)\psi=(-iH_0+T_2)\phi$, now shows that
\begin{equation*}
\partial_t(G(h)\phi)-iH_0\partial_t\phi_x\in O_{2,-2}.
\end{equation*}
Using again the main system \eqref{on1} and previous identities,
\begin{equation}\label{bvd4.7}
\begin{split}
&\partial_th-|\partial_x|\phi\in O_{2,0},\qquad \partial_t\phi-h_{xx}\in O_{2,-1},\qquad\partial_t(G(h)\phi)+|\partial_x|^3h\in O_{2,-2},\\
&\partial_tB+|\partial_x|^3h\in O_{2,-2},\qquad \partial_tV-\partial_x^3h\in O_{2,-2}.
\end{split}
\end{equation}
Finally, the formulas \eqref{symm5} and \eqref{symmquadterms} show that
\begin{equation}\label{bvd4.8}
u\in O_{1,0},\qquad \mathcal{N}_2(h,\omega)\in O_{2,0}, \qquad V-\frac{i}{2}\partial_x|\partial_x|^{-1/2}(u-\overline{u})\in O_{2,-1/2}.
\end{equation}

The following proposition provides suitable bounds on the cubic remainders.

\begin{proposition}\label{CubicSummary}
Let $(G_{\geq 3},\Omega_{\geq 3})$, $U_{\geq 3}$, $\mathcal{O}_{W_k}$, and $\mathcal{O}_{Z_k}$ denote the cubic remainders in Proposition \ref{firstred}, Proposition \ref{proequ}, Proposition \ref{proeqW}, and Lemma \ref{lemeqZ} respectively. Then, with $N=3k/2$,
\begin{equation}\label{cubsum1}
G_{\geq 3}\in O_{3,1}, \qquad \Omega_{\geq 3}\in O_{3,1/2},\qquad U_{\geq 3} \in |\partial_x|^{1/2} O_{3,1/2},
\end{equation}
\begin{equation}\label{cubsum2}
\begin{split}
\langle t\rangle^{1-p_0}\|\mathcal{O}_W(t)\|_{\dot{H}^{N_0-N,-1/2}}+\langle t\rangle^{1-4p_0}\|S\mathcal{O}_W(t)\|_{\dot{H}^{N_1-N,-1/2}}&\lesssim \e_1^3,\qquad\text{ if }N\leq N_1,\\
\langle t\rangle^{1-p_0}\|\mathcal{O}_W(t)\|_{\dot{H}^{N_0-N,-1/2}}&\lesssim \e_1^3,\qquad\text{ if }N\in[N_1,N_0],
\end{split}
\end{equation}
and
\begin{equation}\label{cubsum3}
\langle t\rangle^{1-4p_0}\big\|\mathcal{O}_Z(t)\big\|_{\dot{H}^{0,-1/2}}\lesssim \e_1^3.
\end{equation}
\end{proposition}

\begin{proof}
The desired conclusion $G_{\geq 3}=O_{3,1}$ was already proved in Proposition \ref{ra102}.
We examine now the term $\Omega_{\geq 3}$ in Proposition \ref{firstred}.
Inspecting the calculations that lead from \eqref{dtomega1} to \eqref{on10}, we see that
\begin{align}\label{bvd6}
\begin{split}
\Omega_{\geq 3} & = \sum_{j=1}^5 \Omega_{\geq 3}^j ,
\\
\Omega_{\geq 3}^1 & := \frac{\partial_x^2h}{(1+h_x^2)^{3/2}} - \big( h_{xx} + \partial_x T_\sigma \partial_x h \big) ,
\\
\Omega_{\geq 3}^2 & := \big[ R(B,B) - R(|\partial_x|\omega, |\partial_x|\omega)\big]/2
  - \big[ R(V,V) - R(\partial_x \omega, \partial_x\omega)\big]/2 - R(V, Bh_x) ,
\\
\Omega_{\geq 3}^3 & := -T_{\partial_tB}P_{\geq 1}h -T_{|\partial_x|^3 h}P_{\geq 1}h ,
\\
\Omega_{\geq 3}^4 & := T_B P_{\leq 0} G(h)\phi - T_{|\partial_x|\omega} P_{\leq 0} |\partial_x|\omega,
\\
\Omega_{\geq 3}^5 & := T_B(V h_x) - T_{Bh_x}V - T_V \partial_x (T_B P_{\geq 1} h) .
\end{split}
\end{align}

We examine the terms and use \eqref{bvd1}--\eqref{bvd4.7} and Lemma \ref{OtermsProd} to see
\begin{equation*}
\Omega_{\geq 3}^1=O_{3,1},\quad \Omega_{\geq 3}^2=O_{3,1},\quad\Omega_{\geq 3}^3=O_{3,1},\quad\Omega_{\geq 3}^4=O_{3,1}.
\end{equation*}
Moreover, an argument similar to the proof of \eqref{prods1} in Lemma \ref{OtermsProd} shows that
\begin{equation}\label{bvd7}
T_{f_1f_2}g-T_{f_1}T_{f_2}g=O_{3,1}\qquad \text{ if }\qquad f_1=O_{1,-1},\,f_2=O_{1,-1},\,g=O_{1,-1}.
\end{equation}
Therefore, we can rewrite
\begin{equation*}
\Omega_{\geq 3}^5=\big[T_BT_{h_x}V - T_{Bh_x}V\big]+T_B\big(R(V,h_x)\big)+\big[T_BT_Vh_x-T_VT_BP_{\geq 1}h_x\big]- T_V T_{\partial_xB} P_{\geq 1} h,
\end{equation*}
and conclude that $\Omega_{\geq 3}^5=O_{3,1}$. Therefore $\Omega_{\geq 3}=O_{3,1}$ as desired.

We now move on to examine the cubic terms appearing in \eqref{equ} in Proposition \ref{proequ}.
For this we inspect the computations in the proof of Proposition \ref{prosymm} to retrieve all cubic terms that
have been incorporated in $U_{\geq 3}$. We find
\begin{align}
\label{U3}
U_{\geq 3} := \sum_{j=1}^5 U_{\geq 3}^j ,
\end{align}
where
\begin{align}
\label{U31}
\begin{split}
U_{\geq 3}^1 & := |\partial_x| G_{\geq 3} - i |\partial_x|^{1/2} \Omega_{\geq 3}
  + T_{\partial_t p_1} P_{\geq 1} |\partial_x| h + T_{\partial_t p_0} P_{\geq 1} |\partial_x|^{-1} \partial_x h
\\
& + T_{p_1} P_{\geq 1} |\partial_x|\big( \partial_t h - |\partial_x|\o + \partial_xT_V h \big)
  + T_{p_0} P_{\geq 1} |\partial_x|^{-1} \partial_x \big( \partial_t h - |\partial_x|\o + \partial_xT_V h \big) ,
\end{split}
\\
\label{U32}
\begin{split}
U_{\geq 3}^2 & := - \big[ T_{p_1} P_{\geq 1} |\partial_x| , \partial_xT_V \big] h
  - \big[ T_{p_0} P_{\geq 1} |\partial_x|^{-1} \partial_x , \partial_xT_V \big] h
  + \big[ |\partial_x| , \partial_xT_{- V +\partial_x \o} \big] h
\\
  & +iT_{- \partial_x V + \partial_x^2 \o} |\partial_x|^{1/2} \o
  + i\big[ |\partial_x|^{1/2} , \partial_x T_{V-\partial_x \o} \big] \o ,
\end{split}
\end{align}
\begin{align}
\label{U33}
\begin{split}
U_{\geq 3}^3 & := i|\partial_x|^{3/2} T_\sigma P_{\leq 0} |\partial_x| h
  + i\big( |\partial_x|^{3/2} T_\sigma P_{\geq 1} |\partial_x| h - T_\sigma P_{\geq 1} |\partial_x|^{5/2} h
  + \frac{3}{2} T_{\partial_x \sigma} P_{\geq 1} |\partial_x|^{1/2} \partial_x h \big)
\\
& - i\big( |\partial_x|^{3/2} T_{p_1} P_{\geq 1} |\partial_x| h - T_{p_1} P_{\geq 1} |\partial_x|^{5/2} h
  + \frac{3}{2} T_{\partial_x p_1} P_{\geq 1} |\partial_x|^{1/2} \partial_x h \big)
\\
& - i\big( |\partial_x|^{3/2} T_{p_0} P_{\geq 1} |\partial_x|^{-1}\partial_x h - T_{p_0} P_{\geq 1} |\partial_x|^{1/2}\partial_x h \big) ,
\end{split}
\end{align}
\begin{align}
\label{U34}
\begin{split}
U_{\geq 3}^4 & := -\frac{3i}{4} T_{\partial_x \gamma} P_{\geq 1} |\partial_x|^{-1/2}\partial_x T_{p_0} |\partial_x|^{-1}\partial_x h
\\
& - i \big( T_{\gamma} P_{\geq 1} |\partial_x|^{3/2} T_{p_1} P_{\geq 1} |\partial_x| h
  - T_{\gamma p_1} P_{\geq 1} |\partial_x|^{5/2} h
  + \frac{3}{2} T_{\gamma \partial_x p_1} P_{\geq 1} |\partial_x|^{1/2} \partial_x h \big)
\\
& - i \big( T_{\gamma} P_{\geq 1} |\partial_x|^{3/2} T_{p_0} P_{\geq 1} |\partial_x|^{-1}\partial_x h
  - T_{\gamma p_0} P_{\geq 1} |\partial_x|^{1/2}\partial_x h \big)
\\
& + \frac{3i}{4}\big( T_{\partial_x\gamma} P_{\geq 1} |\partial_x|^{-1/2}\partial_x T_{p_1} P_{\geq 1} |\partial_x| h
  - T_{\partial_x\gamma p_1} P_{\geq 1} |\partial_x|^{1/2}\partial_x h \big),
\end{split}
\end{align}
\begin{align}
\label{U35}
\begin{split}
U_{\geq 3}^5&:=[|\partial_x|,\partial_xT_{\partial_x\omega}](\widetilde{h}-h)+i|\partial_x|^{1/2}[T_{|\partial_x|^3\widetilde{h}}P_{\geq 1}\widetilde{h}-T_{|\partial_x|^3h}P_{\geq 1}h]+|\partial_x|M_2(\omega,h-\widetilde{h}),
\end{split}
\end{align}
where $\widetilde{h}=(1/2)|\partial_x|^{-1}(u+\overline{u})$, see \eqref{eqhou}. The first term comes from the remainders in the first chain of identities in the proof of Proposition \ref{prosymm}.
This is the same remainder that appears also at the end of \eqref{symm20}.
The second term above comes from \eqref{symm25}.
The term $U_{\geq 3}^3$ contains cubic terms that are discarded in \eqref{symm50}.
Quartic terms that have been discarded in \eqref{symm50},
and the last term on the right-hand side of \eqref{symm30}, are in $U_{\geq 3}^4$.
The term $U_{\geq 3}^5$ comes from the formulas after \eqref{eqhou}.

We examine the formulas \eqref{U3}--\eqref{U34} and use Lemma \ref{OtermsProd} and \eqref{bvd7} to conclude that $U_{\geq 3} = |\partial_x|^{1/2} O_{3,1/2}$, as desired. This completes the proof of \eqref{cubsum1}.

We turn now to the proof of \eqref{cubsum2}. We examine the formula \eqref{eqW5} and observe that
\begin{align}\label{hft0}
\big[ \pa_t, \D^k \big] = \sum_{j=0}^{k-1} \D^j [\pa_t, \D] \D^{k-j-1} = \sum_{j=0}^{k-1} \D^j [\partial_t,\Sigma_\gamma] \D^{k-j-1}.
\end{align}
In view of \eqref{bvd4.7} and Lemma \ref{OtermsProd}, $[\partial_t,\Sigma_\gamma]$ is an operator of order $3/2$ that transforms linear terms into cubic terms. Moreover
\begin{equation*}
\mathcal{D}^k-|\partial_x|^N=\sum'_{P_1,\ldots,P_k\in\{|\partial_x|^{3/2},\Sigma_\gamma\}}P_1\ldots P_k,
\end{equation*}
where the sum above is taken over all possible choices of operators $P_1,\ldots,P_k\in\{|\partial_x|^{3/2},\Sigma_\gamma\}$,
not all of them equal to $|\partial_x|^{3/2}$.
Therefore $\mathcal{D}^k-|\partial_x|^N$ is an operator of order $3k/2$ that transforms linear terms into cubic terms. Finally, using Lemma \ref{lembounda}, Lemma \ref{lemboundb}, and Lemma \ref{OtermsProd}, $\mathcal{N}_u$ is a quadratic term that does not lose derivatives, $\mathcal{N}_u\in O_{2,0}$. 
The desired conclusion follows by applying elliptic estimates, as in the proof of Lemma \ref{OtermsProd}, to the terms in the last two lines of \eqref{eqW5}.

The proof of \eqref{cubsum3} is similar, using the formula \eqref{NeqZ26}, Lemma \ref{OtermsProd}, and \eqref{hft0}.
\end{proof}

The next lemma gives bounds on the nonlinear terms in Propositions \ref{proequ}, \ref{proeqW} and \ref{lemeqZ}.

\begin{lemma}\label{lemboundN_u}
With the notation in definition \ref{Oterms} we have
\begin{align}
\label{boundLu}
& \partial_t u - i\Lambda u \in |\partial_x|^{1/2} O_{2,-1} ,
\\
\label{boundN_u}
& \mathcal{N}_u \in |\partial_x|^{3/2} O_{2,3/2}. 
\end{align}
Moreover
\begin{align}
\label{estLSu}
\begin{split}
{\| P_k (\partial_t - i\Lambda) u \|}_{L^2} & \lesssim \e_1^2 2^{k/2}
  {(1+t)}^{-1/2 + p_0} (2^{k/2} + (1+t)^{-1/2}),\\
{\| P_k (\partial_t - i\Lambda) Su \|}_{L^2} & \lesssim \e_1^2 2^{k/2}
  {(1+t)}^{-1/2 + 4p_0} (2^{k/2} + (1+t)^{-1/2}) .
\end{split}
\end{align}

Furthermore, with notation of Proposition \ref{lemeqZ}, let $Z=S\mathcal{D}^{2N_1/3}u$ and
\begin{align}
\label{defNZ}
  \mathcal{N}_Z :=  \mathcal{N}_{Z,1} +  \mathcal{N}_{Z,2} +  \mathcal{N}_{Z,3}.
\end{align}
Then, we have
\begin{align}
\label{estLZ}
{\| P_k (\partial_t - i\Lambda)Z \|}_{L^2} & \lesssim \e_1^2 2^{k/2} 2^{\max(k,0)} {(1+t)}^{-1/2 + 4p_0} ,
\\
\label{estNZ}
{\| P_k \mathcal{N}_Z \|}_{L^2} & \lesssim \e_1^2 2^{\min(k,0)} {(1+t)}^{-1/2 + 4p_0} .
\end{align}

\end{lemma}

\begin{proof}
The bounds \eqref{boundLu}--\eqref{boundN_u} follow from the formulas \eqref{equ}--\eqref{equN}, the bounds on the symbols in Lemmas \ref{lembounda} and \ref{lemboundb}, and Lemma \ref{OtermsProd}. The bounds \eqref{estLSu}, which provide a more refined version of the bounds \eqref{boundLu} when $k\leq 0$, follow by decomposing the nonlinearity into quadratic and cubic components. 

The estimates \eqref{estLZ}--\eqref{estNZ} follow from the formulas in Proposition \ref{lemeqZ}, using the same symbols bounds as before and Lemma \ref{touse}.
\end{proof}


\newcommand{\comment}[1]{\vskip.3cm
\fbox{%
\parbox{0.93\linewidth}{\footnotesize #1}}
\vskip.3cm}

\def\toverify{\vskip10pt \noindent {\bf F:} {\it Need to verify this ...} \vskip10pt}
\def\tofill{\vskip30pt $\cdots$ To fill in $\cdots$ \vskip30pt}


\subsection{Semilinear expansions}\label{CubicBounds2}

Recall the notation for trilinear operators
\begin{align}
 \label{moutMtri}
\mathcal{F}\big[ M(f_1,f_2,f_3) \big](\xi) := \frac{1}{4\pi^2} \iint_{\R\times\R} m(\xi,\eta,\sigma)
  \what{f_1}(\xi-\eta) \what{f_2}(\eta-\sigma) \what{f_3}(\sigma) \, d\eta d\sigma .
\end{align}
The starting point for the semilinear analysis in sections \ref{secdecay} and \ref{secprmainlem} leading to pointwise decay,
is the equation \eqref{equ100} stated at the beginning of section \ref{secdecay}.
In this subsection we show how to derive \eqref{equ100} starting from the main equation \eqref{on1} in Proposition \ref{MainProp}. 
In particular, to derive \eqref{equ100} we will need to expand the nonlinearity as a sum of quadratic, cubic, and higher order terms. We point out that at this point we are not interested in the potential loss of derivatives. 
The following is the main result in this subsection.

\begin{lemma}\label{lemcubicfinal}
Let $U=|\partial_x| h-i|\partial_x|^{1/2}\phi$, where $(h,\phi)$ are as in Proposition \ref{MainProp}. Then we can write
\begin{align}
\label{lemcubicfinal1}
\partial_t U - i |\partial_x|^{3/2} U= \mathcal{Q}_U + \mathcal{C}_U + \mathcal{R}_{\geq 4},
\end{align}
where the following claims hold true:

\setlength{\leftmargini}{1.5em}
\begin{itemize}

\smallskip
\item The quadratic nonlinear terms are
\begin{equation}\label{lemcubicfinal2}
\mathcal{Q}_U=\sum_{(\eps_1\eps_2)\in\{(++),(+-),(--)\}}Q_{\eps_1\eps_2}(U_{\eps_1},U_{\eps_2}),
\end{equation}
where $U_+=U$, $U_-=\overline{U}$, and the operators $Q_{++},Q_{+-},Q_{--}$ are defined by the symbols
\begin{equation}\label{cxz1}
\begin{split}
q_{++}(\xi,\eta)&:=\frac{i|\xi|(\xi\eta-|\xi||\eta|)}{8|\eta|^{1/2}|\xi-\eta|}+\frac{i|\xi|(\xi(\xi-\eta)-|\xi||\xi-\eta|)}{8|\eta||\xi-\eta|^{1/2}}+\frac{i|\xi|^{1/2}(\eta(\xi-\eta)+|\eta||\xi-\eta|)}{8|\eta|^{1/2}|\xi-\eta|^{1/2}},\\
q_{+-}(\xi,\eta)&:=-\frac{i|\xi|(\xi\eta-|\xi||\eta|)}{4|\eta|^{1/2}|\xi-\eta|}+\frac{i|\xi|(\xi(\xi-\eta)-|\xi||\xi-\eta|)}{4|\eta||\xi-\eta|^{1/2}}-\frac{i|\xi|^{1/2}(\eta(\xi-\eta)+|\eta||\xi-\eta|)}{4|\eta|^{1/2}|\xi-\eta|^{1/2}},\\
q_{--}(\xi,\eta)&:=-\frac{i|\xi|(\xi\eta-|\xi||\eta|)}{8|\eta|^{1/2}|\xi-\eta|}-\frac{i|\xi|(\xi(\xi-\eta)-|\xi||\xi-\eta|)}{8|\eta||\xi-\eta|^{1/2}}+\frac{i|\xi|^{1/2}(\eta(\xi-\eta)+|\eta||\xi-\eta|)}{8|\eta|^{1/2}|\xi-\eta|^{1/2}}.
\end{split}
\end{equation}

\smallskip
\item  The cubic terms have the form
\begin{align}
\label{lemcubicfinal3}
\mathcal{C}_U := M_{++-}(U, U, \bar{U}) + M_{+++}(U,U,U) + M_{--+}(\bar{U}, \bar{U}, U) + M_{---}(\bar{U}, \bar{U}, \bar{U}) ,
\end{align}
with purely imaginary symbols $m_{\iota_1\iota_2\iota_3}$ such that
\begin{align}
\label{lemcubicfinal4}
\begin{split}
& {\big\| \mathcal{F}^{-1}\big[m_{\iota_1\iota_2\iota_3}(\xi,\eta,\sigma)\cdot 
  \varphi_k(\xi)\varphi_{k_1}(\xi-\eta)\varphi_{k_2}(\eta-\sigma)\varphi_{k_3}(\sigma)\big] \big\|}_{L^1} \lesssim 2^{k/2} 2^{\max(k_1,k_2,k_3)}
\end{split}
\end{align}
for all $(\iota_1\iota_2\iota_3) \in \{ (++-),(--+),(+++),(---) \}$.
Moreover, with $d_1 = 1/16$,
\begin{align}
\label{lemcubicfinal5}
m_{++-} (\xi,0,-\xi) = id_1|\xi|^{3/2}.
\end{align}

\smallskip
\item $\mathcal{R}_{\geq 4}$ is a quartic remainder satisfying
\begin{align}
\label{lemcubicfinalR_4}
{\| \mathcal{R}_{\geq 4} \|}_{L^2} +{\| S \mathcal{R}_{\geq 4} \|}_{L^2} \lesssim \e_1^4 \langle t \rangle^{-5/4}. 
\end{align}
Moreover
\begin{align}
 \label{equR4}
\mathcal{C}_U+\mathcal{R}_{\geq 4}\in |\partial_x|^{1/2} O_{3,-1} .
\end{align}
\end{itemize}
\end{lemma}

Let $O_4$ denote the set of functions $g\in C([0,T]:L^2)$ that satisfy the "quartic" bounds
\begin{equation}\label{quar}
 {\| g(t)\|}_{L^2} +{\| S g(t)\|}_{L^2} \lesssim \e_1^4 \langle t \rangle^{-5/4},
\end{equation}
for any $t\in[0,T]$, compare with \eqref{lemcubicfinalR_4}. It is easy to see, using the same argument as in the proof of Lemma \ref{OtermsProd}, that
\begin{equation}\label{quar2}
M(O_{1,-2},O_{3,-2})\subseteq O_4,\qquad M(O_{2,-2},O_{2,-2})\subseteq O_4
\end{equation}
for any bilinear operator $M$ associated to a multiplier $m$ satisfying
\begin{equation}\label{quar3}
\|m^{k,k_1,k_2}\|_{S^\infty}+\|\widetilde{m}^{k,k_1,k_2}\|_{S^\infty}\lesssim 1,\qquad \text{ for any }k,k_1,k_2\in\mathbb{Z},
\end{equation} 
where $\widetilde{m}(\xi,\eta)=-(\xi\partial_\xi+\eta\partial_\eta)m(\xi,\eta)$. 

To identify the cubic terms we need the following more precise expansion of the Dirichlet--Neumann map. This follows from the formula \eqref{na35} (recall that $G(h)\phi=-\psi_x$) and the proof of Lemma \ref{boundedness2} (see, in particular, \eqref{ra11.5}, \eqref{na13.3}, \eqref{na15}, and \eqref{na19}). 

\begin{lemma}\label{DNexpansion}
With $(h,\phi)$ as in Proposition \ref{MainProp}, we have
\begin{equation}\label{quar10}
G(h)\phi=|\partial_x|\phi+DN_2[h,\phi]+DN_3[h,h,\phi]+DN_{\geq 4},
\end{equation} 
where
\begin{equation}\label{quar11}
\mathcal{F}\{DN_2[h,\phi]\}=\frac{1}{2\pi}\int_{\mathbb{R}}\widehat{h}(\xi-\eta)\widehat{\phi}(\eta)n_2(\xi,\eta)\,d\eta,\qquad n_2(\xi,\eta)=\xi\cdot\eta-|\xi||\eta|,
\end{equation}
\begin{equation}\label{quar12}
\begin{split}
&\mathcal{F}\{DN_3[h,h,\phi]\}=\frac{1}{4\pi^2}\int_{\mathbb{R}}\widehat{h}(\xi-\eta)\widehat{h}(\eta-\sigma)\widehat{\phi}(\sigma)n_3(\xi,\eta,\sigma)\,d\eta,\\
&\qquad n_3(\xi,\eta,\sigma)=\frac{|\xi||\sigma|}{2}\big(|\eta|+|\xi+\sigma-\eta|-|\xi|-|\sigma|\big),
\end{split}
\end{equation}
and, for any $t\in[0,T]$,
\begin{equation}\label{quar13}
\|DN_{\geq 4}(t)\|_{H^{N_0}}+\|S(DN_{\geq 4})(t)\|_{H^{N_1}}\lesssim \e_1^4(1+t)^{-5/4}.
\end{equation}
\end{lemma}

We can now prove Lemma \ref{lemcubicfinal}.

\begin{proof}[Proof of Lemma \ref{lemcubicfinal}] We start from the equations \eqref{on1},
\begin{equation*}
\partial_th = G(h)\phi,\qquad \partial_t\phi = \dfrac{\partial_x^2 h}{(1+h_x^2)^{3/2}}
  - \dfrac{1}{2}\phi_x^2 + \dfrac{(G(h)\phi+h_x\phi_x)^2}{2(1+h_x^2)}.
\end{equation*}
Letting $U=|\partial_x| h-i|\partial_x|^{1/2}\phi$, it follows that
\begin{equation*}
\partial_tU-i|\partial_x|^{3/2}U=\mathcal{Q}_U+\mathcal{C}_U+\mathcal{R}_{\geq 4},
\end{equation*}
where with $DN_2=DN_2[h,\phi]$, $DN_3=DN_3[h,h,\phi]$,
\begin{equation}\label{cxz2}
\begin{split}
\mathcal{Q}_U&:=|\partial_x|DN_2+\frac{i}{2}|\partial_x|^{1/2}\big[\phi_x^2-(|\partial_x|\phi)^2\big],\\
\mathcal{C}_U&:=|\partial_x|DN_3+\frac{3i}{2}|\partial_x|^{1/2}(h_{xx}h_x^2)-i|\partial_x|^{1/2}\big[|\partial_x|\phi\cdot(DN_2+h_x\phi_x)\big],
\end{split}
\end{equation}
and
\begin{equation*}
\begin{split}
\mathcal{R}_{\geq 4}&:=|\partial_x|DN_4-i|\partial_x|^{1/2}\Big[h_{xx}\Big(\frac{1}{(1+h_x^2)^{3/2}}-1+\frac{3}{2}h_x^2\Big)\Big]\\
&-i|\partial_x|^{1/2}\Big[\dfrac{[G(h)\phi+h_x\phi_x]^2}{2(1+h_x^2)}-\frac{(|\partial_x|\phi)^2+2|\partial_x|\phi\cdot(DN_2+h_x\phi_x)}{2}\Big].
\end{split}
\end{equation*}
Notice also that
\begin{equation*}
h=\frac{1}{2}|\partial_x|^{-1}(U+\overline{U}),\qquad \phi=\frac{i}{2}|\partial_x|^{-1/2}(U-\overline{U}).
\end{equation*}

The conclusions \eqref{lemcubicfinal2} and \eqref{cxz1} follow from the formula for the quadratic term $\mathcal{Q}_U$ (for the symbols $q_{++}$ and $q_{--}$ it is convenient to symmetrize such that $Q_{++}(f,g)=Q_{++}(g,f)$ and $Q_{--}(f,g)=Q_{--}(g,f)$). The conclusions \eqref{lemcubicfinal3}--\eqref{lemcubicfinal4} for the cubic components also follow from the explicit formulas in \eqref{cxz2}. The symbols $m_{++-}, m_{+++}, m_{--+}, m_{---}$ are linear combinations, with purely imaginary coefficients, of the symbols
\begin{equation}\label{cxz40}
\begin{split}
&p_1(\xi,\eta,\sigma)=\frac{|\xi|^2|\sigma|\big(|\eta|+|\xi+\sigma-\eta|-|\xi|-|\sigma|\big)}{|\xi-\eta||\eta-\sigma||\sigma|^{1/2}},\qquad p_2(\xi,\eta,\sigma)=\frac{|\xi|^{1/2}(\xi-\eta)^2(\eta-\sigma)\sigma}{|\xi-\eta||\eta-\sigma||\sigma|},\\
&p_3(\xi,\eta,\sigma)=\frac{|\xi|^{1/2}|\xi-\eta|^{1/2}|\sigma|(|\sigma|-|\eta|)}{|\eta-\sigma||\sigma|^{1/2}}.
\end{split}
\end{equation}
More precisely, we have
\begin{equation*}
\begin{split}
m_{++-}(\xi,\eta,\sigma)&=\frac{i[-p_1(\xi,\eta,\sigma)+p_1(\xi,\eta,\eta-\sigma)+p_1(\xi,\xi-\sigma,\eta-\sigma)]}{16}\\
&+\frac{3i[p_2(\xi,\eta,\sigma)+p_2(\xi,\eta,\eta-\sigma)+p_2(\xi,\xi-\sigma,\eta-\sigma)]}{16}\\
&+\frac{i[-p_3(\xi,\eta,\sigma)+p_3(\xi,\eta,\eta-\sigma)-p_3(\xi,\xi-\sigma,\eta-\sigma)]}{8}.
\end{split}
\end{equation*}
Therefore
\begin{equation*}
m_{++-}(\xi,0,-\xi)=\frac{2i|\xi|^{3/2}}{16}-\frac{3i|\xi|^{3/2}}{16}+\frac{i|\xi|^{3/2}}{8}=\frac{i|\xi|^{3/2}}{16},
\end{equation*}
as claimed in \eqref{lemcubicfinal5}.

The bounds on the quartic term $\mathcal{R}_{\geq 4}$ follow easily from the defining formula, Lemma \ref{OtermsProd}, \eqref{quar2} and Lemma \ref{DNexpansion}.
\end{proof}

\subsubsection{The resonant value}
We now compute the resonant value $c^{++-}(\xi,0,-\xi)$, see \eqref{nf50}, 
from the formulas \eqref{nf} and \eqref{equq_0}.
Since $q_{+-}(0,x) = 0$ and $m_{++-}(\xi,0,-\xi) = id_1 |\xi|^{3/2}$ we have, using \eqref{csymbols},
\begin{align*}
ic^{++-}(\xi,0,-\xi) & =  2m_{++}(\xi,0) q_{+-}(0,-\xi)+ m_{+-}(\xi,-\xi) q_{++}(2\xi, \xi) - m_{+-}(\xi,0) q_{+-}(0, \xi)
\\
& - 2m_{--}(\xi,-\xi) q_{--}(2\xi, \xi)+ m_{++-}(\xi,0,-\xi)
\\
& =  m_{+-}(\xi,-\xi) q_{++}(2\xi, \xi) -  2m_{--}(\xi,-\xi) q_{--}(2\xi, \xi) + i d_1|\xi|^{3/2} . 
\end{align*}
Since $q_{++}(2\xi,\xi)=q_{--}(2\xi,\xi)=i|\xi|^{3/2}\sqrt{2}/4$ and $q_{+-}(\xi,-\xi)=2q_{--}(\xi,-\xi)=i|\xi|^{3/2}/4$, we obtain
\begin{align}
\label{c++-value}
\begin{split}
c^{++-}(\xi,0,-\xi) &= -\frac{q_{+-}(\xi,-\xi) q_{++}(2\xi, \xi)}{|\xi|^{3/2}-|2\xi|^{3/2}+|\xi|^{3/2}} +\frac{2q_{--}(\xi,-\xi) q_{--}(2\xi, \xi)}{|\xi|^{3/2}+|2\xi|^{3/2}+|\xi|^{3/2}} + d_1 |\xi|^{3/2}\\
  &= d_2 |\xi|^{3/2}, 
\end{split}
\end{align}
where $d_2 =-1/16$.


\vskip20pt
\begin{thebibliography}{100}


\bibitem{ABZ1}
Alazard, T., Burq, N. and Zuily, C.
\newblock On the water waves equations with surface tension.
\newblock {\em Duke Math. J.} 158 (2011), no. 3, 413-499.


\bibitem{ABZ2}
Alazard, T., Burq, N. and Zuily, C.
\newblock On the Cauchy problem for gravity water waves.
\newblock {\em Invent. Math.} 198 (2014), no. 1, 71-163. 


\bibitem{ABZ3}
Alazard, T., Burq, N. and Zuily, C.
\newblock Strichartz estimates and the Cauchy problem for the gravity water waves equations.
\newblock Preprint. {\em arXiv:1404.4276}.






\bibitem{ADa}
Alazard, T. and Delort, J.M.
\newblock Global solutions and asymptotic behavior for two dimensional gravity water waves.
\newblock {\em Ann. Sci. \'Ec. Norm. Sup\'er.} (4) 48 (2015), no. 5, 1149-1238. 

\bibitem{ADb}
Alazard, T. and Delort, J.M.
\newblock Sobolev estimates for two dimensional gravity water waves
\newblock {\em Ast\'erisque} No. 374 (2015), viii+241 pp. 


\bibitem{AlMet1}
Alazard, T. and M\'etivier, G.
\newblock Paralinearization of the Dirichlet to Neumann operator, and regularity of three-dimensional water waves.
\newblock {\em Comm. Partial Differential Equations}, 34 (2009), no. 10-12, 1632-1704.



\bibitem{AM}
Ambrose, D.M. and Masmoudi, N.
\newblock The zero surface tension limit of two-dimensional water waves.
\newblock {\em Comm. Pure Appl. Math.} 58 (2005), no. 10, 1287-1315.




\bibitem{BG}
Beyer, K. and G\"{u}nther, M.
\newblock On the Cauchy problem for a capillary drop. I. Irrotational motion.
\newblock {\em Math. Methods Appl. Sci.} 21 (1998), no. 12, 1149-1183.


\bibitem{CCFGG}
Castro, A., C\'ordoba, D., Fefferman, C., Gancedo, F. and G\'{o}mez-Serrano, J.
\newblock Finite time singularities for the free boundary incompressible Euler equations.
\newblock {\em Ann. of Math.} 178 (2013), 1061-1134.


\bibitem{CCFGL}
Castro, A., C\'ordoba, D., Fefferman, C., Gancedo, F. and L\'{o}pez-Fern\'{a}ndez, M.
\newblock Rayleigh-Taylor breakdown for the Muskat problem with applications to water waves.
\newblock {\em Ann. of Math.} 175 (2012), 909-948.


\bibitem{CHS}
Christianson, H., Hur, V. and Staffilani, G.
\newblock Strichartz estimates for the water-wave problem with surface tension.
\newblock {\em Comm. Partial Differential Equations}, 35 (2010), no. 12, 2195-2252.


\bibitem{CL}
Christodoulou, D. and Lindblad, H.
\newblock On the motion of the free surface of a liquid.
\newblock {\em Comm. Pure Appl. Math.} 53 (2000), no. 12, 1536-1602.


\bibitem{CKSTT1}
Colliander, J., Keel, M., Staffilani, G., Takaoka, H. and Tao, T.
\newblock Sharp global well-posedness for KdV and modified KdV on $\R$ and $\mathbb{T}$.
\newblock {\em  J. Amer. Math. Soc.} 16 (2003), no. 3, 705-749.


\bibitem{CKSTT2}
Colliander, J., Keel, M., Staffilani, G., Takaoka, H. and Tao, T.
\newblock Resonant decompositions and the I-method for the cubic nonlinear Schr\"odinger equation on $\R^2$.
\newblock {\em  Discrete Contin. Dyn. Syst.}, 21 (2008), no. 3, 665-686.


\bibitem{CS2}
Coutand, D. and Shkoller, S.
\newblock Well-posedness of the free-surface incompressible Euler equations with or without surface tension.
\newblock {\em  J. Amer. Math. Soc.} 20 (2007), no. 3, 829-930.


\bibitem{CSSplash}
Coutand, D. and Shkoller, S.
\newblock On the finite-time splash and splat singularities for the 3-D free-surface Euler equations.
\newblock {\em  Comm. Math. Phys.} 325 (2014), 143-183.



\bibitem{CraigLim}
Craig, W.
\newblock An existence theory for water waves and the Boussinesq and Korteweg-de Vries scaling limits.
\newblock {\em  Comm. Partial Differential Equations} 10 (1985), no. 8, 787-1003


\bibitem{CSS}
Craig, W., Sulem, C. and Sulem, P.-L.
\newblock Nonlinear modulation of gravity waves: a rigorous approach.
\newblock {\em Nonlinearity} 5 (1992), no. 2, 497-522.


\bibitem{CraSul}
Craig, W. and Sulem, C.
\newblock Numerical simulation of gravity waves.
\newblock {\em J. Comput. Physics} 108 (1993), 73-83.



\bibitem{DelortKG1d}
Delort, J.M.
\newblock Existence globale et comportement asymptotique pour l' \'equation de Klein-Gordon quasi-lin\'eaire \`a donn\'ees
petites en dimension 1.
\newblock {\em Ann. Sci. \'Ecole Norm. Sup.} 34 (2001) 1-61.

\bibitem{Delo}
Delort, J.M. 
\newblock Long-time Sobolev stability for small solutions of quasi-linear Klein-Gordon equations on the circle.
\newblock {\em Trans. Amer. Math. Soc.} 361 (2009), 4299-4365.


\bibitem{DIP}
Deng, Y., Ionescu, A. and Pausader, B.
\newblock The Euler-Maxwell system for electrons: global solutions in 2D.
\newblock Preprint (2014). {\em arXiv:1605.05340}.

\bibitem{IFL}
Fefferman, C., Ionescu, A., and Lie, V.
\newblock  On the absence of ``splash'' singularities in the case of two-fluid interfaces.
\newblock {\em Duke Math. J.} 165 (2016), no. 3, 417-462. 


\bibitem{GM}
Germain, P. and  Masmoudi, N.
\newblock Global existence for the Euler-Maxwell system. 
\newblock {\em Ann. Sci. \'Ecole Norm. Sup.}, to appear. {\em arXiv:1107.1595}.


\bibitem{GMS2}
Germain, P., Masmoudi, N. and Shatah, J.
\newblock Global solutions for the gravity surface water waves equation in dimension 3.
\newblock {\em Ann. of Math.} 175 (2012), 691-754.


\bibitem{GMSC}
Germain, P., Masmoudi, N. and Shatah, J.
\newblock Global solutions for capillary waves equation in dimension 3.
\newblock {\em Comm. Pure Appl. Math.} 68 (2015), no. 4, 625-687.




\bibitem{GNT1}
Gustafson, S., Nakanishi, K. and Tsai, T.
\newblock Scattering for the Gross-Pitaevsky equation in 3 dimensions.
\newblock {\em Comm. Contemp. Math.} 11 (2009), no. 4, 657-707.


\bibitem{HN}
Hayashi, N. and Naumkin, P.
\newblock Asymptotics for large time of solutions to the nonlinear Schr\"{o}dinger and Hartree equations.
\newblock {\em  Amer. J. Math.} 120 (1998), 369-389.






\bibitem{HITW}
Hunter, J., Ifrim, M., Tataru, D. and Wong T.
\newblock Long time solutions for a Burgers-Hilbert equation via a modified energy method.
\newblock {\em Proc. Amer. Math. Soc.} 143 (2015), 3407-3412.

\bibitem{HIT}
Hunter, J., Ifrim, M. and Tataru, D.
\newblock Two dimensional water waves in holomorphic coordinates.
\newblock {\em Comm. Math. Phys.} 346 (2016), no. 2, 483-552.


\bibitem{IT}
Ifrim, M. and Tataru, D.
\newblock Two dimensional water waves in holomorphic coordinates II: global solutions.
\newblock {\em Bull. Soc. Math. France} 144 (2016), no. 2, 369-394. 

\bibitem{IT2}
Ifrim, M. and Tataru, D.
\newblock The lifespan of small data solutions in two dimensional capillary water waves.
\newblock {\em arXiv:1406.5471}.







\bibitem{IoPu1}
Ionescu, A. and Pusateri, F.
\newblock Nonlinear fractional Schr\"{o}dinger equations in one dimension.
\newblock {\em J. Funct. Anal.} 266 (2014), 139-176.


\bibitem{IoPu2}
Ionescu, A. and Pusateri, F.
\newblock Global solutions for the gravity water waves system in 2D.
\newblock {\em Invent. Math.} 199 (2015), no. 3, 653-804. 


\bibitem{IoPunote}
Ionescu, A. and Pusateri, F.
\newblock A note on the asymptotic behavior of gravity water waves in two dimensions.
\newblock Unpublished note. https://web.math.princeton.edu/~fabiop/2dWWasym-f.pdf.


\bibitem{IoPu3}
Ionescu, A. and Pusateri, F.
\newblock Global analysis of a model for capillary water waves in 2D.
\newblock {\em Comm. Pure Appl. Math.} 69 (2016), no. 11, 2015-2071. 


\bibitem{KP}
Kato, J. and Pusateri, F.
\newblock A new proof of long range scattering for critical nonlinear Schr\"{o}dinger equations.
\newblock {\em Diff. Int. Equations}, 24 (2011), no. 9-10, 923-940.

\bibitem{K1}
Klainerman, S.
\newblock The null condition and global existence for systems of wave equations.
\newblock {\em Nonlinear systems of partial differential equations in applied mathematics, Part 1 (Santa Fe, N.M., 1984)}, 293-326.
\newblock Lectures in Appl. Math., 23, Amer. Math. Soc., Providence, RI, 1986.


\bibitem{Lannes}
Lannes, D.
\newblock Well-posedness of the water waves equations.
\newblock {\em J. Amer. Math. Soc.} 18 (2005), no. 3, 605-654.


\bibitem{LannesBook}
Lannes, D.
\newblock The water waves problem. Mathematical analysis and asymptotics.
\newblock Mathematical Surveys and Monographs, Vol. 188. American Mathematical Society, Providence, RI, 2013. xx+321 pp.



\bibitem{Lindblad}
Lindblad, H.
\newblock Well-posedness for the motion of an incompressible liquid with free surface boundary.
\newblock {\em Ann. of Math.} 162 (2005), no. 1, 109-194.



\bibitem{Nalimov}
Nalimov, V. I.
\newblock The Cauchy-Poisson problem.
\newblock {\em Dinamika Splosn. Sredy Vyp.} 18 Dinamika Zidkost. so Svobod. Granicami (1974), 10-210, 254.




\bibitem{BosonStar}
Pusateri, F.
\newblock Modified scattering for the Boson Star equation.
\newblock {\em Comm. Math. Phys.} 332 (2014), no. 3, 1203-1234. 


\bibitem{SW1}
Schneider, G. and Wayne, C.E.
\newblock The long wave limit for the water wave problem. I. the case of zero surface tension.
\newblock {\em Comm. Pure Appl. Math.} 53 (2000), no. 12, 1475-1535.


\bibitem{shatahKGE}
Shatah, J.
\newblock Normal forms and quadratic nonlinear Klein-Gordon equations.
\newblock {\em  Comm. Pure Appl. Math.} 38 (1985), no. 5, 685-696.


\bibitem{ShZ1}
Shatah, J. and Zeng, C.
\newblock Geometry and a priori estimates for free boundary problems of the Euler equation.
\newblock {\em Comm. Pure Appl. Math.} 61 (2008), no. 5, 698-744.




\bibitem{ShZ3}
Shatah, J. and Zeng, C.
\newblock Local well-posedness for the fluid interface problem.
\newblock {\em Arch. Ration. Mech. Anal.} 199 (2011), no. 2, 653-705.


\bibitem{SulemBook}
Sulem, C. and Sulem, P.L.
\newblock The nonlinear Schr\"{o}dinger equation. Self-focussing and wave collapse.
\newblock {\em Applied Mathematical Sciences}, 139.
\newblock Springer-Verlag, New York, 1999.


\bibitem{WuNLS} 
Totz, N and Wu, S.
\newblock A rigorous justification of the modulation approximation to the 2D full water wave problem.
\newblock {\em Comm. Math. Phys.} 310 (2012), no. 3, 817-883.


\bibitem{Totz} 
Totz, N.
\newblock A Justification of the Modulation Approximation to the 3D Full Water Wave Problem.
\newblock {\em arXiv:1309.5995} (2013).


\bibitem{Yosi}
Yosihara, H.
\newblock Gravity waves on the free surface of an incompressible perfect fluid of finite depth.
\newblock {\em Publ. Res. Inst. Math. Sci.} 18 (1982), 49-96.


\bibitem{Wu1}
Wu, S.
\newblock Well-posedness in Sobolev spaces of the full water wave problem in 2-D.
\newblock {\em Invent. Math.}, 130 (1997), 39-72.


\bibitem{Wu2}
Wu, S.
\newblock Well-posedness in Sobolev spaces of the full water wave problem in 3-D.
\newblock {\em J. Amer. Math. Soc.} 12 (1999), 445-495.


\bibitem{WuAG}
Wu, S.
\newblock Almost global wellposedness of the 2-D full water wave problem.
\newblock {\em Invent. Math.} 177 (2009), 45-135.


\bibitem{Wu3DWW}
Wu, S.
\newblock Global wellposedness of the 3-D full water wave problem.
\newblock {\em Invent. Math.} 184 (2011), 125-220.




\end{thebibliography}
\end{document}